\documentclass{colt2022} 

\usepackage{enumitem}

\usepackage{booktabs}       

\hypersetup{
    colorlinks=true
}

\newcommand{\pb}[1]{{\color{magenta} #1}}

\newcommand{\ysc}[1]{{\color{blue} #1}}

\usepackage{tikz}
\usetikzlibrary{graphs}

\usepackage{thmtools}
\usepackage{mathtools}
\usepackage{thm-restate}
\usepackage{mathrsfs}
\usepackage{tensor}
\usepackage{bm}
\usepackage{lineno} 
\usepackage{pgfplotstable}
\pgfplotsset{compat=1.16}

\usepackage{cleveref}
\Crefname{assumption}{Assumption}{Assumptions}
\usepackage{soul}

\usepackage{xifthen}
\renewenvironment{proof}[1][]
{
    \begin{newproof}
    \ifthenelse{\equal{#1}{}}{ }{\textbf{(#1) }}%
}{
    \end{newproof}
}

\makeatletter
\def\cleartheorem#1{%
    \expandafter\let\csname#1\endcsname\relax
    \expandafter\let\csname c@#1\endcsname\relax
}
\makeatother
\cleartheorem{proposition}
\declaretheorem[name=Proposition,sibling=theorem]{proposition}
\declaretheorem[name=Assumption,numberwithin=section]{assumption}

\usepackage{xr,refcount}
\usepackage{tcolorbox}
\usepackage{bold-preambule}
\usepackage{math-preambule}
\usepackage{constants}
\DeclareMathOperator{\Rem}{Rem}
\def\eps{\varepsilon}

\numberwithin{equation}{section}

\title
[Derivatives and residual distribution of regularized M-estimators]
{Derivatives and residual distribution of regularized M-estimators 
with application to adaptive tuning}
\usepackage{times}

\coltauthor{
    \Name{Pierre C.~Bellec} 
    \Email{pierre.bellec@rutgers.edu}\and
    \Name{Yiwei Shen} 
    \Email{ys573@stat.rutgers.edu}\\
\addr Department of Statistics, Rutgers University, Piscataway, NJ 08854, USA.}

\begin{document}

\maketitle

\begin{abstract}
    This paper studies M-estimators with gradient-Lipschitz loss function
    regularized with convex penalty
    in linear models with Gaussian design matrix
    and arbitrary noise distribution.
    A practical example is the robust M-estimator constructed with
    the Huber loss and the Elastic-Net penalty 
    and the noise distribution has heavy-tails.
    Our main contributions are three-fold.
    (i)
    We provide general formulae for the derivatives of regularized
    M-estimators $\hbbeta(\by,\bX)$ where differentiation is
    taken with respect to both $\by$ and $\bX$; this reveals
    a simple differentiability structure shared by all convex regularized
    M-estimators.
    (ii)
    Using these derivatives, we characterize the distribution
    of the residual $r_i  = y_i-\bx_i^\top\hbbeta$
    in the intermediate high-dimensional regime where dimension and sample
    size are of the same order.
    (iii)
    Motivated by the distribution of the residuals, we propose
    a novel adaptive criterion to select tuning parameters
    of regularized M-estimators. The criterion approximates
    the out-of-sample error up to an additive constant independent 
    of the estimator,
    so that minimizing the criterion provides a proxy for minimizing the
    out-of-sample error.
    The proposed adaptive criterion does not require the knowledge
    of the noise distribution or of
    the covariance of the design.
    Simulated data confirms the theoretical findings, regarding
    both the distribution of the residuals and the success
    of the criterion as a proxy of the out-of-sample error.
    Finally our results reveal new relationships between the derivatives
    of $\hbbeta(\by,\bX)$ and the effective degrees of freedom
    of the M-estimator, which are of independent interest.


\end{abstract}

\begin{keywords}
    Robust estimation,
    M-estimator,
    Adaptive tuning,
    High-dimensional statistics,
    Residual distribution.
\end{keywords}

\section{Introduction}

This paper studies properties
of robust estimators in linear models
$
\by = \bX \bbeta^* + \bep
$
with response $\by\in\R^n$, unknown regression vector $\bbeta^*$ and
$\bX$ is a design matrix with $n$ rows $\bx_1,...,\bx_n$.
Each row $\bx_i$ being a high-dimensional feature vector in $\R^p$, centered and normally distributed 
with covariance $\bSigma$, and each $\varepsilon_i$ is independent of $\bX$ with continuous distribution.
Throughout, let $\hbbeta=\hbbeta(\by,\bX)$ be a regularized $M$-estimator given 
as a solution of the convex minimization problem
\begin{equation}
    \hbbeta(\by,\bX) = \argmin_{\bb \in \R^p}
    \frac 1 n\sum_{i=1}^n\rho(y_i - \bx_i^\top \bb) + g(\bb)
    \label{hbeta}
\end{equation}
where $\rho:\R\to\R$ is a convex data-fitting loss function
and $g:\R^p\to\R$ a convex penalty. 
We may write $\hbbeta_{\rho,g}(\by,\bX)$ for \eqref{hbeta}
to emphasize the dependence on the loss-penalty pair $(\rho, g)$;
if the argument $(\by,\bX)$ is dropped then $\hbbeta$ is implicitly understood
at the observed that $(\by,\bX)$.
Typical examples of losses include the square loss $\rho(u)=u^2/2$,
the Huber loss $H(u)=\int_0^{|u|}\min(1,t)dt$
or its scaled version $\rho = \Lambda^{2} H( u/ \Lambda)$ 
for some tuning parameter $\Lambda>0$,
while typical examples of penalty functions include
the Elastic-Net $g(\bb) = \lambda \|\bb\|_1 + \mu\|\bb\|^2/2$
for tuning parameters $\lambda,\mu \ge 0$.

The paper introduces the following criterion to select
a loss-penalty pair $(\rho,g)$ with small out-of-sample error
$\|\bSigma^{1/2}(\hbbeta-\bbeta^*)\|^2$: for a given set of candidate
loss-penalty pairs $\{(\rho, g)\}$ and the corresponding $M$-estimator
$\hbbeta_{\rho,g}$ in \eqref{hbeta}, select the pair $(\rho,g)$ 
that minimizes the criterion
\begin{equation}
\text{Crit}(\rho,g)
= \Big\| \br + 
\frac{\df}{\trace[\bV]}
\psi\bigl(\br\bigr)\Big\|^2
~
\text{ with }
\begin{cases}
    \br  = \by-\bX\hbbeta_{\rho,g} &\in\R^n,\\
    \df = \trace[\bX(\partial/\partial \by)\hbbeta_{\rho,g}] &\in\R, \\
    \bV = \diag\{\psi'(\br)\}(\bI_n-\bX(\partial/\partial \by)\hbbeta_{\rho,g})
      &\in\R^{n\times n}
\end{cases}
\label{crit}
\end{equation}
where $\trace[\cdot]$ is the trace,
$\psi:\R\to\R$ is the derivative of $\rho$,
    $\psi'$ the derivative of $\psi$ and
we extend $\psi$ and $\psi'$ to functions $\R^n\to\R^n$ by
componentwise application
of the univariate function of the same symbol.
Above,
$(\partial/\partial \by) \hbbeta_{\rho,g}\in\R^{p\times n}$
denotes the Jacobian of
\eqref{hbeta} with respect to $\by$ for $\bX$ fixed,
at the observed data $(\by,\bX)$.
As we will see while studying particular examples, for pairs $(\rho,g)$  
commonly used in robust high-dimensional
statistics such as the square loss, Huber loss with the $\ell_1$-penalty
or Elastic-Net penalty, the ratio $\df/\trace[\bV]$ in
\eqref{crit} admits simple, closed-form expressions.
The criterion \eqref{crit} has an appealing adaptivity property:
it does not require any knowledge of the noise $\bep$ or its distribution,
nor any knowledge of the covariance $\bSigma$ of the design.

\begin{figure}[ht]
    \centering
    \includegraphics[width=46mm]{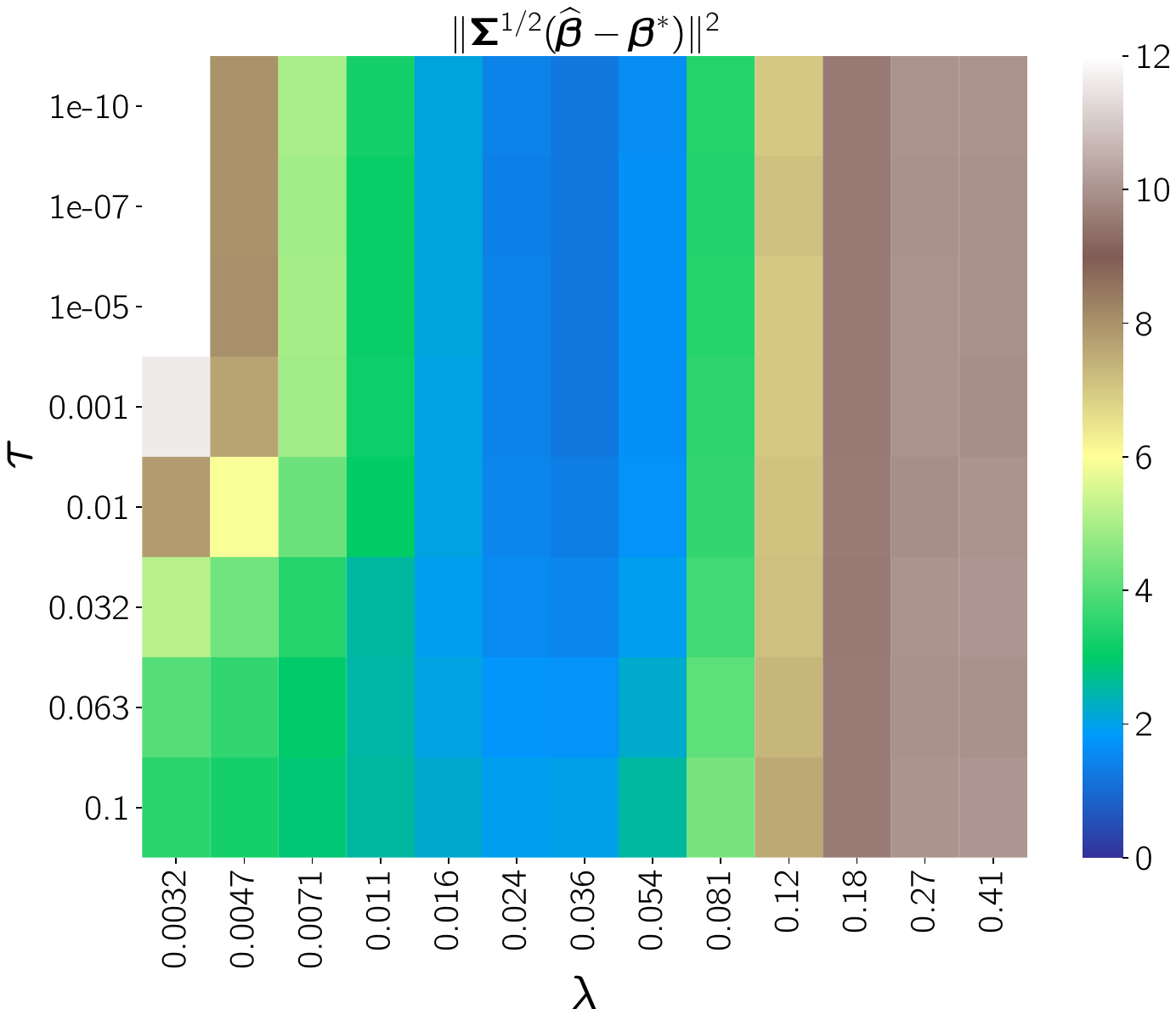}   
    \includegraphics[width=46mm]{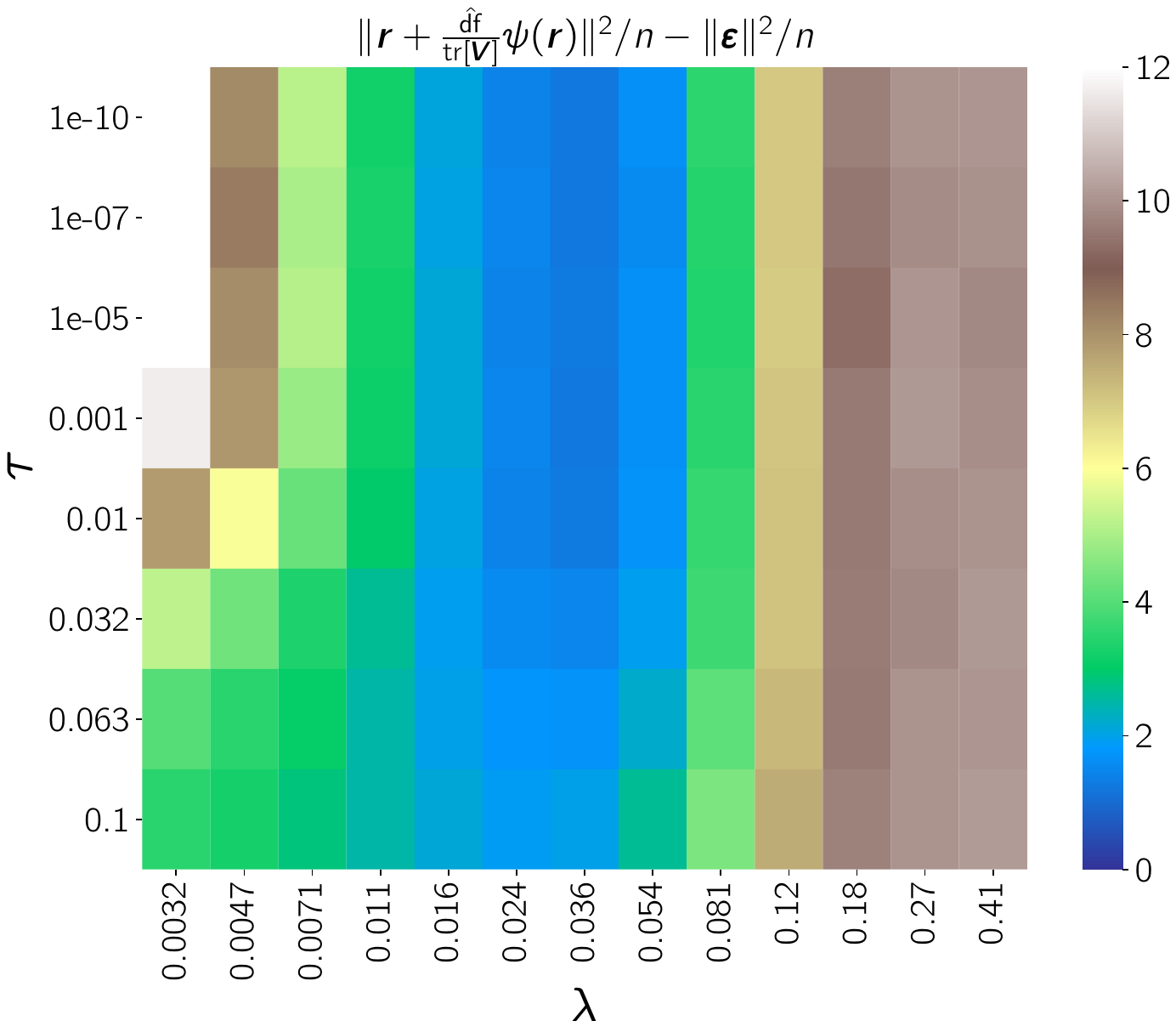}   
    \includegraphics[width=46mm]{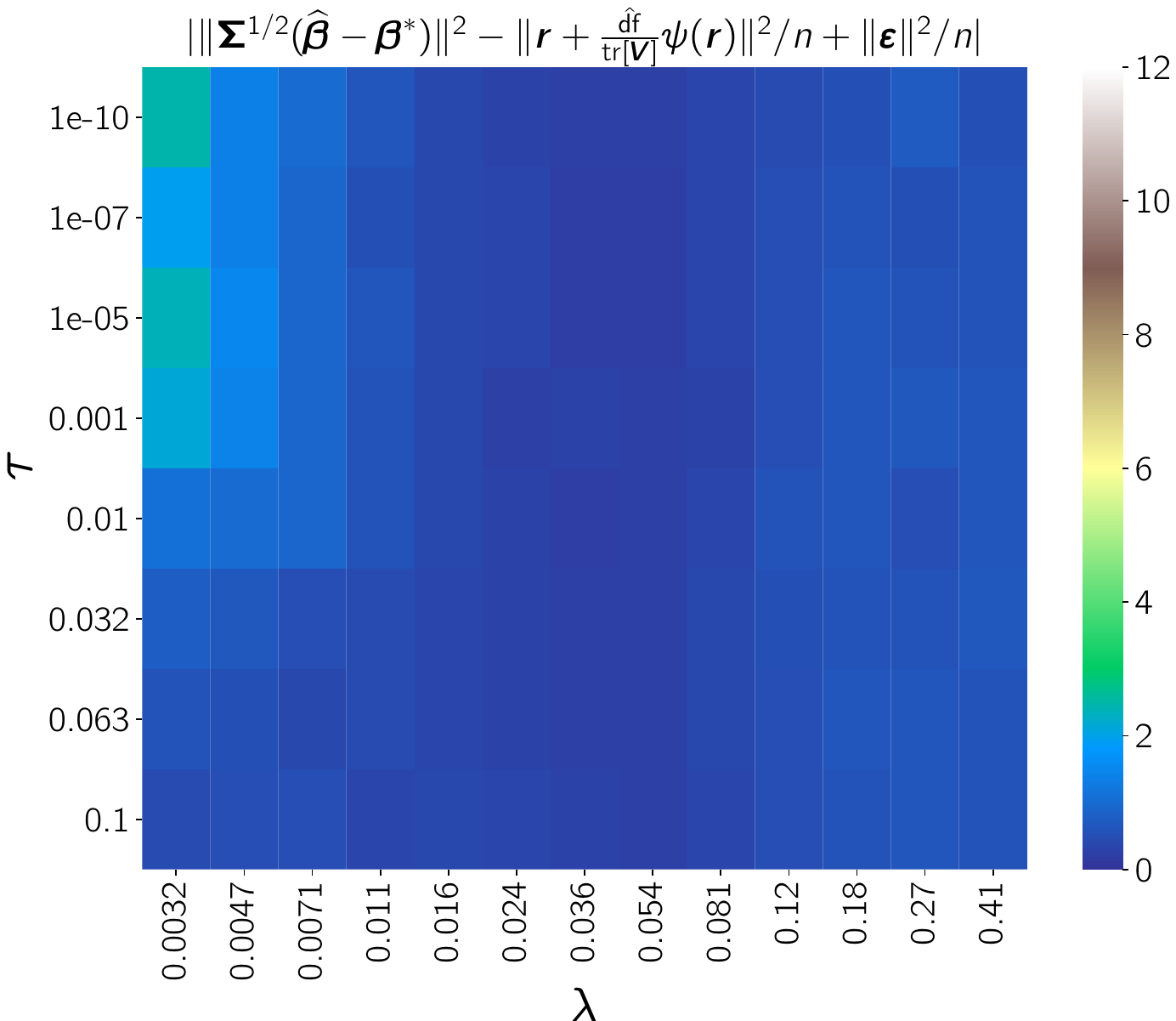}   
    \caption{
    Heatmaps for $\|\bSigma^{1/2} (\hat\bbeta - \bbeta^*)\|^{2}$, 
    its approximation 
        $\|\br+({\df}/{\trace[\bV]})\psi(\br)\|^{2}/n-\|\bep\|^{2}/n$ 
    and the approximation error 
        $|\|\bSigma^{1/2} (\hat\bbeta - \bbeta^*)\|^{2} - \| \br 
        + ({\df}/{\trace[\bV]}) \psi (\br) \|^{2} / n - \| \bep \|^{2} / n|$ 
        for the Huber loss and Elastic-Net penalty 
        on a grid of tuning parameters  $(\lambda, \tau)$
    where 
    $\lambda \in [0.0032, 0.41]$
    and 
    $\tau \in [10^{-10}, 0.1]$.
    Each cell is the average over 100 repetitions.
    See \Cref{sec:simulations} for more details.
    }
    \label{fig:out-of-sample}
\end{figure}

\subsection{Contributions}
\begin{enumerate}[leftmargin=0.5cm]
    \item
The end goal of this paper is to provide theoretical justification and 
theoretical guarantees for the criterion
\eqref{crit} in the high-dimensional regime 
    where the ratio $p/n$ has a finite limit 
    and $\bX$ has anisotropic Gaussian distribution.
The theoretical results will justify the approximation
\begin{equation}
\big\| \br + 
\bigl(\df/\trace[\bV]\bigr)
\psi\bigl(\br\bigr)\big\|^2/n
\approx
\|\bep\|^2/n
+
\|\bSigma^{1/2}(\hbbeta-\bbeta^*)\|^2.
\label{eq:approx-intro}
\end{equation}
\end{enumerate}
\Cref{fig:out-of-sample} illustrates the accuracy
of \eqref{eq:approx-intro} on simulated data.
To study the criterion \eqref{crit} and derive the 
approximation \eqref{eq:approx-intro},
we develop novel results of independent interest regarding
$M$-estimators in \eqref{hbeta}:
\begin{enumerate}[leftmargin=0.5cm]
    \setcounter{enumi}{1}
    \item 
    The paper derives general formula for the derivatives
    $(\partial/\partial y_i) \hbbeta$ and $(\partial/\partial x_{ij}) \hbbeta$.
    This sheds light on the differentiability structure of $M$-estimators
    for general loss-penalty pairs: for any $\rho,g$ with $g$ strongly convex,
    there exists $\hbA\in\R^{p\times p}$ depending on $(\by,\bX)$ such that
    for almost every $(\by,\bX)$,
    \begin{equation*}
        (\partial/\partial y_i)
        \hbbeta(\by,\bX)
        =  \hbA \bX^\top \be_i \psi'(r_i), 
        \quad
        (\partial/\partial x_{ij})
        \hbbeta(\by,\bX)
        =   \hbA \be_j \psi(r_i)
        - \hbA \bX^\top \be_i \psi'(r_i)  \hbeta_j,
    \end{equation*}
    for $r_i = y_{i} - \bx_{i}^{\top} \hat \bbeta$,
    $\forall i\in[n]$,$j\in[p]$ where $\be_j\in\R^p$ and $\be_i\in\R^n$ are canonical basis vectors.
    \item
    The paper obtains a stochastic representation for the
    residual $y_i - \bx_i^\top\hbbeta$ for some fixed $i=1,...,n$,
    extending some results of \cite{karoui2013robust} on unregularized $M$-estimators to penalized ones as in \eqref{hbeta}.
    In short, for each $i=1,...,n$ the $i$-th residual satisfies $r_i = y_i - \bx_i^\top\hbbeta$
    \begin{equation}
    r_i + \tfrac{\df}{\trace[\bV]} \psi(r_i) \approx 
    \eps_i + Z_i \|\bSigma^{1/2}(\hbbeta-\bbeta^*)\|
    \end{equation}
    where $Z_i\sim N(0,1)$ is independent of $\eps_i$.
    This stochastic representation is the motivation for the criterion
    \eqref{crit} as the amplitude of the normal part
    in the right-hand side
    is proportional to the out-of-sample error 
    $\|\bSigma^{1/2}(\hbbeta-\bbeta^*)\|$ that we wish to minimize,
    while the variance of the noise $\eps_i$ does not depend
    on the choice of $(\rho,g)$.
\end{enumerate}

Simulated data in \Cref{fig:qq} confirms that the stochastic representation
for the $i$-th residual $r_i = y_i-\bx_i^\top\hbbeta$ is accurate.
Our working assumption throughout the paper is the following.

\begin{assumption}
    \label{assumMain}
    For constants $\gamma,\mu>0$ 
    independent of $n,p$ we have $p/n\le \gamma$,
    the loss $\rho:\R\to\R$ is convex with a unique minimizer at 0, 
    continuously differentiable
    and its derivative $\psi=\rho'$ is 1-Lipschitz.
    The design matrix $\bX$ has iid $N(\bzero,\bSigma)$ rows for some
    invertible covariance $\bSigma$ and the noise $\bep$ is independent
    of $\bX$ with continuous distribution.
    The penalty $g:\R^p\to\R$ is {$\mu$-strongly} convex w.r.t. $\bSigma$ in the sense that
    $\bb \mapsto g(\bb) - (\mu/2) \bb^\top\bSigma \bb$ is convex in $\bb\in\R^p$.
\end{assumption}

Throughout the paper, we consider a sequence (say, indexed by $n$)
of regression problems with $p$, $\bbeta^*$, $\bSigma$ and the loss-penalty
pair $(\rho,g)$ depending implicitly on $n$.
For some deterministic sequence $(a_n)$,
the stochastically bounded notation $O_P(a_n)$ in this context
may hide constants depending on $\gamma,\mu$ only, that is,
$O_P(a_n)$ denotes a sequence of random variables $W_n$ such that
for any $\epsilon>0$ there exists $K$ depending on $(\epsilon,\gamma,\mu)$
satisfying $\P(|W_n|\ge K a_n ) \le \epsilon$.

Since \Cref{assumMain} requires $p/n\le \gamma$, the Bolzano-Weierstrass theorem lets us extract a subsequence of regression problems 
such that $p/n\to\gamma'$
along this subsequence, for some constant $\gamma'$.
This is the asymptotic regime  we have in mind throughout the paper,
although our results do not require a specific limit for the ratio $p/n$.
For some results, we will require the following additional assumption
which is satisfied by
robust loss functions and penalty that shrink towards 0.

\begin{assumption}
    \label{assumAdditional}
    The penalty is minimized at $\mathbf 0$,
    that is, $g(\mathbf{0}) = \min_{\bb\in\R^p} g(\bb)$;
    the loss is Lipschitz as in
    $|\psi|\le M$ for some constant $M$ independent of $n,p$;
    the signal is bounded as in
    $\|\bSigma^{1/2}\bbeta^*\|^2\le M$.
\end{assumption}

\subsection{Related works}
\label{sec:related}

The context of the present work is the study of $M$-estimators
in the regime $\frac pn$ has a finite limit. This literature pioneered
in \cite{bayati2012lasso,karoui2013robust,donoho2016high,stojnic2013framework}
typically describes the subtle behavior of $\hbbeta$ in this regime by solving a system of nonlinear equations.
This system 
depends
on a prior distribution for the components of $\bbeta^*$,
and either depends on the covariance $\bSigma$ 
\citep{celentano2020lasso,dobriban2018high}
or assume $\bSigma=\bI_p$
\cite[among many others]{bayati2012lasso,thrampoulidis2018precise,celentano2019fundamental}.
Solutions to the nonlinear system are a powerful tool to understand
$\hbbeta$ in theory, e.g., to characterize the deterministic
limit of $\|\bSigma^{1/2}(\hbbeta-\bbeta^*)\|$, see e.g.,
the general results in 
\cite{celentano2019fundamental} for the square loss
and \cite{thrampoulidis2018precise} for general loss-penalty pairs.
However, since the system
and its solution depend on unobservable quantities
($\bSigma$ and prior on $\bbeta^*$), the system solution is not directly usable
for practical purposes such as parameter tuning.

The present work distinguishes itself from most of this literature
as the goal is to describe the behavior of $\hbbeta$
using observable quantities that only depend on the data $(\by,\bX)$
(and not unobservable ones such as $\bSigma$ or a prior 
distribution on $\bbeta^*$ that appear in the aforementioned nonlinear system
of equations). As we will see this view lets us perform
adaptive tuning of parameters in a fully adaptive manner
using the criterion \eqref{crit}.
The criterion \eqref{crit} appeared in previous works for the square loss
only: \cite{bayati2013estimating,miolane2018distribution} studied
\eqref{crit} for the Lasso with $\bSigma=\bI_p$
and \cite[Section 3]{bellec2020out_of_sample} for the square loss and general
penalty (note that for the square loss $\rho(u)=u^2/2$,
\eqref{crit} reduces to $n^{2}\|\br\|^2/(n-\df)^2$
due to $\psi(u)=u$ and $\trace[\bV]=n-\df$.
The property $\psi(u)=u$ of the square loss
hides the subtle interplay between $\br,\psi(\br),\df$ and $\trace[\bV]$
in \eqref{crit} for $\rho$ different than the square loss).

A criterion
different from \eqref{crit} is studied
in \cite{bayati2013estimating,miolane2018distribution} (for the Lasso and
$\bSigma=\bI_p$)
and \cite{bellec2020out_of_sample} (for general loss-penalty pairs),
with the purpose of estimating the out-of-sample error
$\|\bSigma^{1/2}(\hbbeta - \bbeta^*)\|^2$. In the case of an M-estimator
and with the notation in \eqref{crit}, this criterion is the right-hand side
of the approximation
$$
\|\bSigma^{1/2}(\hbbeta-\bbeta^*)\|^2
\approx
\trace[\bV]^{-2}\bigl(
\|\psi(\br)\|^2(2\df - p)
+ \|\bSigma^{-1/2}\bX^\top \psi(\br)\|^2
\bigr).
$$
That criterion from 
\cite{bayati2013estimating,miolane2018distribution,bellec2020out_of_sample}
has the drawback to require the knowledge of the covariance $\bSigma$,
and is thus not readily usable unless $\bSigma$ is known or can be consistently estimated. On the other hand, the criterion \eqref{crit} is fully adaptive:
it does not depend on $\bSigma$, and can thus be used even if $\bSigma$ cannot
itself be consistently estimated.
Another line of work \citep{rad2020error,xu2021consistent,rad2020scalable}
proposed the ALO criterion
\begin{equation}
    \sum_{i=1}^n \Bigl(r_i + \frac{H_{ii}}{V_{ii}} \psi(r_i)\Bigr)^2
\label{ALO}
\end{equation}
(when specialized to linear models),
where $\bV$ is the matrix defined in \eqref{crit} and
$H_{ii} = \bx_i^\top \frac{\partial}{\partial y_i} \hbbeta(\by,\bX)$
in the notation of the present paper.
This criterion differs from \eqref{crit} proposed in the present paper,
in that \eqref{crit}
replaces $H_{ii}$ and $V_{ii}$ by their respective averages, $\df/n=\frac1n \sum_{i=1}^n H_{ii}$ and $\trace[\bV]/n = \frac1n\sum_{i=1}^n V_{ii}$.
This extra averaging step lets us prove, for non-smooth penalty functions,
theoretical guarantees for selecting a loss-penalty pair by minimizing
the criterion \eqref{crit} (cf. \Cref{sec:proxy,sec:edf} below). On the other hand, we are not aware
of similar theoretical guarantees for selecting a loss-penalty pair
based on \eqref{ALO}, since the theoretical analysis of \eqref{ALO}
is so far restricted to twice continuously differentiable loss and penalty
functions, with a uniform upper bound on the Lipschitz constant of the
second derivatives \cite[Assumption 6 required in Theorem 3 and Corollary 1]{rad2020scalable}. This rules
out the Elastic-Net and other non-smooth penalty functions typically used
for high-dimensional data, as well
as the Huber loss which is not twice continuously differentiable.
The criterion \eqref{crit} of the present paper thus improves
upon \eqref{ALO} since it enjoys theoretical guarantees for non-smooth
penalty functions and the Huber loss.

This work leverages probabilistic results on
functions of standard normal random variables
\cite{bellec_zhang2019second_poincare}\cite[\S6, \S7]{bellec2020out_of_sample}
which are consequences of Stein's formula \cite{stein1981estimation}.
Consequently, the main limitation of our work is that it currently requires
Gaussian design for the probabilistic results,
although simulations in \Cref{sec:rademacher-figures}
    suggest that the results hold for more general distributions,
    including design with Rademacher entries.
On the other hand,
the differentiability result \eqref{eq:differentiability-formulae}
is deterministic and does not rely on any probabilistic assumption.

\section{Differentiability of regularized M-estimators}
\label{sec:Differentiability}

The first step towards the study of the criterion \eqref{crit}
is to justify the almost sure existence of the derivatives of $\hbbeta$
that appear in \eqref{crit} through the scalar
scalar $\df$ and the matrix $\bV$ in \eqref{crit}.
Although the criterion \eqref{crit} only involves the derivatives 
of $\hbbeta(\by,\bX)$ with respect to $\by$ for a fixed $\bX$, the proof
of our results rely on the interplay between the derivatives
with respect to $\by$ and with respect to $\bX$:
this \emph{differentiability structure} of $M$-estimators is
the content of the following result.

\begin{restatable}{theorem}{theoremDifferentiability}
    \label{thm:differentiability}
    Let \Cref{assumMain} be fulfilled.
    For almost every $(\by,\bX)$ the map $(\by,\bX)\mapsto \hbbeta(\by,\bX)$
    is differentiable at $(\by,\bX)$ and there exists a matrix $\hbA\in\R^{p\times p}$
    depending on $(\bX, \by)$
    with $\|\bSigma^{1/2}\hbA\bSigma^{1/2}\|_{op} \le (n\mu)^{-1}$
    s.t.
    \begin{equation}
        \begin{split}
        (\partial/\partial y_i)
        \hbbeta(\by,\bX)
        &=  \hbA \bX^\top \be_i\psi'(r_i), 
        \\
        (\partial/\partial x_{ij})
        \hbbeta(\by,\bX)
        &=   \hbA \be_j \psi(r_i)
        - \hbA \bX^\top \be_i \psi'(r_i)  \hbeta_j,
        \end{split}
        \qquad
        \text{ where }
        r_i = y_{i} - \bx_{i}^{\top} \hat \bbeta,
        \label{eq:differentiability-formulae}
    \end{equation}
    $\be_i\in\R^n, \be_j\in\R^p$ are canonical basis vectors
    , $\psi := \rho'$ and $\psi'$ denote the derivatives.
    Furthermore,
\begin{align}
    \df = \trace [ \bX ( \partial / \partial \by ) \hat\bbeta ]
        &= \trace[\bX\hbA\bX^{\top}\diag\{\psi'(\br)\}],
        \label{df}
    \\
    \bV =  \diag\{\psi'(\br)\} (\bI_n - \bX(\partial/\partial \by)\hbbeta)
        &=
        \diag\{\psi'(\br)\} - \diag\{\psi'(\br)\}\bX\hbA\bX^{\top}\diag\{\psi'(\br)\}
        \label{V}
\end{align}
    satisfy $0\le\df\le n$
    and $0\le \trace[\bV] \le n$.

\end{restatable}

    Since the same matrix $\hbA$ appears in both the derivatives with respect
    to $y_i$ and to $x_{ij}$,
    \eqref{eq:differentiability-formulae} provides 
    relationship between $(\partial/\partial y_i)\hbbeta$
    and $(\partial/\partial x_{ij})\hbbeta$, for instance
    $(\partial/\partial x_{ij})
    \hbbeta = \hbA \be_j \psi(r_i) - \hbeta_j (\partial/\partial y_i)\hbbeta$.
    Although the matrix $\hbA$ is not explicit for arbitrary loss-penalty
    pair, closed-form expressions are available for particular
    examples such as the Elastic-Net penalty as discussed 
    in \Cref{sec:simulations}.
\begin{remark}
    For the square loss $\rho(u)=u^2/2$,
    the differentiability formulae \eqref{eq:differentiability-formulae}
    reduce to
    \begin{equation}
        \begin{split}
        (\partial/\partial y_l) \hbbeta (\by,\bX)
        &= \hbA\bX^\top \be_l,
        \\
        (\partial/\partial x_{ij}) \hbbeta (\by,\bX)
        &= \hbA \be_j (y_i - \bx_i^\top\hbbeta) - \hbA\bX^\top \be_i \hbeta_j
        \end{split}
    \end{equation}
    for almost every $(\by,\bX)$ and some matrix $\hbA\in\R^{p\times p}$
    depending on $(\by,\bX)$, since in this case
    $\psi' = 1$.
\end{remark}

    In the simple case where $g$ is twice continuously differentiable,
    \eqref{eq:differentiability-formulae} follows 
    \citep{bellec_zhang2019second_poincare}
    with 
    \begin{equation}
    \hbA=
    \bigl(\bX^\top\diag\{\psi'(\br)\}\bX + n \nabla^2 g(\hbbeta)\bigr)^{-1}
    \label{A-twice-diff}
    \end{equation}
    by differentiating the KKT conditions
    $\bX^\top\psi(\by-\bX\hbbeta) = n \nabla g(\hbbeta)$.
    To illustrate why this is true, provided that $\hbbeta(\by,\bX)$ is differentiable,
    if $(\by(t),\bX(t))$ are smooth perturbations of $(\by,\bX)$
    with $(\by(0),\bX(0))=(\by,\bX)$ and $\frac{d}{dt}(\by(t),\bX(t))|_{t=0}
    =(\dot\by,\dot\bX)$, differentiation of
    $\bX(t){}^\top\psi(\by(t)-\bX(t)\hbbeta(\by(t),\bX(t))) = n \nabla g(\hbbeta(\by(t),\bX(t)))$ at $t=0$ and the chain rule yields
    $$\dot\bX{}^\top \psi(\br) - \bX^\top\diag\{ \psi'(\br) \}
    (\dot\by - \dot\bX\hbbeta(\by,\bX))
    =\hbA{}^{-1}
    \tfrac{d}{dt}\hbbeta(\by(t),\bX(t))\big|_{t=0}
    $$
    with $\hbA$ in \eqref{A-twice-diff}.
    This gives \eqref{eq:differentiability-formulae}
    if the penalty $g$ is twice-differentiable.
    \Cref{thm:differentiability} reveals that for \emph{arbitrary}
    convex penalty functions including non-differentiable ones,
    the differentiability structure
    \eqref{eq:differentiability-formulae} always holds, as in 
    the case of twice differentiable penalty $g$,
    even for penalty functions such as
    $g(\bb) = \mu \|\bb\|^2/2 + \lambda\|\text{mat}(\bb)\|_{\rm nuc}$
    where $\text{mat}:\R^p\to\R^{d_1\times d_2}$ is a linear
    isomorphism to the space of $d_1\times d_2$ matrices
    and $\|\cdot\|_{\rm nuc}$ is the nuclear norm:
    in this case by \Cref{thm:differentiability}
    there exists a matrix $\hbA\in\R^{p\times p}$ such that
    \eqref{eq:differentiability-formulae} holds 
    although no closed-form expression for $\hbA$ is known.

    The representation \eqref{eq:differentiability-formulae}
    is a powerful tool as it provides explicit derivatives
    of quantities of interest such as
    $\br=\by-\bX\hbbeta$, $\|\psi(\br)\|^2$ or
    $\|\bSigma^{1/2}(\hbbeta-\bbeta^*)\|^2$.
    These explicit derivatives
    can then be used in probabilistic identities and inequalities
    that involve derivatives, for instance Stein's formulae
    \citep{stein1981estimation}, the
    Gaussian Poincar\'e inequalty \cite[Theorem 3.20]{boucheron2013concentration},
    or normal approximations \citep{chatterjee2009fluctuations,bellec_zhang2019second_poincare}.


\begin{remark}
    \label{rem:intercept}
    Similar derivative formulae hold if an intercept is included in the minimization, as in
    \begin{equation}
        \bigl(\hbeta_0(\by,\bX), ~\hbbeta(\by,\bX)\bigr) 
        = \argmin_{b_0\in\R, \bb\in\R^p}
        \frac 1 n\sum_{i=1}^n\rho(y_i - b_0 - \bx_i^\top \bb) + g(\bb)
    \end{equation}
    Let \Cref{assumMain} be fulfilled,
    and assume further $ 
        \| \psi' (\br)\|_{2}
        > 0 
    $ with 
    $\br := \by - \bone_{n} \hat\beta_0 - \bx_i^\top \hbbeta$ 
    where  $\bone_n=(1,...,1)^\top\in\R^n$.
    For almost every $(\by,\bX)$ the map $(\by,\bX)\mapsto \hbbeta(\by,\bX)$
    is differentiable at $(\by,\bX)$,
    and there exists $\hbA \in \R^{p\times p}$
    depending on $(\by,\bX)$
    with $\|\bSigma^{1/2} \hbA \bSigma^{1/2}\|_{op} \le (n\mu)^{-1}$
    such that
    \begin{equation}
        (\partial/\partial y_i)
        \hbbeta(\by,\bX)
        = \hbA \bX^\top \bPsi' \be_i,
        \quad
        (\partial/\partial x_{ij})
        \hbbeta(\by,\bX)
        = \hbA \be_j \psi(r_i)
        - \hbA \bX^\top \bPsi' \be_i \hbeta_j
        ,
    \end{equation}
    where $\be_i \in\R^n, \be_j\in\R^p$ are canonical basis vectors,
    $\psi = \rho'$ and $ \bPsi' := \diag\{\psi' (\br)\} - \psi'(\br) \psi'(\br)^{\top} / \sum_{i \in [n]} \psi'(r_i)$.
\end{remark}

\section{Distribution of individual residuals}
\label{sec:residual}

We now turn to the distribution of a single residual $r_i = y_i - \bx_i^\top\hbbeta$ for some fixed observation $i\in\{1,...,n\}$ (for instance, fix $i=1$).
By leveraging the differentiability structure \eqref{eq:differentiability-formulae} and the normal approximation from \cite{bellec_zhang2019second_poincare},
the following result provides a clear picture of the distribution
of $r_i$.

\begin{theorem}
    \label{thm:normality}
    Let \Cref{assumMain} be fulfilled
    and let $\hbA\in\R^{p\times p}$ be given by \Cref{thm:differentiability}.
    Then for every $i=1,...,n$ there exists $Z_i\sim N(0,1)$ 
    independent of $\varepsilon_i$ such that
    \begin{equation}
        \Big|
        \Bigl(r_{i} + \trace[\bSigma \hbA] \psi(r_{i})\Bigr)
        - \Bigl(\eps_i + \|\bSigma^{1/2}(\hbbeta-\bbeta^*)\| Z_i\Bigr)
        \Big|
        \le O_{P}(n^{-1/4})
        (
        |\psi(\eps_i)| + \|\bSigma^{1/2}(\hbbeta-\bbeta^*)\|
        )
        \label{eq:Rem_i}
    \end{equation}
    Furthermore, if $\eps_i$ has a fixed distribution
    $F$, there exists a bivariate variable $(\tilde\eps_i^n,\tilde Z_i^n)$
    converging in distribution to the product measure
    $F\otimes N(0,1)$ such that
    \begin{equation}
        r_{i} + \trace[\bSigma \hbA] \psi(r_{i})
        = \tilde \eps_i^n + \|\bSigma^{1/2}(\hbbeta-\bbeta^*)\| \tilde Z_i^n.
        \label{eq:representation-tilde}
    \end{equation}
    If $\eps_i$ has a fixed distribution
    $F$ and \Cref{assumAdditional} holds
    then 
    $|\psi(\eps_i)| + \|\bSigma^{1/2}(\hbbeta-\bbeta^*)\|
    = O_P(1)$.
\end{theorem}

\begin{figure}[ht]
\centering
\includegraphics[width=34mm]{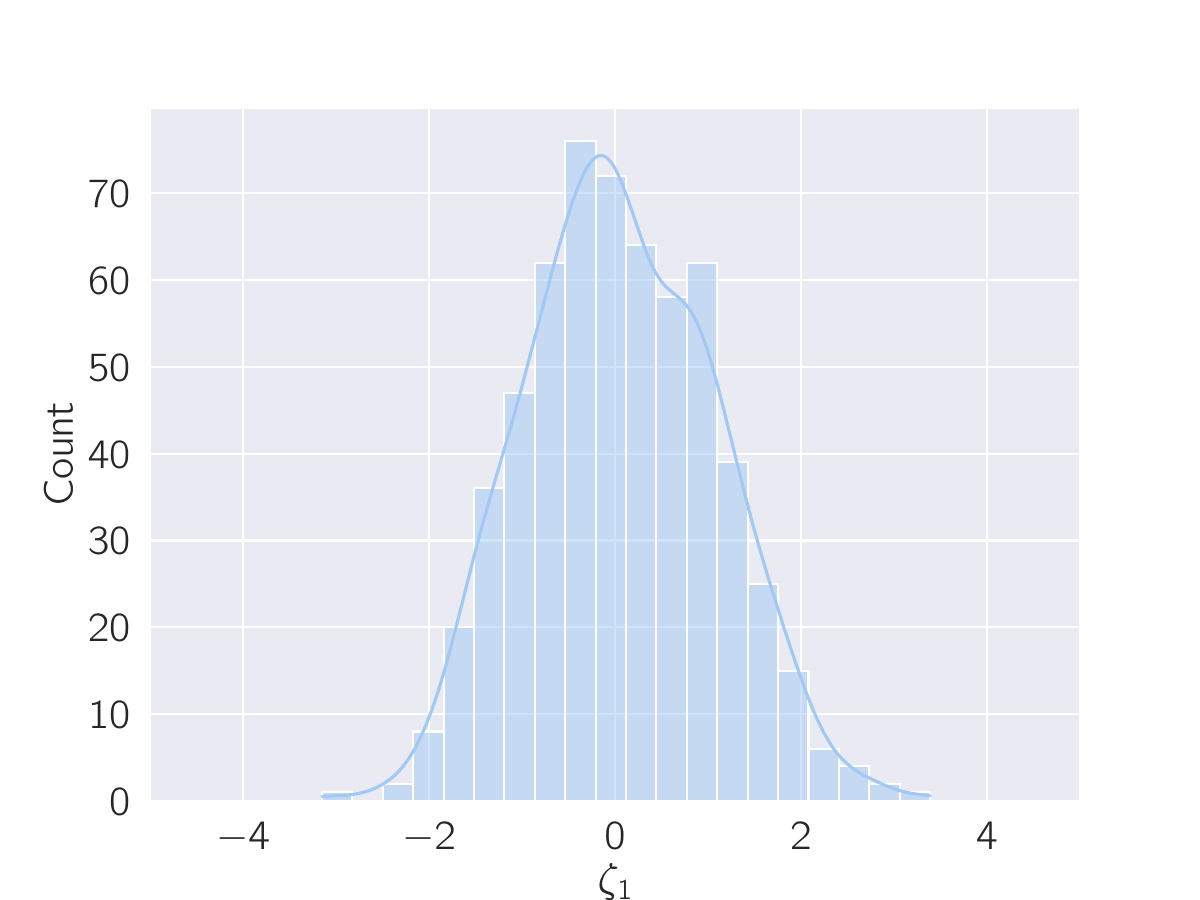}   
\includegraphics[width=34mm]{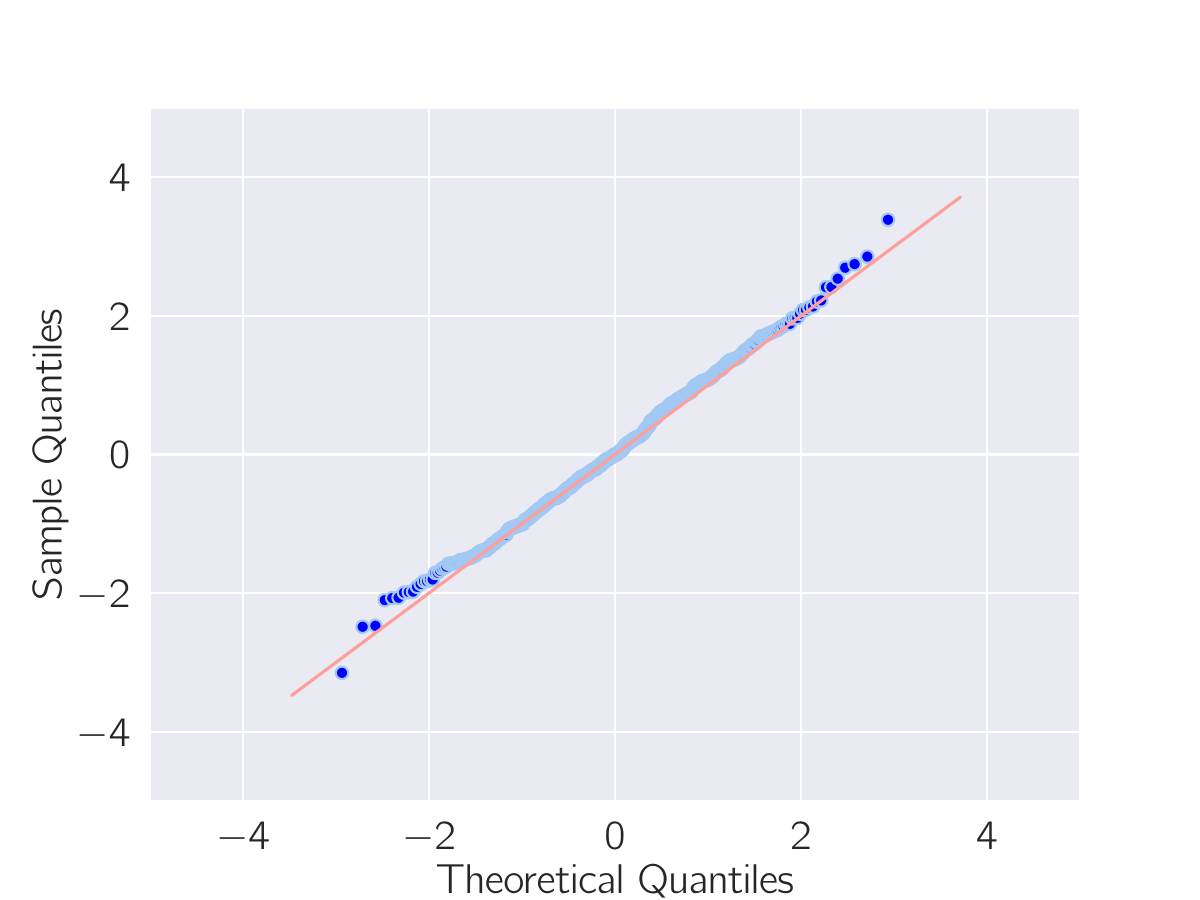}   
\includegraphics[width=34mm]{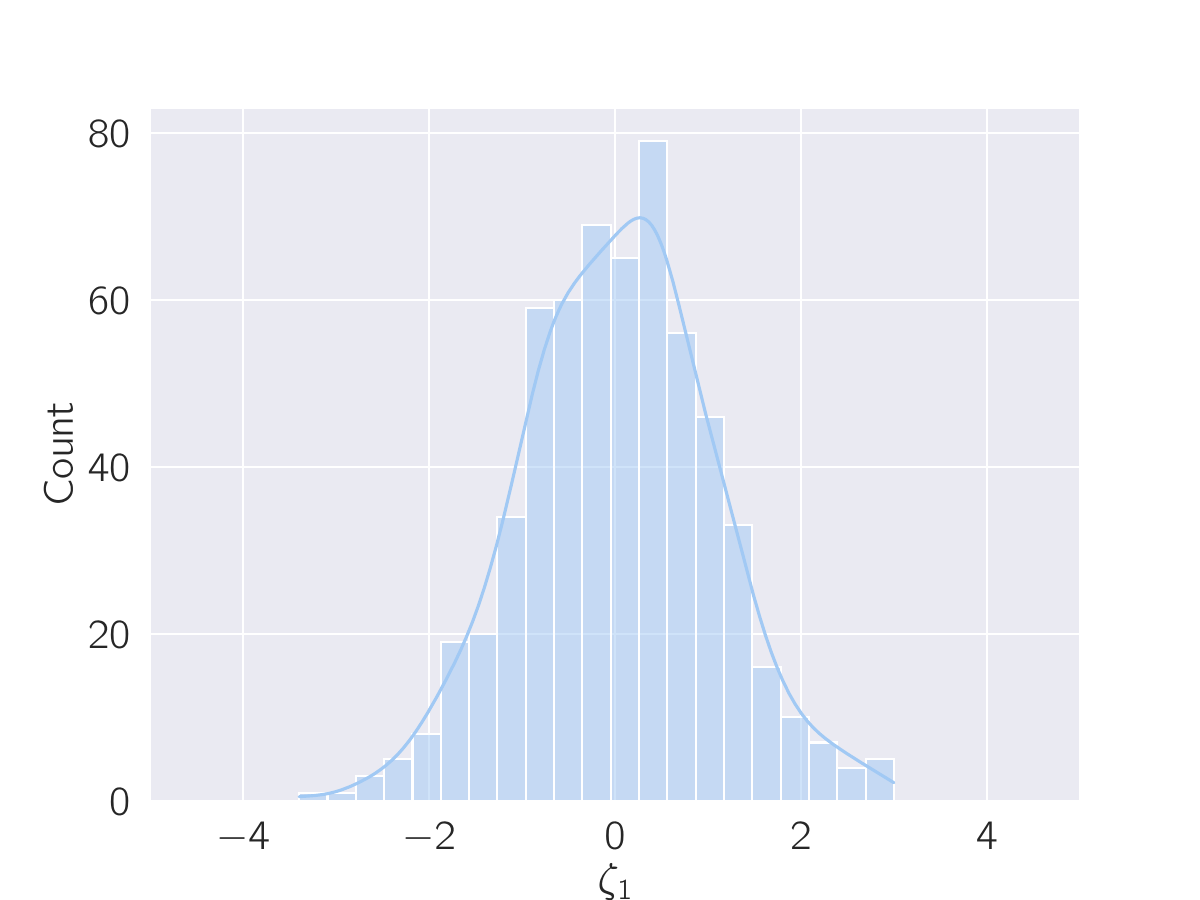}   
\includegraphics[width=34mm]{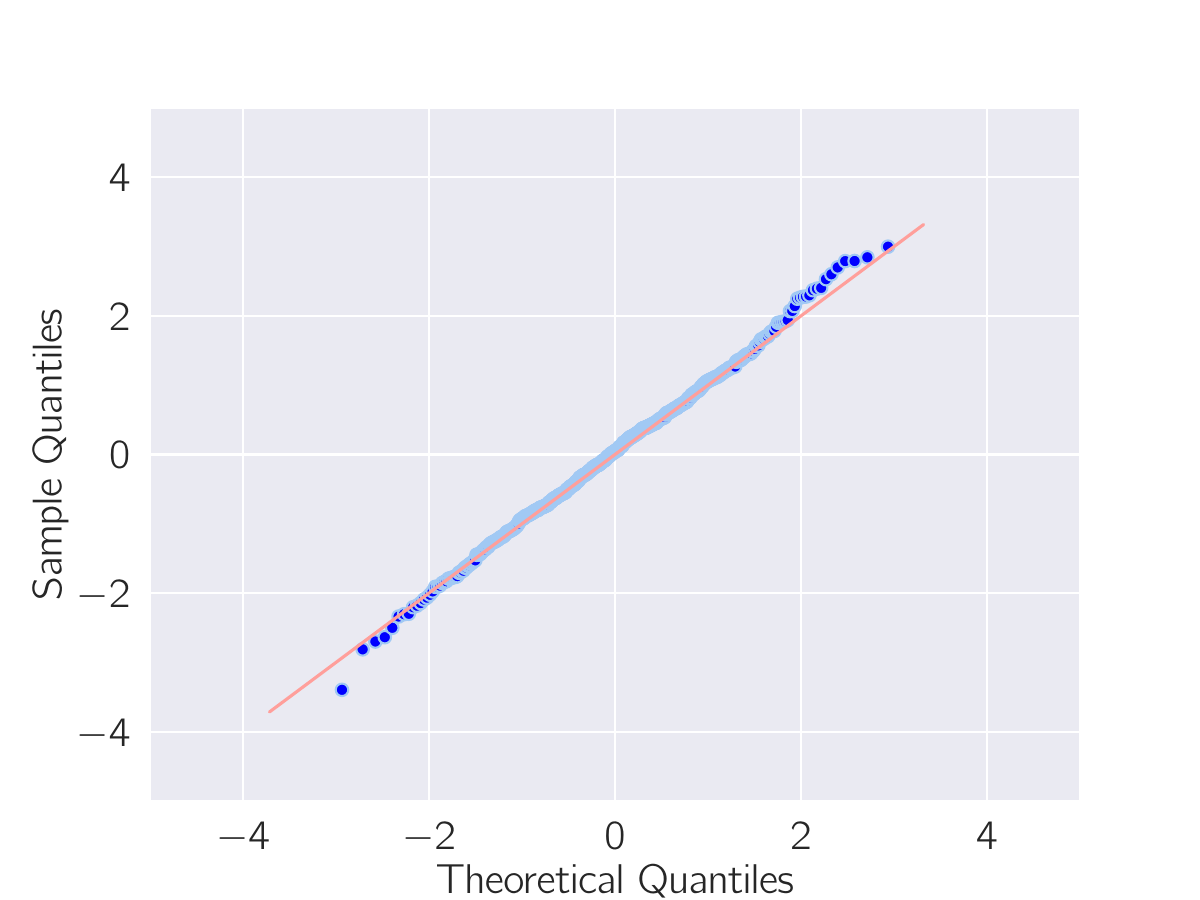}   
\includegraphics[width=34mm]{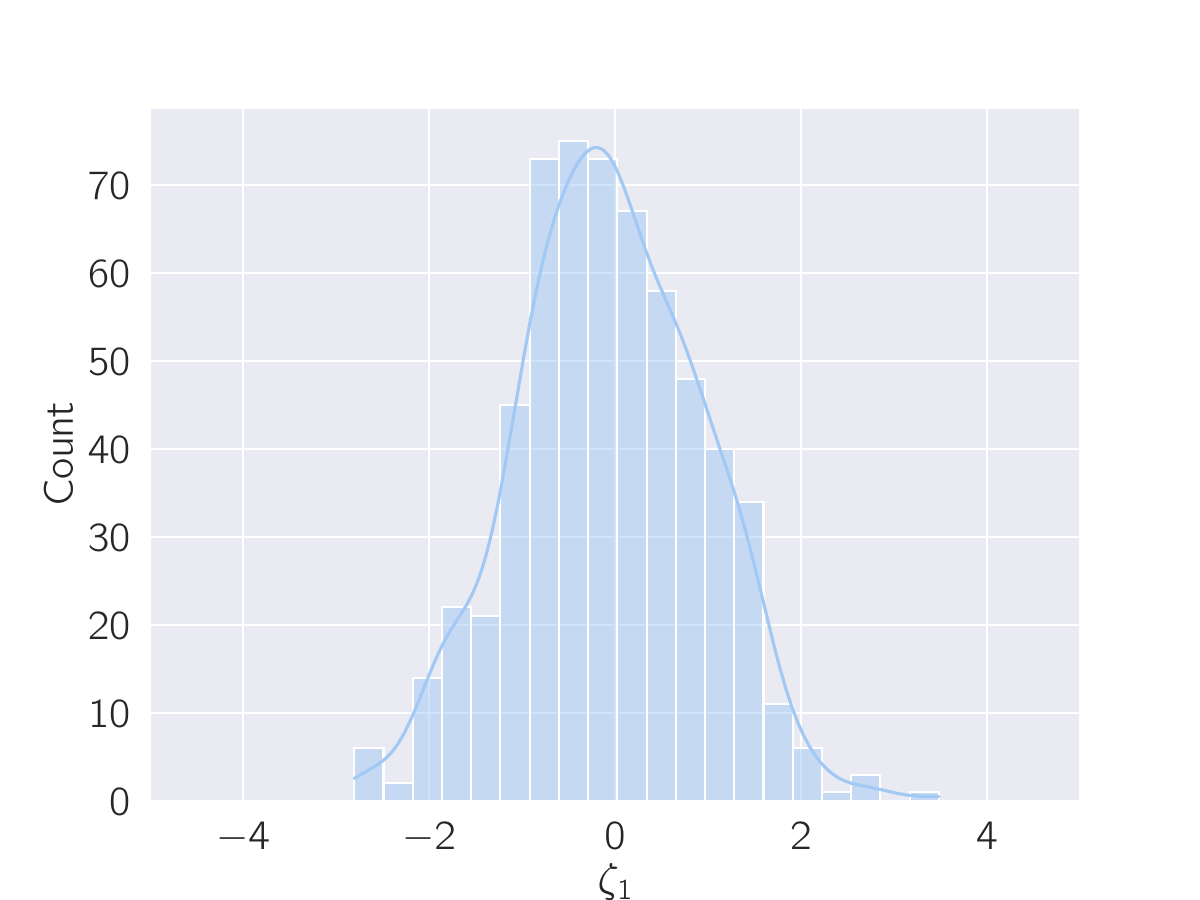}   
\includegraphics[width=34mm]{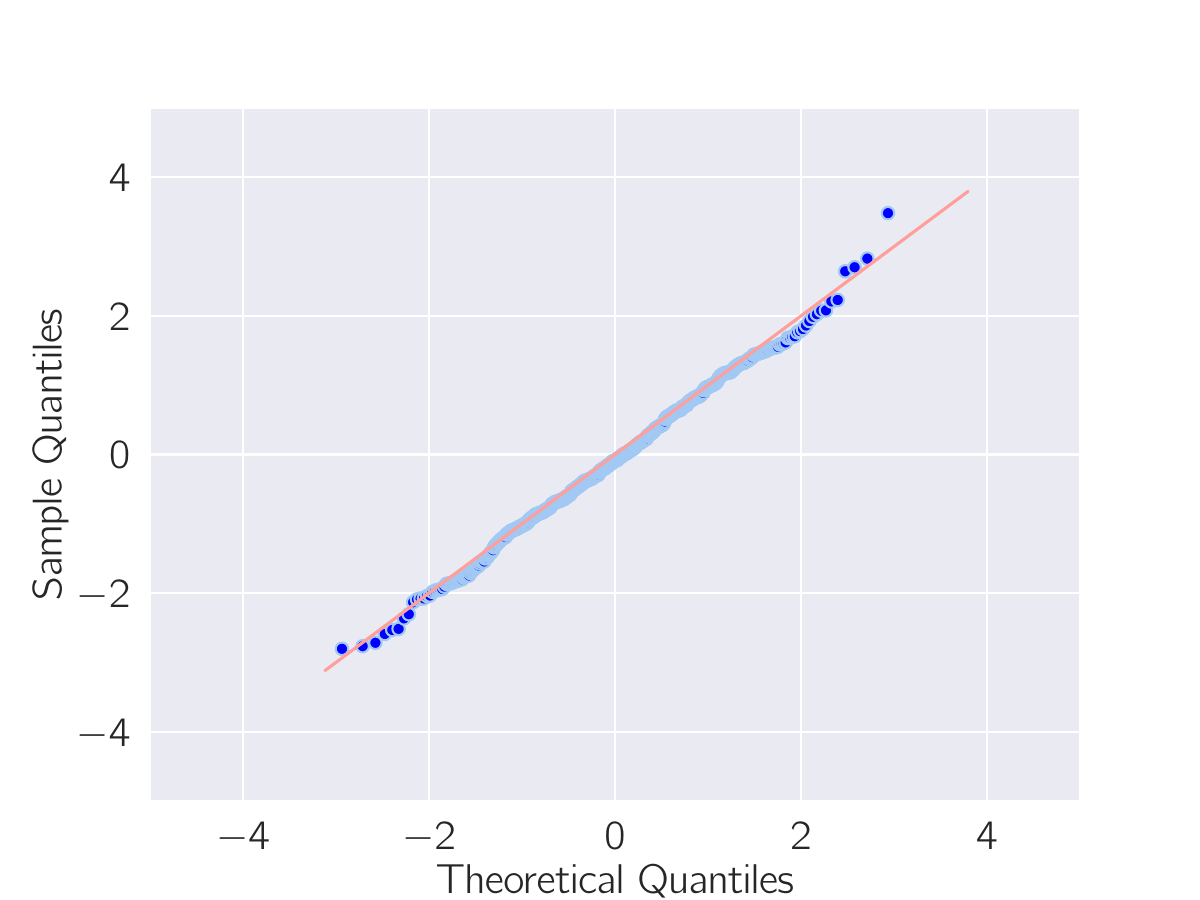}   
\includegraphics[width=34mm]{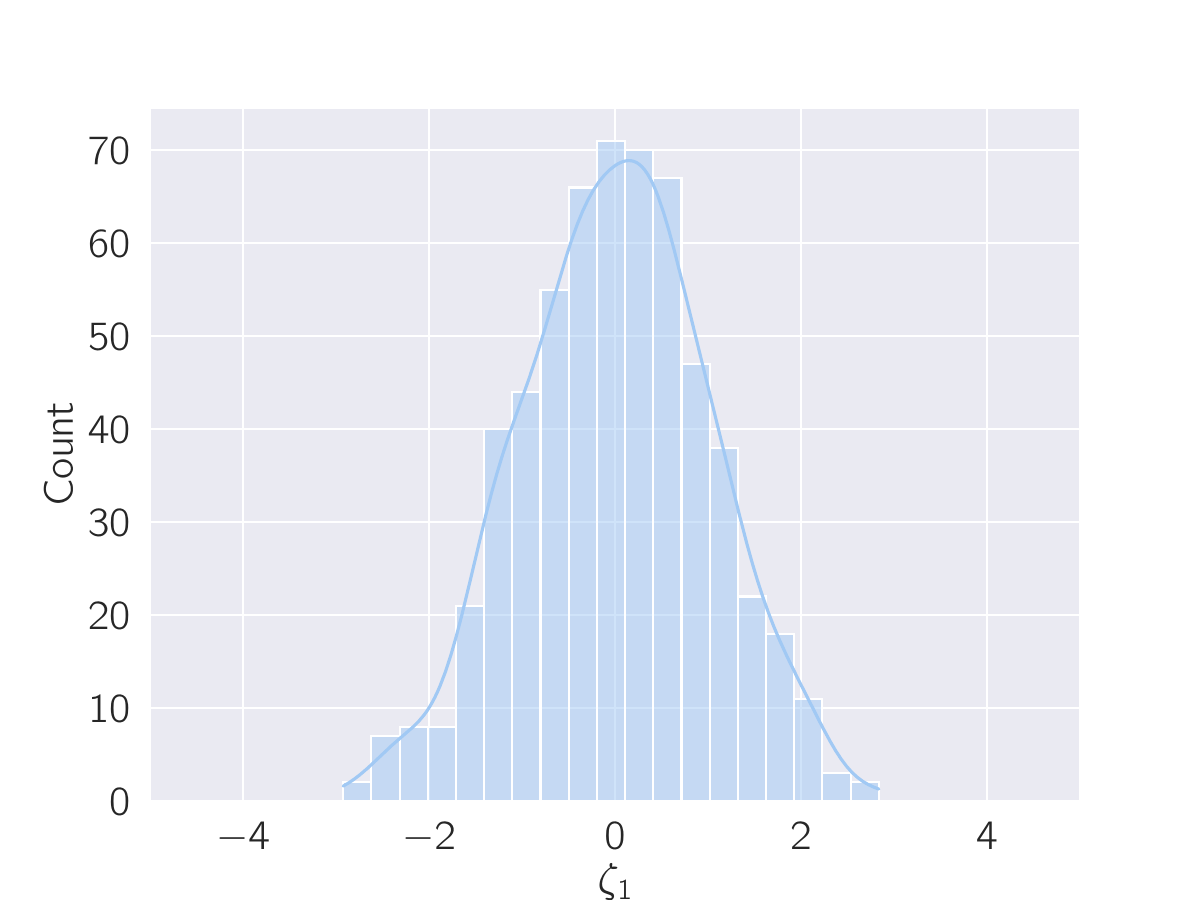}   
\includegraphics[width=34mm]{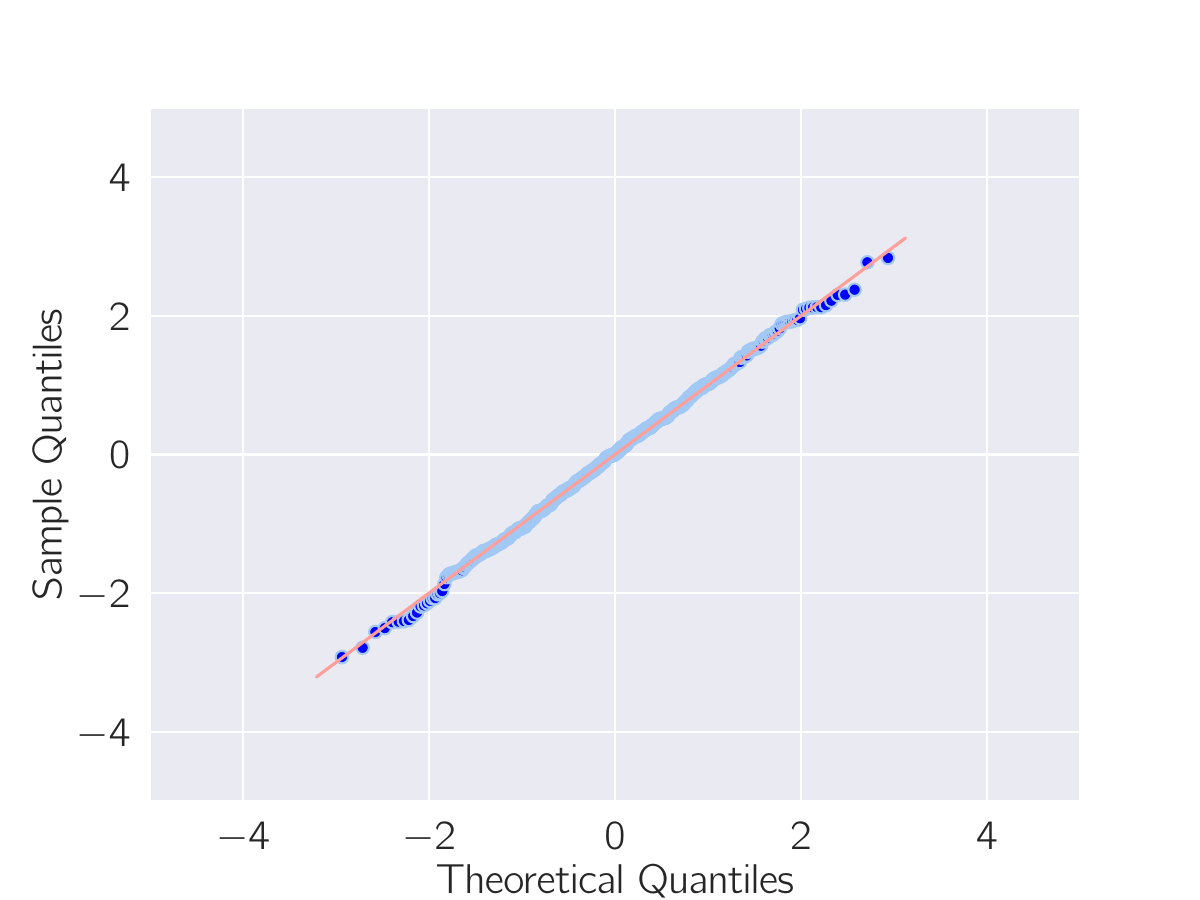}   
\caption{Histogram and QQ-plot for $\zeta_{1}$
     in \eqref{zeta_i}
under Huber Elastic-Net regression for different choices of
tuning parameters $(\lambda, \tau)$. 
Left Top: $(0.036, 10^{-10})$, 
Right Top: $(0.054,0.01)$, 
Left Bottom: $(0.036, 0.01)$, 
Right Bottom: $(0.024, 0.1)$.
Each figure contains 600 data points generated with anisotropic design matrix
and iid $\eps_i$ from the $t$-distribution with 2 degrees of freedom.
A detailed setup is provided in \Cref{sec:simulations}.
}
\label{fig:qq}
\end{figure}

\Cref{thm:normality} is a formal statement regarding the
informal normal approximation
\begin{equation}
    \zeta_{i}
    :=
    \frac{
        r_{i} + \trace [ \bSigma \hbA ] \psi (r_{i}) - \eps_{i}
    }{
        \| \bSigma^{1/2} (\hat \bbeta - \bbeta^*) \|
    }
    \approx N(0,1)
    .
    \label{zeta_i}
\end{equation}
Simulations in \Cref{fig:qq} confirm the normality
of $\zeta_i$ for the Huber loss with Elastic-Net penalty
and four combinations of tuning parameters.
For the square loss $\rho(u)=u^2/2$, because $\psi(u)=u$,
asymptotic normality of the residuals hold in the following form.
\begin{theorem}
    \label{thm:residual-distribution}
    Let \Cref{assumMain} hold with
    $\rho(u)=u^2/2$ and $\bep\sim N(\mathbf{0},\sigma^2\bI_n)$. Then
    for $i=1$,
    \begin{equation}
        \label{eqSquareLossAsymptoticNormality}
         \frac{ 
         (1+\trace [ \bSigma \hbA ])(y_i - \bx_i^\top\hbbeta)
         }{
         (\sigma^2 + \|\bSigma^{1/2}(\hbbeta-\bbeta^*)\|^2)^{1/2}
         }
         \to^d N(0,1)
        \qquad
        \text{ as } n  \to+\infty.
    \end{equation}
\end{theorem}

It is informative to provide a sketch of the proof
of \Cref{thm:normality} to explain
the appearance of $\psi(r_i)$ and $\trace[\bSigma\hbA]$
in the normal approximation results
\eqref{eq:Rem_i} and \eqref{zeta_i}.
A variant of the normal approximation of \cite{bellec_zhang2019second_poincare}
proved in the supplement states that for a differentiable function
$\mathbf{f}:\R^q\to\R^q\setminus\{\mathbf0\}$ and $\bz\sim N(\bzero,\bI_q)$,
there exists $Z\sim N(0,1)$ such hat
\begin{equation}
\E\Bigl[
\Bigl|
\frac{\mathbf{f}(\bz)^\top\bz - \sum_{k=1}^{q}(\partial/\partial z_{k}) f_{k} (\bz)}{\|\mathbf{f}(\bz)\|}
    - Z
    \Bigr|^2
\Bigr]
\le
\C \E\Bigl[\frac{\sum_{k=1}^q\|(\partial/\partial z_k)\mathbf{f}(\bz)\|^2}{\|\mathbf{f}(\bz)\|^{2}}\Bigr]
.
\end{equation}
Some technical hurdles aside,
the proof sketch is the following:
Apply the previous display to $q=p$,
$\bz = \bSigma^{-1/2}\bx_i$ conditionally on 
$(\bep,(\bx_l)_{l\in[n]\setminus\{i\}})$
and to $\mathbf{f}(\bz) = \bSigma^{1/2}(\hbbeta-\bbeta^*)$ in the simple case where $\bbeta^*=0$
(this amounts to performing a change of variable
by translation of $\hbbeta$ to $\hbbeta-\bbeta^*$).
Then the right-hand side of the previous display is negligible
in probability compared to $Z$, and in the left-hand side
$\mathbf{f}(\bz)^\top\bz = \bx_i^\top(\hbbeta-\bbeta^*)$
and 
$\sum_{k=1}^{q}(\partial/\partial z_{k}) f_{k} (\bz)
\approx \trace[\bSigma\hbA]\psi(r_i)$
as the second term in \eqref{eq:differentiability-formulae}
is negligible. This completes the sketch of the proof
of \eqref{zeta_i}.

\paragraph{Proximal operator representation.}
From the above asymptotic normality results,
a stochastic representation for the $i$-th residual
$r_i=y_i - \bx_i^\top\smash{\hbbeta}$ can be obtained as follows:
With $\prox[t\rho](u)$ the proximal operator of $x\mapsto t\rho(x)$
defined as the unique solution $z\in\R$ of equation
$z + t\psi(z) = u$,
$$
r_i = y_i - \bx_i^\top\hbbeta
= \prox[\hat t \rho]\bigl(
    \tilde \eps_i^n + \|\bSigma^{1/2}(\hbbeta-\bbeta^*)\| \tilde Z_i^n
\bigr) 
\qquad \text{ with } \hat t = \trace[\bSigma\hbA]
$$
where
$(\tilde \eps_i^n,\tilde Z_i^n)$ converges in distribution to 
the product measure $F\otimes N(0,1)$ where $F$ is the law of $\eps_i$.

\section{A proxy of the out-of-sample error if 
\texorpdfstring{$\mathbf{\Sigma}$}{bSigma} is known}
\label{sec:proxy}
The approximations of the previous sections
for $r_i + \trace[\bSigma\hbA]\psi(r_i)$
and the fact that $\eps_i$ is independent of $Z_i\sim N(0,1)$
in \eqref{eq:Rem_i} suggest that
$(r_i + \trace[\bSigma\hbA]\psi(r_i))^2
\approx \eps_i^2 + \|\bSigma^{1/2}(\hbbeta-\bbeta^*)\|^2Z_i^2$;
and averaging over $\{1,...,n\}$ one can hope for the approximation
$ \|\br + \trace[\bSigma\hbA]\psi(\br)\|^2/n
\approx \|\bep\|^2/n + \|\bSigma^{1/2}(\hbbeta-\bbeta^*)\|^2$.
The following result makes this heuristic precise.

\begin{theorem}
    \label{thm:out-of-sample}
    Let \Cref{assumMain} be fulfilled
    and $\hbA$ be given by \Cref{thm:differentiability}.
    Then 
    \begin{align*}
        \|\bSigma^{1/2}(\hbbeta-\bbeta^*)\|^2 + \|\bep\|^2/n
        = \big\| \br + \trace[\bSigma\hbA]\psi(\br)\big\|^2/n
        + O_P(n^{-1/2}) \Rem,
    \end{align*}
    where 
    $
    \Rem := 
     \| \bSigma^{1/2} (\hat \bbeta - \bbeta^*) \|^{2} 
    + \frac1n \| \psi (\br) \|^{2} 
    + (\| \bSigma^{1/2} (\hat \bbeta - \bbeta^*) \|^{2} 
    + \frac1n\| \psi (\br) \|^{2} )^{1/2} \| \frac{1}{\sqrt n}\bep\|
    $.
    Thus
    $$
        \|\bSigma^{1/2}(\hbbeta-\bbeta^*)\|^2 + \|\bep\|^2/n
        = (1+ O_P(n^{-1/2}))
        \big\| \br + \trace[\bSigma\hbA]\psi(\br)\big\|^2/n.
    $$
\end{theorem}

\Cref{thm:out-of-sample} provides a first candidate,
$\big\| \br + \trace[\bSigma\hbA]\psi(\br)\big\|^2/n$
to estimate
\begin{equation}
    \|\bSigma^{1/2}(\hbbeta-\bbeta^*)\|^2 + \|\bep\|^2/n.
    \label{target}
\end{equation}
Estimation of \eqref{target} is useful
as $\|\bep\|^2/n$ is independent of the choice of the estimator
$\hbbeta$ and in particular independent of the chosen loss-penalty
pair in \eqref{hbeta}. 
Given two or more estimators \eqref{hbeta},
choosing the one with smallest 
$\big\| \br + \trace[\bSigma\hbA]\psi(\br)\big\|^2$ is thus a
good proxy for minimizing the out-of-sample error.

\begin{corollary}
    \label{corSelection}
    Let $\hbbeta, \tbbeta$ be two $M$-estimators \eqref{hbeta}
    \Cref{assumMain}
    with loss-penalty pair $(\rho,g)$
    and $(\tilde \rho,\tilde g)$ respectively.
    Assume that both satisfy \Cref{assumMain}
    and let $\psi=\rho'$ and $\tilde \psi=\tilde\rho'$.
    Let $\br=\by-\bX\hbbeta, \tbr=\by-\bX\tbbeta$ be the 
    residuals, $\hbA,\tbA$ be the corresponding matrices
    of size $p\times p$ given by \Cref{thm:differentiability}.
    Further assume that both estimators satisfy
    \Cref{assumAdditional}
    and 
    that $\bep$ has iid coordinates
    independent with
    $\E[|\eps_i|^{1+q}] \le M$
    for constants $q\in(0,1),M>0$ independent of $n,p$.
    Let
    $\Omega = \{
        \|\bX\bSigma^{-1/2}\|_{op}\le 2\sqrt n + \sqrt p
        \} \cap
        \{
            \|\bep\|^2 \le n^{2/(1+q)}
        \}$.
    Then for any $\eta>0$ independent of $n,p$
    there exists $C(\gamma,\mu,\eta,q,M)>0$
    depending only on  $\{\gamma,\mu,\eta,q,M\}$ such that
    \begin{multline*}
    \P\Bigl(
    \|\bSigma^{1/2}(\hbbeta-\bbeta^*)\|^2
    -
    \|\bSigma^{1/2}(\tbbeta-\bbeta^*)\|^2 > \eta,
    ~
    ~
    \|\br + \trace[\bSigma\hbA] \psi( \br )\|^2
    \le
    \|\tbr + \trace[\bSigma\tbA] \tilde\psi( \tbr )\|^2
    \Bigr)
    \\\le
    C(\gamma,\mu,\eta,q,M)n^{-q/(1+q)}
    +
    \P(\Omega^c)
    \to 0.
    \end{multline*}
\end{corollary}

Provided that the noise random variables $\eps_i$
have at least $1+q$ moments,
\Cref{corSelection} implies that with probability approaching one
given two $M$-estimators $\hbbeta$ and $\tbbeta$,
choosing the estimator corresponding to the
smallest criteria among 
$\|\br + \trace[\bSigma\hbA]\br\|^2$
and 
$\|\tbr + \trace[\bSigma\tbA]\tbr\|^2$
leads to the smallest out-of-sample error, up to any small constant $\eta>0$.
This allows noise random variables $\eps_i$ with infinite variance.
A similar result can be obtained to select among $K$ different
$M$-estimators \eqref{hbeta}.

\begin{corollary}
    \label{corSelection2}
    As in \Cref{corSelection}, assume $\E[|\eps_i|^{1+q}]\le M$
    and let $\hbbeta_1,...,\hbbeta_K$ be $M$-estimators of the form
    \eqref{hbeta} with loss-penalty pair $(\rho_k,g_k)$
    satisfying \Cref{assumMain,assumAdditional}.
    For each $k=1,...,K$, let $\br_k = \by - \bX\hbbeta_k$
    be the residuals and $\hbA_k$ be the corresponding matrix of size $p\times p$
    from \Cref{thm:differentiability}.
    Let $\hat k \in \argmin_{k=1,...,K}\|\br_k +\trace[\bSigma\hbA_k]\psi_k(\br_k)\|$ where $\psi_k=\rho_k'$.
    Then if $(\gamma,\mu,\eta,q,M)$ are constants independent of $n,p$
    \begin{equation*}
    \P\bigl(
        \|\bSigma^{1/2}(\hbbeta_{\hat k}-\bbeta^*)\|^2
        >
        \min_{k=1,...,K}
        \|\bSigma^{1/2}(\hbbeta_k-\bbeta^*)\|^2 
        + \eta
    \bigr)
    \to 0 
    \qquad \text{ if }
    K=o(n^{q/(1+q)}).
    \end{equation*}
    In other words,
    if $K=o(n^{q/(1+q)})$, the selector $\hat k$ picks an optimal
    M-estimator in the sense
    $$\|\bSigma^{1/2}(\hbbeta_{\hat k}-\bbeta^*)\|^2
    -
    \min_{k=1,...,K}
    \|\bSigma^{1/2}(\hbbeta_k-\bbeta^*)\|^2 
    \to^P 0.
    $$
\end{corollary}

Given $K$ different loss-penalty pairs
and the corresponding $M$-estimators in \eqref{hbeta},
minimizing the criterion $\|\br + \trace[\bSigma \hbA]\br\|$ thus provably
selects a loss-penalty pair that leads to an optimal out-of-sample
error, up to an arbitrary small constant $\eta>0$ independent of $n,p$.
The requirement $K=o(n^{q/(1+q)})$ means that the cardinality
of the collection of $M$-estimators to select from should grow
more slowly than a power of $n$. This is typically satisfied
for default tuning parameter grids in popular libraries
(e.g., \verb|sklearn.linear_model.Lasso| from \cite{pedregosa2011scikit})
with tuning parameters evenly spaced in a log-scale that
consequently have cardinality logarithmic in the parameter range.
The major drawback of the criterion $\|\br + \trace[\bSigma \hbA]\br\|$
is the dependence through $\trace[\bSigma \hbA]$
on the covariance $\bSigma$ of the design,
which is typically unknown. The next section introduces an
estimator of $\trace[\bSigma \hbA]$ that does not require the knowledge of
$\bSigma$.

\section{Degrees of freedom and estimating 
\texorpdfstring{$\trace[\bSigma\hbA]$}{tr[bSigma hbA]}
without the knowledge of 
\texorpdfstring{$\bSigma$}{bSigma}}
\label{sec:edf}

This section focuses on estimating $\trace[\bSigma\hbA]$.
The matrix $\hbA$ from \Cref{thm:differentiability}
can be estimated from the data $(\by,\bX)$ in the sense
that $\hbA$ is a measurable function of $(\by,\bX)$
(thanks to the observation that derivatives are limits,
and limits of measurable functions are again measurable).
The difficulty is thus to estimate $\trace[\bSigma\hbA]$
without the knowledge of $\bSigma$. To illustrate this difficulty,
consider Ridge regression with square loss
$\rho(u)=u^2/2$ and penalty $g(\bb)=\tau\|\bb\|^2/2$.
Then 
$\hbbeta(\by,\bX) = (\bX^\top\bX + \tau n \bI_p)^{-1}\bX^\top\by$
and $\hbA$ in \Cref{thm:differentiability}
is given explicitly by $\hbA= (\bX^\top\bX + \tau n \bI_p)^{-1}$ and
$$\trace[\bSigma\hbA] = \trace[(\bG^\top\bG + n\tau \bSigma^{-1})^{-1}],
\qquad
\text{ where }
\bG=\bX\bSigma^{-1/2}.$$
Above, $\bG$ is a random matrix with iid $N(0,1)$ entries
the value of $\trace[\bSigma\hbA]$ is
highly dependent on the spectrum of $\bSigma^{-1}$.
In this particular case, the limit
of $\trace[(\bG^\top\bG + n\tau \bSigma^{-1})^{-1}]$
can be obtained using random matrix theory \citep{marvcenko1967distribution}
as the limiting behavior of the Stieltjes transform of 
$\bG^\top\bG/n + \tau\bSigma^{-1}$ and its spectral distribution is known;
however the limit of the spetral distribution 
depends on the spectrum of $\tau\bSigma^{-1}$.
This is not desirable here as we wish to construct estimators that 
require no knowledge on $\bSigma$.
For more involved loss-penalty pairs such as the
Elastic-Net in \Cref{prop:huber-A-V},
such random matrix theory results do not apply 
as $\trace[\bSigma\hbA]$ depends on the random support of
$\hbbeta$. 

Instead, we do not rely on known
random matrix theory results.
With the matrix $\hbA\in\R^{p\times p}$
given by \Cref{thm:differentiability},
our proposal to estimate
$\trace[\bSigma\hbA]$ is the ratio $\df/\trace[\bV]$
with $\df$ and $\bV$ in \eqref{df}-\eqref{V}.
Both the scalar $\df$ and the matrix $\bV\in\R^{n\times n}$
are observable; in particular they do not depend on $\bSigma$.

\begin{theorem}
    \label{thm:df-trAtrV}
    Let \Cref{assumMain} be fulfilled
    and $\hbA$ be given by \Cref{thm:differentiability}.
    Then
    \begin{equation}
        \label{eq:df-trAtrV}
    \E[|\trace[\bSigma \hbA]\trace[\bV]/n
        -\df / n|] \le \C(\gamma,\mu) n^{-1/2}
    .
    \end{equation}
\end{theorem}
\begin{figure}[ht]
    \centering
    \includegraphics[width=68mm]{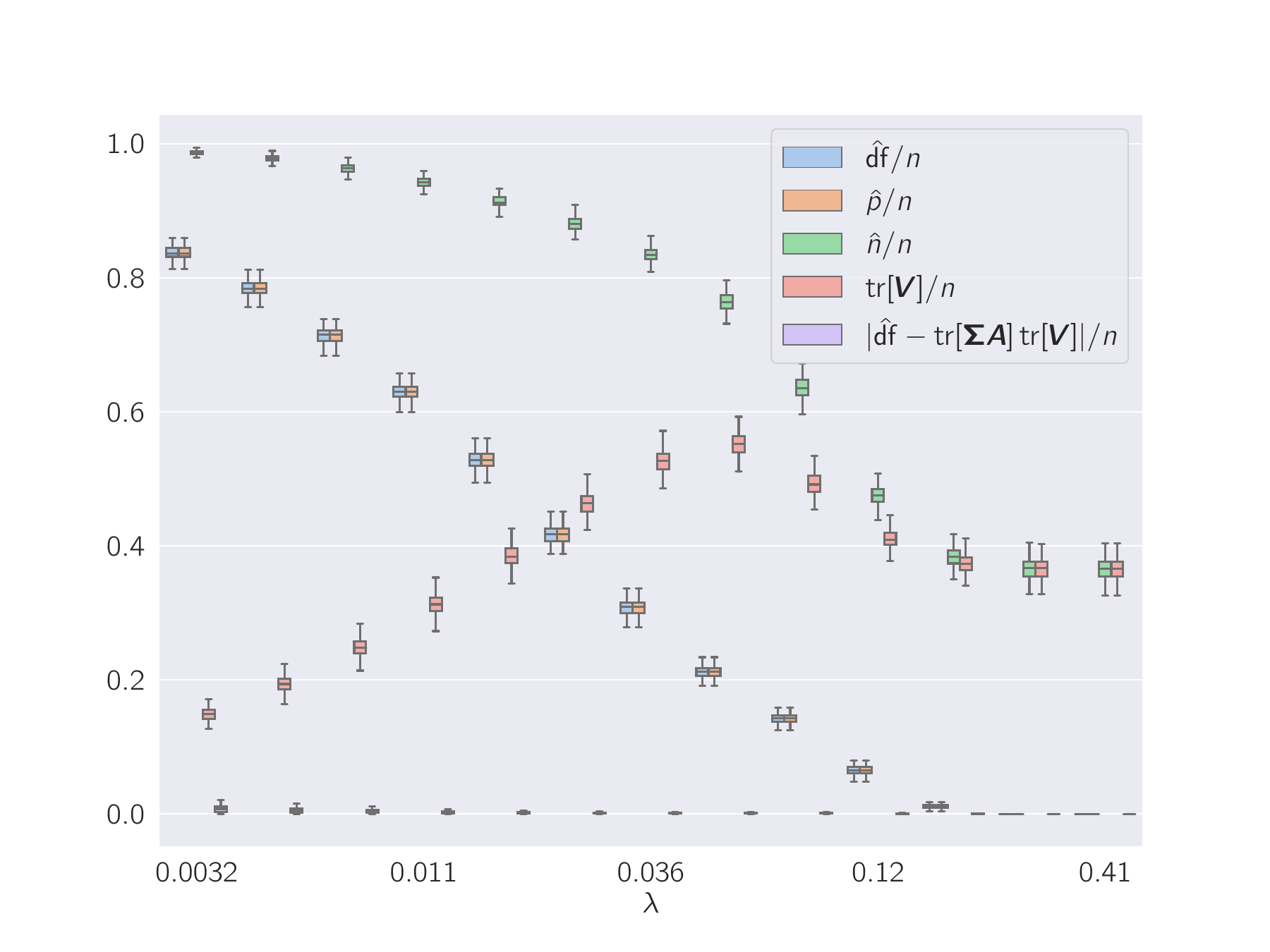}   
    \includegraphics[width=68mm]{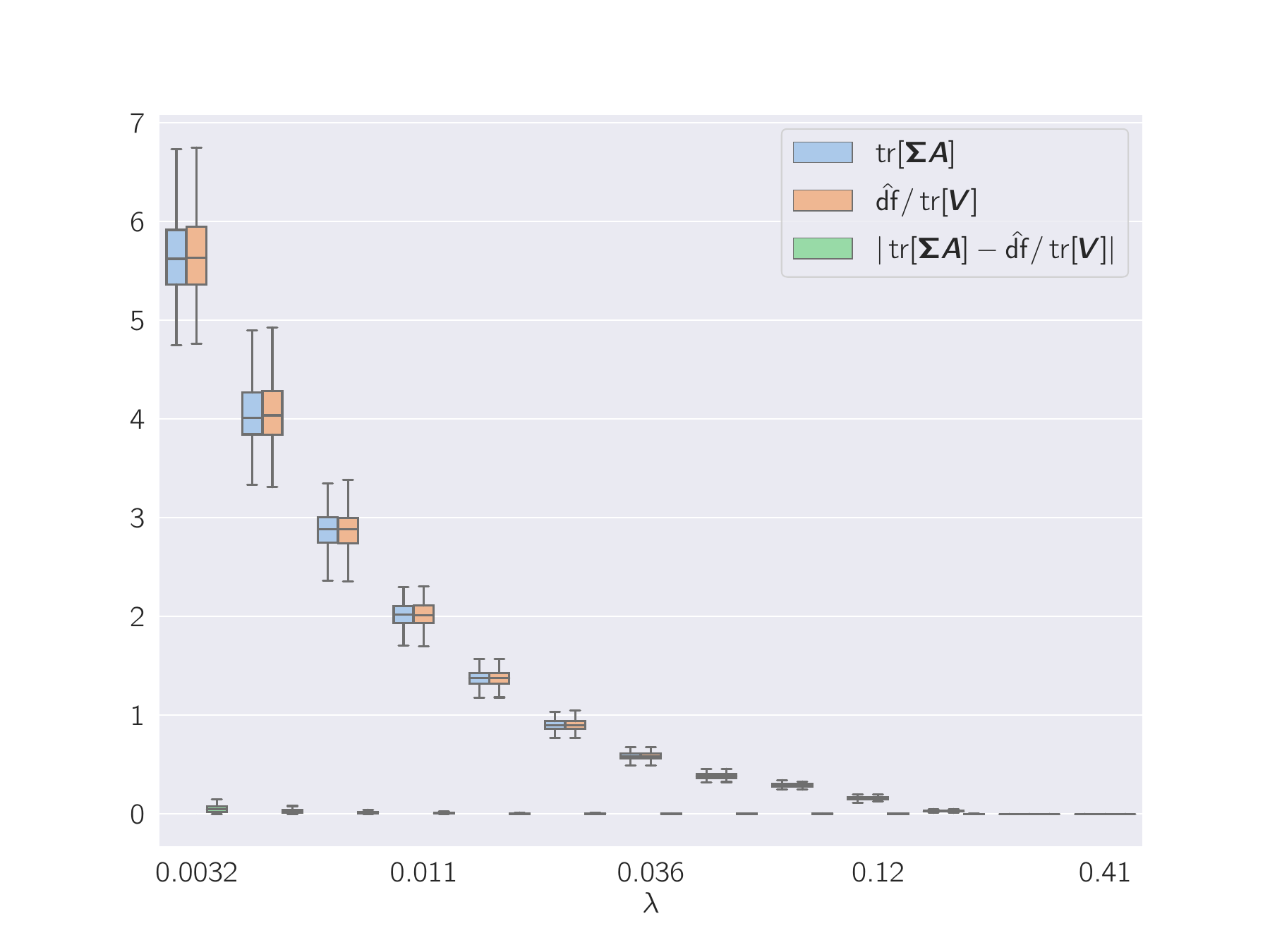}   
    \caption{
    Above: Boxplots for $\df, \hat p, \hat n, \trace[\bV], \trace[\bSigma \hbA]$ and $|\trace [ \bSigma \hbA ] - \df / \trace [ \bV]|$ 
    in Huber Elastic-Net regression with $\tau = 10 ^{-10}$ and $\lambda \in [0.0032, 0.41].$
    Each box contains 200 data points.
    Below: heatmaps
    for $\df/n$, $\trace[\bV]/n$
    and $\hat n/n =\sum_{i=1}^n\psi'(r_i)/n$ under the simulation setup in \Cref{fig:out-of-sample}.
The detailed simulation setup is given in \Cref{sec:simulations}.
    }
    \label{fig:boxplots-trA}
    \includegraphics[width=50mm]{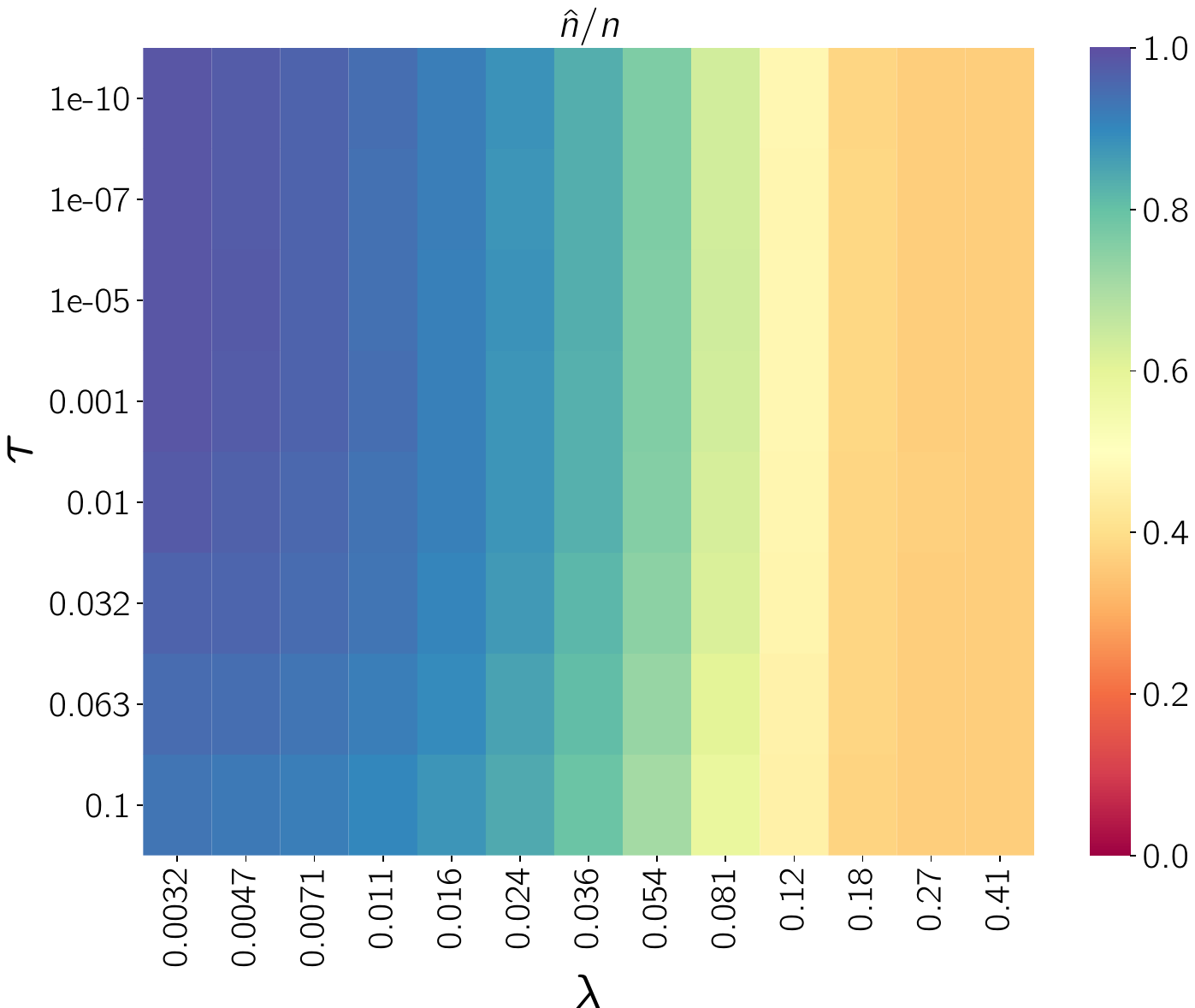}   
    \hspace{-0.3in}
    \includegraphics[width=50mm]{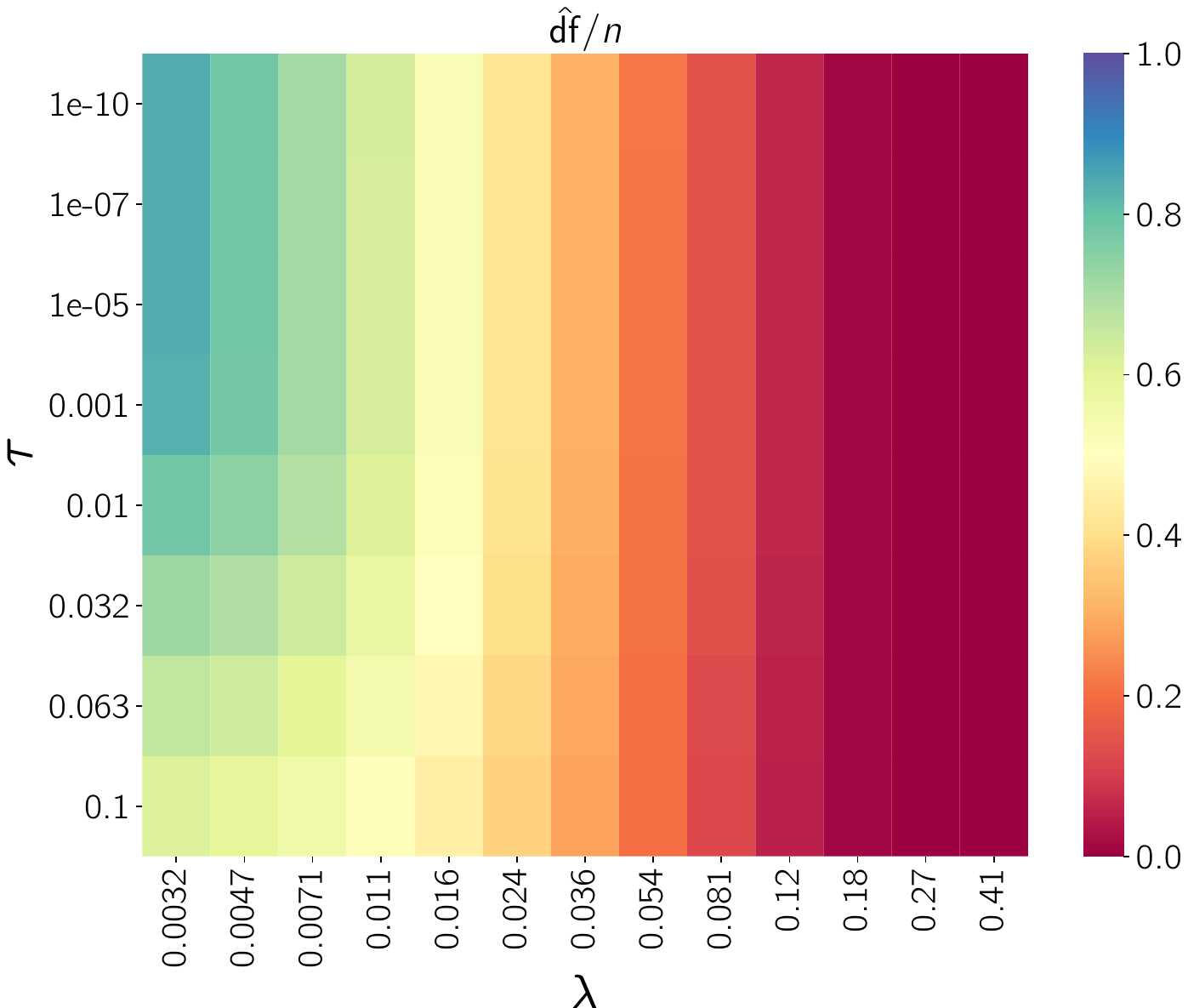}   
    \hspace{-0.3in}
    \includegraphics[width=50mm]{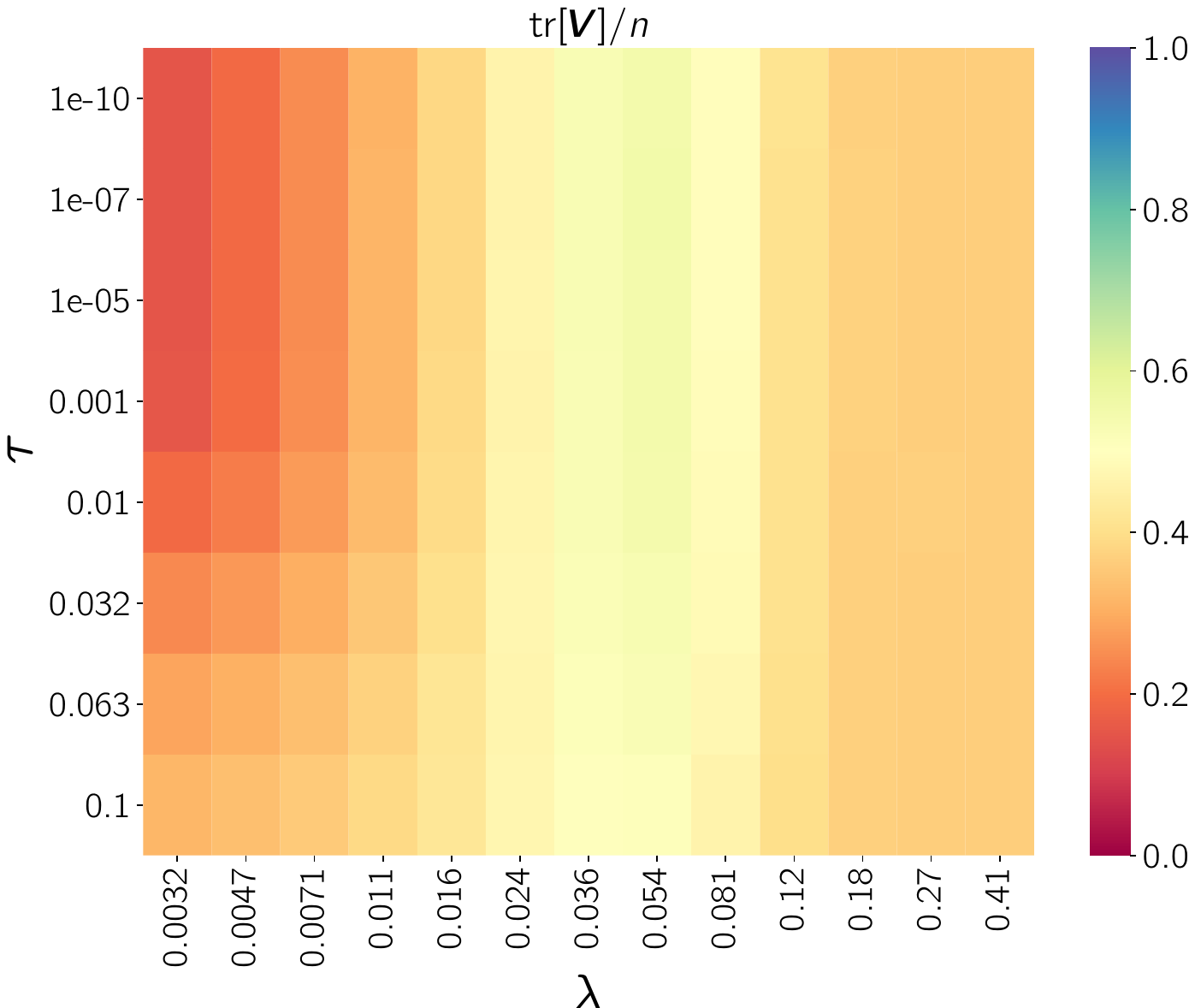}   
\end{figure}
Simulations in \Cref{fig:boxplots-trA,tab:table1}
confirm 
that the approximation $\trace[\bSigma\hbA] \approx \df/\trace[\bV]$
is accurate for the Huber loss with Elastic-Net penalty.
For the square loss, $\psi'=1$ and $\trace[\bV] = n - \df$
so that \eqref{eq:df-trAtrV}
becomes
$\E|(1-\df/n)(1+\trace[\bSigma\hbA]) - 1| \le \C(\gamma,\mu) n^{-1/2}$ 
and the following result holds.
\begin{corollary}
    \label{col:residual-distribution}
    Let \Cref{assumMain} be fulfilled with
    $\rho(u)=u^2/2$ and $\bep\sim N(\mathbf{0},\sigma^2\bI_n)$. Then
    $(1-\df/n)(1+\trace[\bSigma\hbA])\to^\P 1$ and 
    the normality \eqref{eqSquareLossAsymptoticNormality} holds
    with $1+\trace[\bSigma\hbA]$ replaced by
    $(1-\df/n)^{-1}$.
\end{corollary}

For general loss $\rho$, the criterion \eqref{crit} replaces $\trace[\bSigma\hbA]$
by $\df/\trace[\bV]$
in the proxy of the out-of-sample error $\|\br + \trace[\bSigma\hbA]\psi(\br)\|^2$
studied in the previous section.
Thanks to \eqref{eq:df-trAtrV}, this replacement preserves
the good properties of $\|\br + \trace[\bSigma\hbA]\psi(\br)\|^2$
proved in \Cref{corSelection,corSelection2}.

\begin{theorem}
    \label{thm:selection-crit}
    For $k=1,...,K$,
    let $(\rho_k,g_k)$ be a loss-penalty pair
    satisfying \Cref{assumMain,assumAdditional}
    with $\psi_k=\rho_k'$,
    let $\hbbeta_k,\br_k,\hbA_k$ be the corresponding $M$-estimator
    residual vector and matrix of size $p\times p$ given by
    \Cref{thm:differentiability} as in \Cref{corSelection2}
    and let 
    $\df_k = \trace[\bX\hbA_k \bX^\top\diag\{\psi_k'(\br_k)\}]$
    and 
    $\bV_k =\diag\{\psi_k'(\br_k)\}(\bI_n -  \bX\hbA_k \bX^\top\diag\{\psi_k'(\br_k)\})$.
    For a small constant $\eta>0$ independent of $n,p$, say $\eta=0.05$,
    define
    $$\hat k \in \argmin_{k=1,...,K} 
    \Big\|\br_k + \frac{\df_k}{\trace[\bV_k]}\psi_k(\br_k)\Big\|^2
    \qquad \text{subject to}
    \qquad
    \frac 1 n \sum_{i=1}^n \psi_k'(r_{ki}) \ge \eta.$$
    If $\eps_i$ has $1+q$ moments in the sense
    that $\E[|\eps_i|^{1+q}]\le M$ for constants $q\in (0,1),M>0$.
    If $(M,q,\eta,\mu,\gamma)$ and $\tilde\eta>0$ are independent of $n,p$ then
    \begin{equation*}
    \P\Bigl(
        \|\bSigma^{1/2}(\hbbeta_{\hat k}-\bbeta^*)\|
        >
        \min_{k=1,...,K:
        \frac 1 n \sum_{i=1}^n \psi_k'(r_{ki}) \ge \eta
        }
        \|\bSigma^{1/2}(\hbbeta_k-\bbeta^*)\| 
        + \tilde \eta
    \Bigr)
    \to 0 
    \qquad \text{ if }
    K=o(n^{q/(1+q)}).
    \end{equation*}
\end{theorem}

\Cref{fig:out-of-sample} illustrates on simulations the success
of the criterion \eqref{crit} over a grid of tuning parameters
for $M$-estimators with the Huber loss and Elastic-Net penalty.
The criterion \eqref{crit} is thus successful
at selecting a $M$-estimator with smallest out-of-sample error
up to an additive constant $\tilde\eta$,
among those $M$-estimators indexed in $\{1,...,K\}$
that are such that $\frac 1 n \sum_{i=1}^n \psi_k'(r_{ki}) \ge \eta$.
On the one hand it is unclear to us whether the restriction
$\frac 1 n \sum_{i=1}^n \psi_k'(r_{ki}) \ge \eta$ can be omitted.
On the other hand there is a practical meaning in excluding
$M$-estimators with small $\frac 1 n \sum_{i=1}^n \psi_k'(r_{ki})$:
For the Huber loss
$H(u):= u^{2} / 2$ for $ |u| \leq 1$ and $|u| -1/2$ for $|u|\geq1$
the quantity $\frac 1 n \sum_{i=1}^n \psi_k'(r_{ki})$ is the number of 
of data points in $\{1,...,n\}$ such that the residual $y_i-\bx_i^\top\hbbeta_k$ fall within the quadratic regime of the loss function.
Observations $i\in\{1,...,n\}$ that fall in the linear regime
of the loss are excluded from the fit, in the sense 
that for some $i$ with $r_{ki} = y_i-\bx_i^\top\hbbeta_k > 1$, replacing
$y_i$ by $\tilde y_i = y_i + 1000$ (or any positive value)
does not change the $M$-estimator solution $\hbbeta_k$ 
(this can be seen from the KKT conditions directly, or by integration the
derivative with respect to $y_i$ in \eqref{eq:differentiability-formulae}).
Thus the constraint
$\frac 1 n \sum_{i=1}^n \psi_k'(r_{ki}) \ge \eta$ requires 
that at most a constant fraction of the observations 
are excluded from the fit
(or equivalently, at least a constant fraction of the $n$ observations
participate in the fit).
For scaled versions of the Huber loss,
$\rho_k(u)=a^2 H(a^{-1}u)$ for some $a>0$,
the value $\hat n = \frac 1 n \sum_{i=1}^n \psi_k'(r_{ki})$ again counts the number
of residuals falling in the quadratic regime of the loss, i.e.,
the number of observations participating in the fit.
The heatmaps of \Cref{fig:boxplots-trA} illustrate
$\hat n$ in a simulation for a wide range of parameters.
Similarly, for smooth robust loss functions such as $\rho_k(u)=\sqrt{1+u^2}$,
the constraint $\frac 1 n \sum_{i=1}^n \psi_k'(r_{ki}) \ge \eta$
requires that at most a constant fraction of the $n$ observations are such that
$\psi_k'(r_{ki}) < \eta/2$, i.e., such that the second derivative $\psi_k'$ is too small (and the loss $\rho_k$ too flat).

\Cref{thm:differentiability,thm:residual-distribution,thm:out-of-sample,thm:df-trAtrV}
provide our general results applicable to a single regularized $M$-estimator
\eqref{hbeta}
while corollaries such as
\Cref{thm:selection-crit} are obtained using the union bound.
The next section specializes our results and notation to 
the Huber loss with Elastic-Net penalty and details
the simulation setup used in the figures.

\section{Example and simulation setting: Huber loss with Elastic-Net penalty}
\label{sec:simulations}

In simulations and in the example below, we focus on the loss-penalty pair
\begin{equation}
    \rho(u; \Lambda) = \Lambda^{2} H ( \Lambda^{-1} u ),
    \qquad
    g(\bb; \lambda, \tau) = \lambda \| \bb \|_{1}
                           + (\tau / 2) \| \bb \|_{2}^{2} 
                           \label{eq:choice-loss-penalty}
\end{equation}
for tuning parameters $\Lambda,\lambda,\tau\ge 0$ where $H(u):= u^{2} / 2$ for $ |u| \leq 1$ and $|u| -1/2$ for $|u|\geq1$.

\begin{example}
    \label{prop:huber-A-V}
    With $(\rho,g)$ in \eqref{eq:choice-loss-penalty}, matrix
    $\hbA$ in \eqref{eq:differentiability-formulae}
    matrix $\bV$ in \eqref{V} and
    $\df$ in \eqref{df} 
    we have 
    \begin{equation}
        \label{eq:huber-AVD}
    \begin{aligned}
        &\hbA_{\hat S, \hat S} = ( \bX_{\hat S}^{\top} \diag\{\psi'(\br)\} \bX_{\hat S} + n \tau \bI_{\hat p})^{-1},\quad
        A_{i,j} = 0 \text{ if $i \not \in \hat S$ or $j \not \in \hat S$,}
        \\
        &\bV = \diag\{\psi'(\br)\}
        -
            \diag\{\psi'(\br)\}
            \bX_{\hat S}
            \hbA_{\hat S, \hat S}
            \bX_{\hat S}^{\top}
            \diag\{\psi'(\br)\}
            ,
        \\
        &\df = \trace  [ \bX_{\hat S} 
            \hbA_{\hat S, \hat S}
        \bX_{\hat S}^{\top} \diag\{\psi'(\br)\} ],
    \end{aligned}
    \end{equation}
    where $\hat S$ is the active set $\{j \in [p] : \hat \beta _{j} \neq 0 \}$ and $\hat p $ is the size of $\hat S$;
    $\bX_{\hat S}$ is the submatrix of $\bX$ selecting columns with index in $\hat S$
    and 
    $\hbA_{\hat S, \hat S}$ is the submatrix of $\hbA$
    with entries indexed in $\hat S\times \hat S$.
\end{example}

\begin{table}[ht]
    \centering
    \setlength{\tabcolsep}{3pt}
    \begin{tabular}{llllllllll}
\toprule
 $(\lambda, \tau)$                         & $(0.036, 10^{-10})$ & $(0.054,0.01)$      & $(0.036,0.01)$      & $(0.024, 0.1)$      \\
 $\df/n$                                   & $0.31 \pm 0.012$    & $0.21 \pm 0.0095$   & $0.3 \pm 0.011$     & $0.37 \pm 0.0093$   \\
 $\hat{p}/n$                               & $0.31 \pm 0.012$    & $0.22 \pm 0.0098$   & $0.31 \pm 0.012$    & $0.47 \pm 0.014$    \\
 $\hat{n}/n$                               & $0.83 \pm 0.011$    & $0.76 \pm 0.014$    & $0.83 \pm 0.012$    & $0.84 \pm 0.012$    \\
 $\trace [\bSigma \bA]$                        & $0.58 \pm 0.039$    & $0.39 \pm 0.027$    & $0.58 \pm 0.038$    & $0.8 \pm 0.038$     \\
 $|\trace [\bSigma \bA] - \df/\trace[\bV]|$        & $0.0019 \pm 0.0015$ & $0.0015 \pm 0.0012$ & $0.0021 \pm 0.0016$ & $0.0023 \pm 0.0017$ \\
 $\|\bSigma^{1/2}(\hbbeta-\bbeta^*)\|^2$ & $1.3 \pm 0.18$      & $1.7 \pm 0.25$      & $1.3 \pm 0.19$      & $1.9 \pm 0.21$      \\
 $\zeta_1$                                 & $0.056 \pm 1$       & $0.021 \pm 1$       & $0.0044 \pm 1$      & $0.042 \pm 0.97$    \\
\bottomrule
\end{tabular}

    \vspace{2mm}
    \caption{Simulation for Huber Elastic-Net regression under different choices of $(\lambda, \tau)$. 
    $(n,p) = (1001,1000)$.
    For each choice of $(\lambda, \tau)$, 
    600 data points are simulated with anisotropic design matrix
    and i.i.d. $t$-distributed noises with 2 degrees of freedom.
A detailed setup is provided in \Cref{sec:simulations}.
    }
    \label{tab:table1}
\end{table}

The identities \eqref{eq:huber-AVD} are proved in \cite[\S2.6]{bellec2020out_of_sample}.
Simulations in \Cref{fig:qq,fig:boxplots-trA,fig:out-of-sample,tab:table1} 
illustrate typical values
for $\df,\trace[\bV],\trace[\bSigma\hbA]$,
the out-of-sample error and the criterion \eqref{crit},
$\hat n = \sum_{i=1}^n \psi'(r_i)$
and $\hat p =|\hat S|$
under anisotropic Gaussian design and heavy-tailed $\eps_i$.
The simulation setup is as follows.

\textbf{Data Generation Process.}
Simulation data are generated from a linear model $\by = \bX \bbeta^* + \bep $ 
with anisotropic Gaussian design $\bSigma$ and heavy-tail noise vector $\bep$.
The design matrix $\bX$ has $n = 1001$ rows 
and $p = 1000$ columns.
Each row of $\bX$ is i.i.d. $N(\bzero, \bSigma)$,
with the same $\bSigma$ across all repetitions, generated once
by $\bSigma = \bR^{\top}\bR / (2p)$
with $\bR \in \R^{2p \times p}$ being a Rademacher matrix with i.i.d. entries $\P(\bR_{ij} = \pm 1)=\frac12$.
The true signal vector $\bbeta^* \in \R^{p}$ has its first 100 coordinates set to $ p^{1/2} / 100 = \sqrt{10}/10$ and the rest 900 coordinates set to $0$.
The noise vector $\bep \in \R^{n}$ has i.i.d. entries from the t-distribution with 2 degrees of freedom (so that $\Var[\eps_i]=\infty$, i.e., $\eps_i$ is heavy-tailed).

\textbf{Estimation Process.}
Each dataset $(\by, \bX)$
is fitted by a Huber Elastic-Net estimator
with loss-penalty pair in \eqref{eq:choice-loss-penalty}.
We focus on 2d heatmaps with respect to the two penalty parameters
$(\lambda,\tau)$ of the penalty; to this end
the Huber loss parameter $\Lambda$ is set to $\Lambda = 0.054 n^{1/2}$
and a grid for $(\lambda,\tau)$ in then set so that
$\df/n$ varies on the grid from 0 to 1 (cf. the middle heatmap in
\Cref{fig:boxplots-trA}).
The Elastic-Net penalty $g(\bb;\lambda, \tau) = \lambda \| \bb \|_{1} + (\tau / 2) \| \bb \|_{2}^{2}$ 
is used with 
$(\lambda, \tau) \in \{ (0.036, 10^{-10}), (0.054, 0.01), (0.036, 0.01), (0.024, 0.1)\}$ in \Cref{fig:qq,tab:table1},
$(\lambda, \tau) \in [0.0032, 0.41] \times \{10 ^{-10}\}$ in \Cref{fig:boxplots-trA},
and $(\lambda, \tau) \in [0.0032, 0.041] \times [10^{-10}, 0.1]$ in \Cref{fig:out-of-sample}.
More simulation results are provided in the supplementary materials. 
In these simulations, the criterion \eqref{ALO} from \cite{rad2020scalable}
was also computed and was not noticeably different from \eqref{crit},
cf. the lower half of \Cref{fig6,fig7}.

\section{Relaxing the strong convexity assumption}

While previous results rely heavily on the $\mu$-strong convexity assumption
(with respcet to $\bSigma$, as stated in the last part of \Cref{assumMain}),
the proof of the following proposition presents a device that lets us generalize the results
under the following condition: For any $\bep$, there exists an open set
$U_{\bep}\subset\R^{n\times p}$ such that the mapping
\begin{equation}
    \label{Phi}
    \Phi_{\bep}:
    \left\{
\begin{aligned}
        U_{\bep} &\to \R^{n+p}, \qquad
        \\ \bX &\mapsto 
        \frac{
            \bigl(
                n^{-1/2}\psi(\by-\bX\hbbeta(\by,\bX)), ~
            \bSigma^{1/2}(\hbbeta(\by,\bX) -\bbeta )
            \bigr)
        }{
            (\frac1n \|\psi(\by-\bX\hbbeta(\by,\bX))\|^2
            + \|\bSigma^{1/2}(\hbbeta(\by,\bX) -\bbeta )\|^2
            )^{1/2}
        }
\end{aligned}
\right.
\text{is $\frac{L}{\sqrt n}$-Lipschitz.}
\end{equation}
In this definition, $\bep$ is held fixed
and $\hbbeta(\by,\bX)$ is the composition of $\hbbeta$
with the function $\bX\mapsto (\bep+\bX\bbeta, \bX)$
so that $\Phi_{\bep}$ is a function of $\bX$ only.
The following proposition shows that strong convexity on the penalty
function can be relaxed, provided that the above Lipschitz condition
holds and the expectations are restricted to the event $\{\bX \in U_{\bep}\}$.
\begin{proposition}
    \label{prop12}
    Let $L,\gamma>0$ be constants and assume $p/n\le\gamma$.
    Consider a convex differentiable loss $\rho:\R\to\R$
    such that $\psi=\rho'$ is 1-Lipschitz and a convex penalty $g$,
    assume $\bX$ has iid $N(\bzero,\bSigma)$ rows with 
    invertible $\bSigma$ and the noise $\bep$ is independent
    of $\bX$ with continuous distribution.
    Assume that for some open $U_{\bep}\subset\R^{n\times p}$,
    the Lipschitz condition \eqref{Phi} holds and that
    almost everywhere in $U_{\bep}$, the derivative formulae \eqref{eq:differentiability-formulae} hold for some matrix $\hbA\in\R^{p\times p}$
    satisfying $\|\bSigma^{1/2}\hbA\bSigma^{1/2}\|_{op}\le L/n$. Then
    \begin{align}
        \E\Bigl[I\{\bX\in U_\bep\}\Big|\frac{\trace[\bSigma \hbA]\trace[\bV]}{n}
        -\frac{\df}{n}\Big|\Bigr] &\le \frac{\C(\gamma,L) }{\sqrt n} \\
        \E\Bigl[\frac{I\{\bX\in U_\bep\}}{\Rem}
        \Big|
        \|\bSigma^{1/2}(\hbbeta-\bbeta^*)\|^2 + \frac{\|\bep\|^2}{n}
        - \frac{\big\| \br + \trace[\bSigma\hbA]\psi(\br)\big\|^2}{n}
        \Big|
        \Bigr]
                              &\le \frac{\C(\gamma, L)}{\sqrt n}
    \end{align}
    where $I\{\bX\in U_\bep\}$ is the indicator function of the event
    $\{\bX\in U_\bep\}$ and $\Rem$ is defined in
    \Cref{thm:out-of-sample}.

\end{proposition}
\Cref{prop12} is proved in \Cref{appendix:proof-prop12}.
Consequently, if the event $\{\bX\in U_\bep\}$ has high probability
and the Lipschitz condition \eqref{Phi} holds in this event,
the main results \Cref{thm:out-of-sample,thm:df-trAtrV} still hold,
with no strong convexity assumption on the penalty.
The proof relies on an application of Kirszbraun's theorem
already presented in \citep{bellec2020out_of_sample}.

The Lipschitz condition \eqref{Phi}
and inequality $\|\bSigma^{1/2}\hbA\bSigma^{1/2}\|_{op}\le L/n$ have
been proved
to hold in the regime $n\asymp p$ for covariance $\bSigma$ such that
$\bSigma_{jj}=1,\forall j\in[p]$
in the following two cases:
\begin{itemize}
\item The Lasso (i.e., square loss and L1 penalty) under the assumption
    that $\|\bbeta^*\ysc{\|}_0 \le s_* n$ for some \ysc{small enough} constant $s_*$;
\item The Huber Lasso (i.e., Huber Loss and L1 penalty)
    under the assumption
    that $\|\bbeta^* \ysc{\|}_0 \le s_* n$ for some small enough constant $s_*$,
    and that
    at least $(1-s_*)n$ components of the noise (the ``inliers") are
    iid standard normal;
\end{itemize}
cf. \citep[Assumption 2.3 and Proposition 12.1]{bellec2020out_of_sample}.
In both cases, the constant $s_*$ only depends on $\gamma$,
the condition number of $\bSigma$ and  the multiplicative constant 
of the tuning parameters.

Our results can thus be extended on a case-by-case basis for loss-penalty pairs
such that the Lipschitz condition \eqref{Phi} holds, for instance
for the two above examples.


\acks{P.C.B.'s research was partially supported by the NSF Grant DMS-1811976
and DMS-1945428 and by Ecole Des Ponts ParisTech.}

\bibliography{residuals}

\begin{thebibliography}{24}
\providecommand{\natexlab}[1]{#1}
\providecommand{\url}[1]{\texttt{#1}}
\expandafter\ifx\csname urlstyle\endcsname\relax
  \providecommand{\doi}[1]{doi: #1}\else
  \providecommand{\doi}{doi: \begingroup \urlstyle{rm}\Url}\fi

\bibitem[Bayati and Montanari(2012)]{bayati2012lasso}
Mohsen Bayati and Andrea Montanari.
\newblock The lasso risk for gaussian matrices.
\newblock \emph{IEEE Transactions on Information Theory}, 58\penalty0
  (4):\penalty0 1997--2017, 2012.

\bibitem[Bayati et~al.(2013)Bayati, Erdogdu, and
  Montanari]{bayati2013estimating}
Mohsen Bayati, Murat~A Erdogdu, and Andrea Montanari.
\newblock Estimating lasso risk and noise level.
\newblock In \emph{Advances in Neural Information Processing Systems}, pages
  944--952, 2013.

\bibitem[Bellec(2020)]{bellec2020out_of_sample}
Pierre~C Bellec.
\newblock Out-of-sample error estimate for robust m-estimators with convex
  penalty.
\newblock \emph{arXiv:2008.11840}, 2020.
\newblock URL \url{https://arxiv.org/pdf/2008.11840.pdf}.

\bibitem[Bellec and Zhang(2021)]{bellec_zhang2018second_stein}
Pierre~C Bellec and Cun-Hui Zhang.
\newblock Second-order stein: Sure for sure and other applications in
  high-dimensional inference.
\newblock \emph{The Annals of Statistics}, 49\penalty0 (4):\penalty0
  1864--1903, 2021.

\bibitem[Bellec and Zhang(2023)]{bellec_zhang2019second_poincare}
Pierre~C Bellec and Cun-Hui Zhang.
\newblock Debiasing convex regularized estimators and interval estimation in
  linear models.
\newblock \emph{The Annals of Statistics}, 51\penalty0 (2):\penalty0 391--436,
  2023.

\bibitem[Boucheron et~al.(2013)Boucheron, Lugosi, and
  Massart]{boucheron2013concentration}
St{\'e}phane Boucheron, G{\'a}bor Lugosi, and Pascal Massart.
\newblock \emph{Concentration inequalities: A nonasymptotic theory of
  independence}.
\newblock Oxford University Press, 2013.

\bibitem[Celentano and Montanari(2022)]{celentano2019fundamental}
Michael Celentano and Andrea Montanari.
\newblock Fundamental barriers to high-dimensional regression with convex
  penalties.
\newblock \emph{The Annals of Statistics}, 50\penalty0 (1):\penalty0 170--196,
  2022.

\bibitem[Celentano et~al.(2023)Celentano, Montanari, and
  Wei]{celentano2020lasso}
Michael Celentano, Andrea Montanari, and Yuting Wei.
\newblock The lasso with general gaussian designs with applications to
  hypothesis testing.
\newblock \emph{The Annals of Statistics}, 51\penalty0 (5):\penalty0
  2194--2220, 2023.

\bibitem[Chatterjee(2009)]{chatterjee2009fluctuations}
Sourav Chatterjee.
\newblock Fluctuations of eigenvalues and second order poincar{\'e}
  inequalities.
\newblock \emph{Probability Theory and Related Fields}, 143\penalty0
  (1):\penalty0 1--40, 2009.

\bibitem[Davidson and Szarek(2001)]{davidson2001local}
Kenneth~R Davidson and Stanislaw~J Szarek.
\newblock Local operator theory, random matrices and banach spaces.
\newblock \emph{Handbook of the geometry of Banach spaces}, 1\penalty0
  (317-366):\penalty0 131, 2001.

\bibitem[Dobriban and Wager(2018)]{dobriban2018high}
Edgar Dobriban and Stefan Wager.
\newblock High-dimensional asymptotics of prediction: Ridge regression and
  classification.
\newblock \emph{The Annals of Statistics}, 46\penalty0 (1):\penalty0 247--279,
  2018.

\bibitem[Donoho and Montanari(2016)]{donoho2016high}
David Donoho and Andrea Montanari.
\newblock High dimensional robust m-estimation: Asymptotic variance via
  approximate message passing.
\newblock \emph{Probability Theory and Related Fields}, 166\penalty0
  (3-4):\penalty0 935--969, 2016.

\bibitem[El~Karoui et~al.(2013)El~Karoui, Bean, Bickel, Lim, and
  Yu]{karoui2013robust}
Noureddine El~Karoui, Derek Bean, Peter~J Bickel, Chinghway Lim, and Bin Yu.
\newblock On robust regression with high-dimensional predictors.
\newblock \emph{Proceedings of the National Academy of Sciences}, 110\penalty0
  (36):\penalty0 14557--14562, 2013.

\bibitem[Mar{\v{c}}enko and Pastur(1967)]{marvcenko1967distribution}
Vladimir~A Mar{\v{c}}enko and Leonid~Andreevich Pastur.
\newblock Distribution of eigenvalues for some sets of random matrices.
\newblock \emph{Mathematics of the USSR-Sbornik}, 1\penalty0 (4):\penalty0 457,
  1967.

\bibitem[Miolane and Montanari(2021)]{miolane2018distribution}
L{\'e}o Miolane and Andrea Montanari.
\newblock The distribution of the lasso: Uniform control over sparse balls and
  adaptive parameter tuning.
\newblock \emph{The Annals of Statistics}, 49\penalty0 (4), 2021.

\bibitem[Pedregosa et~al.(2011)Pedregosa, Varoquaux, Gramfort, Michel, Thirion,
  Grisel, Blondel, Prettenhofer, Weiss, Dubourg, et~al.]{pedregosa2011scikit}
Fabian Pedregosa, Ga{\"e}l Varoquaux, Alexandre Gramfort, Vincent Michel,
  Bertrand Thirion, Olivier Grisel, Mathieu Blondel, Peter Prettenhofer, Ron
  Weiss, Vincent Dubourg, et~al.
\newblock Scikit-learn: Machine learning in python.
\newblock \emph{the Journal of machine Learning research}, 12:\penalty0
  2825--2830, 2011.

\bibitem[Pinelis(2021)]{390939}
Iosif Pinelis.
\newblock Large deviations: Growth of empirical average of iid non-negative
  random varialbes with infinite expectations?
\newblock MathOverflow, 2021.
\newblock URL \url{https://mathoverflow.net/q/390939}.
\newblock URL:https://mathoverflow.net/q/390939 (version: 2021-05-24).

\bibitem[Rad and Maleki(2020)]{rad2020scalable}
Kamiar~Rahnama Rad and Arian Maleki.
\newblock A scalable estimate of the out-of-sample prediction error via
  approximate leave-one-out cross-validation.
\newblock \emph{Journal of the Royal Statistical Society: Series B (Statistical
  Methodology)}, 82\penalty0 (4):\penalty0 965--996, 2020.

\bibitem[Rad et~al.(2020)Rad, Zhou, and Maleki]{rad2020error}
Kamiar~Rahnama Rad, Wenda Zhou, and Arian Maleki.
\newblock Error bounds in estimating the out-of-sample prediction error using
  leave-one-out cross validation in high-dimensions.
\newblock In \emph{International Conference on Artificial Intelligence and
  Statistics}, pages 4067--4077. PMLR, 2020.

\bibitem[Stein(1981)]{stein1981estimation}
Charles~M Stein.
\newblock Estimation of the mean of a multivariate normal distribution.
\newblock \emph{The annals of Statistics}, pages 1135--1151, 1981.

\bibitem[Stojnic(2013)]{stojnic2013framework}
Mihailo Stojnic.
\newblock A framework to characterize performance of lasso algorithms.
\newblock \emph{arXiv preprint arXiv:1303.7291}, 2013.

\bibitem[Thrampoulidis et~al.(2018)Thrampoulidis, Abbasi, and
  Hassibi]{thrampoulidis2018precise}
Christos Thrampoulidis, Ehsan Abbasi, and Babak Hassibi.
\newblock Precise error analysis of regularized $ m $-estimators in high
  dimensions.
\newblock \emph{IEEE Transactions on Information Theory}, 64\penalty0
  (8):\penalty0 5592--5628, 2018.

\bibitem[Xu et~al.(2021)Xu, Maleki, Rad, and Hsu]{xu2021consistent}
Ji~Xu, Arian Maleki, Kamiar~Rahnama Rad, and Daniel Hsu.
\newblock Consistent risk estimation in moderately high-dimensional linear
  regression.
\newblock \emph{IEEE Transactions on Information Theory}, 67\penalty0
  (9):\penalty0 5997--6030, 2021.

\bibitem[Ziemer(1989)]{ziemer2012weakly}
William~P Ziemer.
\newblock \emph{Weakly differentiable functions: Sobolev spaces and functions
  of bounded variation}, volume 120.
\newblock Springer-Verlag New York, 1989.
\newblock \doi{10.1007/978-1-4612-1015-3}.

\end{thebibliography}

\appendix

\section{Proof of the main results}

\textbf{Notation.}
For vectors in $\R^q$ or $\R^n$,
the Euclidean norm is $\|\cdot\|$
and $\|\cdot\|_q$ is the $\ell_q$-norm for $1\le q \le +\infty$.
For matrices, $\|\cdot\|_{op}$ is the operator norm (largest singular value),
$\|\cdot\|_F$ the Frobenius norm.
We use index $i$ only to loop or sum over $[n] =\{1,...,n\}$
and $j$ only to loop or sum over $[p]=\{1,...,p\}$,
so that $\be_i\in\R^n$ refers to the $i$-th
canonical basis vector in $\R^n$ and $\be_j\in\R^p$ 
the $j$-th canonical basis vector in $\R^p$.
Positive absolute constants are denoted $C_0, C_1, C_2,...,$,
constants that depend on $\gamma$ only are denoted $C_0(\gamma),C_1(\gamma),...$ and constant that depend on $\gamma,\mu$ only are denoted
by $C_0(\gamma,\mu), C_1(\gamma,\mu),\dots$
If $\bff:\R^q\to\R^n$ is differentiable at $\bz\in\R^q$,
we denote the Jacobian matrix in $\R^{n\times q}$ by
$\frac{\partial\bff}{\partial \bz}$
or $\partial\bff/\partial \bz$.
For an event $\Omega$, its indicator function is denoted by
$I_\Omega$ or $I\{\Omega\}$.

\textbf{Organization of the proofs.}
\Cref{sec:7} provides the proof of the main results from the main text
(\Cref{thm:normality,thm:residual-distribution,thm:out-of-sample,corSelection,corSelection2,thm:df-trAtrV,thm:selection-crit})
and the overall proof strategy.
\Cref{sec:more-lemmas} gives the proof of the probabilistic tools used
in \Cref{sec:7}.
\Cref{sec:9} proves the differentiability formulae in
\Cref{thm:differentiability,rem:intercept}.

\textbf{Additional simulations.}
Additional simulations and figures are given in \Cref{additional-figures}
for Gaussian designs and in \Cref{sec:rademacher-figures} for non-Gaussian 
Rademacher design. The simulations for Rademacher design suggests that
our results generalize to non-Gaussian design, although 
it is unclear at this point how to extend the proofs to non-Gaussian $\bX$.

\section{Proof of the main results}
\label{sec:7}

We perform the following change of variable
to reduce the anisotropic design regression problem to an isotropic one,
$\bG = \bX\bSigma^{-1/2}\in\R^{n\times p}$
a Gaussian matrix with iid $N(0,1)$ entries and
\begin{equation}
\bh(\bep,\bG) = \argmin_{\bu\in\R^p}
\frac1n\sum_{i=1}^n\rho(\eps_i - \be_i^\top\bG\bu) + g(\bbeta^* + \bSigma^{-1/2}\bu)
\label{bh}    
\end{equation}
and denote by $(h_j)_{j=1,...,p}$ the components of \eqref{bh}.
Then $\bSigma^{1/2}(\hbbeta(\by,\bX)-\bbeta^*) = \bh(\bep,\bX)$
with $\hbbeta(\by,\bX)$ the $M$-estimator in \eqref{hbeta}.
With $\by=\bG\bSigma^{1/2}\bbeta^* + \bep$,
by the chain rule and \eqref{eq:differentiability-formulae},
\begin{align*}
    &\bSigma^{-1/2}(\partial/\partial g_{ij})\bh(\bep,\bG)
\\&= (\partial/\partial g_{ij}) \hbbeta(\bG\bSigma^{1/2}\bbeta^* + \bep,\bG\bSigma^{1/2})
\\&=
\hbA\bX^\top\be_i \psi'(r_i)(\bSigma^{1/2}\bbeta^*)\be_j
+ \hbA\bSigma^{1/2}\be_j \psi(r_i)
- \hbA\bX^\top\be_i \psi'(r_i)(\bSigma^{1/2}\hbbeta)\be_j
\end{align*}
where $\be_i\in\R^n,\be_j\in\R^p$ denote canonical
basis vectors. 
Define $\bpsi(\bep,\bG) = \psi(\bep - \bG\bh)$ and let
$$\bA := \bSigma^{1/2}\hbA\bSigma^{1/2}.$$
Then we have
\begin{align}
(\partial/\partial g_{ij})\bh(\bep,\bG)
&= \bA\be_j \psi(r_i) 
- \bA\bG^\top \be_i \psi'(r_i) h_j
\label{dgij_h}
\\
(\partial/\partial g_{ij})\bpsi(\bep,\bG)
    &= - \diag\{\psi'(\br)\} \bG\bA\be_j \psi(r_i)
- \bV\be_i h_j
\label{dgij_psi}
\end{align}
where the second line follows by the chain rule
for Lipschitz functions in 
in \cite[Theorem 2.1.11]{ziemer2012weakly}.
The crux of the argument is that the
quantities of interest appearing in our results,
$\|\bh\|^2=\|\bSigma^{1/2}(\hbbeta-\bbeta^*)\|^2,
\|\psi(\br)\|^2$,
$\trace[\hbA\bSigma]=\trace[\bA],\trace[\bV]$ and $\df$
naturally appear from tensor contractions involving the derivatives
in \eqref{dgij_h}-\eqref{dgij_psi}. For instance, denoting
$\bD=\diag\{\psi'(\br)\}\in\R^{n\times n}$ if $h_j,\psi_i$
are the $j$-th and $i$-th component of \eqref{bh}
and $\bpsi(\bep,\bG)$
and denoting $\sum_{i=1}^n\sum_{j=1}^p$ by $\sum_{ij}$ for brevity,
\begin{align}
    \sum_{j=1}^p
    \frac{\partial h_j}{g_{ij}}
    &=
    \trace[\bA] \psi_i
    - {\color{purple}\bh^\top\bA\bG^\top\bD\be_i}
    \qquad \text{ for a given $i=1,...,n$},
    \label{contraction-1}
    \\
    \sum_{i=1}^n
    \frac{\partial \psi_i}{\partial g_{ij}}
    &=
    -{\color{purple}\bpsi^\top\bD\bG\bA\be_j}
    - \trace[\bV] h_j
    \qquad \text{ for a given $j=1,...,p$},
    \label{contraction-2}
    \\\sum_{ij}
    \frac{\partial (h_j \psi_i)}{g_{ij}}
    &= 
    \|\bpsi\|^2 \trace[\bA]
    - {\color{purple}\bh^\top\bG^\top\bD\bpsi}
    - {\color{purple}\bpsi^\top\bD\bG\bA\bh}
    -  \|\bh\|^2 \trace[\bV],
    \label{contraction-3}
\\
    \sum_{ij}
    \frac{\partial (h_j \be_i^\top\bG\bh)}{g_{ij}}
    &=
    \trace[\bA] \bpsi^\top\bG\bh
    -{\color{purple}\bh^\top\bA\bG^\top\bG\bh}
    + n\|\bh\|^2
    +
    {\color{purple}\bpsi^\top\bG\bA\bh}
    -
    \|\bh\|^2\df,
    \label{contraction-4}
    \\
    \sum_{ij}
    \frac{\partial (\psi_i \be_j^\top\bG^\top\bpsi)}{g_{ij}}
    &=
    - {\color{purple}\bpsi^\top\bD\bG\bA\bG^\top\bpsi}
    - \trace[\bV] \bpsi^\top\bG\bh
    - {\color{purple}\bh^\top\bG^\top\bV\bpsi}
     +(p-\df) \|\bpsi\|^2
    \label{contraction-5}
\end{align}
where we used that $\df=
\sum_{i=1}^n\be_i^\top\bG\bA\bG^\top\bD\be_i
=
\trace[\bG\bA\bG^\top\bD]$
in the fourth line
and $\df=\sum_{j=1}^p \be_j^\top\bG^\top\bD\bG\bA\be_j = \trace[\bG^\top\bD\bG\bA]$ in the fifth thanks to the commutation property of the trace.
The terms in colored purple indicate terms that will
be proved to be negligible later on. The probabilistic tool
that leads to asymptotic normality of the residuals is the following.

\begin{restatable}{proposition}{propVariantNormalApprox}
    [Variant of \cite{bellec_zhang2019second_poincare}]
    \label{prop:a-lemma}
    Let $\bz \in N(\bzero, \bI_{q})$ and $\bff := \bff (\bz) : \R ^{q} \to \R^{q} \setminus \{\bzero\}$ be locally 
    Lipschitz in $\bz$ with 
    $ 
    \E [ \| \bff \|^{-2}\sum_{k=1}^q \| \frac{\partial \bff}{\partial z_k} \|^{2} ] < +\infty.
    $
    Then 
    \begin{align}
        \label{eq:rhs-stein-normality}
        \E 
        \Bigl[
             \Bigl(
                 \frac{ \bff ^{\top} \bz - \sum_{k=1}^q(\partial/\partial z_k) f_k }{ \| \bff \|_{2} }
                  - 
                  Z
             \Bigr)^{2}
        \Bigr]
        \leq 
        ( 7  + 2 \sqrt{6} )
    \E \Bigl[ \| \bff \|^{-2}\sum_{k=1}^q \| \frac{\partial \bff}{\partial z_k} \|^{2} \Bigr] 
    < +\infty.
    \end{align}
\end{restatable}
\Cref{prop:a-lemma} is proved in \Cref{sec:more-lemmas}.
From here, asymptotic normality of the residuals in the square loss
case is readily obtained using the explicit formulae for the derivatives
and the contraction \eqref{contraction-1}. We start with the square loss
and the proof of \Cref{thm:residual-distribution}.

\begin{proof}[Proof of \Cref{thm:residual-distribution}]
    Apply \Cref{prop:a-lemma} with $q=p+1$
    and $\bz=(\bg_i,\eps_i/\sigma)\sim N(\mathbf0,\bI_{p+1})$
    conditionally on $(\bg_l,\eps_l)_{l\in[n]\setminus\{i}\}$,
    and with $\bff = (\bh,-\sigma)\in\R^{p+1}$.
    Note that the last component of $\bff$ is constant
    and $\|\bff\|^2 = \|\bh\|^2 + \sigma^2$.
    By \eqref{contraction-1} and $\bD=\bI_n$ for the square loss,
    $\trace[\partial \bff / \partial \bz]=
    \trace[\bA]\psi_i - \bh^\top\bA\bG^\top\be_i$
    and by symmetry in $i=1,...,n$, 
    $\E[|\bh^\top\bA\bG^\top\be_i|^2/\|\bff\|^2]
    =\frac 1 n \sum_{l=1}^n
    \E[|\bh^\top\bA\bG^\top\be_l|^2/\|\bff\|^2]
    = \frac 1 n\|\bG\bA^\top\bh\|^2/\|\bff\|^2]
    \le \frac 1 n \E[\|\bG\|_{op}^2\|\bA\|_{op}^2]
    \le n^{-2} \C(\gamma,\mu)$
    thanks to $\|\bA\|_{op}\le 1/(n\mu)$
    and $\E[\|\bG\|_{op}^2]\le \C(\gamma) n$.
    Similarly, for the square loss $r_i=\psi_i=\eps_i - \bg_i^\top\bh$
    and
    \begin{align*}
    \| \bff \|^{-1}
             \| \partial \bff / \partial \bz \|_{F}
             &=
             (\|\bh\|^2 + \sigma^2)^{-1/2}
             \|\bA\psi_i - \bA\bG^\top\be_i \bh^\top\|_F
             \\&\le
             \|\bA\|_{op}
             [
             \sqrt p |\eps_i|/\sigma + \sqrt p \|\bh\|^{-1}|\bg_i^\top\bh|
             + \|\bG\|_{op}].
    \end{align*}
    By the triangle inequality, $\|\bA\|_{op}\le 1/(n\mu)$
    and $p\le \gamma n$,
    $$\E[
    \| \bff \|^{-2}
             \| \partial \bff / \partial \bz \|_{F}^2
    ]^{1/2}
    \le \tfrac{\sqrt p}{n\mu}(
    \E[\eps_i^2/\sigma^2]^{1/2}
    + \E[(\bg_i^\top\bh)^2/\|\bh\|^2]^{1/2}
    )
    + \tfrac{1}{n\mu} \E[\|\bG\|_{op}^2]^{1/2}.
    $$
    By symmetry in $i=1,...,n$,
    $\E[(\bg_i^\top\bh)^2/\|\bh\|^2]
    =\frac1n\sum_{l=1}^n\E[(\bg_l^\top\bh)^2/\|\bh\|^2]
    \le \frac 1 n \E[\|\bG\|_{op}^2]$.
    Since $\frac 1 n \E[\|\bG\|_{op}^2] \le \C(\gamma)$,
    the right-hand side in the previous display is 
    bounded from above by $\C(\gamma,\mu) n^{-1/2}$.
    Since $\bff^\top\bz = -r_i$ we obtain
    $-r_i  - \trace[\bA]r_i = (\|\bh\|^2 + \sigma^2)^{1/2}(Z + O_P(n^{-1/2}))$
    which completes the proof
    of \eqref{eqSquareLossAsymptoticNormality}.
\end{proof}

\begin{proof}[Proof of \Cref{thm:normality}]
    Let $U\sim N(0,1)$ be independent of everything else.
    We apply the previous proposition with $\bz=(\bg_i,U)\sim N(\mathbf 0, \bI_{p+1})$ 
    conditionally on $(\bep,\bg_l, l\in[n]\setminus\{i\})$
    to $\bff = (\bh,n^{-1/4}\psi(\eps_i))$.
    Note that the last component of $\bff$
    is constant. By \eqref{contraction-1},
    $\trace[\partial \bff / \partial \bz]=
    \trace[\bA]\psi_i - \bh^\top\bA\bG^\top\bD\be_i$
    and by \eqref{dgij_h},
    \begin{align}
    \| \bff \|^{-1}
             \| \partial \bff / \partial \bz \|_{F}
             &=
             (\|\bh\|^2 + n^{-1/2}\psi(\eps_i)^2)^{-1/2}
             \|\bA\psi_i - \bA\bG^\top\bD\be_i \bh^\top\|_F
             \\&\le
             \|\bA\|_{op}
             [
             n^{1/4}\sqrt{p} + \sqrt p \|\bh\|^{-1}|\bg_i^\top\bh|
             + \|\bG\|_{op}]
    \end{align}
    where we used
    $\|\bA\|_F\le\sqrt{p} \|\bA\|_{op}$ and
    $|\psi_i|\le \psi(\eps_i) + |\bg_i^\top\bh|$
    thanks to $\psi$ being 1-Lipschitz.
    We have $\|\bA\|_{op} \le 1/(n\mu)$
    and $\E[\|\bh\|^{-2}|\bg_i^\top\bh|^2]
    = \frac 1 n \sum_{l=1}^n \E[\|\bh\|^{-2}|\bg_l^\top\bh|^2]
    = \frac 1 n \E[\|\bh\|^{-2}\|\bG\bh\|^2]
    \le \frac 1 n \E[\|\bG\|_{op}^2]$
    by symmetry in $i=1,...,n$,
    so that 
    $\E[
    \| \bff \|^{-2}
             \| \partial \bff / \partial \bz \|_{F}^2
             ]\le n^{-1/2}\C(\gamma,\mu)$.
    Thus by \Cref{prop:a-lemma},
    \begin{align*}
    (- r_i - \trace[\bA]\psi_i) + (\eps_i - \|\bh\|Z)
     &=
    \bg_i^\top\bh - \trace[\bA]\psi_i - \|\bh\|Z
      \\  &= 
      - U n^{-1/4}\psi(\eps_i) +
        [\|\bff\| -  \|\bh\|]Z
        +
        \|\bff\|\Rem
        - \bh^\top\bA\bG^\top\bD\be_i
    \end{align*}
    where $\E[\Rem^2]\le \C
    \E[\|\bff\|^{-2}
             \| \partial \bff / \partial \bz \|_{F}^2
             ]\le n^{-1/2}\C(\gamma,\mu)$.
    By properties of the operator norm and symmetry in $i=1,...,n$,
    \begin{equation}
    \E[\|\bh\|^{-2}|\bh^\top\bA\bG^\top\bD\be_i|^2]
    = \tfrac 1 n\E[\|\bh\|^{-2}\|\bD\bG\bA^\top\bh\|^2]
    \le \tfrac 1 n \E[\|\bG\|_{op}^2\|\bA\|_{op}^2]
    \le \tfrac{\C(\gamma,\mu)}{n^{2}}.
    \label{eq:purple-term-contraction-1}
    \end{equation}
    By the triangle inequality,
    $|\|\bff\|-\|\bh\||\le n^{-1/4}|\psi(\eps_i)|$
    so that the right-hand side is of the form
    $O_P(n^{-1/4})(|\psi(\eps_i)| + \|\bh\|)$
    as desired. The previous display can be rewritten as
    $
    r_i + \trace[\bA]\psi_i
    =\tilde \eps_i^n
    + \|\bh\| \tilde Z_i^n
    $
    for $$\tilde\eps_i^n = \eps_i
    + U n^{-1/4} \psi(\eps_i)
    - [\|\bff\| - \|\bh\|](Z + \Rem),
    \qquad
    \tilde Z_i^n = -Z -\Rem
    + \|\bh\|^{-1}\bh^\top\bA\bG^\top\bD\be_i.
    $$
    If $\eps_i$ has a fixed distribution $F$,
    then $|\psi(\eps_i)|\le |\psi(0)| + |\eps_i|
    =|\eps_i| = O_P(1)$ thanks to $\psi(0)=0$ and $\psi$ being 1-Lipschitz
    so that
    $(\tilde\eps_i^n,\tilde Z_i^n)
    = (\eps_i, -Z) + O_P(n^{-1/4})$.
    Since $(\eps_i, -Z)$ are independent,
    by Slutsky's theorem
    this proves that $(\tilde\eps_i^n,\tilde Z_i^n)$
    converges weakly to the product measure $F\otimes N(0,1)$.
\end{proof}

\begin{restatable}{proposition}{propSteinThree}
    \label{prop:Stein-3}
    Let $\bh:\R^{n\times p}\to\R^p$,
        $\bpsi:\R^{n\times p}\to\R^n$
        be locally Lipschitz functions.
        If $\bG\in\R^{n\times p}$ has iid $N(0,1)$ entries then
\begin{align}
    &\E\Bigl[
    \Bigl(
    \frac{\bpsi^\top\bG\bh-\sum_{ij}\frac{\partial(\psi_ih_j)}{g_{ij}}}{\|\bh\|^2 + \|\bpsi\|^2/n}
    \Bigr)^2
    +
    \Bigl(
    \frac{\|\bG\bh\|^2 - \sum_{ij}\frac{\partial(h_j\be_i^\top\bG\bh)}{g_{ij}}}{\|\bh\|^2 + \|\bpsi\|^2/n}
    \Bigr)^2
    +
    \Bigl(
    \frac{\|\bG^\top\bpsi\|^2 - \sum_{ij}\frac{\partial(\psi_i\be_j^\top\bG^\top\bpsi)}{g_{ij}}}{n\|\bh\|^2 + \|\bpsi\|^2}
    \Bigr)^2
    \Bigr]
    \nonumber
    \\
    &\le
    \C
    \E\Bigl[
    n+p+\|\bG\|_{op}^2
    +
    { (n+p) }
    \sum_{i=1}^n
    \sum_{j=1}^p
    \frac{1+\|\bG\|_{op}^2/n}{(\|\bh\|^2 + \|\bpsi\|^2/n)^2}
        \Bigl(
        \Big\|
        \frac{\partial \bh}{\partial g_{ij}}
        \Big\|^2
        +
        \frac1n
        \Big\|
        \frac{\partial \bpsi}{\partial g_{ij}}
        \Big\|^2
        \Bigr)
    \Bigr]
    \label{upper-bound-Stein}
\end{align}
        for some positive absolute constant in the second line.
\end{restatable}

\Cref{prop:Stein-3} is proved in \Cref{sec:more-lemmas}; 
it is a consequence
of \cite[Proposition 6.3]{bellec2020out_of_sample}.
By \Cref{prop:Stein-3} combined with
the identities \eqref{contraction-3}-\eqref{contraction-4}-\eqref{contraction-5}, and by showing
that the purple-colored terms in
\eqref{contraction-3}-\eqref{contraction-4}-\eqref{contraction-5}
are negligible, we obtain the following.

\begin{proposition}
    \label{prop:first-three-equations}
    Let \Cref{assumMain} be fulfilled.
    Then
    \begin{align}
        \label{eq:five-1}
        \E 
        \bigl[
            \bigl\{
            n^{-\frac12} 
            (\|\bh\|^2 + \|\bpsi\|^2/n)^{-1}&
            \bigl(
                \bpsi^{\top} \bG \bh  
                - \trace [ \bA ] \| \bpsi \|^{2} 
                + \trace [ \bV ] \| \bh \|^{2} 
            \bigr)
            \bigr\}^{2} 
        \bigr]
        \le\C(\gamma,\mu)
        ,
        \\
        \E
        \bigl[
            \bigl\{
            n^{-\frac12} 
            (\|\bh\|^2 + \|\bpsi\|^2/n)^{-1}&
            \bigl(
                \tfrac 1 n \|\bG^{\top} \bpsi\| ^{2}
                - \tfrac{p - 
                \df}{n}
                \| \bpsi \|^{2} 
                +
                \tfrac{\trace [ \bV ]}{n} \bpsi^{\top} \bG \bh
            \bigr)
            \bigr\}^{2}
        \bigr]
        \le\C(\gamma,\mu),
        \label{eq:five-2}
        \\
        \label{eq:five-3}
        \E \bigl[
            \bigl\{
            n^{-\frac12} 
            (\|\bh\|^2 + \|\bpsi\|^2/n)^{-1}&
            \bigl(
                \| \bG \bh \|^{2}
                 - 
                 \trace [ \bA ] \bpsi^{\top} \bG \bh 
                 - 
                 ( n - \df ) \| \bh \|^{2}
            \bigr)
            \bigr\}^{2}
        \bigr]
        \le\C(\gamma,\mu).
    \end{align}
\end{proposition}
\begin{proof}
    We bound from above the derivatives in \eqref{upper-bound-Stein}.
    For the norm of $(\partial/\partial g_{ij})\bh$ and
    $(\partial/\partial g_{ij})\bpsi$,
    by \eqref{dgij_psi}-\eqref{dgij_h} and $\frac 1 2(a+b)^2\le a^2+b^2$,
    \begin{align*}
        \sum_{ij}
        \frac12
        \Big\|
        \frac{\partial \bh}{\partial g_{ij}}
        \Big\|^2
        \le
        \|\bA\|_F^2 \|\bpsi\|^2
        + \|\bA\bG^\top\bD\|_F^2 \|\bh\|^2,
        ~
        \sum_{ij}
        \frac1{2n}
        \Big\|
        \frac{\partial \bpsi}{\partial g_{ij}}
        \Big\|^2
        \le
        \frac{\|\bD\bG\bA\|_F^2\|\bpsi\|^2 + \|\bV\|_F^2 \|\bh\|^2}{n}
        .
    \end{align*}
    Using $\|\bA\|_{op}\le 1/(n\mu)$, $\|\bD\|_{op}\le 1$, $p/n\le \gamma$
    and $\bV$ in \eqref{V},
    it follows that in \eqref{upper-bound-Stein} we have
    \begin{equation}
    \frac{1}{\|\bh\|^2 + \|\bpsi\|^2/n}
    \sum_{i=1}^n
    \sum_{j=1}^p
        \Bigl(
        \Big\|
        \frac{\partial \bh}{\partial g_{ij}}
        \Big\|^2
        +
        \frac1n
        \Big\|
        \frac{\partial \bpsi}{\partial g_{ij}}
        \Big\|^2
        \Bigr)
    \le
    \C(\gamma,\mu)
    \Bigl(1+  \frac{\|\bG\|_{op}^2}{n} \Bigr).
    \label{Xi-upper-bound}
    \end{equation}
    Since $\E[\|n^{-1/2}\bG\|_{op}^4]\le \C(\gamma)$
    \cite[Theorem II.13]{davidson2001local},
    this shows that \eqref{upper-bound-Stein} is bounded from above
    by $\C(\gamma,\mu) n$. The contractions appearing in
    the left-hand side of \eqref{upper-bound-Stein} are given in
    \eqref{contraction-3}-\eqref{contraction-4}-\eqref{contraction-5},
    so that it remains to bound from above the purple colored terms
    in these three equations. This is done by using the
    upper bounds on the operator norms
    $\|\bA\|_{op}\le 1/(n\mu)$, $\|\bD\|_{op}\le 1$
    and again that $\E[\|n^{-1/2}\bG\|_{op}^4]\le \C(\gamma)$,
    so that \eqref{upper-bound-Stein}
    yields the three inequalities in \Cref{prop:first-three-equations}.
\end{proof}

The next result is another probabilistic result
where the contractions in \eqref{contraction-1}-\eqref{contraction-2}
appear.

\begin{restatable}{proposition}{thmChi}
    \label{prop:chi2-custom}
    Let $\bh:\R^{n\times p}\to\R^p$,
    $\bpsi:\R^{n\times p}\to\R^n$
    be locally Lipschitz functions.
    If $\bG\in\R^{n\times p}$ has iid $N(0,1)$ entries then
    \begin{align*}
    &\E\Bigl[
    \frac{
    \big|
    \frac p n\|\bpsi\|^2 -
    \frac 1 n
    \sum_{j=1}^p
    \bigl(
        \bpsi^\top\bG\be_j - \sum_{i=1}^n\frac{\partial \psi_i}{\partial g_{ij}}
    \bigr)^2
    \big|
    }{\|\bh\|^2 + \|\bpsi\|^2/n}
    \Bigr]
    +
    \E\Bigl[
    \frac{
    \big|
    n \|\bh\|^2 - \sum_{i=1}^n
    \bigl(
        \bg_i^\top\bh - \sum_{j=1}^p\frac{\partial h_j}{\partial g_{ij}}
    \bigr)^2
    \big|
    }{\|\bh\|^2 + \|\bpsi\|^2/n}
    \Bigr]
    \\&\le
    \C \Bigl(
        \sqrt{n+p}
        (1+\Xi^{1/2}) + \Xi
    \Bigr)
    \text{ where }\Xi = 
    \E\Bigl[
    \frac{1}{\|\bh\|^2 + \|\bpsi\|^2/n}
    \sum_{i=1}^n
    \sum_{j=1}^p
        \Bigl(
        \Big\|
        \frac{\partial \bh}{\partial g_{ij}}
        \Big\|^2
        +
        \frac1n
        \Big\|
        \frac{\partial \bpsi}{\partial g_{ij}}
        \Big\|^2
    \Big)
        \Bigr].
    \end{align*}
\end{restatable}
\Cref{prop:chi2-custom} is proved in \Cref{sec:more-lemmas}; 
it is a consequence
of \cite[Theorem 7.1]{bellec2020out_of_sample}.
Using the contractions \eqref{contraction-1}-\eqref{contraction-2}
in the left-hand side of \Cref{prop:chi2-custom},
and by showing that the purple colored terms are negligible,
we obtain the following two inequalities.

\begin{proposition}
    \label{prop:next-two-equations}
    Let \Cref{assumMain} be fulfilled.
    Then
    \begin{align}
        \label{eq:five-4}    
        \E 
        \bigl|
        n^{-\frac12} (\|\bh\|^2 + \|\bpsi\|^2/n)^{-1}&
        \bigl(
            \tfrac p n \| \bpsi \|^{2}
            -
            \tfrac 1 n \| \bG^{\top} \bpsi + \trace[\bV] \bh \|^{2}
        \bigr)
        \bigr|
        \le \C(\gamma,\mu),
        \\
        \label{eq:five-5}
        \E 
        | 
        n^{-\frac12}
        (\|\bh\|^2 + \|\bpsi\|^2/n)^{-1}&
        \bigl(
            n \| \bh \|^{2} 
            - \| \bG \bh - \trace[\bA]\bpsi \|^{2} 
        \bigr)
        |
        \le \C(\gamma,\mu)
        .
    \end{align}
\end{proposition}
\begin{proof}
    For $\Xi$ in \Cref{prop:chi2-custom},
    the fact that $\Xi\le \C(\gamma,\mu)$ is already proved in
    \eqref{Xi-upper-bound}.
    For the first inequality we use
    \Cref{prop:chi2-custom} and
    the contraction \eqref{contraction-2}.
    To control the purple terms in \eqref{contraction-2}
    inside the left-hand side of \Cref{prop:next-two-equations},
    \begin{align*}
        &
    \Big|
     \sum_{j=1}^p(\bpsi^\top\bG\be_j - \sum_{i=1}^n\frac{\partial \psi_i}{\partial g_{ij}})^2
    -
    \|\bG^\top\bpsi + \trace[\bV]\bh\|^2
    \Big|
    =
    \Big|\bpsi^\top\bD\bG\bA
    \Bigl(
    2\bG^\top \bpsi + 2\trace[\bV] \bh
    + \bA^\top\bG^\top\bD\bpsi
    \Bigr)
    \Big|
  \\&\le
  (\|\bpsi\|^2/n + \|\bh\|^2)
  (
  2n\|\bG\|_{op}^2\|\bA\|_{op}
  +
  2\sqrt n \|\bG\|_{op}\|\bA\|_{op}
  +
  n \|\bA\|_{op}^2\|\bG\|_{op}^2
  )
    \end{align*}
    thanks to $|\trace\bV| \le n$ in \Cref{thm:differentiability}.
    With the bound obtained by multiplying
    the previous display by
    $n^{-3/2}(\|\bh\|^2 + \|\bpsi\|^2/n)^{-1}$,
    and using the previous bounds on $\|\bA\|_{op}$
    and $\E[\|n^{-1/2}\bG\|_{op}^2]$,
    we obtain \eqref{eq:five-4}
    from \Cref{prop:chi2-custom}
    and \eqref{contraction-2}.
    The second claim is obtained
    by \Cref{prop:chi2-custom},
    the contraction \eqref{contraction-1}
    and an argument similar to the previous display
    bound the purple term in \eqref{contraction-1}.
\end{proof}

We are now ready to prove \Cref{thm:df-trAtrV}.

\begin{proof}[Proof of \Cref{thm:df-trAtrV}]
    Define
    \begin{align*}
        \xi_I&=
                \bpsi^{\top} \bG \bh  
                - \trace [ \bA ] \| \bpsi \|^{2} 
                + \trace [ \bV ] \| \bh \|^{2} 
             &\text{(bounded in \eqref{eq:five-1})},
        \\\xi_{II}&=
                \tfrac 1 n \|\bG^{\top} \bpsi\| ^{2}
                - \tfrac{p - 
                \df}{n}
                \| \bpsi \|^{2} 
                +
                \tfrac{\trace [ \bV ]}{n} \bpsi^{\top} \bG \bh
                  &\text{(bounded in \eqref{eq:five-2})},
        \\\xi_{III}&=
                \| \bG \bh \|^{2}
                 - 
                 \trace [ \bA ] \bpsi^{\top} \bG \bh 
                 - 
                 ( n - \df ) \| \bh \|^{2}
                  &\text{(bounded in \eqref{eq:five-3})},
        \\\xi_{IV}&=
            \tfrac p n \| \bpsi \|^{2}
            -
            \tfrac 1 n \| \bG^{\top} \bpsi + \trace[\bV] \bh \|^{2}
                  &\text{(bounded in \eqref{eq:five-4})},
        \\\xi_{V}&=
            n \| \bh \|^{2} 
            - \| \bG \bh - \trace[\bA]\bpsi \|^{2} 
                  &\text{(bounded in \eqref{eq:five-5})}.
    \end{align*}
    Then by expanding the square in $\xi_{IV}$ and $\xi_V$
    and simple algebra
    (for instance by computing first $\xi_{II}+\xi_{IV}$
    and $\xi_{III}+\xi_{V}$ separately),
    $$
    (\trace[\bV]/n - \trace\bA)
    \xi_I
    +\xi_{II}
    +\xi_{III}
    +\xi_{VI}
    +\xi_{V}
    =
    (\|\bpsi\|^2/n + \|\bh\|^2)
    (\df - \trace[\bA]\trace[\bV]).
    $$
    Since $|\trace[\bV]/n\le 1$, $\trace[\bA]\le\gamma/\mu$
    by \Cref{thm:differentiability},
    the previous display divided by $n^{1/2}
    (\|\bpsi\|^2/n + \|\bh\|^2)$ and
    the bounds
    \eqref{eq:five-1},
    \eqref{eq:five-2}, 
    \eqref{eq:five-3}, 
    \eqref{eq:five-4} and
    \eqref{eq:five-5}
    complete the proof.
\end{proof}

To prove \Cref{thm:out-of-sample}, we need this extra proposition
whose proof is closely related to \Cref{prop:first-three-equations}.
\begin{restatable}{proposition}{propXiVI}
    \label{prop:xi_VI}
    Let \Cref{assumMain} be fulfilled.
    Then
    \begin{equation}
        \E\bigl[\bigl\{(\|\bh\|^2 + \|\bpsi\|^2/n)^{-\frac12} \|\bep\|^{-1} \xi_{VI} \bigr\}^2\bigr]\le
    \C(\gamma,\mu)
    \quad
    \text{ for }
    \quad
    \xi_{VI} = \bep^\top(\bG\bh - \trace[\bA]\bpsi)
    .
    \label{eq:xi_VI}
    \end{equation}
\end{restatable}

\Cref{prop:xi_VI} is proved in \Cref{sec:more-lemmas}. 
We are now ready to prove \Cref{thm:out-of-sample}.

\begin{proof}[Proof of \Cref{thm:out-of-sample}]
    We have
    $
    n \|\bh\|^2 + \|\bep\|^2
    -
    \|\br +\trace[\bA]\bpsi\|^2
    =
    \xi_V
    + 2\xi_{VI} 
    $
    by simple algebra and the definitions of $\xi_V$ and $\xi_{VI}$.
    Hence
    \begin{equation}
    \E\Bigl[
        \frac{
    \big|
    \|\bh\|^2 + \|\bep\|^2/n
    -
    \|\br +\trace[\bA]\bpsi\|^2/n
\big|}{
    \max\{\|\bh\|^2 + \|\bpsi\|^2/n, (\|\bh\|^2 + \|\bpsi\|^2/n)^{1/2}(\|\bep\|^2/n)^{1/2} \}
}
    \Bigr]
    \le  n^{-1/2}\C(\gamma,\mu)
    \label{eq:bound-xi_VI}
    \end{equation}
    thanks to \eqref{eq:xi_VI} and \eqref{eq:five-5}.
\end{proof}

\begin{proof}[Proof of \Cref{corSelection}]
    We perform the change of variable \eqref{bh}
    to $\tbbeta$ as well, giving $\tbh$ (the counterpart of $\bh$),
    $\tbpsi$ (counterpart of $\bpsi$) and
    $\tbA$ (counterpart of $\bA$).
    Let $\Omega$
    be the event defined
    in the theorem, i.e,
    \begin{equation}
        \Omega = \{
            \|\bG\|_{op}\le 2\sqrt n + \sqrt p
        \} \cap
        \{
            \|\bep\|^2 \le n^{2/(1+q)}
        \}.
    \end{equation}
    Then $\mathbb P (\Omega^c)\to 0$
    by \cite[Theorem 2.13]{davidson2001local}
    for the first event and \cite{390939} to show that 
    $\|\bep\|^2/ n^{2/(1+q)}\to^\P 0$
    under the assumption that $\E[|\eps_i|^{1+q}]$ is bounded.

    Under \Cref{assumAdditional}, $I_{\Omega}(\|\bpsi\|^2/n + \|\bh\|^2)$
    is bounded by a constant. Indeed,
    since the penalty $g$ is minimized at $\mathbf 0$,
    $(\hbbeta- \mathbf0)^\top\bX^\top\bpsi \in n (\hbbeta - \mathbf{0})^\top(\partial g (\hbbeta) - \partial g(\mathbf 0))$ since $\mathbf0\in\partial g(\mathbf 0)$. By strong convexity of $g$ in \Cref{assumMain},
    $(\hbbeta- \mathbf0)^\top\bX^\top\bpsi
    \ge \mu \|\bSigma^{1/2}\hbbeta\|^2$.
    In $\Omega$, this implies
    $\|\bSigma^{1/2}\hbbeta\| \le \frac{1}{\mu n} \|\bG\|_{op} \|\bpsi\|
    \le \C(\gamma,\mu) \|\bpsi\|/\sqrt n
    $
    and $\|\bpsi\|/\sqrt n \le M$ in \Cref{assumAdditional}.
    Since $\|\bSigma^{1/2}\bbeta^*\|^2\le M$ in \Cref{assumAdditional},
    this yields $I_\Omega(\|\bh\|^2+\|\bpsi\|^2/n) \le \C(\gamma,\mu,M)$ 
    and the same holds for $\tbh,\tbpsi$:
    $I_\Omega(\|\tbh\|^2+\|\tbpsi\|^2/n) \le \C(\gamma,\mu,M)$.

    Inequality \eqref{eq:bound-xi_VI} thus implies
    \begin{multline*}
    \E[I_{\Omega}
    (
    \big|
    \|\bh\|^2 + \|\bep\|^2/n
    -
    \|\br +\trace[\bA]\bpsi\|^2/n
\big|
+
    \big|
    \|\tbh\|^2 + \|\bep\|^2/n
    -
    \|\tbr +\trace[\tbA]\tbpsi\|^2/n
\big|
)
]
\\
\le \C(\gamma,\mu,M)(n^{-1/2} \vee n^{-q/(1+q)}).
    \end{multline*}
Since $q\in (0,1)$ we have $n^{-1/2} \vee n^{-q/(1+q)} = n^{-q/(1+q)}$
in the right-hand side.
    Let $\hat\Omega = \{\|\bh\|^2 - \|\tbh\|^2>\eta,
        \|\br+\trace[\bA]\bpsi\|^2\le
        \|\tbr+\trace[\tbA]\tbpsi\|^2
    \}$ be the event for which we are trying to control the probability.
    By the triangle inequality,
    $$
    \E[I_{\Omega}
    \big|
    \|\bh\|^2 - \|\tbh\|^2 
    -
    \|\br +\trace[\bA]\bpsi\|^2/n
    +
    \|\tbr +\trace[\tbA]\tbpsi\|^2/n
\big|
]
\\
\le \C(\gamma,\mu,M) n^{-q/(1+q)}.
$$
In $\hat\Omega$, the random variable in the expectation sign is
larger than
$\eta I_{\Omega}$. Thus
$\eta \E[I_{\Omega}I_{\hat\Omega}] \le \C(\gamma,\mu,M) n^{-q/(1+q)}$
and $\P(\hat\Omega) \le \eta^{-1}\C(\gamma,\mu,M) n^{-q/(1+q)} +  \P(\Omega^c)$.
\end{proof}

\begin{proof}[Proof of \Cref{corSelection2}]
    We follow the same strategy. Let $\Omega$ be the same event
    as in the previous proof, so that $\P(\Omega^c)\to 0$
    as before.
    We perform the change of variable \eqref{bh}
    for each $k=1,...,K$  giving $\bh_k$,
    $\bpsi_k$ and $\bA_k$.
    We have $I_\Omega\max_{k=1,...,K}(\|\bh_k\|^2 + \|\bpsi_k\|^2/n)
    \le \C(\gamma,\mu,M)$ as explained in the previous proof.

    Summing over $k$ the inequality \eqref{eq:bound-xi_VI} gives
    $\E[I_\Omega \sum_{k=1}^K |\|\bh_k\|^2 + \|\bep\|^2 - \|\br_k +\trace[\bA_k]\bpsi_k\|^2| ] \le K \C(\gamma,\mu, M) n^{-q/(1+q)}$.
    Let $\hat k$ be the minimizer
    of $\|\br_k + \trace[\bA_k] \bpsi_k\|^2$
    as defined in the statement
    of \Cref{corSelection2} and let $\tilde k\in \{1,...,K\}$
    be such that
    $\|\bh_{\hat k}\|^2 \ge \|\bh_{\tilde k}\|^2 + \eta$
    in the event $\tilde \Omega$ where such $\tilde k$ exists,
    so that
    $$
    \eta I_{\tilde \Omega}
    \le
    I_{\tilde \Omega}
    \bigl(\|\bh_{\hat k}\|^2-\|\bh_{\tilde k}\|^2\bigr)
    \le
    I_{\tilde \Omega}
    \Bigl(\|\bh_{\hat k}\|^2 + 
        \tfrac1n
    \Bigl[\|\bep\|^2-\|\br_{\hat k}+\trace[\bA_{\hat k}]\bpsi_{\hat k}\|^2
    + \|\br_{\tilde k}+\trace[\bA_{\tilde k}]\bpsi_{\tilde k}\|^2 - \|\bep\|^2
    \Bigr]
    -\|\bh_{\tilde k}\|^2
    \Bigr)
    .
    $$
    Then by the triangle inequality,
    $\eta\E[I_\Omega I_{\tilde \Omega} ]
    \le \pb{K} \C(\gamma,\mu,M) n^{-q/(1+q)}$.
    It follows that
    $\P(\tilde \Omega) \le \pb{K} \eta^{-1}\C(\gamma,\mu,M)n^{-q/(1+q)} +
    \P(\Omega^c) \to 0$ as desired.
\end{proof}

\begin{proof}[Proof of \Cref{thm:selection-crit}]
    Using $\|\ba\|^2 - \|\bb\|^2 = (\ba-\bb)^\top(\ba+\bb)$ we have
    \begin{align*}
    \| \bG \bh - \trace[\bA]\bpsi \|^{2}
    - \|\bG\bh - (\df/\trace[\bV])\|^2
    &= (\df/\trace[\bV] - \trace[\bA]) \bpsi^\top(2\bG\bh - (\trace[\bA]+\df/\trace[\bV])\bpsi).
    \end{align*}
    Hence using $|\trace[\bA]|\le \gamma/\mu$,
    $|\df|\le n$ and the Cauchy-Schwarz inequality
    \begin{align*}
    &|\| \bG \bh - \trace[\bA]\bpsi \|^{2}
    - \|\bG\bh - (\df/\trace[\bV])\|^2|
    \\&\le
    \C(\gamma,\mu)(\tfrac{n}{\trace\bV}\vee 1)
    | \df/n - \trace[\bV]\trace[\bA]/n| (\|\bpsi\|^2 + \|\bG\|_{op}\|\bh\|^2).
    \end{align*}
    Let $\Omega$ be the event in \Cref{corSelection}.
    Using the bound on the operator norm of $\bG$ in $\Omega$,
    for any deterministic $\eta>0$ we have proved
    $$\E\Bigl[
    I\{\Omega\}
    I\{{\trace[\bV]}{n} \ge \eta\}
    \frac{|\| \bG \bh - \trace[\bA]\bpsi \|^{2}
    - \|\bG\bh - (\df/\trace[\bV])\|^2|}{
    \|\bh\|^2 + \|\bpsi\|^2/n}
    \Bigr]
    \le
    \frac{\C(\gamma,\mu)}{\eta\wedge 1} n^{1/2}
    $$
    thanks to \Cref{thm:df-trAtrV}.
    By \eqref{eq:lower-bond-trace-V}, in the event $\Omega$
    where the operator norm of $\|n^{-1/2}\bG\|_{op}$ is bounded
    by a constant,
    $\trace[\bV] \ge \trace[\diag\{\psi'(\br)\}]/\C(\gamma,\mu)$.
    Hence combining the previous display with \eqref{eq:bound-xi_VI},
    we have proved
    \begin{equation*}
    \E\Bigl[
        \frac{
        I\{\Omega\}
        I\{\sum_{i=1}^n\psi'(r_i)\ge n\eta\}
    \big|
    \|\bh\|^2 + \|\bep\|^2/n
    -
    \|\br +\frac{\df}{\trace\bV}\bpsi\|^2/n
\big|}{
    \max\{\|\bh\|^2 + \|\bpsi\|^2/n, (\|\bh\|^2 + \|\bpsi\|^2/n)^{1/2}(\|\bep\|^2/n)^{1/2} \}
}
    \Bigr]
    \le  \frac{\C(\gamma,\mu,\eta)}{\sqrt n}.
    \end{equation*}
    At this point the proof is similar to that of \Cref{corSelection2}:
    We perform the change of variable \eqref{bh}
    for each $k=1,...,K$  giving $\bh_k$,
    $\bpsi_k$, $\df_k$ and $\bV_k$.
    We have $I_\Omega\max_{k=1,...,K}(\|\bh_k\|^2 + \|\bpsi_k\|^2/n)
    \le \C(\gamma,\mu,M)$ as explained in the previous proofs.
    Summing over $k=1,...,K$ the previous display,
    using $I_\Omega\max_{k=1,...,K}(\|\bh_k\|^2 + \|\bpsi_k\|^2/n)
    \le \C(\gamma,\mu,M)$
    and $I_{\Omega} \|\bep\|^2\le  n^{2/(1+q)}$ we find
    \begin{equation*}
    \E\Bigl[
        \sum_{k=1}^K
        I\{\Omega\}
        I\{\sum_{i=1}^n\psi_k'(r_{ki})\ge n\eta\}
    \Big|
    \|\bh_k\|^2 + \|\bep\|^2/n
    -
    \|\br_k +\frac{\df_k}{\trace\bV_k}\bpsi_k\|^2/n
\Big|
    \Bigr]
    \le  \frac{K \C(\gamma,\mu,\eta)}{n^{q/(1+q)}}.
    \end{equation*}
    Let $\tilde \Omega$ be the event that there exists
    $\tilde k$ with $\frac1n\sum_{i=1}^n\psi_{\tilde k}'(r_{\tilde ki})\ge\eta$
    satisfying $\|\bh_{\tilde k}\|^2+\tilde\eta \le \|\bh_{\hat k}\|^2$,
    then by the previous display and the triangle inequality,
    using
    $\|\br_{\hat k} +\frac{\df_{\hat k}}{\trace\bV_{\hat k}}\bpsi_{\hat k}\|^2
    \le
    \|\br_{\tilde k} +\frac{\df_{\tilde k}}{\trace\bV_{\tilde k}}\bpsi_{\tilde k}\|^2$ by definition of $\hat k$,
    we obtain
    $\tilde\eta \P(I_\Omega I_{\tilde\Omega}) = O(K/n^{q/(1+q)})$.
    Since $\tilde\eta$ is a constant independent of $n,p$
    and $\P(\Omega)\to 1$, the probability $\P(\tilde \Omega)$
    converge to 0 if $K=o(n^{q/(1+q)})$.
\end{proof}

\section{Probabilistic results and their proofs}
\label{sec:more-lemmas}

\propVariantNormalApprox*

\begin{proof}
    Let $ \bg := \bg(\bz) = \frac{\bff(\bz)}{\|\bff(\bz)\|} - \E[
    \frac{\bff(\bz)}{\|\bff(\bz)\|}
    ]$
    and set
    $$Z= \bz^\top \E[
        \frac{\bff(\bz)}{\|\bff(\bz)\|}
        ] \Big/ \sqrt V,
    \qquad
    V=\bigl\|\E\Bigl[
        \frac{\bff(\bz)}{\|\bff(\bz)\|}
    \Bigr]\bigr\|^2
    $$
    so that $Z\sim N(0,1)$ and $V$ is deterministic
    with $V\le 1$ by Jensen's inequality.
    As a first step, we proceed to prove inequality
    \begin{equation}
        \label{eq:6470}
    \E\Bigl[
    \Bigl(
        \frac{\bff^\top \bz -
        \sum_{k=1}^q(\partial/\partial z_k) f_k
        }{\|\bff\|_{2}}
        - \sqrt{V} Z
    \Bigr)^2
\Bigr]
    \le
    6 ~
    \E \Bigl[ \| \bff \|^{-2}\sum_{k=1}^q \| \frac{\partial \bff}{\partial z_k} \|^{2} \Bigr].
    \end{equation}
    Then at any point $\bz$ where $\bff$ is differentiable we have
    $$
        \frac{ \partial \bg }{ \partial z_k }
        = \|\bff(\bz)\|^{-1} \hbP 
        \frac{ \partial \bff }{ \partial z_k },
        \qquad
        \text{ where }
        \qquad
        \hbP = \bI_q - \frac{\bff \bff ^\top}{\| \bff \|^2}.
    $$
    This implies that almost surely,
    $$
        \frac{
            \bff^T \bz 
        -
        \sum_{k=1}^q(\partial/\partial z_k) f_k
    }{\| \bff \|_{2}}
    - \sqrt{V} Z
    =
    \bg ^T \bz -
    \sum_{k=1}^q(\partial/\partial z_k) g_k
    -
    \frac{\bff^T (\partial \bff / \partial \bz) \bff}{\|\bff\|^{3}}
     $$
    where $\partial \bff / \partial \bz$ is the matrix with entries
    $(l,k)$ entry 
    $(\partial/ \partial z_k  )  f_l$
    for all, $k,l=1,...,q$.

    By the triangle inequality 
    and $(a+b)^2\le 2a^2 + 2b^2$,
    this implies that the left-hand side of \eqref{eq:6470}
    is bounded from above by
    $
    2\E[(\bz^T \bg - \trace[ \partial \bg / \partial \bz ])^2]
    +
    2\E[\|\bff\|^{-2}\|\partial \bff / \partial \bz \|_{F}^{2}]$.
    The first term can be bounded using the main result
    of \cite{bellec_zhang2018second_stein}
    and the Gaussian Poincar\'e inequality
    \cite[Theorem 3.20]{boucheron2013concentration}
    \begin{align}
        \E[(\bz^T \bg - \trace[ \partial \bg / \partial \bz ])^2]
         &=  
         \E[\|\bg\|^2]  + \E\trace[(\partial \bg / \partial \bz)^2]
         \le 
         2 \E[\| \partial \bg / \partial \bz \|_F^2].
         \nonumber
    \end{align}
    This proves
    \eqref{eq:6470}.
    To bound $|\sqrt{V} - 1|$, we have by the triangle inequality
    $$
    |\sqrt{V} - 1|
    =
    |\sqrt{V} - \bigl\| \tfrac{\bff}{ \| \bff \|} \bigr\| |
    \le
    \big\|\E[\tfrac{ \bff}{ \| \bff \|}] - \tfrac{\bff}{\|\bff\|} \big\|
    =
    \| \bg \|.
    $$
    By another application of the Gaussian Poincar\'e inequality,
    \begin{equation}
        \label{eq:8507}
        |\sqrt{V} - 1|^2
        \leq
        \E[\| \bg \|_{2}^{2}]
        \leq
        \E[\| \partial \bg / \partial \bz \|_F^2]
        \leq
        \E[\| \bff \|^{-2}\| \partial \bff / \partial \bz \|_F^2].
    \end{equation}
    Combining \Cref{eq:8507,eq:6470} using 
    $ 
        (a + b)^{2} = a^{2} + 2 ab  + b^{2} \leq a^{2} + 1 / \sqrt{6} a^{2} + \sqrt{6} b^{2} + b^{2},
    $
    we obtain the constant $ 7 + 2 \sqrt{6}$.

\end{proof}

\propSteinThree*
\begin{proof} [Proof of \Cref{prop:Stein-3}]
    We prove the claim separately for the three terms in the left-hand side
    of \Cref{prop:Stein-3}; we start with the first of the three terms.
    We will apply the probabilistic result given in
    Proposition 6.3 in \cite{bellec2020out_of_sample}:
    if $\bfeta:\R^{n\times p}\to\R^p$ and $\brho:\R^{n\times p}\to\R^n$
    are locally Lipschitz and $\bG\in\R^{n\times p}$ has iid $N(0,1)$
    entries,
    \begin{equation}
    \E\Bigl[\Bigl(\brho^\top\bG\bfeta - \sum_{ij} \frac{\partial (\rho_i\eta_j)}{g_{ij}}\Bigr)^2\Bigr]
    \le
    \E\Bigl[\|\brho\|^2 \|\bfeta\|^2\Bigr]
    +
    2\E\Bigl[\sum_{ij}
    \|\bfeta\|^2 \| \frac{\partial \brho}{\partial g_{ij}}\|^2
    +
    \|\brho\|^2 \| \frac{\partial \bfeta}{\partial g_{ij}}\|^2
    \Bigr].
    \label{eq:prop63}
    \end{equation}
    The proof only relies on Gaussian integration by parts
    to transform the left-hand side.
    Let $\bff:\R^{n\times p}\to \R^{n+p}$ be locally Lipschitz. For any $i,j$
    and at a point where both $\bh$ and $\bpsi$ are differentiable
    and $\bff\ne \mathbf0$,
    \begin{equation*}
    \frac{\partial}{\partial g_{ij}}
    \Bigl(\frac{\bff}{\|\bff\|}\Bigr)
    =
    \frac{1}{\|\bff\|}
    \Bigl(\bI_{n+p} - \frac{\bff\bff^\top}{\|\bff\|^2}\Bigr)
    \frac{\partial \bff}{\partial g_{ij}}
    \quad
    \text{ so that }
    \Big\|
    \frac{\partial}{\partial g_{ij}}
    \Bigl(\frac{\bff}{\|\bff\|}\Bigr)
    \Big\|^2
    \le \frac{1}{\|\bff\|^2}
    \Big\|
    \frac{\partial \bff}{\partial g_{ij}}
    \Big\|^2.
    \end{equation*}
    We use this inequality applied with
    \begin{equation}
        \label{bff-brho-bh}
    \bff=(\bh,\tfrac{1}{ \sqrt{n} } \bpsi),\qquad
    \brho = \tfrac{1}{\sqrt n} \tfrac{\bpsi}{\|\bff\|},\qquad
    \bfeta = \tfrac{\bh }{\|\bff\|}.
    \end{equation}
    To bound from above 
    the right-hand side of \eqref{eq:prop63},
    the inequality
    in the previous display can be rewritten
    \begin{equation}
    \|
    \frac{\partial \bfeta}{\partial g_{ij}}
    \|^2
    +
    \|
    \frac{\partial \brho}{\partial g_{ij}}
    \|^2
    \le
    \frac{1}{\|\bh\|^2 + \|\bpsi\|^2/n}
    \Bigl(
    \|\frac{\partial \bh}{\partial g_{ij}}\|^2
    +
    \frac 1 n \|\frac{\partial \bpsi}{\partial g_{ij}}\|^2
    \Bigr).
    \label{eq:argument-bff}
    \end{equation}
    Since $\|\brho\|\le 1$ and $\|\bfeta\|\le 1$ by definition,
    the right-hand side of \eqref{eq:prop63}
    is bounded from above by
    $1+2 \E[
    \frac{1}{\|\bh\|^2 + \|\bpsi\|^2/n}
    (
    \|\frac{\partial \bh}{\partial g_{ij}}\|^2
    +
    \frac 1 n \|\frac{\partial \bpsi}{\partial g_{ij}}\|^2
    )
    ]$.
    Thus the proof of \Cref{prop:Stein-3} for the first term in the left-hand
    side is almost complete; it remains to control
    inside the parenthesis of the left-hand side,
    $$
    \sum_{ij}
    \frac{1}{\|\bh\|^2+\|\bpsi\|^2/n}\frac{\partial(\psi_i n^{-1/2} h_j)}{\partial g_{ij}}
    -
    \frac{\partial}{\partial g_{ij}}
    \Bigl(
\frac{
\psi_i n^{-1/2} h_j
}{\|\bh\|^2+\|\bpsi\|^2/n}
    \Bigr)
    =
    2
    \sum_{ij}
    \psi_i n^{-1/2}h_j
    \frac{  
    \bh^\top\frac{\partial\bh}{\partial g_{ij}}
    +
    \frac 1 n\bpsi^\top\frac{\partial\bpsi}{\partial g_{ij}}
    } {(\|\bh\|^2+\|\bpsi\|^2/n)^2}.
    $$
    By multiple applications of the Cauchy-Schwartz inequality, the absolute
    value of the previous display is bounded from above by
    $2(\|\bh\|^2+\|\bpsi\|^2/n)^{-1/2}
    (
    \sum_{ij}
    \|\frac{\partial\bh}{\partial g_{ij}}\|
    +\frac1n\|\frac{\partial\bpsi}{\partial g_{ij}}\|
    )^{1/2}
    $.
    This completes the proof of \Cref{prop:Stein-3}
    for the first term in the left-hand side.

    For the second and third term in the left-hand side of \Cref{prop:Stein-3},
    apply instead \eqref{eq:prop63} to
    $\brho=\bG\bfeta$ and  $\bfeta = \bG^\top\brho$ to obtain
    \begin{equation*}
    \E\Bigl[\Bigl(\|\bG\bfeta\|^2 - \sum_{ij} \frac{\partial ( \eta_j\be_i^\top\bG\bfeta)}{g_{ij}}\Bigr)^2\Bigr]
    \le
    \E\Bigl[\|\bG\bfeta\|^2 \|\bfeta\|^2\Bigr]
    +
    2\E\Bigl[\sum_{ij}
    \|\bfeta\|^2
    \|
    \be_i\eta_j
    + \bG\frac{\partial \bfeta}{\partial g_{ij}}
    \|^2
    +
    \|\bG\bfeta\|^2 \| \frac{\partial \bfeta}{\partial g_{ij}}\|^2
    \Bigr],
    \end{equation*}
    \begin{equation*}
    \E\Bigl[\Bigl(\|\bG^\top\brho\|^2 - \sum_{ij} \frac{\partial ( \rho_i \brho^\top\bG\be_j)}{g_{ij}}\Bigr)^2\Bigr]
    \le
    \E\Bigl[\|\bG^\top\brho\|^2 \|\brho\|^2\Bigr]
    +
    2\E\Bigl[\sum_{ij}
    \|\bG^\top\brho\|^2
    \|
    \frac{\partial \brho}{\partial g_{ij}}
    \|^2
    +
    \|\brho\|^2 \|\be_j\rho_i+ \bG^\top\frac{\partial \brho}{\partial g_{ij}}\|^2
    \Bigr].
    \end{equation*}
    Setting $\brho = \frac{1}{\sqrt n} \bpsi / \|\bff\|$,
    $\bfeta = \bh /\|\bff\|$ we obtain
    the claim in \Cref{eq:prop63} by bounding the right-hand side
    of the previous displays using the operator norm of $\bG$
    and arguments similar to \eqref{eq:argument-bff}.
    The term involving
    $\frac{\partial}{\partial g_{ij}}
\bigl( 
\frac{ 1 }{\|\bh\|^2+\|\bpsi\|^2/n}
\bigr)$ in the left-hand side is controlled similarly to
the previous paragraph.

\end{proof}

\thmChi*

\begin{proof} [Proof of \Cref{prop:chi2-custom}]
    We first focus on the first term in the left-hand side.
    Theorem 7.1 in \cite{bellec2020out_of_sample}
    provides that if $\brho:\R^{n\times p}\to\R^n$ is locally Lipschitz
    with $\|\brho\|\le 1$ then
    \begin{equation}
        \E\Big| p \|\brho\|^2 - \sum_{j=1}^p
        \Bigl(\brho^\top\bG\be_j -\sum_{i=1}^n \frac{\partial \rho_i }{\partial g_{ij}}\Bigr)^2\Big|
        \le \C \sqrt p \Bigl( 1+ \E\sum_{ij} \Big\|\frac{\partial \brho}{\partial g_{ij}}\Big\|^2\Bigr)^{1/2}
        +\C \E\sum_{ij} \Big\|\frac{\partial \brho}{\partial g_{ij}}\Big\|^2.
        \label{eq:cor71}
    \end{equation}
    Let $\brho = n^{-1/2}\bpsi / \|\bff\|$ as in \eqref{bff-brho-bh}.
    Inequality \eqref{eq:argument-bff}
    lets us bound from above the right-hand side of the previous display
    by the right-hand side of
    \Cref{prop:chi2-custom}.
    In the left-hand side, $p\|\brho\|^2 = \frac pn \|\bpsi\|^2 / (\|\bh\|^2 + \|\bpsi\|^2/n)$ as desired.
    For the left-hand side, using some algebra
    in \cite[Section 7]{bellec2020out_of_sample},
    for any random vectors $\ba,\bb\in\R^p$
    by the triangle and Cauchy-Schwarz inequalities we have
    \begin{align*}
    |p\|\brho\|^2 - \|\ba\|^2|
    -|p\|\brho\|^2 - \|\bb\|^2|
    &\le 
    \|\ba-\bb\| \|\ba + \bb\|
  \\&\le \|\ba-\bb\|^2 + 2\|\ba-\bb\|\|\bb\|
  \\&\le \|\ba-\bb\|^2 + 2\|\ba-\bb\|(\sqrt{|\|\bb\|^2-p\|\brho\|^2|}+\sqrt{p\|\brho\|^2})
  \\&\le 3\|\ba-\bb\|^2 + \tfrac 1 2|\|\bb\|^2 - p\|\brho\|^2|
  + 2 \|\ba-\bb\| \sqrt{p\|\brho\|^2}
    \end{align*}
    so that
    $|p\|\brho\|^2 - \|\ba\|^2|
    \le\frac32 |p\|\brho\|^2 - \|\bb\|^2|
    + 3 \|\ba-\bb\|^2
    + 2 \|\ba-\bb\| \sqrt{p\|\brho\|^2}.
    $
    Applying this to $b_j = \brho^\top\bG\be_j - \sum_{i=1}^n\frac{\partial \rho_i}{\partial g_{ij}}$
    we use \eqref{eq:cor71}
    to bound $|p\|\brho\|^2 - \|\bb\|^2|$
    and $\|\brho\|\le 1$ to bound
    $\sqrt{p\|\brho\|^2}\le \sqrt p$. It remains to specify
    $\ba$ so that $|p\|\brho\|^2 - \|\ba\|^2|$ coincides
    with the first term in the left-hand side of \Cref{prop:chi2-custom}
    and bound $\|\ba-\bb\|$. Consequently, we set
    $$a_j =\frac{  
    \bpsi^\top\bG\be_j - \sum_{i=1}^n\frac{\partial \psi_i}{\partial g_{ij}}}{\sqrt n (\|\bh\|^2 + \|\bpsi\|^2/n)^{1/2}}
    =
    \brho^\top\bG\be_j - 
\frac{  
    \sum_{i=1}^n\frac{\partial \psi_i}{\partial g_{ij}}}{\sqrt n (\|\bh\|^2 + \|\bpsi\|^2/n)^{1/2}}
    = b_j
    - \sum_{i=1}^n \frac{\psi_i}{\sqrt n}
    \frac{\partial (D^{-1})}{\partial g_{ij} }
    $$
    where $D=(\|\bh\|^2 + \|\bpsi\|^2/n)^{1/2}$ so that
    by the Cauchy-Schwarz inequality
    $\|\ba-\bb\|^2 \le \frac 1 n \|\bpsi\|^2
    \sum_{ij}(\frac{\partial(D^{-1})}{\partial g_{ij}})^2$
    and
    \begin{equation}
    \sum_{ij}\Bigl(\frac{\partial(D^{-1})}{\partial g_{ij}}\Bigr)^2
    =
    \frac{1}{D^6} \sum_{ij} \Bigl(
      \bh^\top\frac{\partial \bh}{\partial g_{ij}}
      +\frac{\bpsi}{\sqrt n}^\top\frac{\partial \bpsi}{\partial g_{ij}}
    \Bigr)^2
    \le \frac{2}{D^4}
    \sum_{ij} \|\frac{\partial \bh}{\partial g_{ij}}\|^2
    + \frac 1 n \|\frac{\partial \bpsi}{\partial g_{ij}}\|^2
    .
    \label{eq:derivative-D-1}
    \end{equation}
    using again the Cauchy-Schwarz inequality
    and $\max\{\|\bh\|^2,\|\bpsi\|^2/n\}\le D^2$. We obtain
    $\|\ba-\bb\|^2\le D^{-2}
    \sum_{ij} \|\frac{\partial \bh}{\partial g_{ij}}\|^2
    + \frac 1 n \|\frac{\partial \bpsi}{\partial g_{ij}}\|^2$
    which completes the proof for the first term
    in the left-hand side of \Cref{prop:chi2-custom}.
    For the second term in the left-hand side, the proof
    is similar with by exchanging the role of $n$ and $p$
    in \eqref{eq:cor71} and applying
    \eqref{eq:cor71} to $\bh/D$ instead of $\bpsi/(\sqrt n D)$.
\end{proof}

\propXiVI*
\begin{proof} [Proof of \Cref{prop:xi_VI}]
    Apply \eqref{eq:prop63}
    with $\brho=\bep/\|\bep\|$
    and $\bfeta = \bh/D$ where $D=(\|\bh\|^2 + \|\bpsi\|^2/n)^{1/2}$
    as in the previous proof
    (this scalar $D$ is not related to the diagonal matrix $\bD=\diag\{\psi'(\br)\}$).
    Since $\bep$ has 0 derivative with respect to $\bG$ we find
    \begin{align*}
        \E 
        \Bigl[ \Bigl( \frac{\bep ^{\top} \bG \bh}{\|\bep\|D} - \sum_{ij} \frac{\eps_i}{\|\bep\|} \frac{ \partial (h_{j} D^{-1}) }{ \partial g_{ij} } \Bigr)^{2} \Bigr] 
        &\leq
        1 + 2\sum_{ij}\E[
        \|\frac{\partial \bfeta}{\partial g_{ij}}\|^2
        ].
    \end{align*}
    The right-hand side
    is bounded from above by $\C(\gamma,\mu)$ thanks to \eqref{eq:argument-bff} and \eqref{Xi-upper-bound}.
    For the second term above we use product rule and \eqref{contraction-1},
    \begin{align*}
        \sum_{ij}
        \frac{\eps_i}{\|\bep\|}
        \frac{ \partial ( h_{j} D^{-1}) }{ \partial g_{ij} }
        =
        \frac{\trace [\bA ] \bpsi ^{\top} \bep }{D\|\bep\|}
        -{\color{purple}\frac{\bh^{\top} \bA \bG^{\top} \diag(\psi'(\br)) \bep}{D\|\bep\|} }
        +
        {\color{purple}
            \sum_{ij} \frac{\eps_{i} h_{j}}{\|\bep\|}
        \frac{ \partial (D^{-1}) }{ \partial g_{ij} }
        }
        .
    \end{align*}
    To complete the proof we need to prove that 
    the expectation of the square of the second and third terms
    colored in purple are bounded by $\C(\gamma,\mu)$.
    Since $\|\bh\|\le D$,
    the second term is bounded from above by
    $\|\bA\|_{op} \|\bG\|_{op}$ since $|\psi'|\le 1$
    and $\E[\|\bA\|_{op}^2 \|\bG\|_{op}^2]\le \C(\gamma,\mu)$
    thanks to $\|\bA\|_{op}\le 1/(n\mu)$ and
    \cite[Theorem II.13]{davidson2001local}.
    For the third term, we use
    the Cauchy-Schwarz inequality
    $(\sum_{ij} \frac{\eps_ih_j}{\|\bep\|}\frac{\partial(D^{-1})}{\partial g_{ij}})^2
    \le \|\bh\|^2 \sum_{ij} (\frac{\partial(D^{-1})}{\partial g_{ij}})^2$,
    \eqref{eq:derivative-D-1}
    and \eqref{Xi-upper-bound}.
\end{proof}

\section{Proof of differentiability results}
\label{sec:9}

\theoremDifferentiability*

The first part of the following proof is similar to the argument
using the KKT conditions in \cite{bellec2020out_of_sample}.
After \eqref{eq:9938}, the argument is novel and lets us
derive the convenient formula \eqref{eq:differentiability-formulae}
and the existence of matrix $\hbA$ which plays a central
role in the contractions \eqref{contraction-1}-\eqref{contraction-5}.

\begin{proof}
    [Proof of \Cref{thm:differentiability}]
    $\bX_t = \bX + t \bU$ and $\by_{t} = \by + t \bv$ with $t\in\R$
    where $\bU \in \R^{n\times p}$ and $\bv \in \R^{n}$ are fixed.
    Let
    $ \hbbeta_{t} = \hbbeta(\by_{t}, \bX_{t}) $
    and 
    $ 
        \hbr_{t} = \by_{t} - \bX_{t} \hbbeta (\by_{t}, \bX_{t})
    $
    and 
    $\hbpsi (\by_{t}, \bX_{t})  = \psi (\hbr_{t})$.
    By convention, without arguments $\hbbeta,\bpsi$
    refer to $(\by,\bX)$ which is $(\by_t,\bX_t)$ at $t=0$.
    By the KKT conditions, 
    $ 
        \bX^{\top} \hbpsi \in n \partial g (\hbbeta)
    $ 
        and 
    $
        \bX_{t}^{\top} \hbpsi_{t} \in n \partial g (\hbbeta_{t}), 
    $
    by strong convexity of $g$,
    we have 
    \begin{align}
        \label{eq:0063}
        n \mu \| \bSigma^{1/2} (\hbbeta_{t} - \hbbeta) \|^{2}
        \leq 
        ( \hbbeta_{t} - \hbbeta ) ^{\top} 
        ( \bX_{t}^{\top} \hbpsi_{t} - \bX ^{\top} \hbpsi)
        .
    \end{align}

    By the fact that $\psi$ is non-decreasing and $1$-Lipschitz,
    for any two real numbers $a < b$,
    $ 
        0 \leq \psi(b) - \psi(a) \leq b - a.
    $
    Multiplying $\psi(b) - \psi(a)$, 
    we have
    $ 
        (\psi(b) - \psi(a))^{2} \leq (\psi (b) - \psi(a)) (b - a) 
        .
    $
    Thus
    \begin{align*}
        \| \hbpsi_{t} - \hbpsi \|^{2}
        \leq
       (\hbpsi_{t} - \hbpsi)^{\top} (\hbr_{t} - \hbr).
    \end{align*}
    Adding up the above two displays we have
    \begin{align}
        \label{eq:3363}
        &n\mu \|\bSigma^{1/2}(\hbbeta_{t} - \hbbeta)\|^{2}
        +
        \| \hbpsi _{t} - \hbpsi \|^{2}
        \leq
        (\hbbeta_{t} - \hbbeta)^{\top}
        (\bX_{t}^{\top} \hbpsi_{t} - \bX^{\top} \hbpsi)
        +
        (\hbpsi_{t} - \hbpsi)^{\top} (\hbr_{t} - \hbr)
        .
    \end{align}
    By
    $
        \bX_{t}^{\top} \hbpsi_{t}
        -
        \bX^{\top} \hbpsi
        =
        (\bX_{t}- \bX)^{\top} \hbpsi
        +
        \bX^{\top}_{t} (\hbpsi _{t} - \hbpsi)
    $ and
    $\bX_{t} (\hbbeta_{t} - \hbbeta) + \hbr_{t} - \hbr 
        = 
        \by_{t} - \by - (\bX_{t} - \bX)^{\top} \hbbeta$,
    we have
    \begin{align*}
        n\mu \|\bSigma^{1/2}( \hbbeta _{t} - \hbbeta )\|^{2}
        + \| \hbpsi_{t} - \hbpsi \|^{2}
        \leq 
        (\hbbeta_{t} - \hbbeta)^{\top} 
        (\bX_{t} - \bX)^{\top} \hbpsi
        + 
        (\by_{t} - \by - (\bX_{t} - \bX )^{\top} \hbbeta ) ^{\top} (\hbpsi _{t} - \hbpsi).
    \end{align*}
    By the Cauchy-Schwartz inequality, the above implies 
    \begin{align*}
        \bigl(
            n \mu \|\bSigma^{1/2}( \hbbeta _{t} - \hbbeta )\|^{2}
             + 
             \| \hbpsi _{t} - \hbpsi \|^{2}
        \bigr)^{1/2}
        \leq
        (n\mu)^{-1/2} \| \bSigma^{-1/2}(\bX_{t} - \bX)^{\top} \hbpsi \|_{2}
        +
        \|\by_{t} - \by - (\bX_{t} - \bX )^{\top} \hbbeta \|_{2},
    \end{align*}
    Since $t,\bU,\bv$ are arbitrary,
    for $(\by_{t}, \bX_{t})$ and $(\by,\bX)$ both in a compact subset $K$ of $\R^{p} \times \R^{n \times p}$,
    the above display also implies
    \begin{align*}
        \bigl(
            n \mu \| \bSigma^{1/2} (\hbbeta _{t} - \hbbeta) \|^{2}
             + 
             \| \hbpsi _{t} - \hbpsi \|^{2}
        \bigr)^{1/2}
        \leq
        \text{const}(K)
        \bigl(
            \| \bSigma^{-1/2} (\bX_{t} - \bX) \|_{op}
            + 
            \| \by _{t} - \by \|_{2}
        \bigr)
        ,
    \end{align*}
    where
    $ 
        \text{const}(K)
        :=
        \sup_{(\by,\bX) \in K}
        \{ 
            ( n \mu )^{-1/2} \| \hbpsi \|_{2}
            +
            1 
            + 
            \| \bSigma^{1/2} \hbbeta \|_{2}
        \}
        .
    $
    This says that $\hbbeta(\by, \bX), \hbpsi (\by, \bX)$ are locally Lipschitz in $(\by,\bX)$.
    By Rademacher's Theorem, $\partial \hbbeta / \partial y_{i}$ and $\partial \hbbeta / \partial x_{ij}$ exist almost everywhere.

    Taking the limit $t \to 0^+$ in \eqref{eq:0063} and using the chain rule,
    where the derivatives exist we have
    \begin{equation}
        \label{eq:9938}
    \begin{aligned}
        &n \mu 
        \bigl\|\bSigma^{1/2} \bigl(
            \frac{ \partial \hbbeta }{ \partial \by } \bv
            + \frac{ \partial \hbbeta }{ \partial \bX } (\bU)
            \bigr)
        \bigr\|_{2}^{2} 
        \\
        &\leq 
        \Bigl(
            \frac{ \partial \hbbeta }{ \partial \by } \bv
            + \frac{ \partial \hbbeta }{ \partial \bX } (\bU)
        \Bigr)^{\top}
        \Bigl(
            \bU ^{\top} \hbpsi  
            +
            \bX ^{\top} \diag (\hbpsi')
            \bigl(
                 - \bU \hbbeta
                 - \bX \frac{ \partial \hbbeta }{ \partial \bX } (\bU)
                 +
                 \Bigl(
                      I_{n}
                      - \bX \frac{ \partial \hbbeta }{ \partial \by }
                 \Bigr)
                 \bv
            \bigr)
        \Bigr)
        \\
        &=
        \Bigl(
            \frac{ \partial \hbbeta }{ \partial \by } \bv
            + \frac{ \partial \hbbeta }{ \partial \bX } (\bU)
        \Bigr)^{\top}
        B (\bU,\bv)
        - 
        \Bigl\|
            \diag(\hbpsi')^{\frac12}
            \bX
            \Bigl(
                \frac{ \partial \hbbeta }{ \partial \by } \bv
                +
                \frac{ \partial \hbbeta }{ \partial \bX } (\bU)
            \Bigr)
        \Bigr\|_{2}^{2}
    \end{aligned}
    \end{equation}
    where  
    $ 
       (\partial \hbbeta / \partial \by) \bv := \sum_{i \in [n]} (\partial \hbbeta / \partial y_{i}) v_{i}
    $, the Jacobian with respect to $\bX$ 
    and the linear map $B:\R^{n\times p} \times \R^n\to\R^p$ are defined as 
    \begin{align*}
        \frac{ \partial \hbbeta }{ \partial \bX } (\bU)
        := 
        \sum _{i,j \in [n]\times[p]} 
        \frac{
            \partial \hbbeta
        }{
            \partial x_{ij} 
        } u_{ij} \in \R^{p},
        \quad
        B (\bU,\bv)
        := \bU^{\top} \hbpsi + \bX^{\top} \diag (\hbpsi') ( - \bU \hbbeta + \bv)
        \in \R^{p}
    \end{align*}
    where $(u_{ij})_{i=1,...,n,j=1,...,p}$ are the entries of $\bU$.
    By the Cauchy-Schwartz inequality,
    \eqref{eq:9938} provides us the following two main ingredients:
    \begin{equation}
        \label{eq:4416}
        \frac{ \partial \hbbeta }{ \partial \by } \bv 
        + 
        \frac{ \partial \hbbeta }{ \partial \bX } (\bU)
        = 0
        \text{ for all }
        (\bU,\bv)
        \text{ such that }
        B (\bU,\bv) = 0,
    \end{equation}
    \begin{equation}
        \label{eq:0607}
        \Bigl\|
        \bSigma^{1/2}\Bigl(
             \frac{
                 \partial \hbbeta
             }{
                 \partial \by
             } \bv 
             + 
             \frac{
                 \partial \hbbeta
             }{
                 \partial \bX
             } (\bU)
             \Bigr)
        \Bigr\|_{2}
        \leq 
        \mu^{-1} n^{-1} 
        \| \bSigma^{-1/2} B (\bU,\bv) \|_{2}.
    \end{equation}
    Since both $ \frac{ \partial \hbbeta }{ \partial \by } \bv + \frac{ \partial \hbbeta }{ \partial \bX } (\bU) $ and $ B( \bU,\bv )$ are linear 
    in $ (\bU,\bv) \in \R^{n \times p} \times \R ^{n} $ into $\R ^{p}$, 
    \Cref{prop:6890} implies 
    that there exists 
    a matrix $\hbA \in \R^{p \times p}$ such that
    $ 
        \frac{ \partial \hbbeta }{ \partial \by } \bv + \frac{ \partial \hbbeta }{ \partial \bX } (\bU) = \hbA B( \bU,\bv )
    $
    for all $(\bU, \bv)$,
    and by \eqref{eq:0607},  $\hbA$ can be chosen such that
    $ 
    \|\bSigma^{1/2} \hbA \bSigma^{1/2} \|_{op} \leq (n\mu)^{-1} 
    $
    thanks to the operator norm identity in \Cref{prop:6890}.
    With $(\bU,\bv) = (\be_{i} \be_{j}^{\top}, \bzero)$ for $(i,j) \in [n] \times [p]$
    and $(\bU,\bv) = (\bzero, \be_k)$ for $k \in [n]$, we obtain
    the stated formulae for $(\partial / \partial x_{ij}  )\hbbeta$ 
    and $(\partial / \partial  y_{k}) \hbbeta$ in \eqref{eq:differentiability-formulae}.

    Now we show that both $\trace[\bV] := \trace[\bD - \bD \bX \hbA \bX^{\top} \bD ]$ and $\df := \trace [ \bX \hbA \bX^{\top} \bD ]$ are in $[0,n]$ where $\bD := \diag \{ \psi'  ( \br ) \}$.
    Using the symmetric part of $\bA$ defined as
    $\tbA := ( \hbA + \hbA{}^{\top}) / 2$ we have
    $ 
        \trace [ \bV ] = \trace [ \bD - \bD \bX \tbA \bX^{\top} \bD ]
    $
    and 
    $ 
        \df = \trace [ \bD^{1/2} \bX \tbA \bX^{\top} \bD^{1/2}]
    $ by property of the trace.
    In \eqref{eq:9938}, take
    $\bU = \bzero $ so that
    $ 
        \frac{ \partial \hbbeta }{ \partial \by } \bv + \frac{ \partial \hbbeta }{ \partial \bX } (\bU) = \hbA B( \bU,\bv )
        =
        \hbA \bX^{\top} \bD \bv
    $
    and we have with $\bG=\bX\bSigma^{-1/2}$
    \begin{align} 
        (1 + \tfrac{n \mu}{\|\bD^{1/2}\bG\|_{op}^2})
        \|\bD^{1/2}\bX \hbA \bX^{\top} \bD \bv \|^{2}
        &\le
        n \mu \| \hbA \bX^{\top} \bD \bv \|^{2}
        +
        \| \bD^{1/2} \bX \hbA \bX^{\top} \bD \bv \|^{2}
        \\
        &\leq 
        \bv ^{\top} \bD \bX \hbA \bX^{\top} \bD \bv
        =
        \bv ^{\top} \bD \bX \tbA \bX^{\top} \bD \bv
        \label{eq:forallv}
    \end{align}
    for all $\bv$.
    This implies the positive semi-definite property of the symmetric matrix $\bD \bX \tbA \bX^{\top} \bD$, 
    and thus $\df \geq 0$ and $\trace [ \bV ] \leq \trace [ \bD ] \leq n$.
    With $\tbv = \bD^{1/2} \bv$, it also implies
    $$ 
        (1 + n \mu / \|\bD^{1/2}\bG\|_{op}^2)
        \| \bD ^{1/2} \bX \hbA \bX^{\top} \bD^{1/2} \tbv \|^{2}
        \leq 
        \tbv ^{\top} \bD^{1/2} \bX  \hbA \bX^{\top} \bD^{1/2} \tbv,
    $$
    which, by the Cauchy-Schwartz inequality, yields
    $
    (1 + n \mu / \|\bD^{1/2}\bG\|_{op}^2)
    \| \bD^{1/2} \bX \hbA \bX^{\top} \bD^{1/2} \|_{op} \leq 1$.
    The same operator norm inequality with $\hbA$ replaced by
    its symmetric part $\tbA=(\hbA+\hbA^\top)/2$
    thanks to the triangle inequality.
    Thus
    $\df \le
        \pb{n}
        (1 + n \mu / \|\bD^{1/2}\bG\|_{op}^2)^{-1}
    \le n$ as well as
    \begin{align}
    \trace[\bV]
            =\trace[\bD^{1/2}(\bI_n - \bD^{1/2}\bX\tbA\bX^\top\bD^{1/2})\bD^{1/2}]
        \nonumber
          &\ge \trace[\bD](1-
          (1 + n \mu / \|\bD^{1/2}\bG\|_{op}^2)^{-1}
            )
        \nonumber
          \\&= \trace[\bD]/(\|\bD^{1/2}\bG\|_{op}^2/(n\mu) + 1)
        \nonumber
        \\&\ge \trace[\bD]/(\|\bG\|_{op}^2/(n\mu) + 1)
            \label{eq:lower-bond-trace-V}
        \\&\ge 0
        \nonumber
    \end{align}
    thanks to $\psi'\in[0,1]$.
    Inequality \eqref{eq:forallv} with $\tbv = \bD^{1/2}\bv$
    and $\bM = \bI_n - \bD^{1/2}\bX\hbA\bX^\top\bD^{1/2}$ 
    implies
    $
    \|(\bM-\bI_n) \tbv\|^2 \le \tbv^\top(\bI_n - \bM)\tbv$.
    As the left-hand side is
    $\|\bM\tbv\|^2 - 2 \tbv^\top\bM\tbv + \|\tbv\|^2$,
    this yields
    $\|\bM\tbv\|^2
    \le \tbv^\top\bM \tbv
    \le \|\tbv\| \|\bM\tbv\|$.
    If $\tbv$ has unit norm and is such that $\|\bM\tbv\|=\|\bM\|_{op}$
    this gives $\|\bM\|_{op}\le 1$ so that
    $\|\bV\|_{op}=\|\bD^{1/2}\bM\bD^{1/2}\|_{op}\le \|\bD\|_{op}\le 1$.
    This gives another proof of $\trace[\bV]\le n$.
\end{proof}

\begin{proof} [Proof of \Cref{rem:intercept}] 

    The proof for the intercept term included is the same to that of 
    \Cref{thm:differentiability}.
    The only difference is that when computing the derivatives, 
    \begin{align*}
        &\frac{ d \hbpsi_{t} }{ d t }|_{t=0}
        =
        \bU^{\top} \hbpsi
        +
        \bX^{\top}
        \Bigl(
            \frac{ \partial \hbpsi }{ \partial \by }
            \bv
            +
            \frac{ \partial \hbpsi }{ \partial \bX } (\bU)
        \Bigr), \quad
        \frac{ \partial \hbpsi }{ \partial \by }
        \bv
        =
        \diag(\hbpsi') (\bI_{n} - \bone 
        \frac{ \partial \hbeta_{0} }{ \partial \by }
        - \bX \frac{ \partial \hbbeta }{ \partial \by }
        )
        \bv
        ,
    \\
        &\frac{
            \partial \hbpsi 
        }{
            \partial \bX
        } (\bU)
        =
        \diag(\hbpsi')
        ( - \bone 
        \frac{
            \partial \hbeta_{0}
        }{
            \partial \bX
        } (\bU)
        - \bU \hbbeta 
        - \bX \frac{ \partial \hbbeta }{ \partial \bX } (\bU)
        )
    \end{align*}
    \begin{align*}
        \implies 
        \frac{ d \hbpsi_{t} }{ d t }|_{t = 0}
        =
        - \hbpsi' 
        \frac{ d \hbeta_{0,t} }{ d t  }|_{t = 0}
        - \diag(\hbpsi') \bX 
        \frac{ d \hbbeta_{t} }{ d t  }|_{t=0}
        + \diag(\hbpsi') \bv 
        - \diag(\hbpsi') \bU \hbbeta.
    \end{align*}
    We have an additional KKT conditions providing us
    $ 
        0
        =
        \bone^{\top} 
        (d \hbpsi_{t} / d t) |_{t=0}
        .
    $
    Multiplying $\bone^{\top}$
    on both sides of the above display, 
    we have 
    \begin{align*}
        \frac{ d \hbeta_{0,t} }{ d t }|_{t=0}
        &=
        - \frac{ \hbpsi'^{\top} \bX }{ \bone ^{\top} \hbpsi' } \frac{ d \hbbeta_{t} }{ d t }|_{t=0}
        + \frac{ \hbpsi'^{\top} \bv }{ \bone^{\top} \hbpsi' } 
        - \frac{ \hbpsi'^{\top} \bU \hbbeta }{ \bone^{\top} \hbpsi' }
        ,
    \\
        \implies 
        \frac{ d \hbpsi_{t} }{ d t }|_{t = 0}
        &=  
        - \bPsi' \bX \frac{ d \hbbeta_{t} }{ d t }|_{t = 0}
        + \bPsi' \bv
        - \bPsi' \bU \hbbeta 
        ,
    \end{align*}
    where 
    $ 
        \bPsi' := \diag(\hbpsi') - \hbpsi' \hbpsi'^{\top} / \bone^{\top} \hbpsi'.
    $
    By taking limit of $t \to 0$ in \Cref{eq:3363},
    \begin{align*}
        n \mu 
        \Bigl\| 
        \frac{
            d \hbbeta_{t}
        }{
            d t
        }|_{t=0}
        \Bigr\|_{2}^{2}
    &\leq 
        \frac{ d \hbbeta_{t} }{ d t }|_{t=0} ^{\top}
        \frac{ d (\bX^{\top} \hbpsi) }{ d t  }|_{t = 0}
        =
        \frac{ d \hbbeta_{t} }{ d t }|_{t=0} ^{\top}
        \Bigl( \bU^{\top} \hbpsi + \bX^{\top} \frac{ d \hbpsi_{t} }{ d t }|_{t = 0} \Bigr)
    \\
        &=
        \frac{ d \hbbeta_{t} }{ d t }|_{t=0} ^{\top}
        \Bigl( \bU^{\top} \hbpsi + 
            \bX^{\top} 
            \bigl(
                - \bPsi' \bX \frac{ d \hbbeta_{t} }{ d t }|_{t = 0}
                + \bPsi' \bv
                - \bPsi' \bU \hbbeta
            \bigr)
        \Bigr)
    \\  
        &=
        \frac{ d \hbbeta_{t} }{ d t }|_{t=0} ^{\top}
        \Bigl( 
            \bU^{\top} \hbpsi 
            + \bX^{\top} \bPsi' \bv
            - \bX^{\top} \bPsi' \bU \hbbeta
        \Bigr)
        - 
        \Bigl\| \bPsi'^{1/2} \bX \frac{ d \hbbeta_{t} }{ d t }|_{t=0}\Bigr\|^{2}
    .
    \end{align*}

\end{proof}

\begin{proposition} 
    \label{prop:6890}
    Let $\bA$ and $\bB$ be two real matrices with shape $n$ by $p$.
    Assume that $\bB \bv = \bzero$ for all $\bv\in\R^p$ such that $\bA \bv= \bzero$.
    Then the matrix $\bC := \bB \bA^{+}$ 
    where $\bA^{+}$ is the Moore-Penrose pseudoinverse of $\bA$
    satisfies 
    $ 
        \bB = \bC \bA
    $
    and 
    $ 
        \| \bC \|_{op} = \max_{\bu \in \R^{n}:\bA\bu\ne \mathbf0} \{ \| \bB \bu \|_{2} / \| \bA \bu \|_{2} \}.
    $
\end{proposition}

\begin{proof}
    Let $r$ be the rank of $\bA$.
    We let 
    $
        \bA = \bU \bD \bV^{\top}
    $
    be the SVD of $\bA$
    with $\bU\in\R^{n\times r}$,
    $\bD\in\R^{r\times r}$ diagonal with positive entries, 
    $\bV \in\R^{p\times r}$ and
    $\bU,\bV$ both with orthonormal columns.
    Then $\bA^+ = \bV\bD^{-1}\bU^\top$ and $\bC \bA = \bB\bV\bV^\top=\bB - \bB(\bI_p-\bV\bV^\top)$. Since $\ker \bA \subset \ker \bB$
    and $(\bI_p-\bV\bV^\top)$ is the orthogonal projection onto $\ker \bA$
    we have
    $\bB(\bI_p-\bV\bV^\top)=\mathbold{0}$. This proves $\bB=\bC\bA$.

    For $\|\bB\bA^+\|_{op}$,
    for any vector $\bu$,
    $\|\bB \bu\| = \|\bC\bA \bu\|
    \le \|\bC\|_{op} \|\bA\bu\|$ by definition of $\|\bC\|_{op}$.
    This proves
    $\| \bC \|_{op} \ge M$
    for $M=\max_{\bu \in \R^{n}:\bA\bu\ne \mathbf0} \{ \| \bB \bu \|_{2} / \| \bA \bu \|_{2} \}$.
    For the inequality $M\ge \|\bC\|_{op}$, if $\|\bC\|_{op} = \|\bC\bv\|$
    for some unit vector $\bv$ then $\bA\bA^+\bv = \bU\bU^\top\bv \ne \mathbf0$
    since $\bv$ is a right singular vector of $\bC=\bB\bV\bD\bU^\top$ and
    cannot belong to $\ker(\bU^\top)$. Next,
    $\|\bC\bv\| = \|\bB\bA^+\bv\|
    \le 
    M
    \|\bA\bA^+\bv\|
    $
    and the conclusion follows since $\|\bA\bA^+\|_{op}\le 1$
    and $\|\bv\|=1$.
\end{proof}

\section{Relaxing strong convexity: Proof of \Cref{prop12}}
\label{appendix:proof-prop12}
Consider the notation for $\bG, \bpsi(\bep,\bG),\bh(\bep,\bG),D$ defined around \eqref{bh}.
Let $\Omega=\{\bX\in U_{\bep}\}$.
Let us first rewrite the Lipschitz condition \eqref{Phi}
using the change of variable $\bG=\bX\bSigma^{-1/2}$ explained around
equation \eqref{bh}:
the mapping
$$
\tilde\Phi_{\bep}
    \left\{
\begin{aligned}
    \{\bM\bSigma^{-1/2} \ysc{\mid} \bM\in U_{\bep} \} &\to \R^{n+p}, \qquad
        \\ \bG &\mapsto 
            D^{-1}
            \bigl(
                n^{-1/2}\bpsi(\bep,\bG), ~
                \bh(\bep,\bG)
            \bigr)
\end{aligned}
\right.
\text{is $\frac{L}{\sqrt n}$-Lipschitz}
$$
where we recall the notation
$D=(\frac1n \|\bpsi(\bep,\bG)\|^2
+ \|\bh(\bep,\bG)\|^2
)^{1/2}$.
After the change of variable, the identities
\eqref{dgij_h}-\eqref{contraction-5} all hold in the event $\Omega$.

All previous calculations made in the strongly convex case hold
here is well, but only in the event $\Omega$. Outside of event $\Omega$,
the derivatives may not exist at all.
As in \Cref{thm:out-of-sample}, the device that lets us work around
this is Kirszbraun's theorem: there exists
an $L/\sqrt n$ Lipschitz function
$\bar\Phi_{\bep}$  (the \ysc{``}extension") such that
$\bar\Phi_{\bep}(\bG)=\tilde\Phi_{\bep}(\bG)$
for all $\bG$ in the domain 
$\{\bM\bSigma^{-1/2} \ysc{\mid} \bM\in U_{\bep} \}$ of $\tilde\Phi_{\bep}$.

Now define $\brho:\R^{n\times p}\to \R^n$ and $\bfeta:\R^{n\times p}\to\R^p$
by $\bar\Phi\ysc{_{\bep}}(\bG)=(\brho(\bG), \bfeta(\bG))$.
Using the Lipschitz condition and the fact that the Frobenius norm
of a Jacobian is bounded from above by \ysc{its} rank times the square of
its operator norm,
$$
\sum_{ij}
\|
\frac{\partial \brho}{\partial g_{ij}}
\|^2
\le L^2,
\qquad
\sum_{ij}
\|
\frac{\partial \bfeta}{\partial g_{ij}}
\|^2
\le \frac{p}{n} L^2
$$
so that the right-hand side of \eqref{eq:cor71} is bounded above
by $\C(L,\gamma) \sqrt p$ and the right-hand side
of \eqref{eq:prop63} is bounded from above
by $\C(L,\gamma)$. After we have bounded the right-hand side,
we are allowed to add the indicator function $I\{\Omega\}$
in the left-hand sides of \eqref{eq:cor71} and
of \eqref{eq:prop63} since adding an indicator function only
makes it smaller.
This \ysc{device} lets us obtain analogs of
\Cref{prop:Stein-3}
and \Cref{prop:chi2-custom}
where the right-hand sides are of the same order as in the strongly convex
case, provided that we add the indicator function $I\{\Omega\}$
in the left-hand sides.

From here, in the event $\Omega$
we use the bound $\|\bSigma^{1/2}\hbA\bSigma^{1/2}\|_{op}\le L/n$
assumed in \Cref{prop12} instead of 
the bound $\|\bSigma^{1/2}\hbA\bSigma^{1/2}\|_{op}\le 1/(n\mu)$
from \Cref{thm:differentiability}.
This device provides the bounds
\eqref{eq:five-1}, 
\eqref{eq:five-2}, 
\eqref{eq:five-3}, 
\eqref{eq:five-4}, 
\eqref{eq:five-5} and
\eqref{eq:xi_VI}
on $\xi_I, ... \xi_{VI}$, with the modification that the left-hand sides
present the indicator function $I\{\Omega\}$ and the right-hand sides
are $\C(\gamma, L)$ instead of $\C(\gamma,\mu)$ in the strongly convex case.
The algebra is then the same as in the strongly convex case
and \Cref{prop12} follows.

\newpage
\section{Additional Figures (anisotropic Gaussian design)}
\label{additional-figures}

\begin{figure}[ht]
    \centering

    \includegraphics[width=0.32\textwidth]{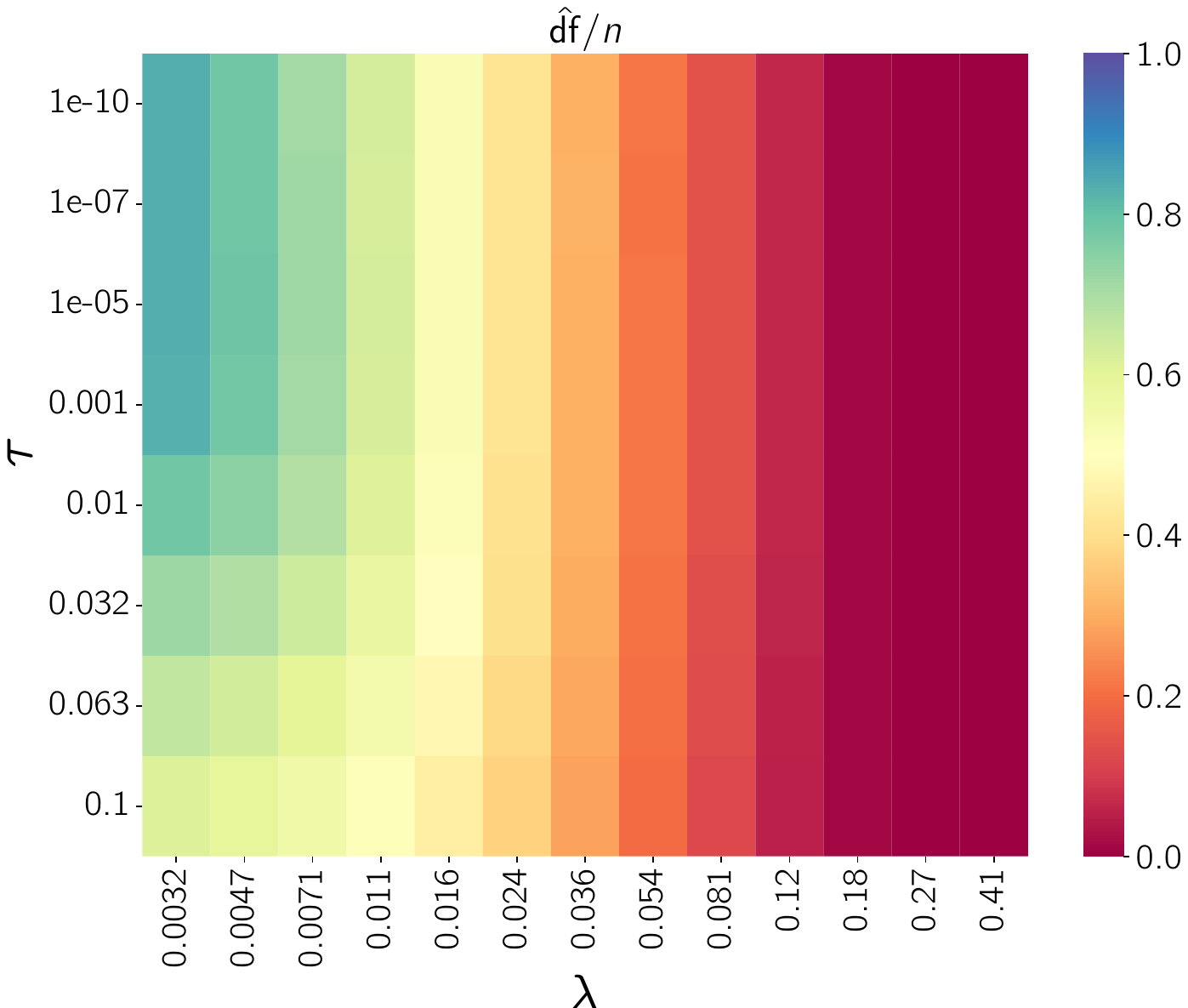}
    \includegraphics[width=0.32\textwidth]{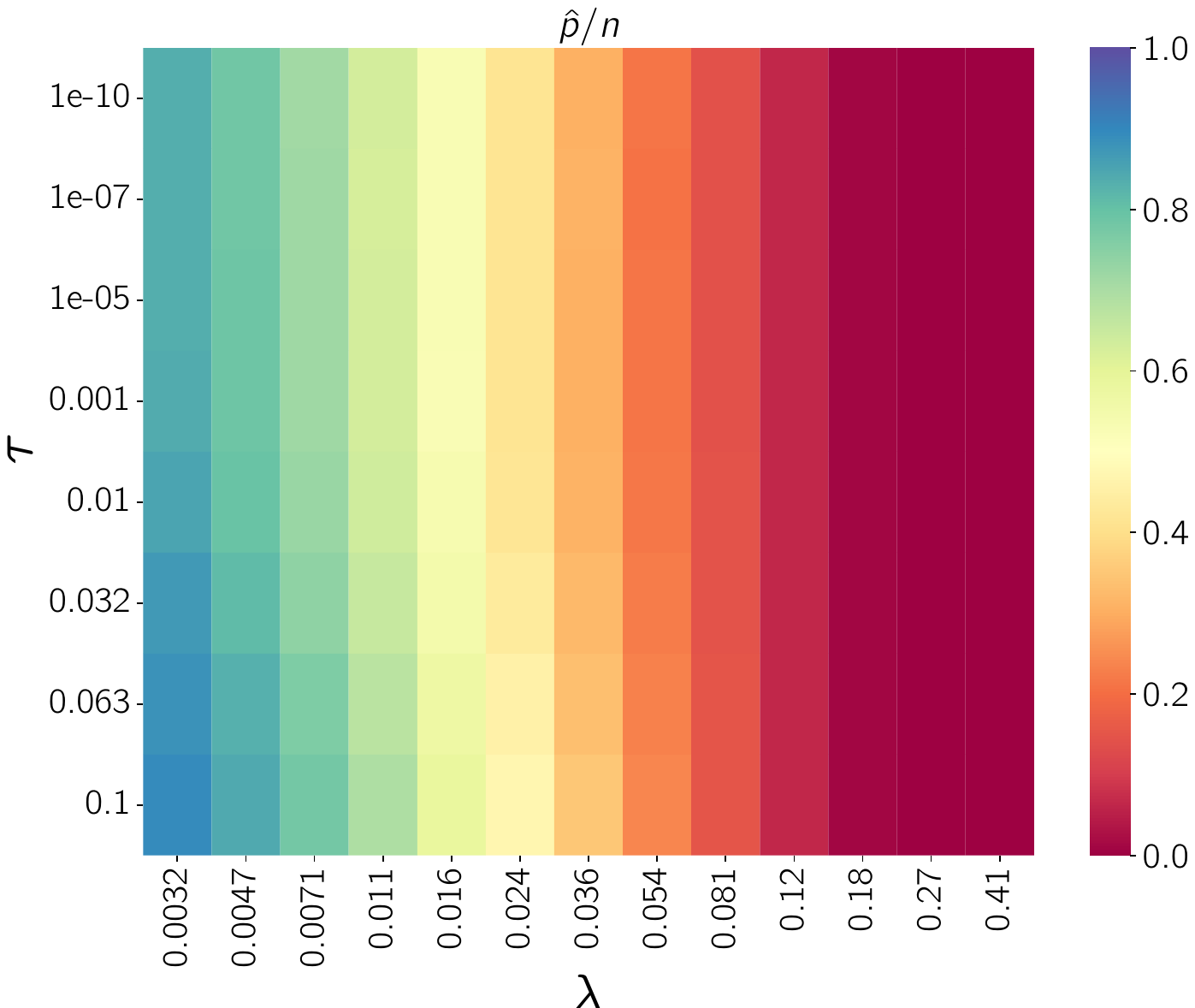}
    \includegraphics[width=0.32\textwidth]{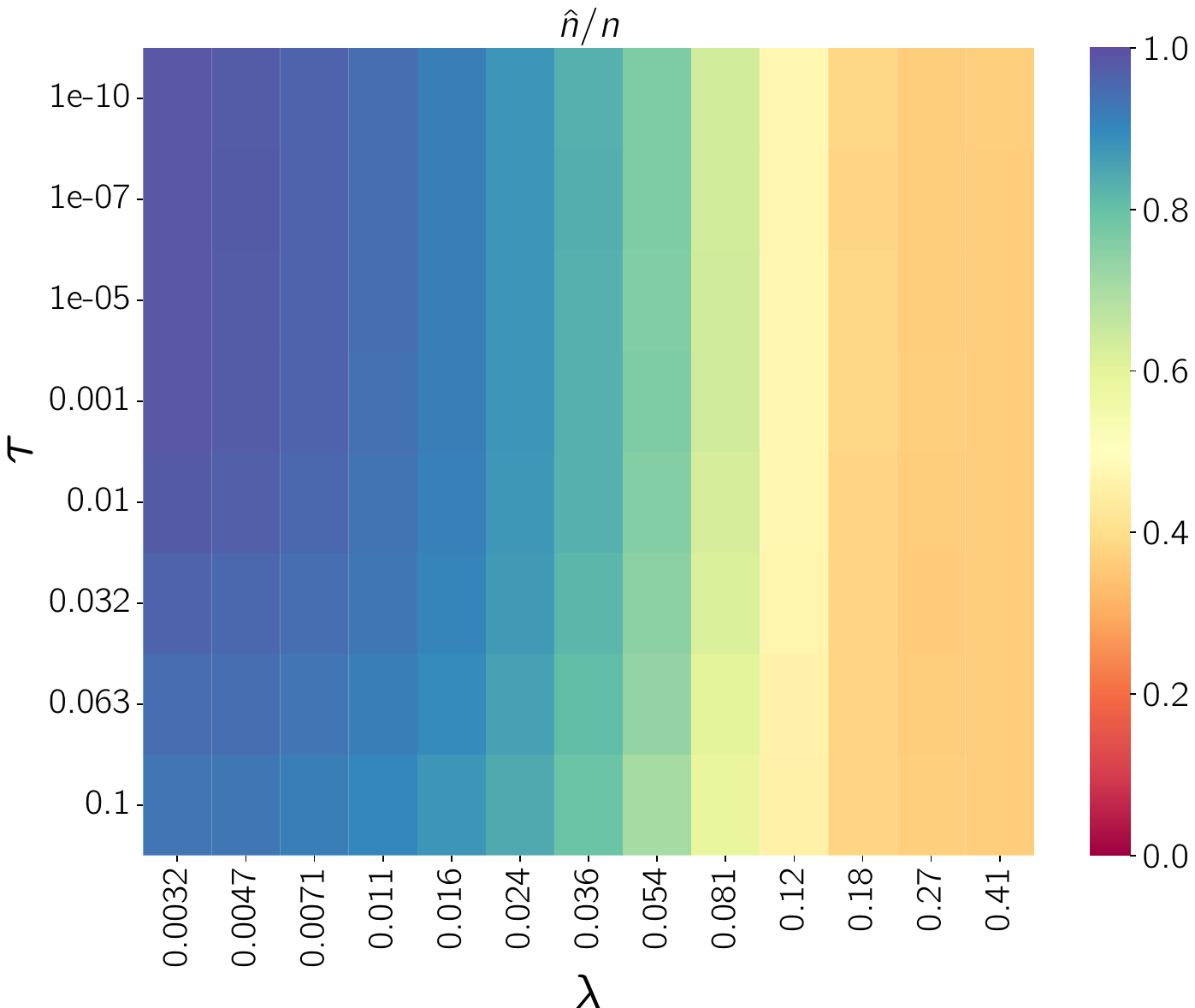}

    \includegraphics[width=0.32\textwidth]{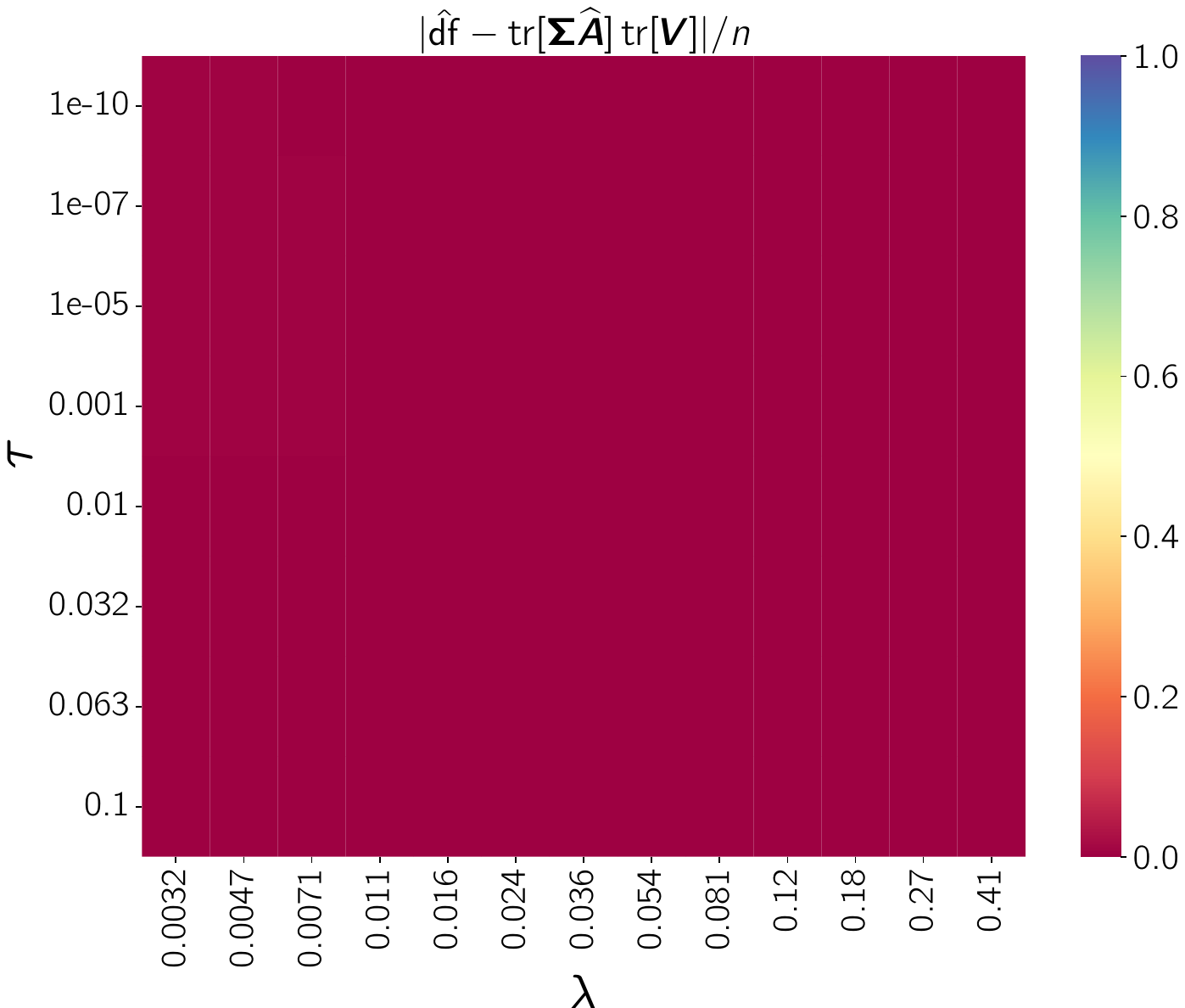}
    \includegraphics[width=0.32\textwidth]{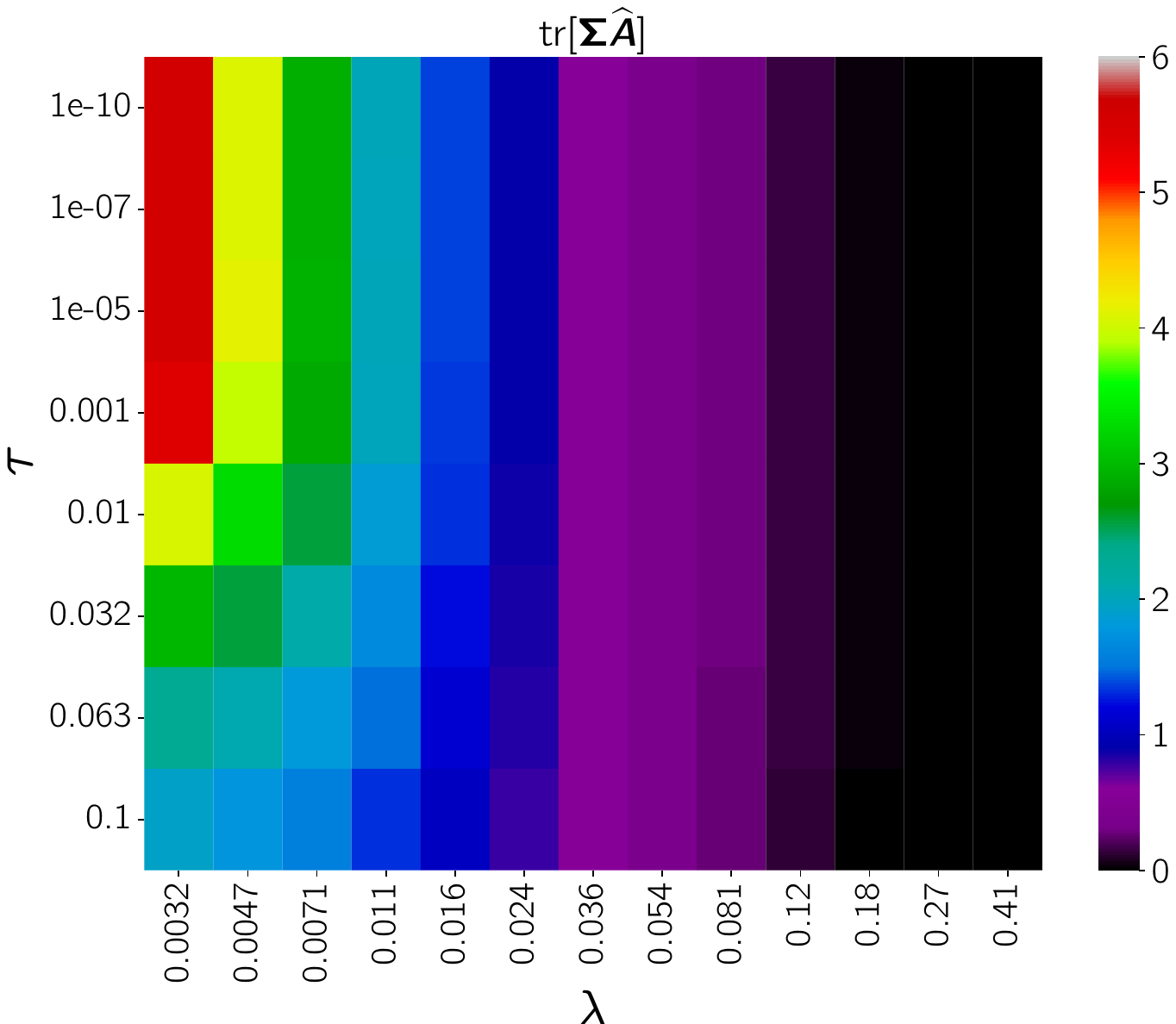}
    \includegraphics[width=0.32\textwidth]{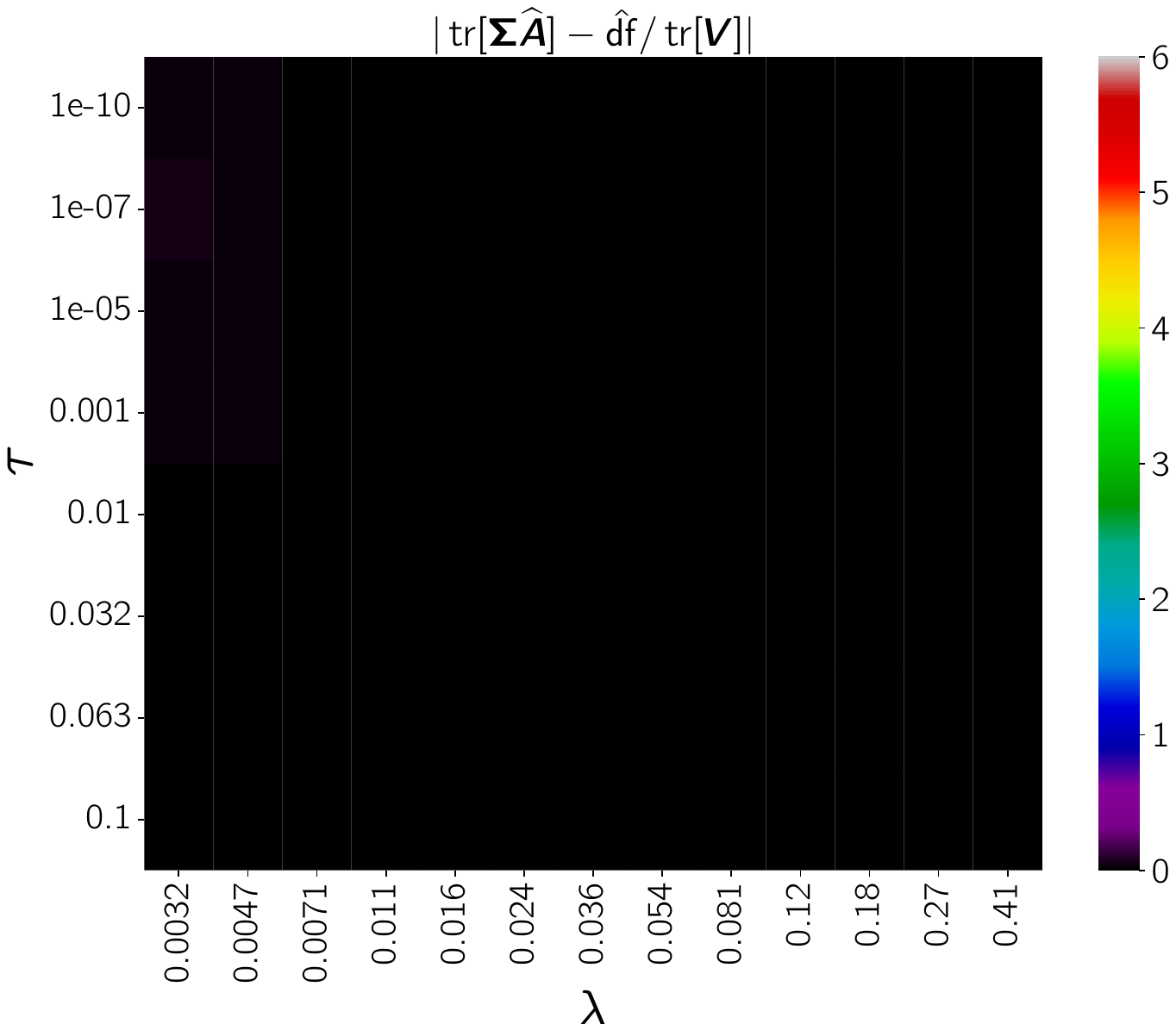}

    \includegraphics[width=0.32\textwidth]{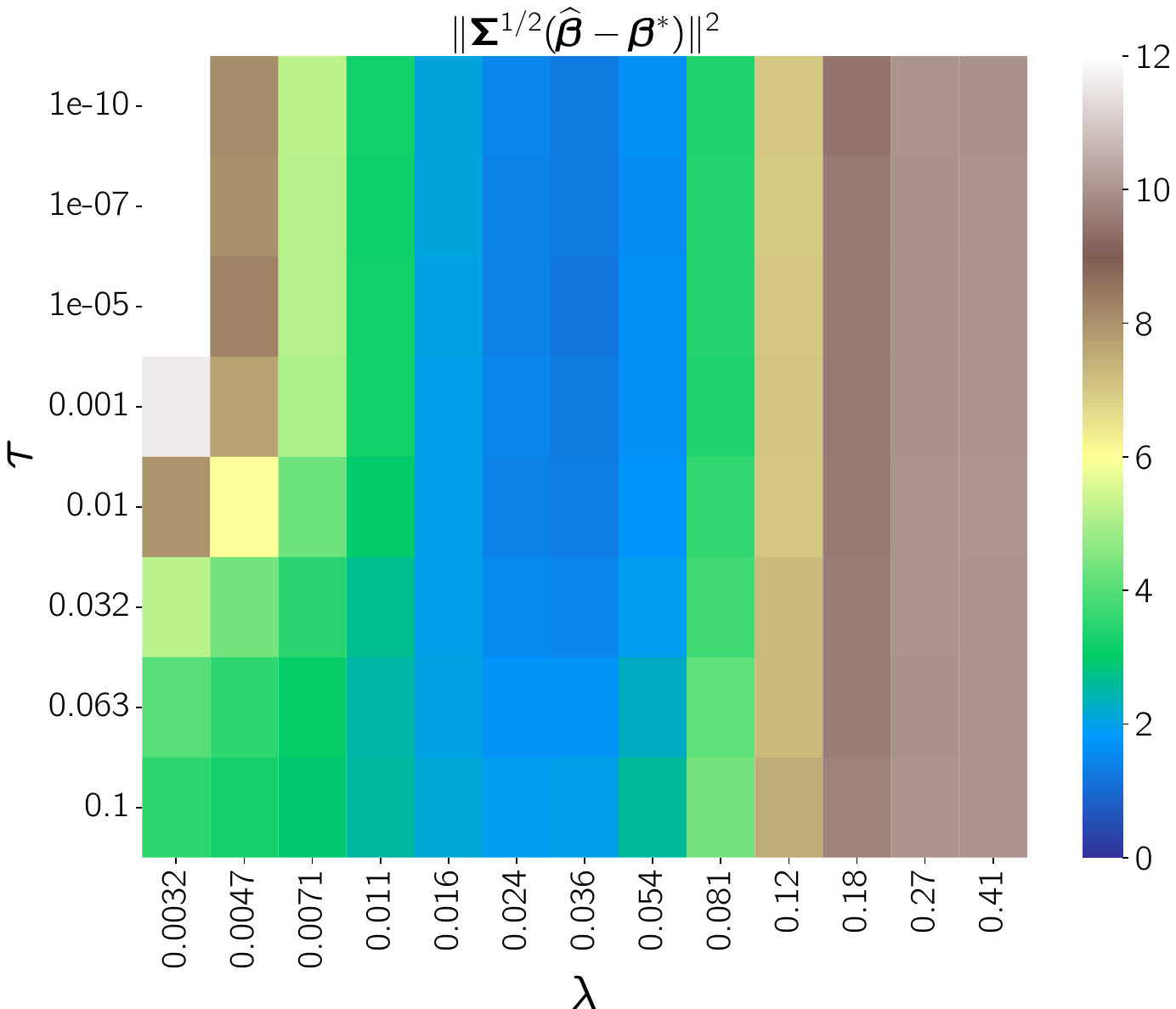}
    \includegraphics[width=0.32\textwidth]{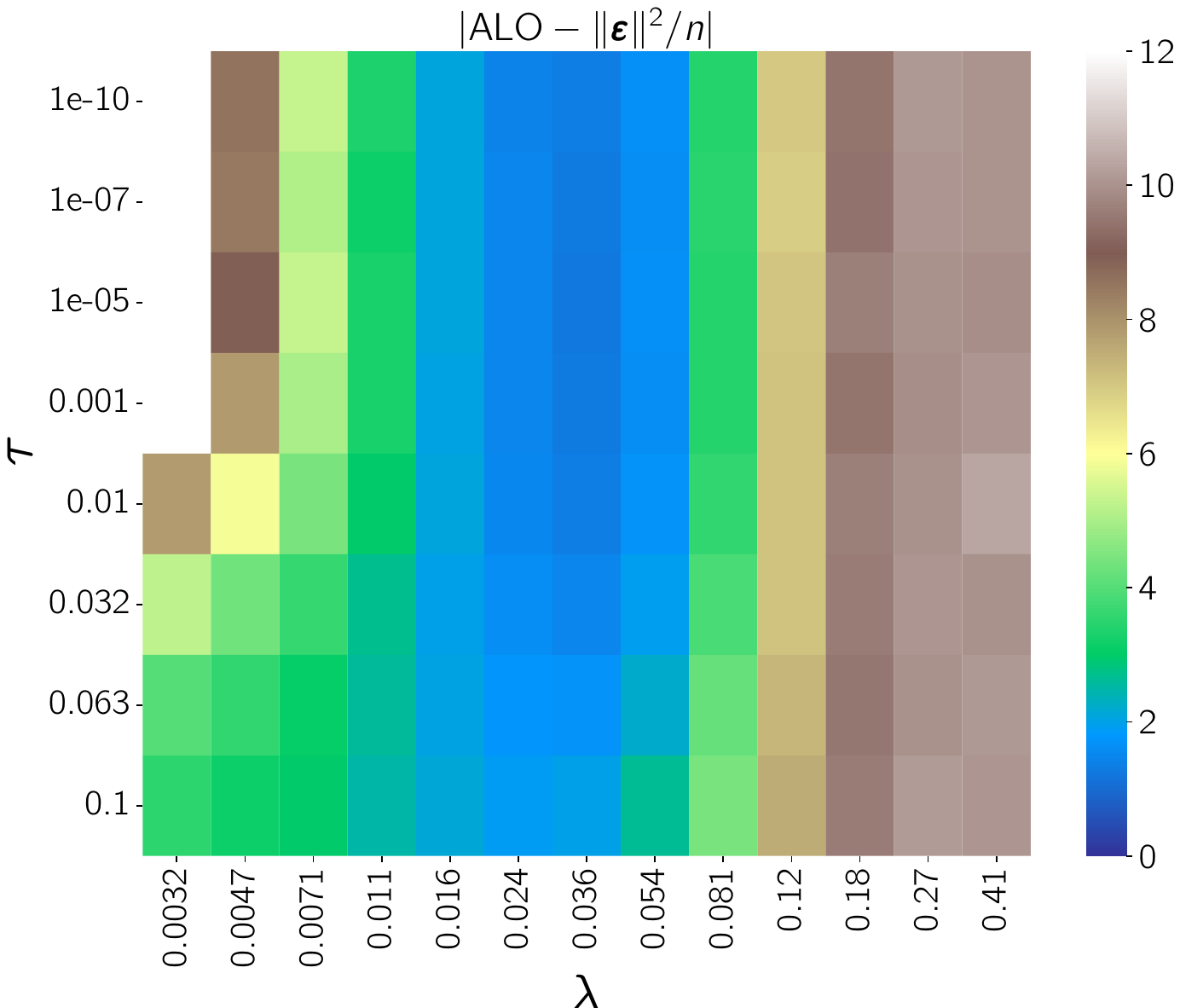}
    \includegraphics[width=0.32\textwidth]{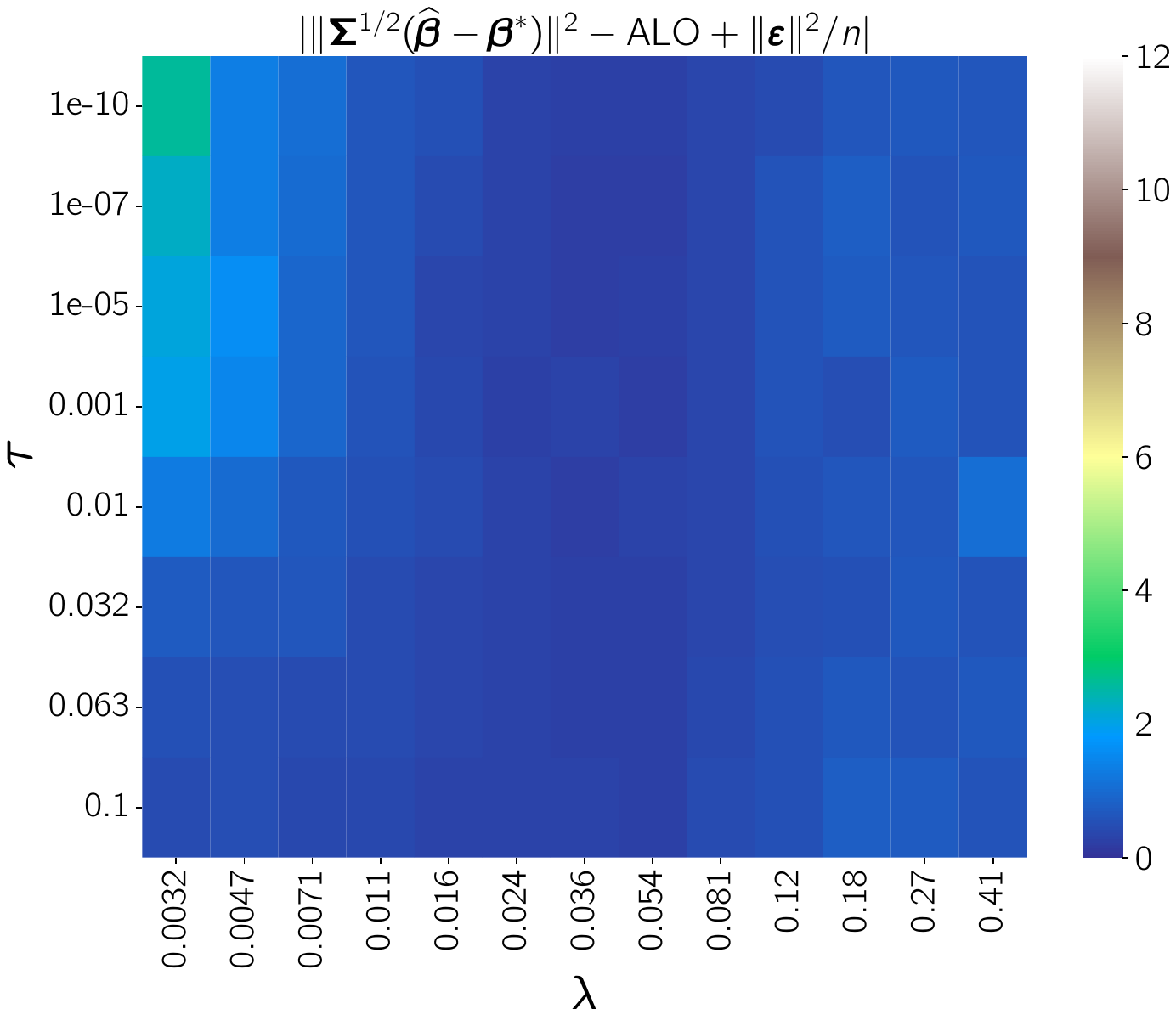}

    \includegraphics[width=0.32\textwidth]{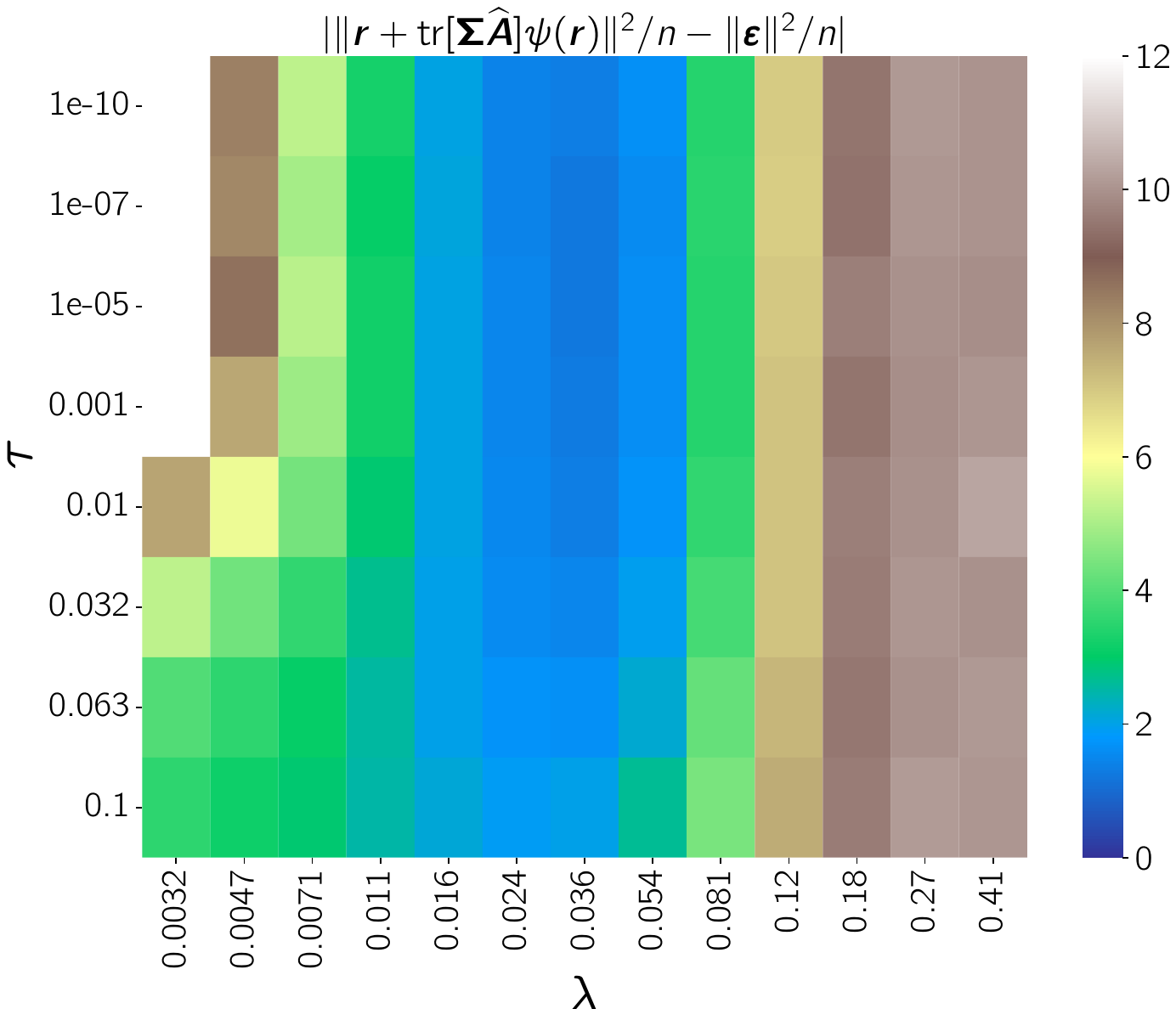}
    \includegraphics[width=0.32\textwidth]{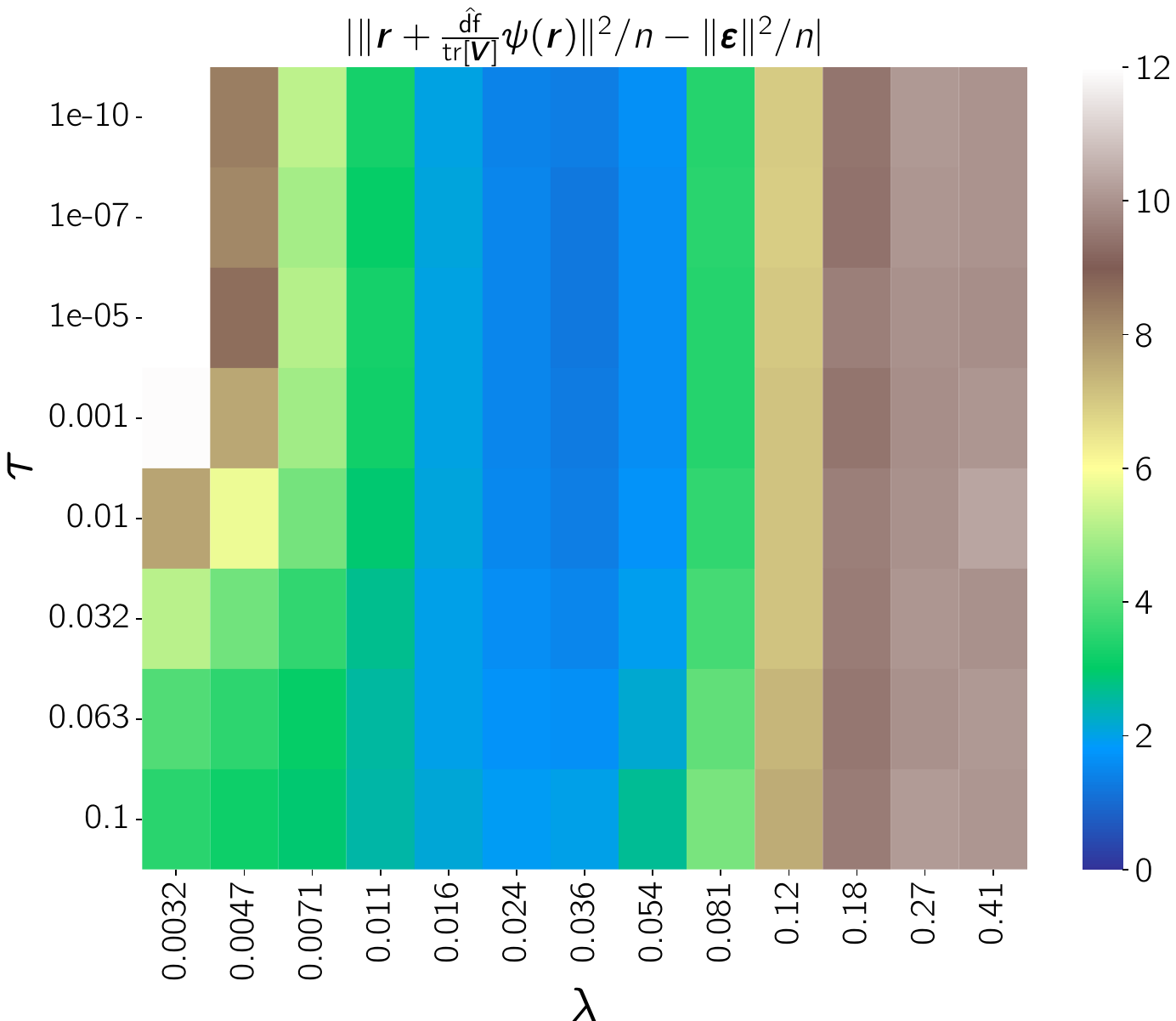}
    \includegraphics[width=0.32\textwidth]{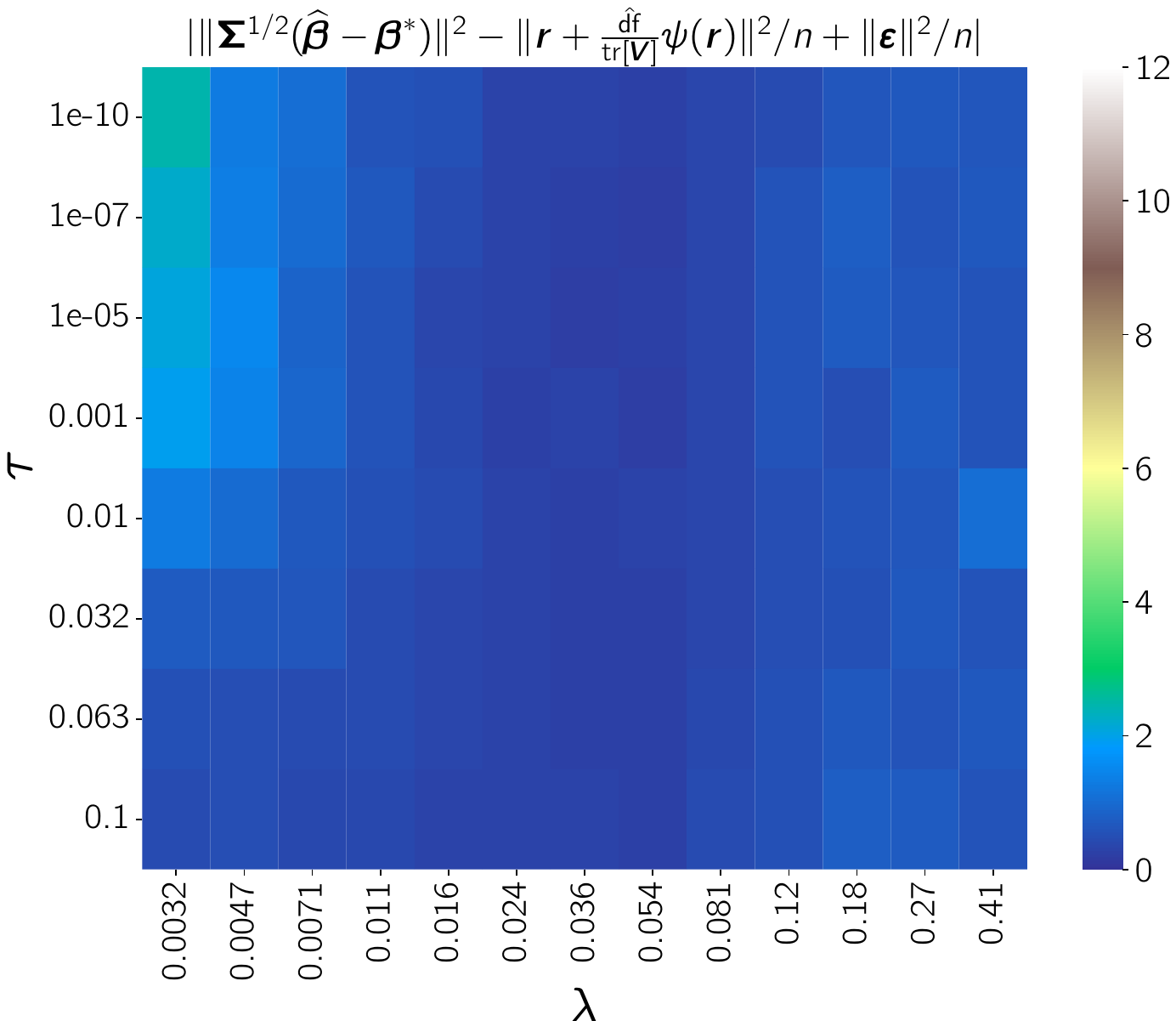}

    \caption{
    Heatmaps for the Huber loss and Elastic-Net penalty on a grid of tuning parameters  
    with $\Lambda = 0.054 n^{1/2}$ 
    and $(\lambda, \tau)$ where $\lambda \in [0.0032, 0.41]$ and $\tau \in [10^{-10}, 0.1]$.
    Each cell is the average over 100 repetitions.
    See the simulation setup in \Cref{sec:simulations} in the paper for more details.}
\end{figure}

\newpage

\begin{figure}[ht]
    \centering

    \includegraphics[width=0.32\textwidth]{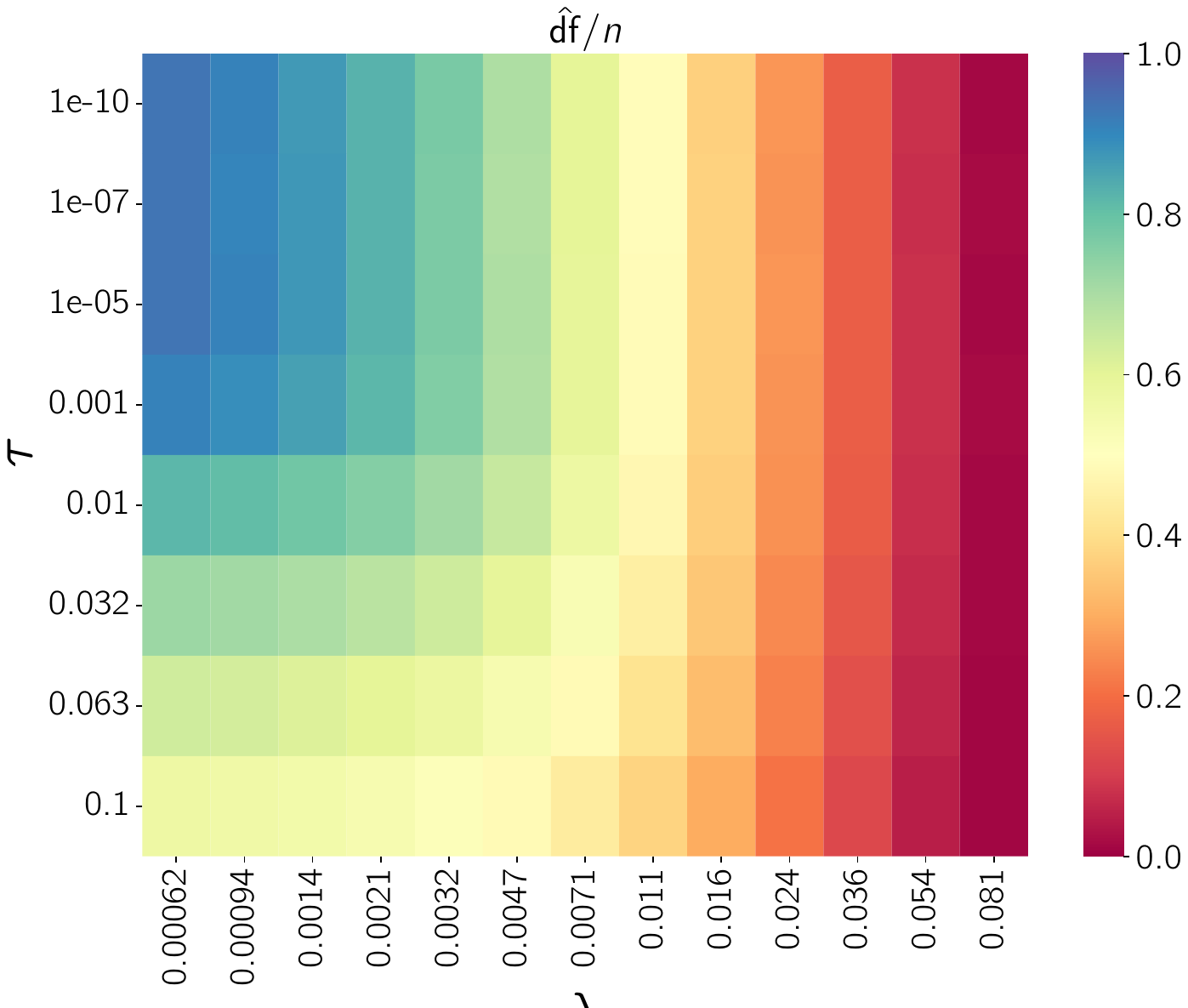}
    \includegraphics[width=0.32\textwidth]{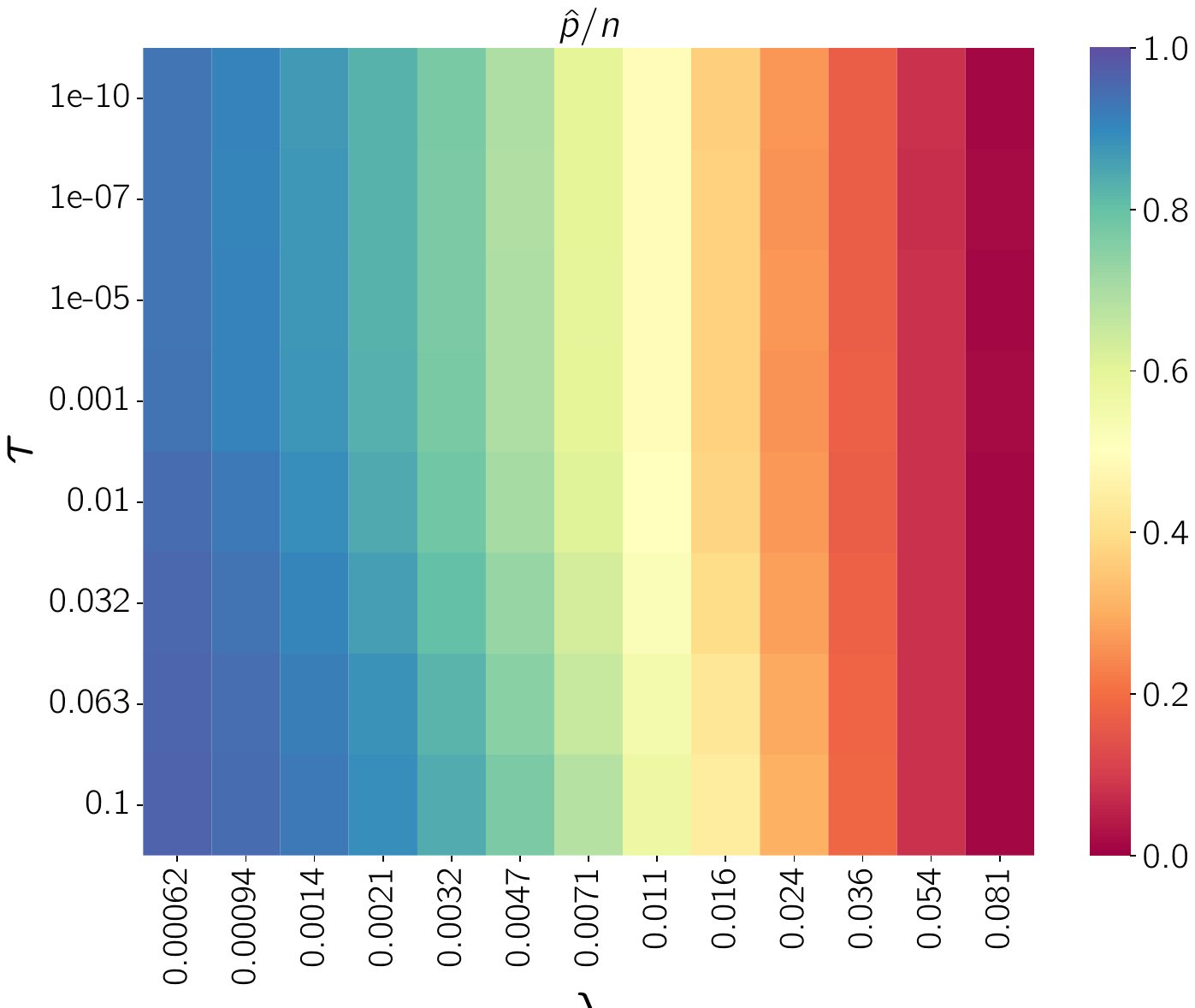}
    \includegraphics[width=0.32\textwidth]{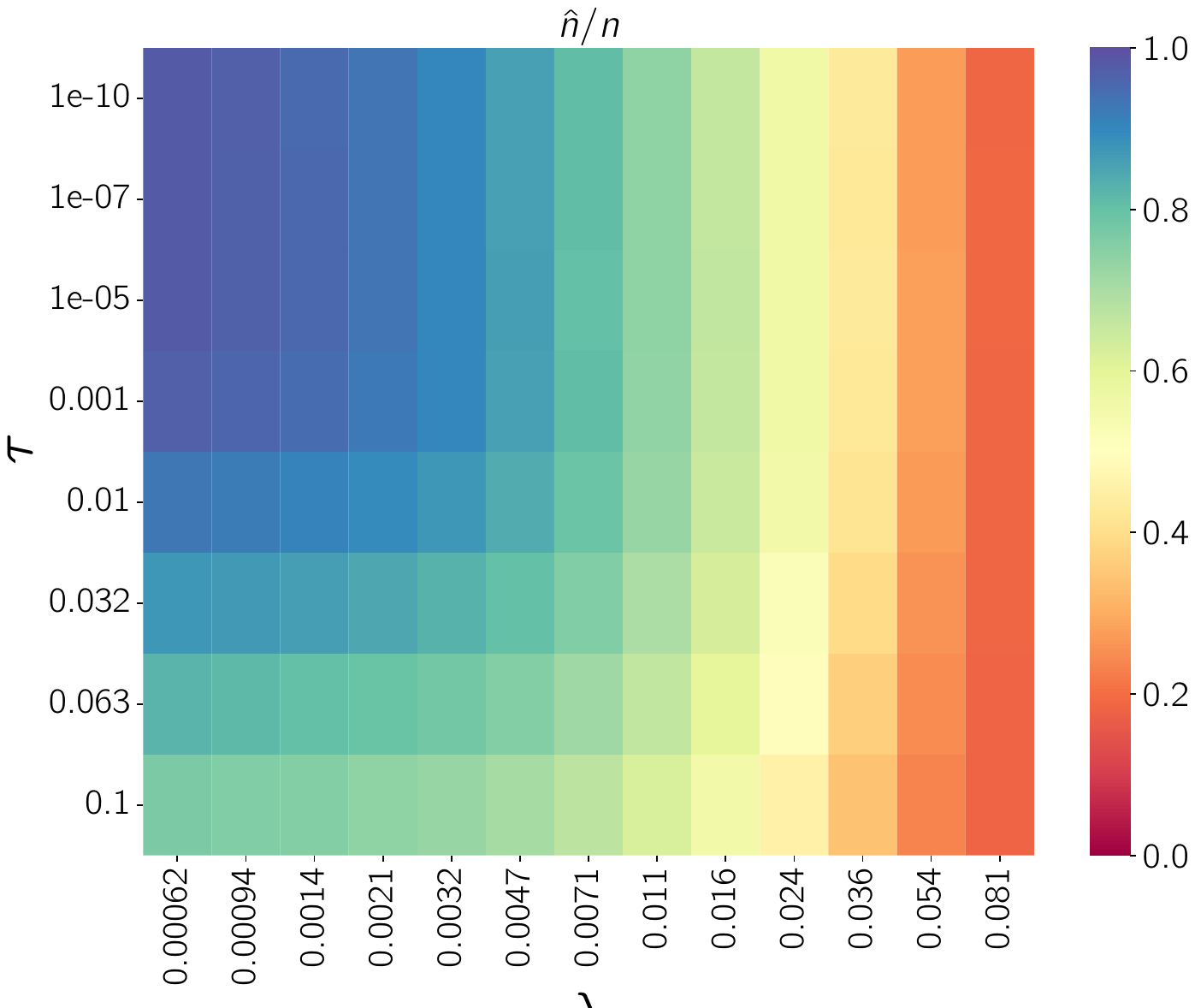}

    \includegraphics[width=0.32\textwidth]{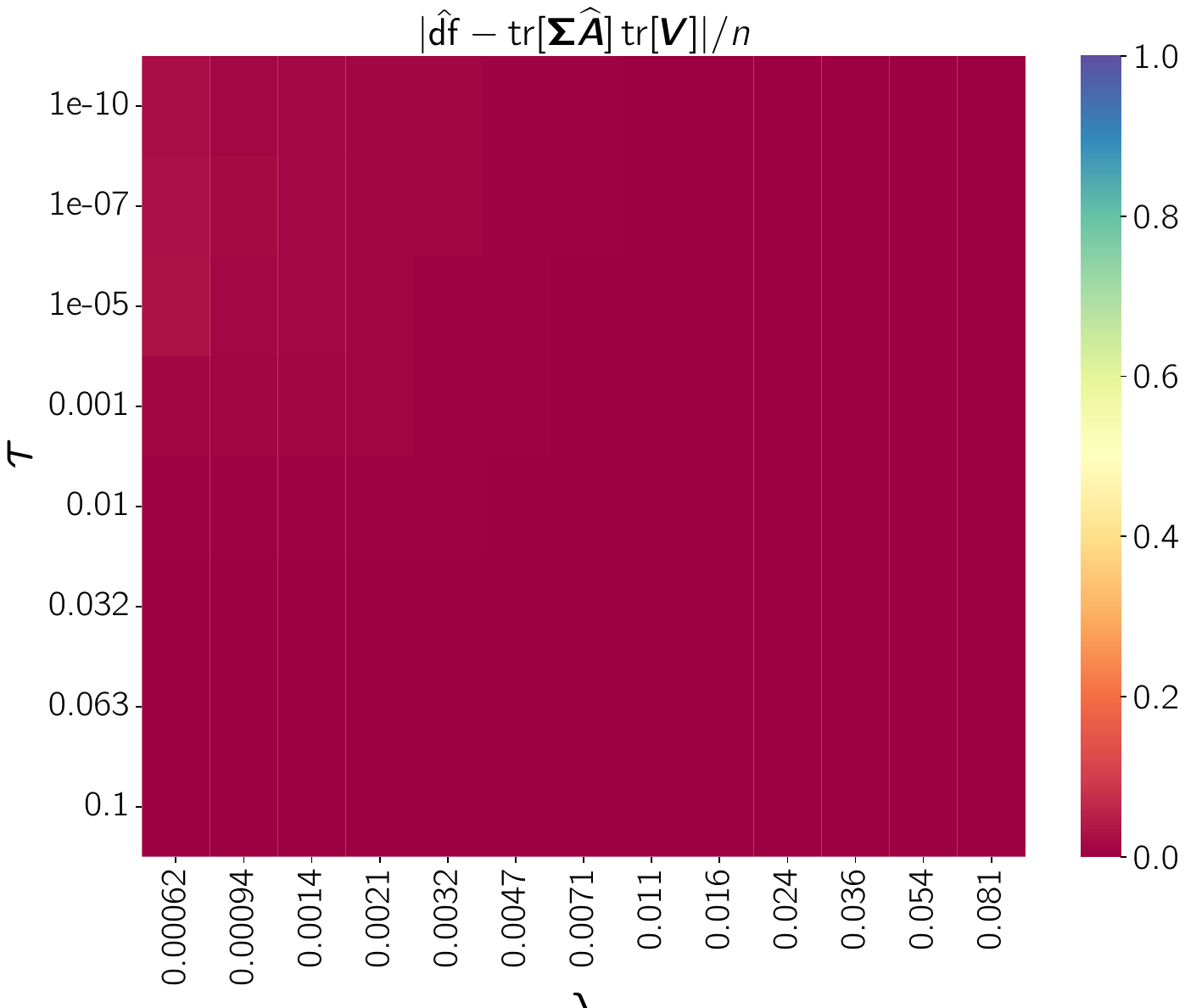}
    \includegraphics[width=0.32\textwidth]{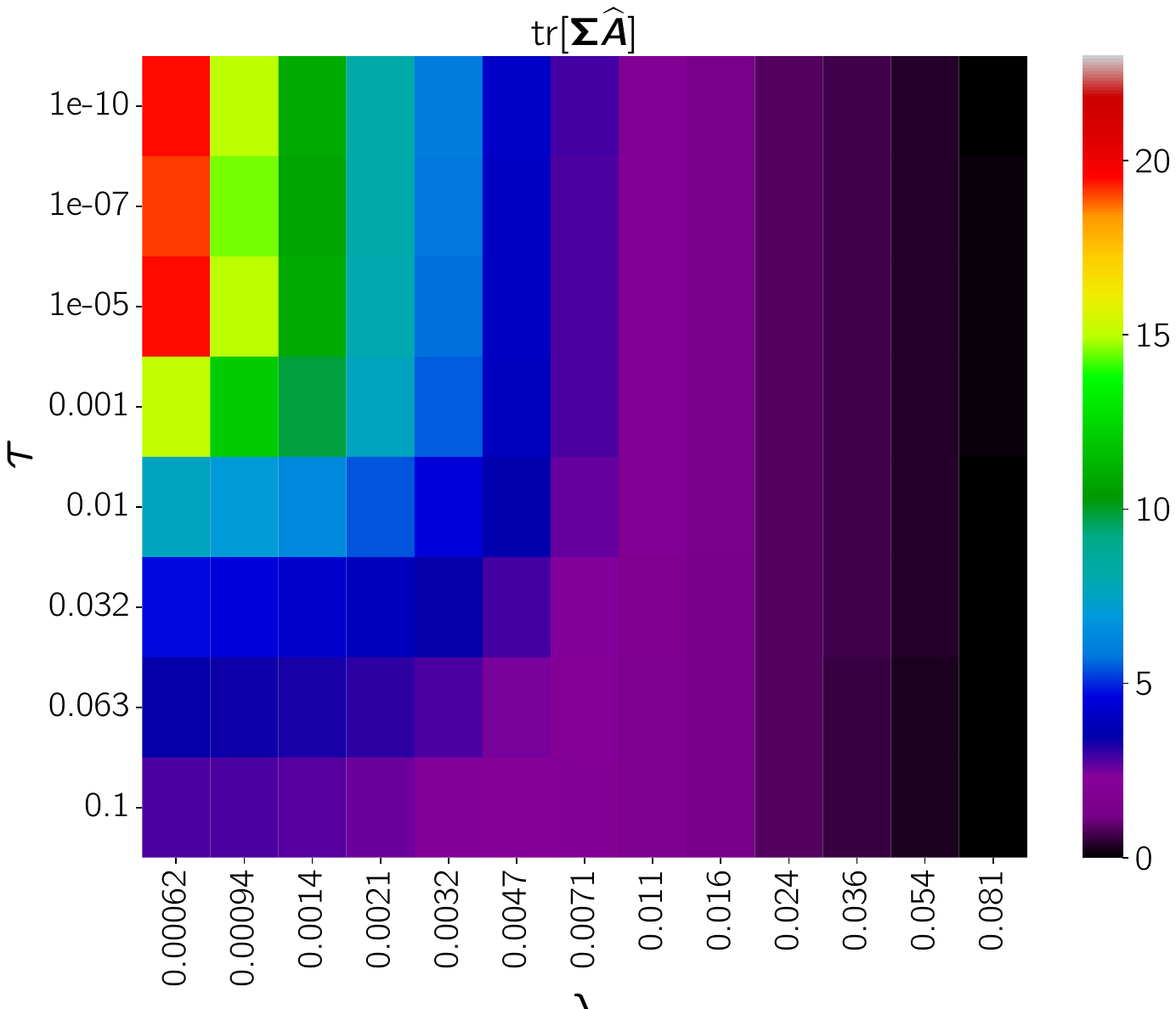}
    \includegraphics[width=0.32\textwidth]{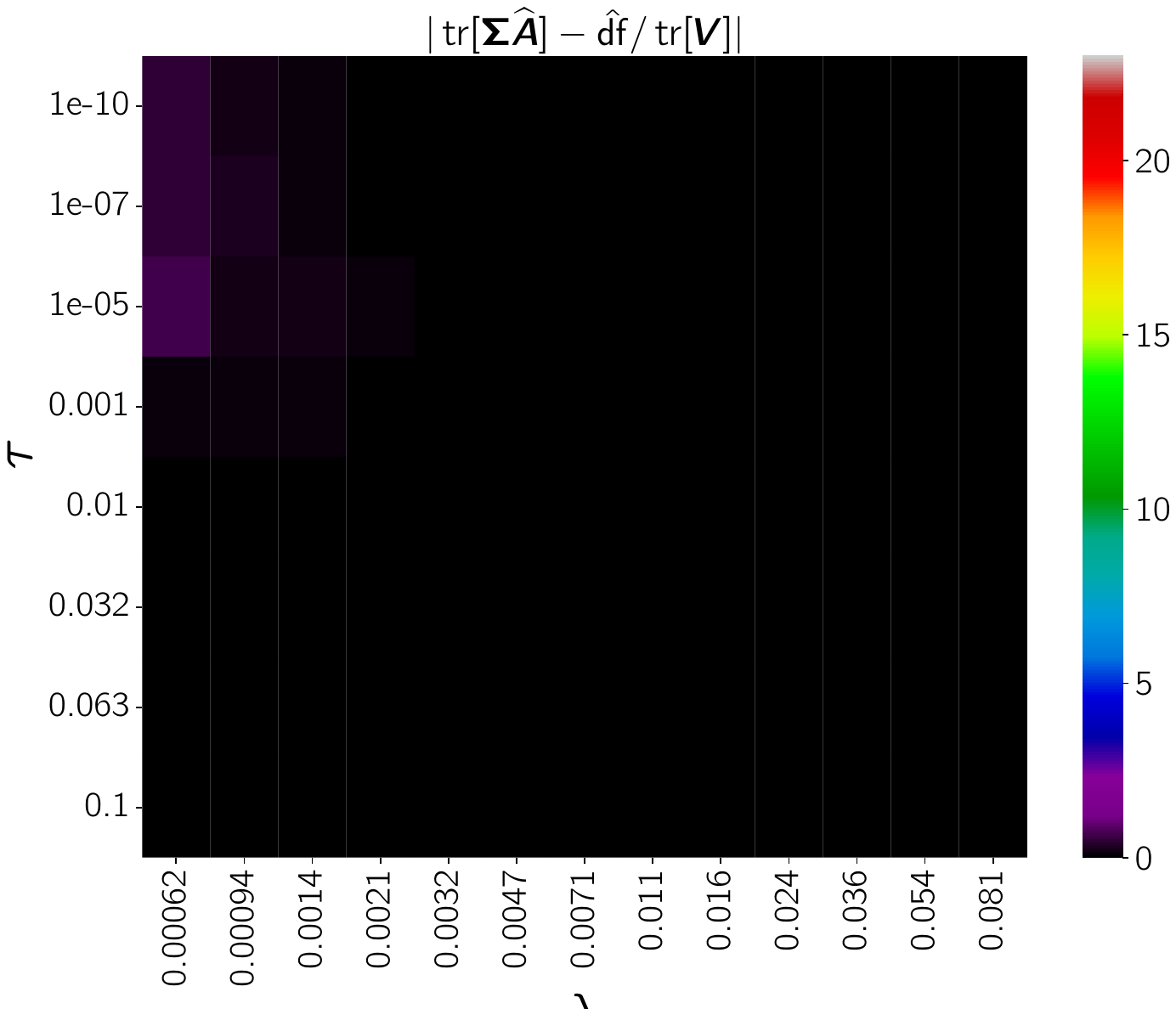}

    \includegraphics[width=0.32\textwidth]{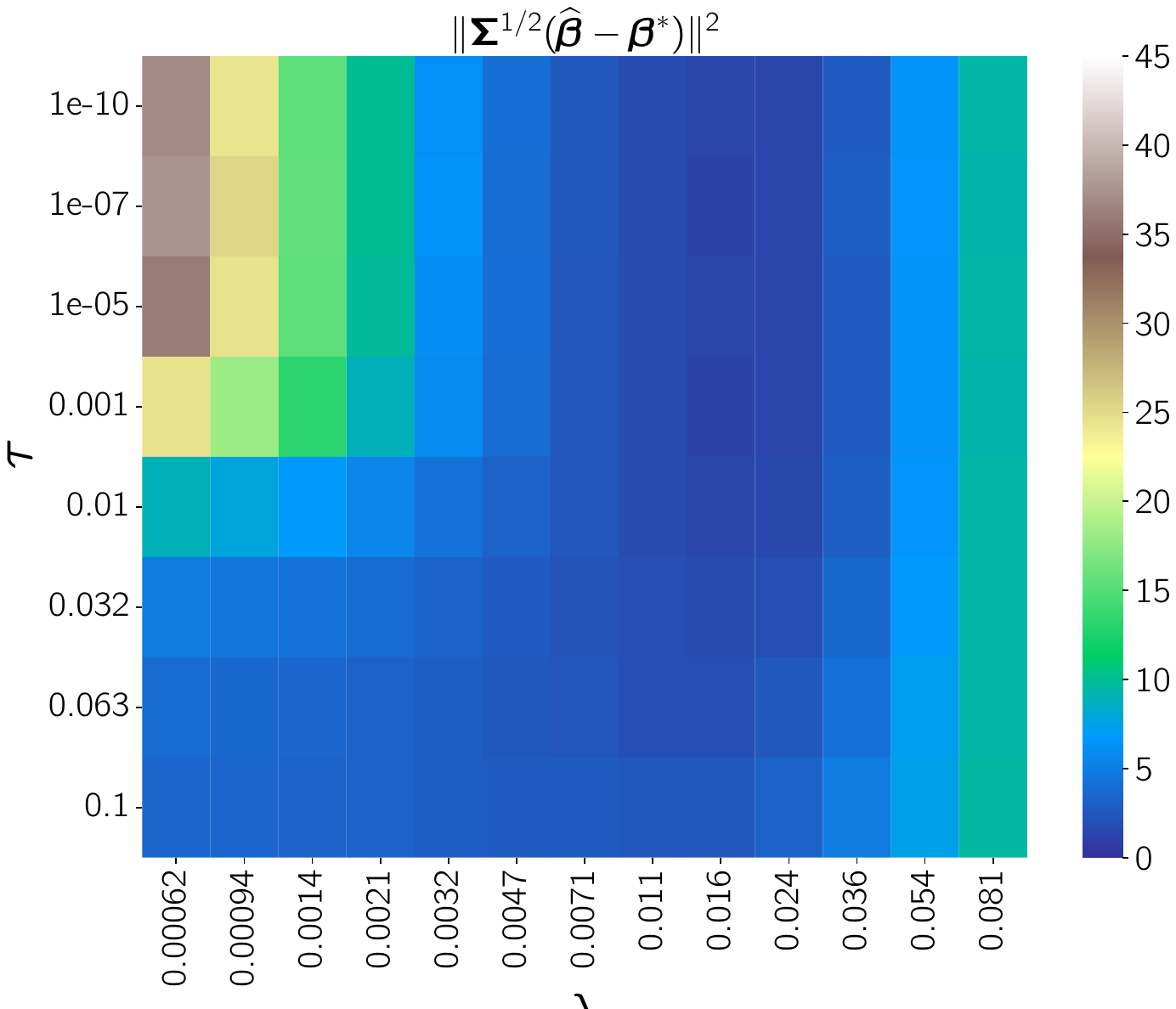}
    \includegraphics[width=0.32\textwidth]{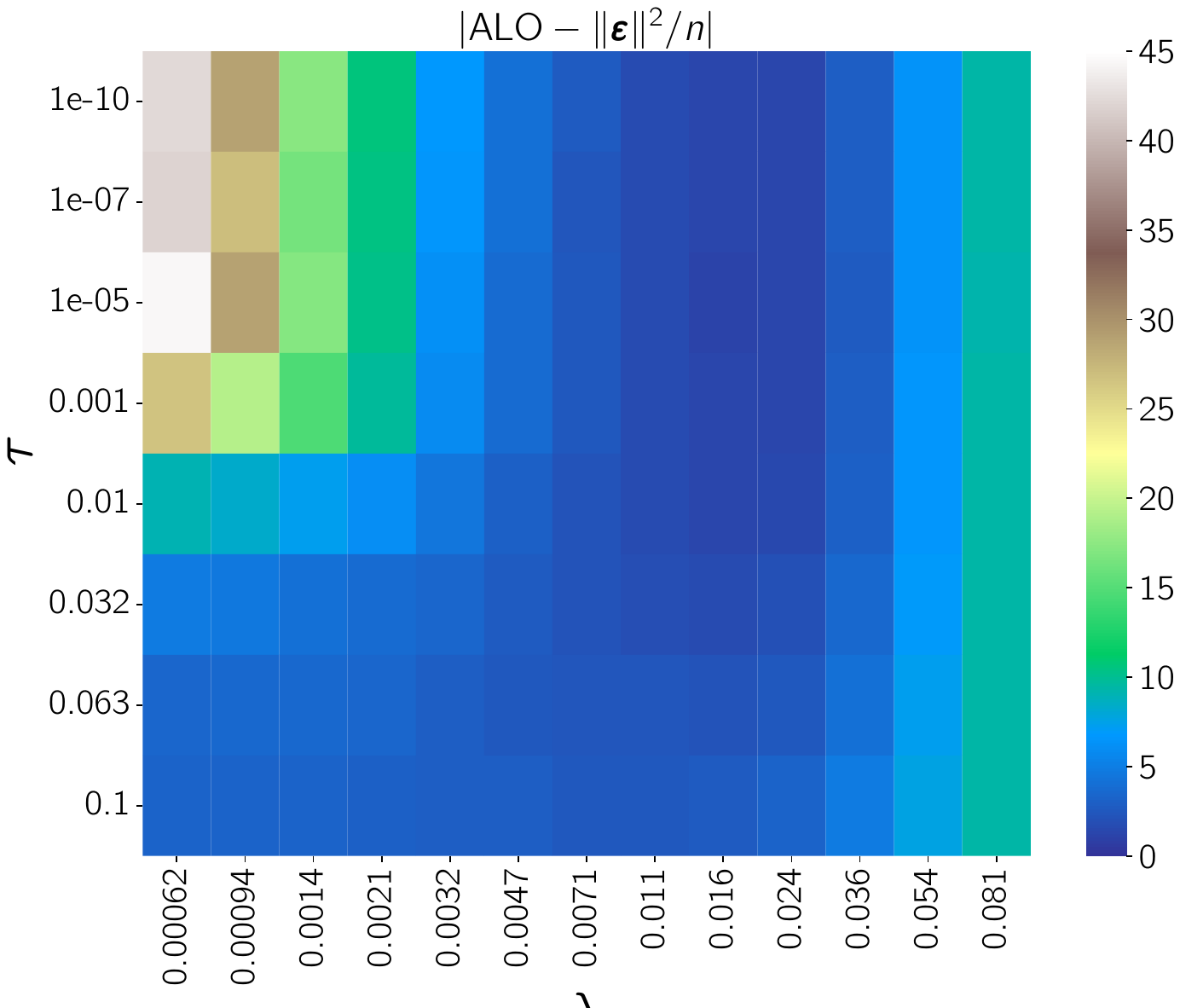}
    \includegraphics[width=0.32\textwidth]{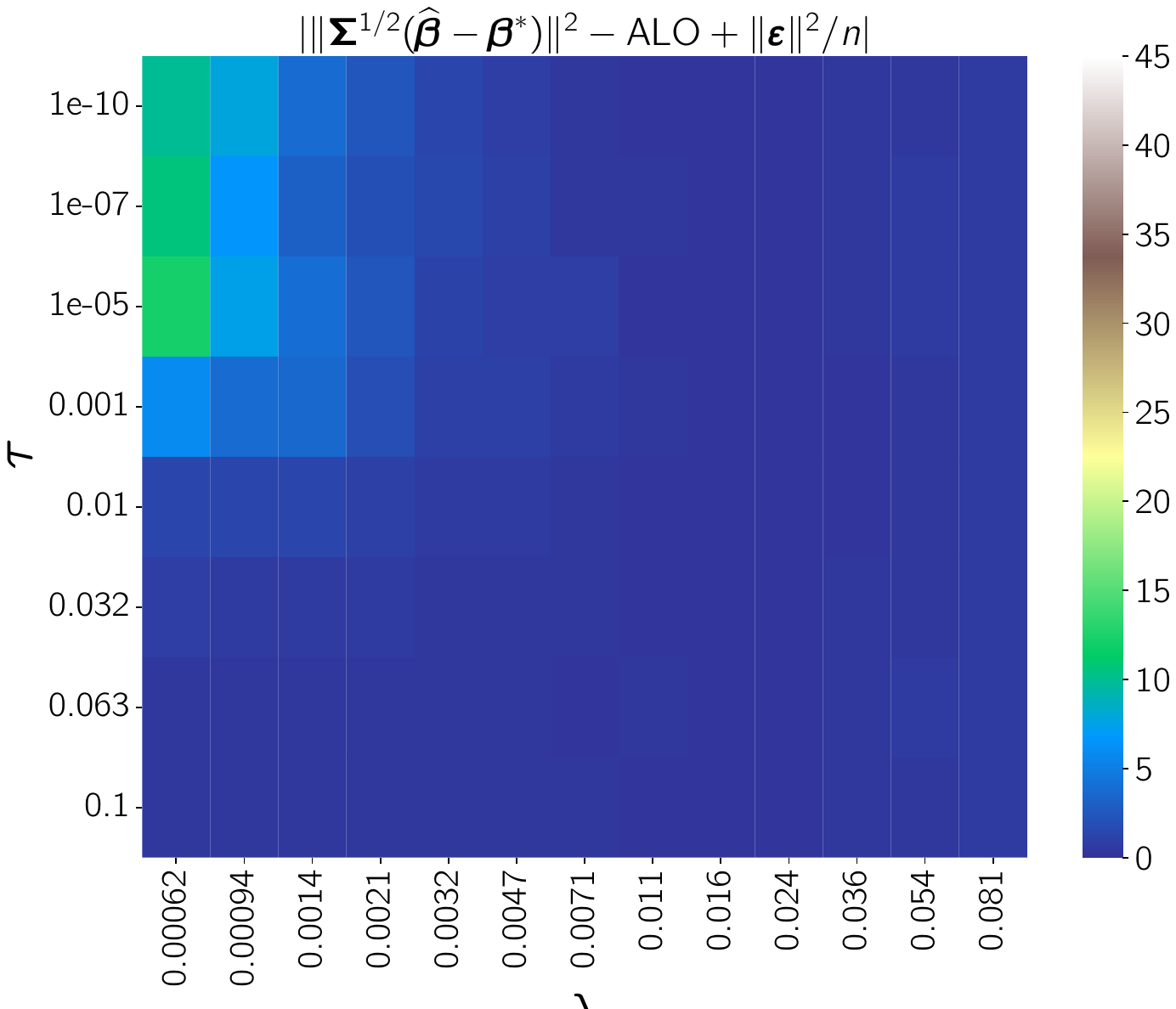}

    \includegraphics[width=0.32\textwidth]{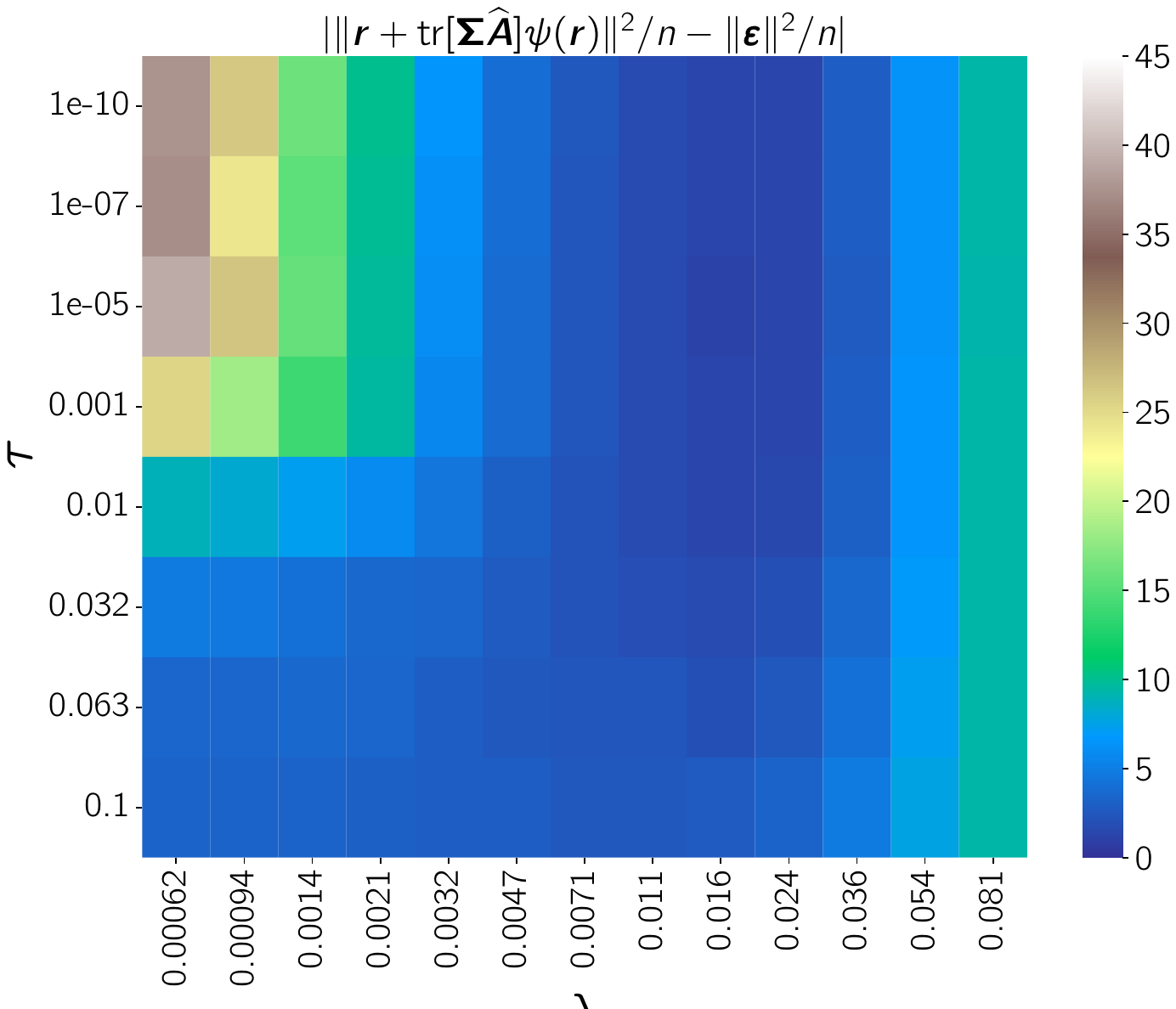}
    \includegraphics[width=0.32\textwidth]{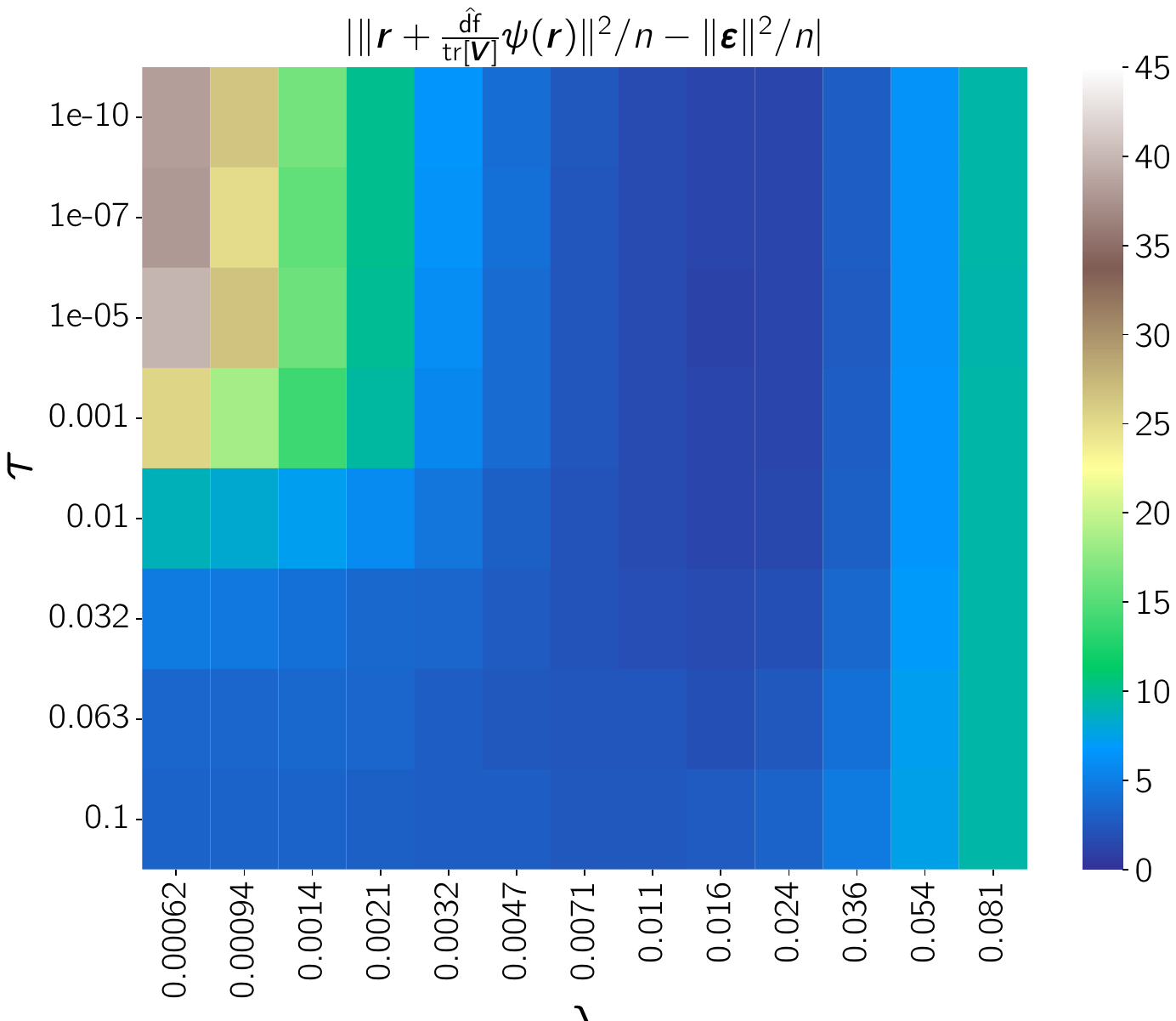}
    \includegraphics[width=0.32\textwidth]{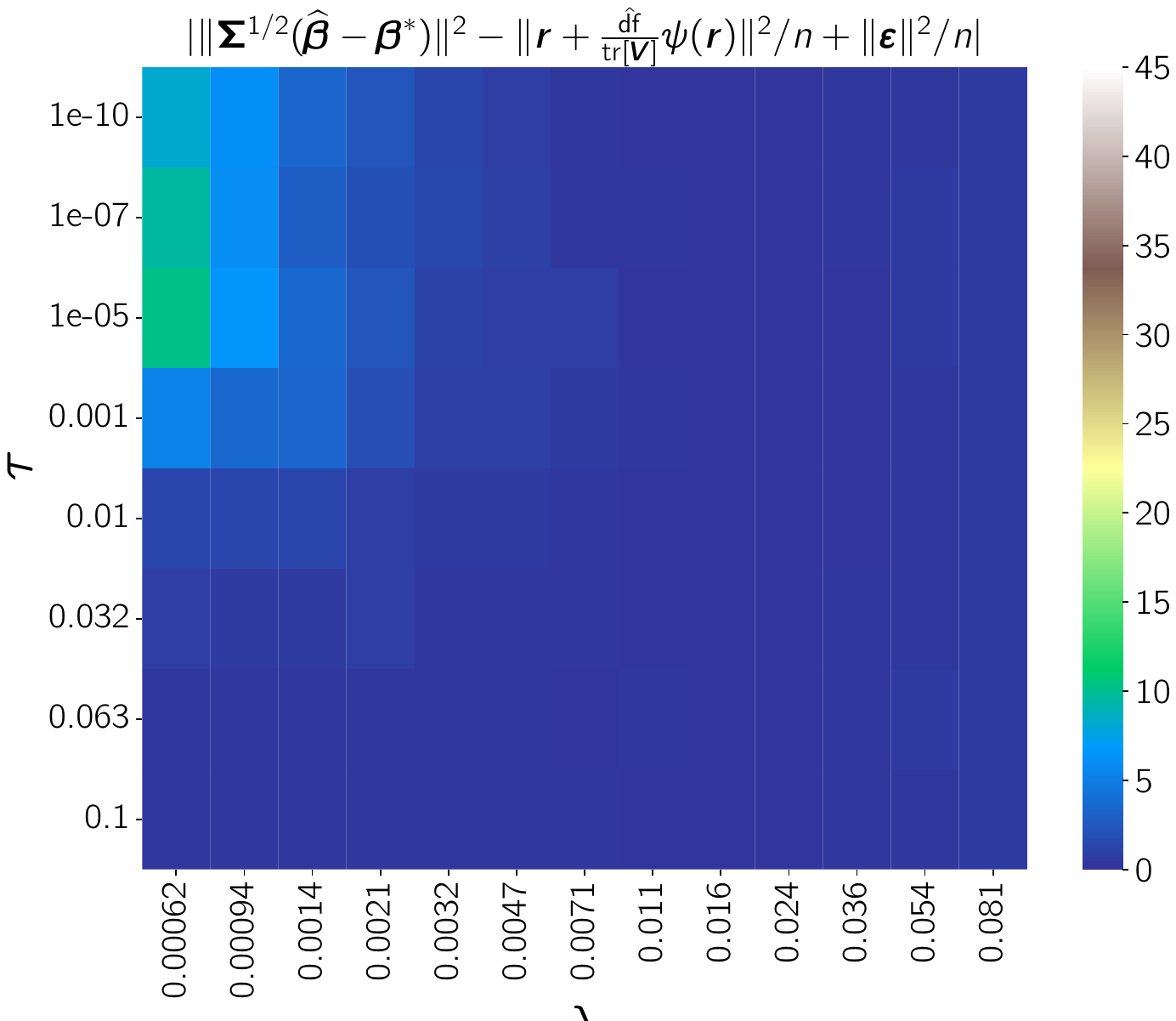}

    \caption{
    Heatmaps for the Huber loss and Elastic-Net penalty on a grid of tuning parameters 
    with $\Lambda = 0.024 n^{1/2}$ and $(\lambda, \tau)$ where 
    $\lambda \in [0.00062, 0.081]$
    and 
    $\tau \in [10^{-10}, 0.1]$.
    Each cell is the average over 50 repetitions.
    See the simulation setup in \Cref{sec:simulations} in the paper for more details.}
\end{figure}

\newpage

\begin{figure}[ht]
    \centering

    \includegraphics[width=0.32\textwidth]{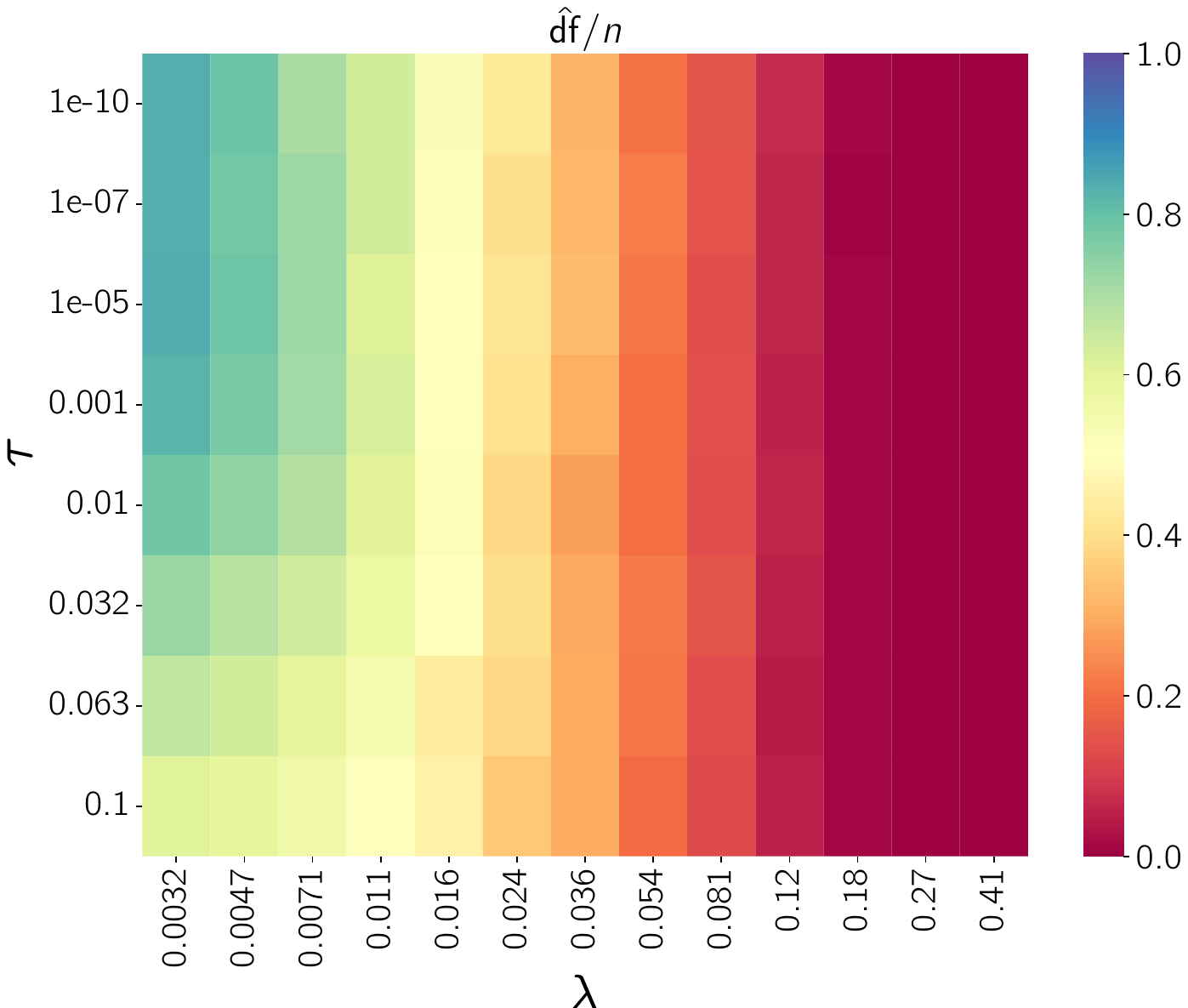}
    \includegraphics[width=0.32\textwidth]{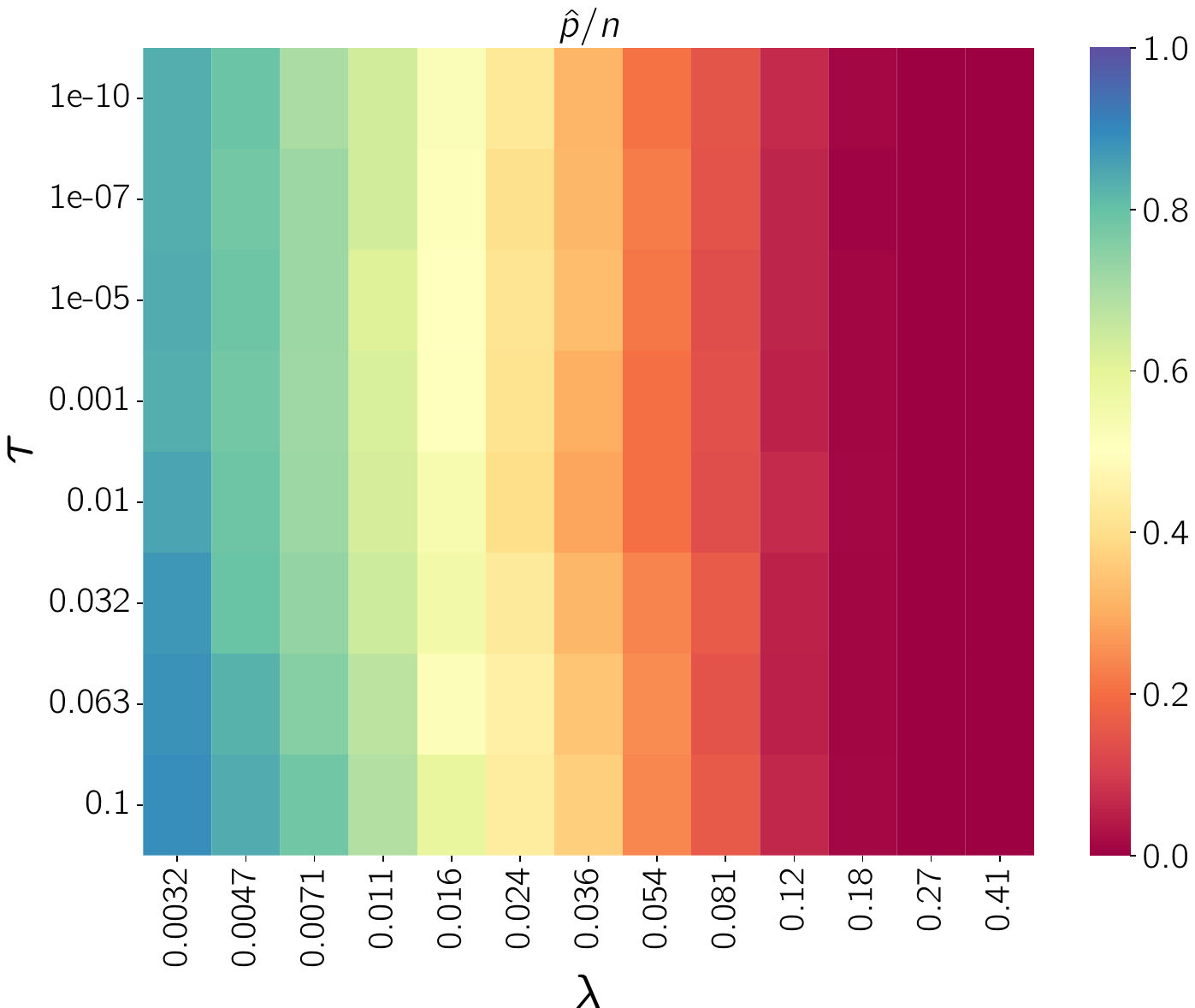}
    \includegraphics[width=0.32\textwidth]{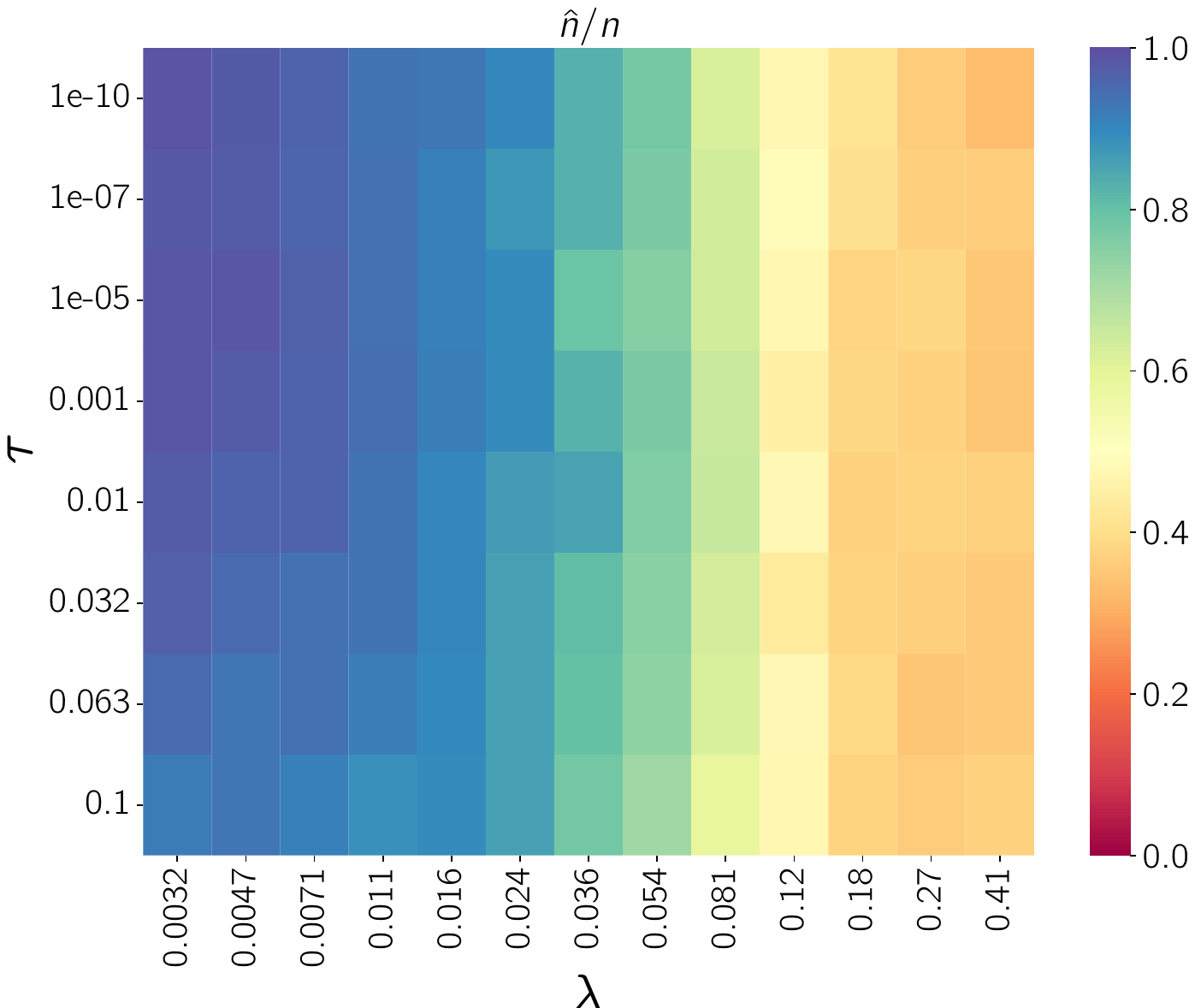}

    \includegraphics[width=0.32\textwidth]{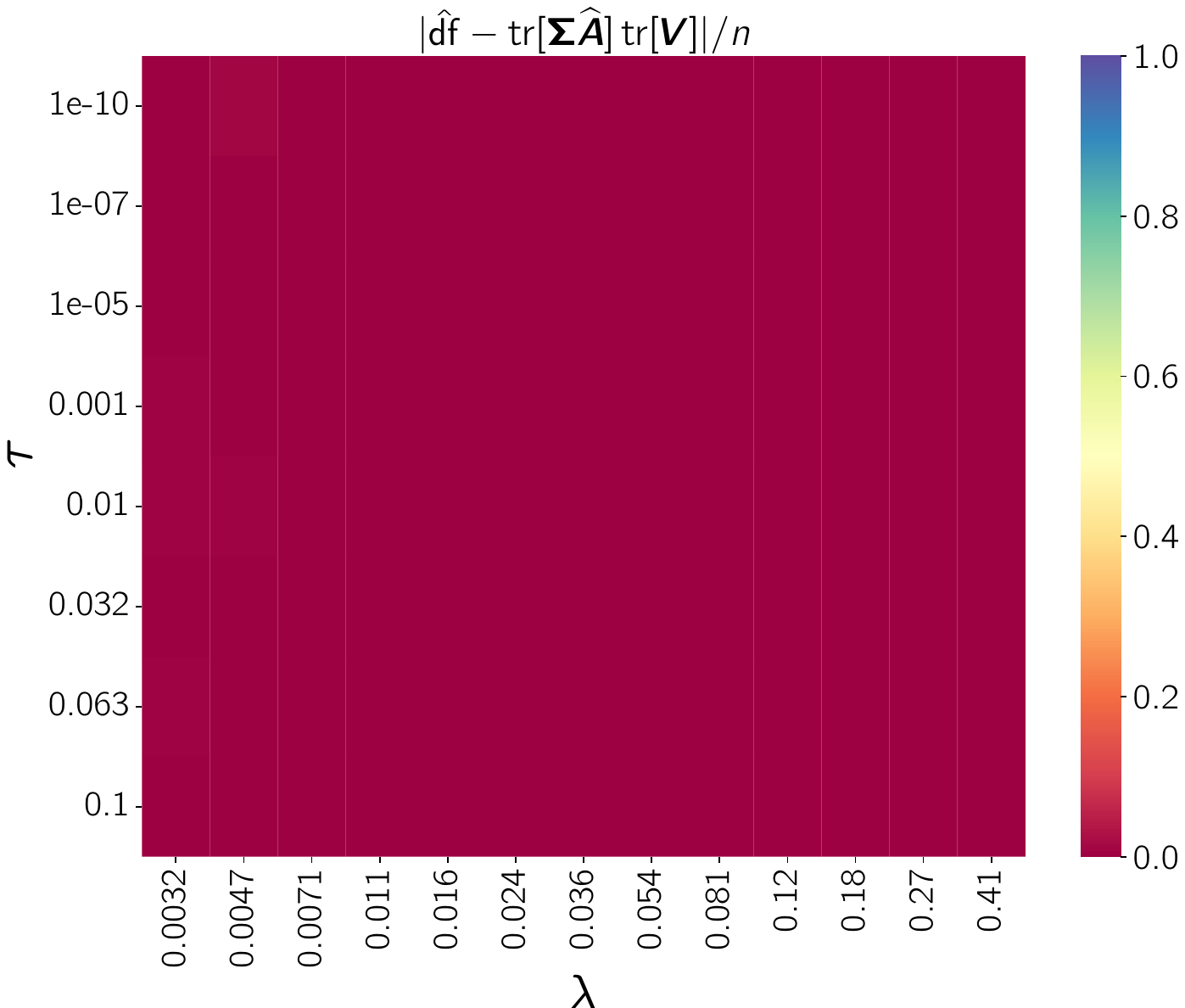}
    \includegraphics[width=0.32\textwidth]{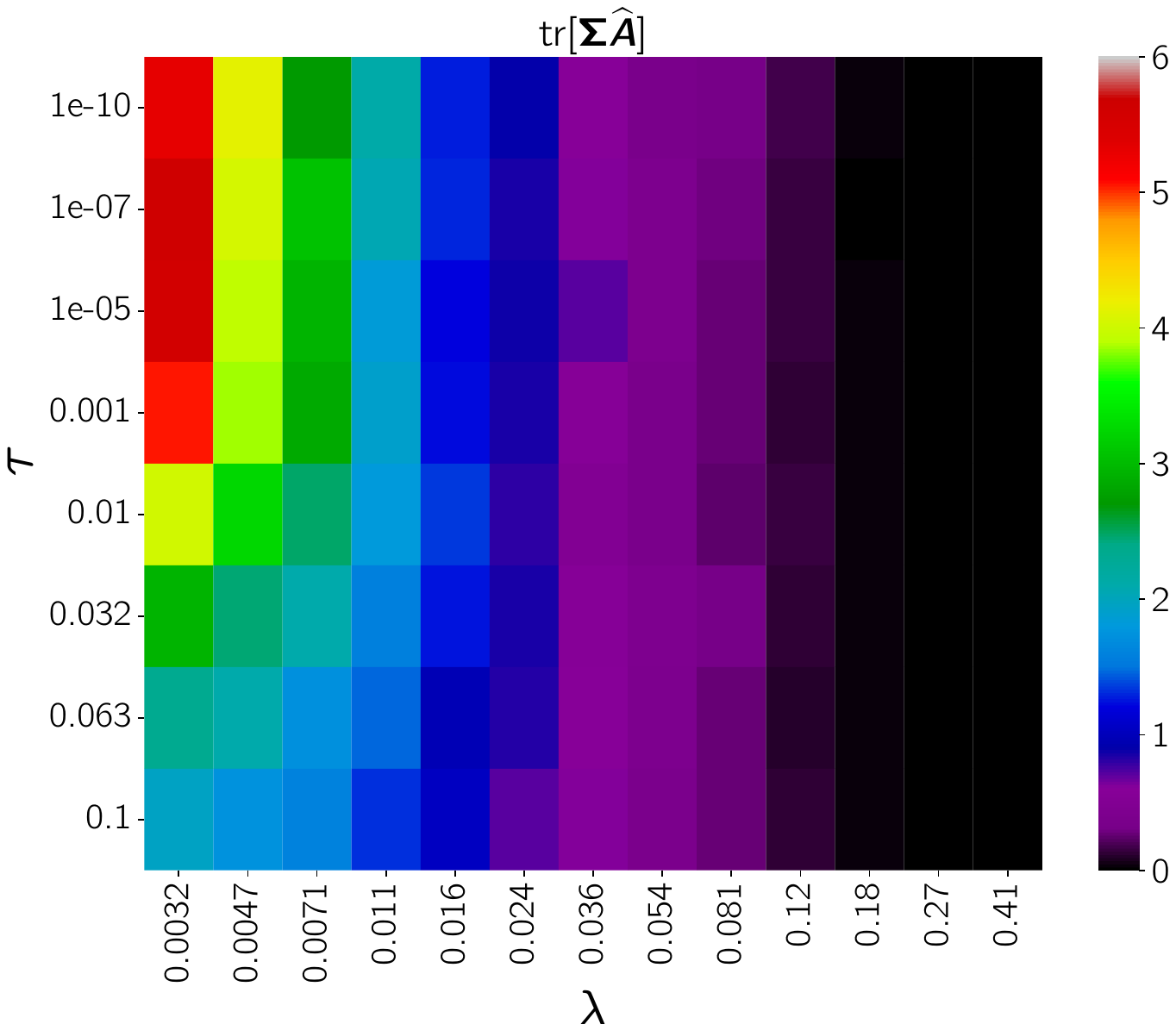}
    \includegraphics[width=0.32\textwidth]{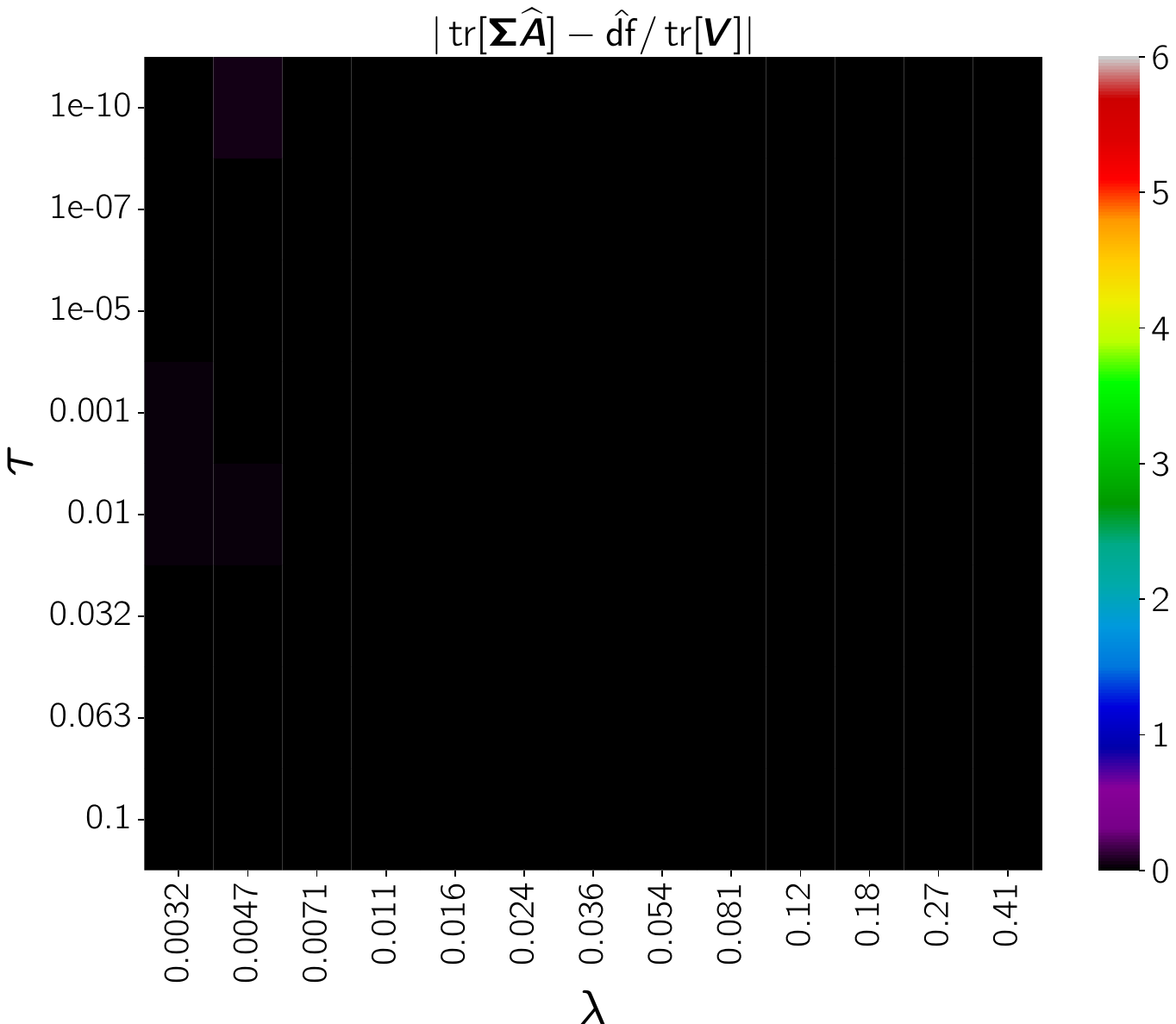}

    \includegraphics[width=0.32\textwidth]{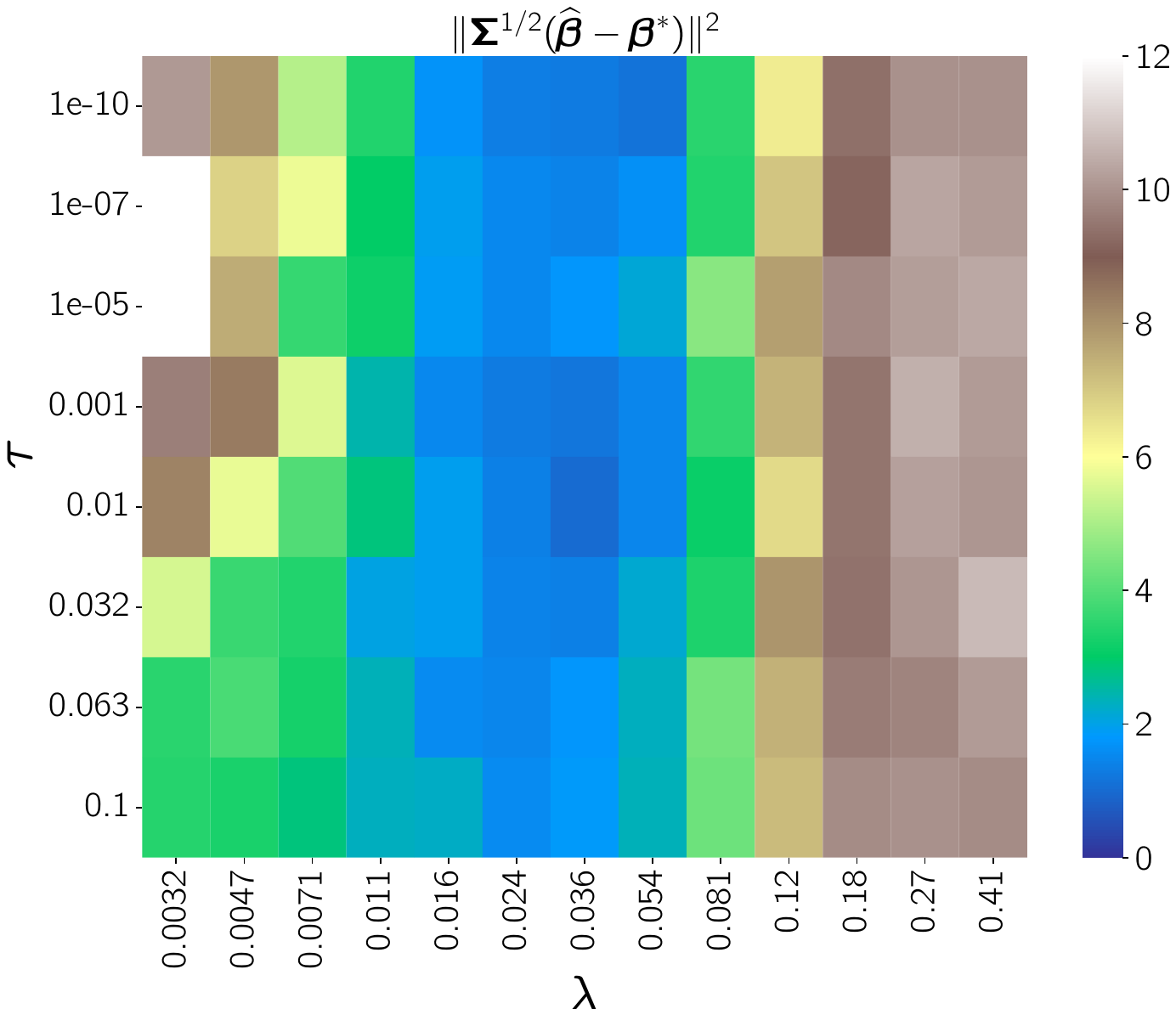}
    \includegraphics[width=0.32\textwidth]{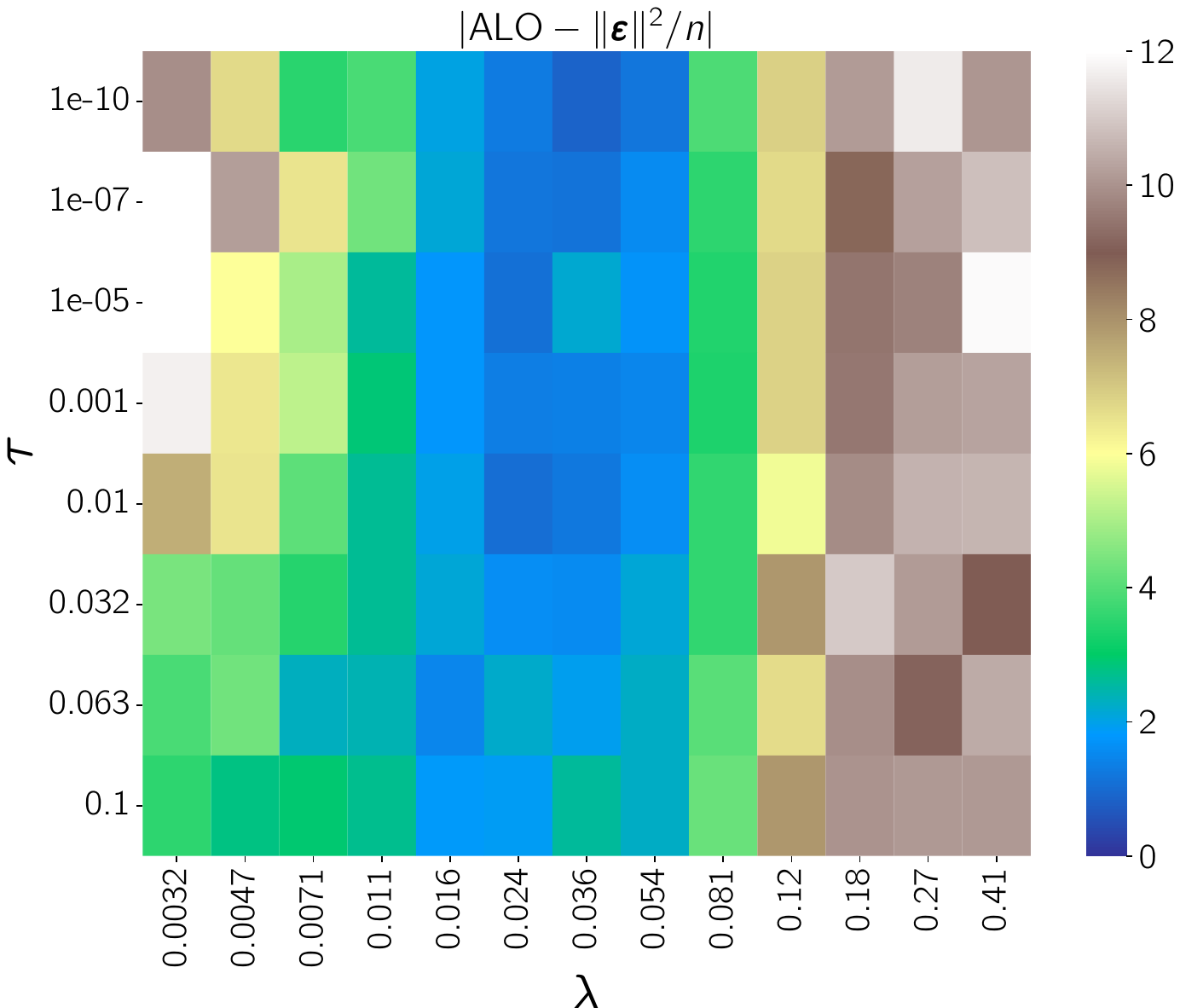}
    \includegraphics[width=0.32\textwidth]{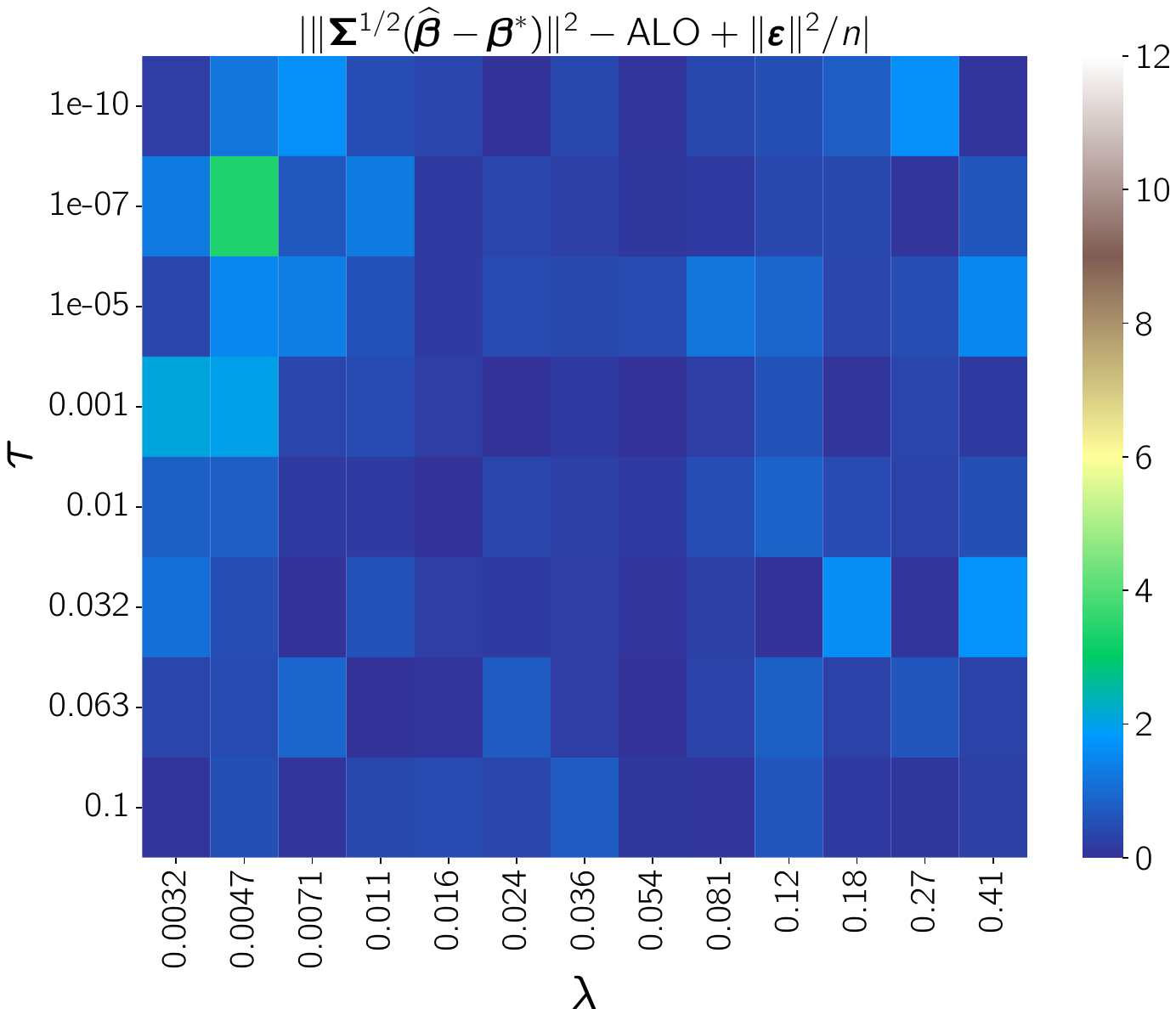}

    \includegraphics[width=0.32\textwidth]{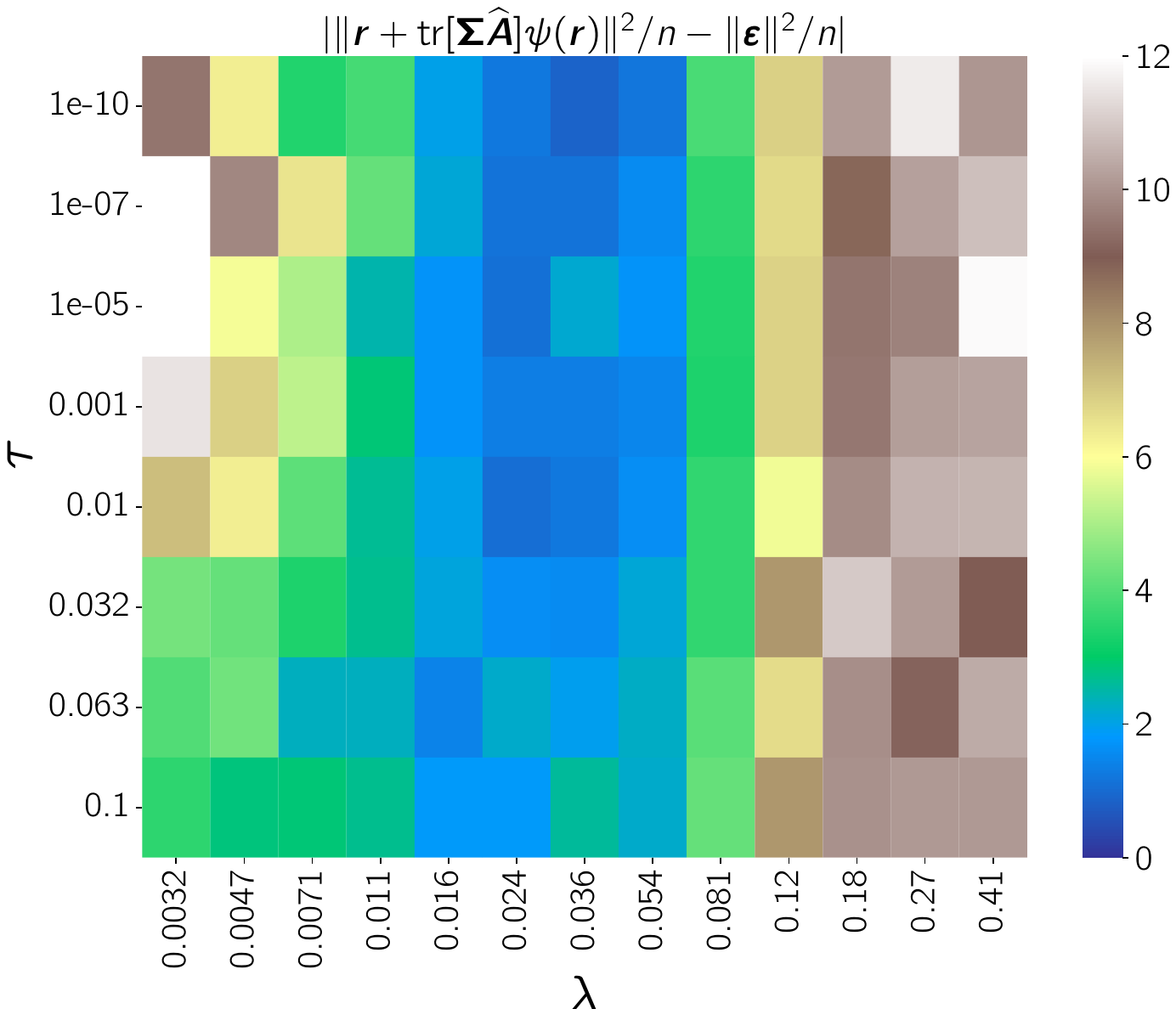}
    \includegraphics[width=0.32\textwidth]{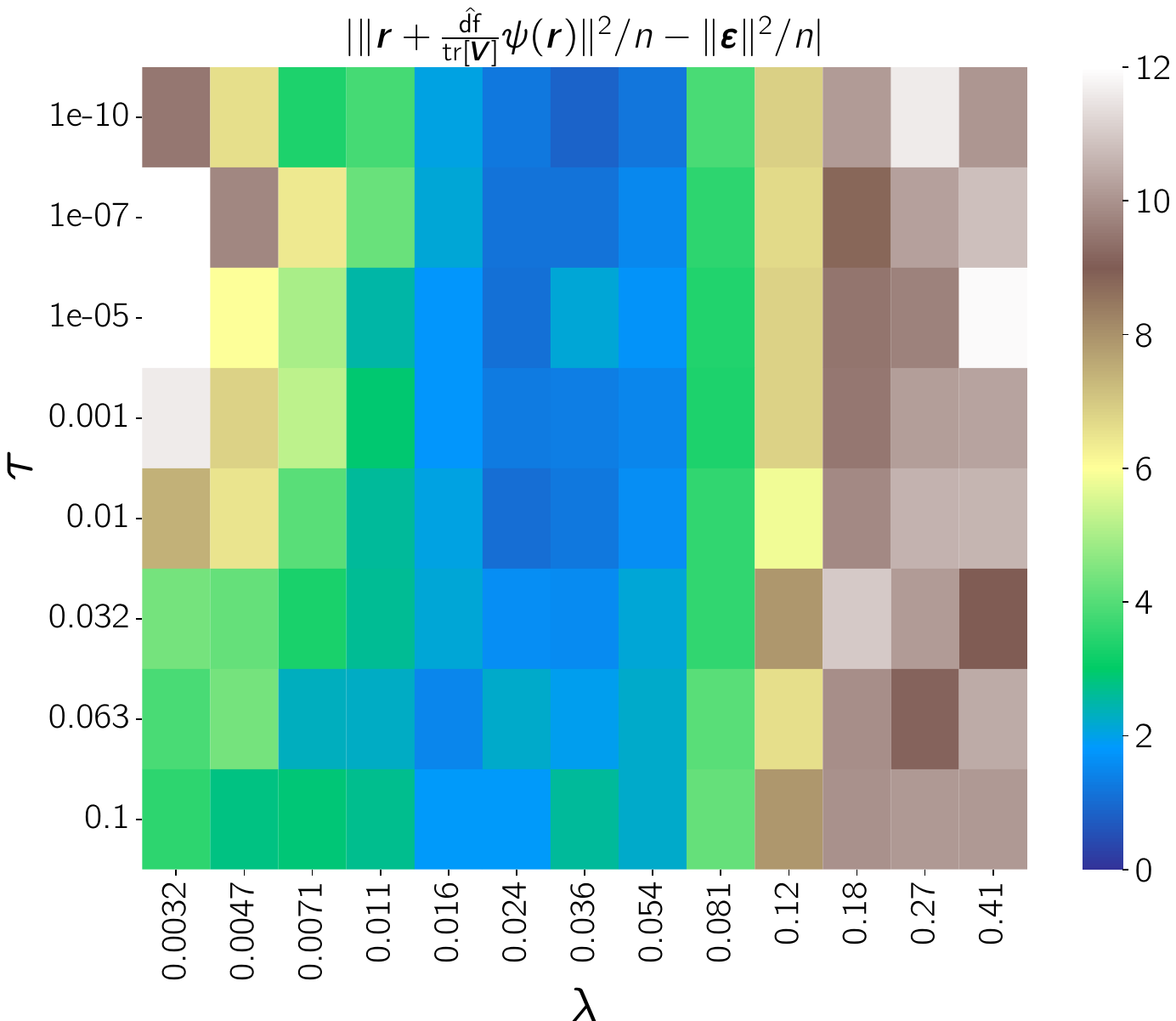}
    \includegraphics[width=0.32\textwidth]{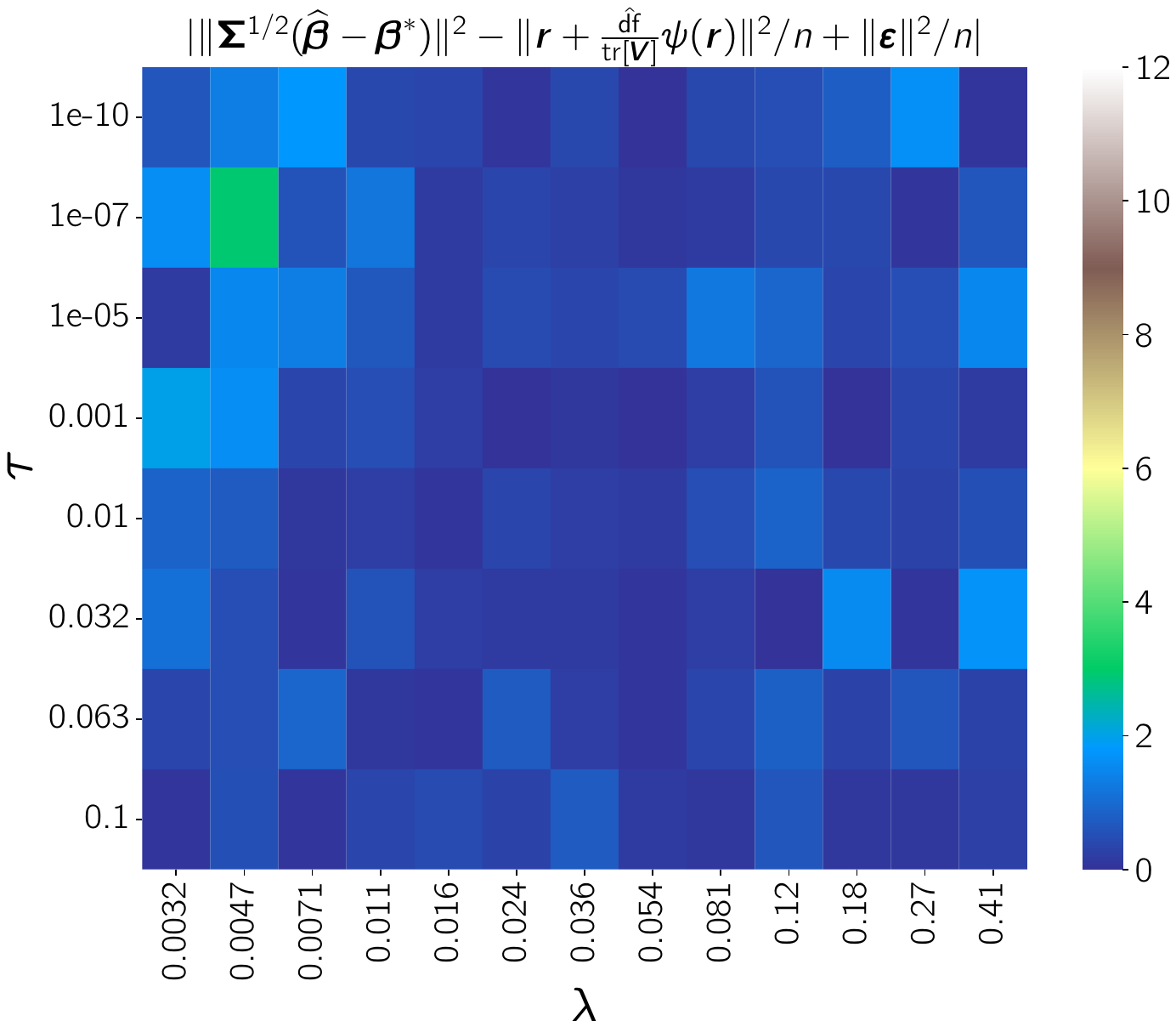}

    \caption{
        \label{fig6}
    Heatmaps for the Huber loss and Elastic-Net penalty on a grid of tuning parameters  
    with $\Lambda = 0.054 n^{1/2}$ 
    and $(\lambda, \tau)$ where $\lambda \in [0.0032, 0.41]$ and $\tau \in [10^{-10}, 0.1]$.
    Each cell is over 1 repetition.
    See the simulation setup in \Cref{sec:simulations} in the paper for more details.}
\end{figure}

\newpage

\begin{figure}[ht]
    \centering

    \includegraphics[width=0.32\textwidth]{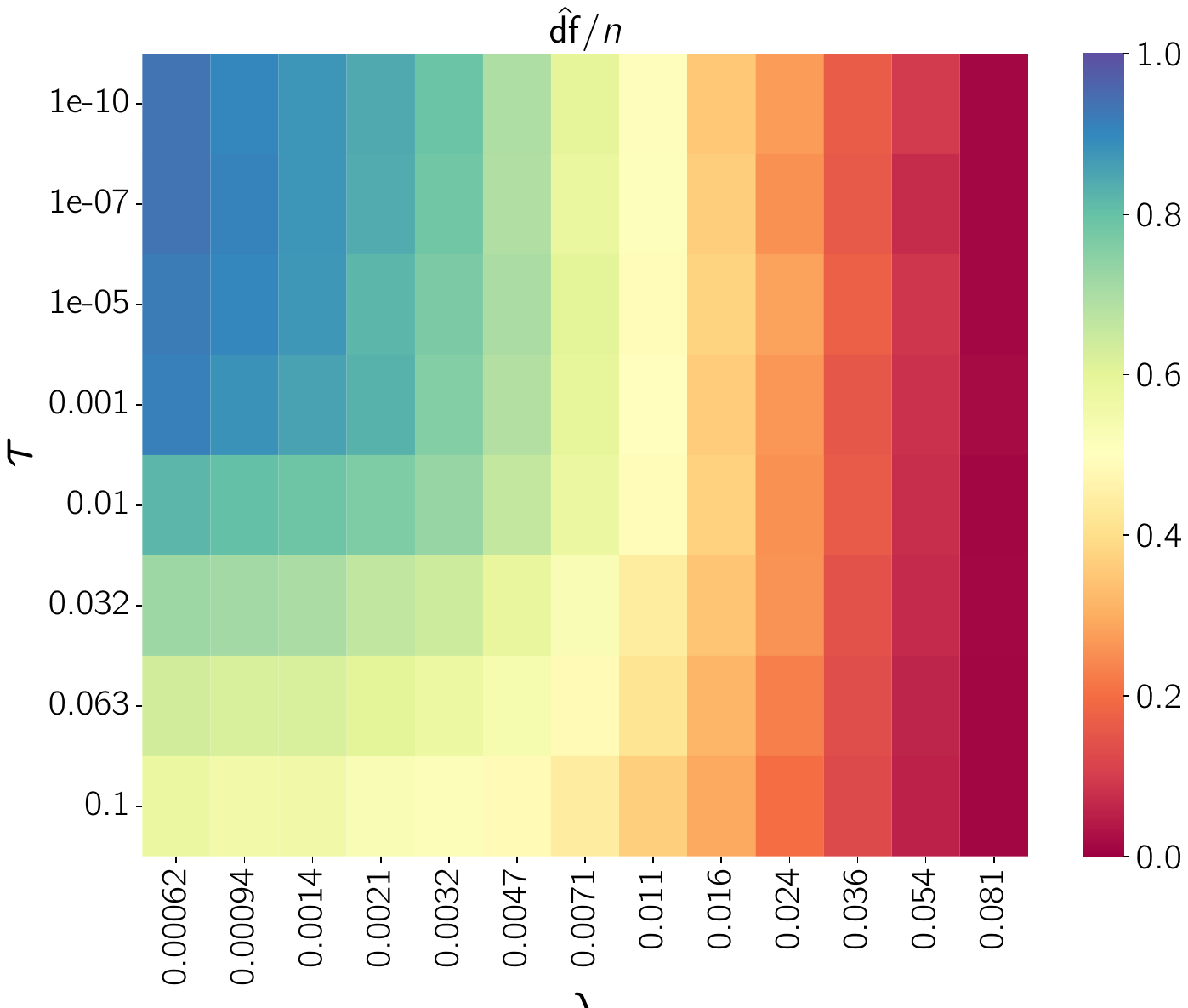}
    \includegraphics[width=0.32\textwidth]{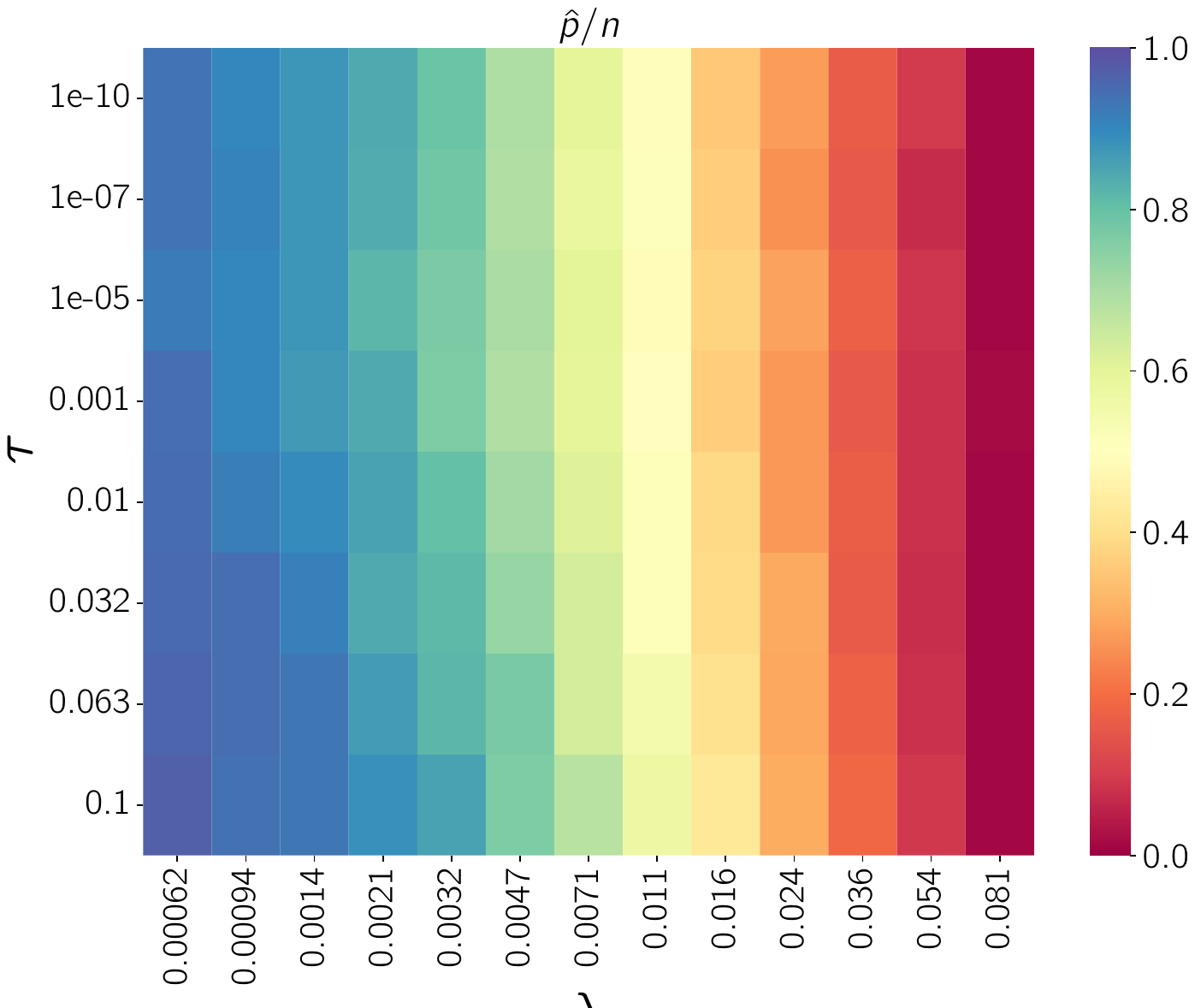}
    \includegraphics[width=0.32\textwidth]{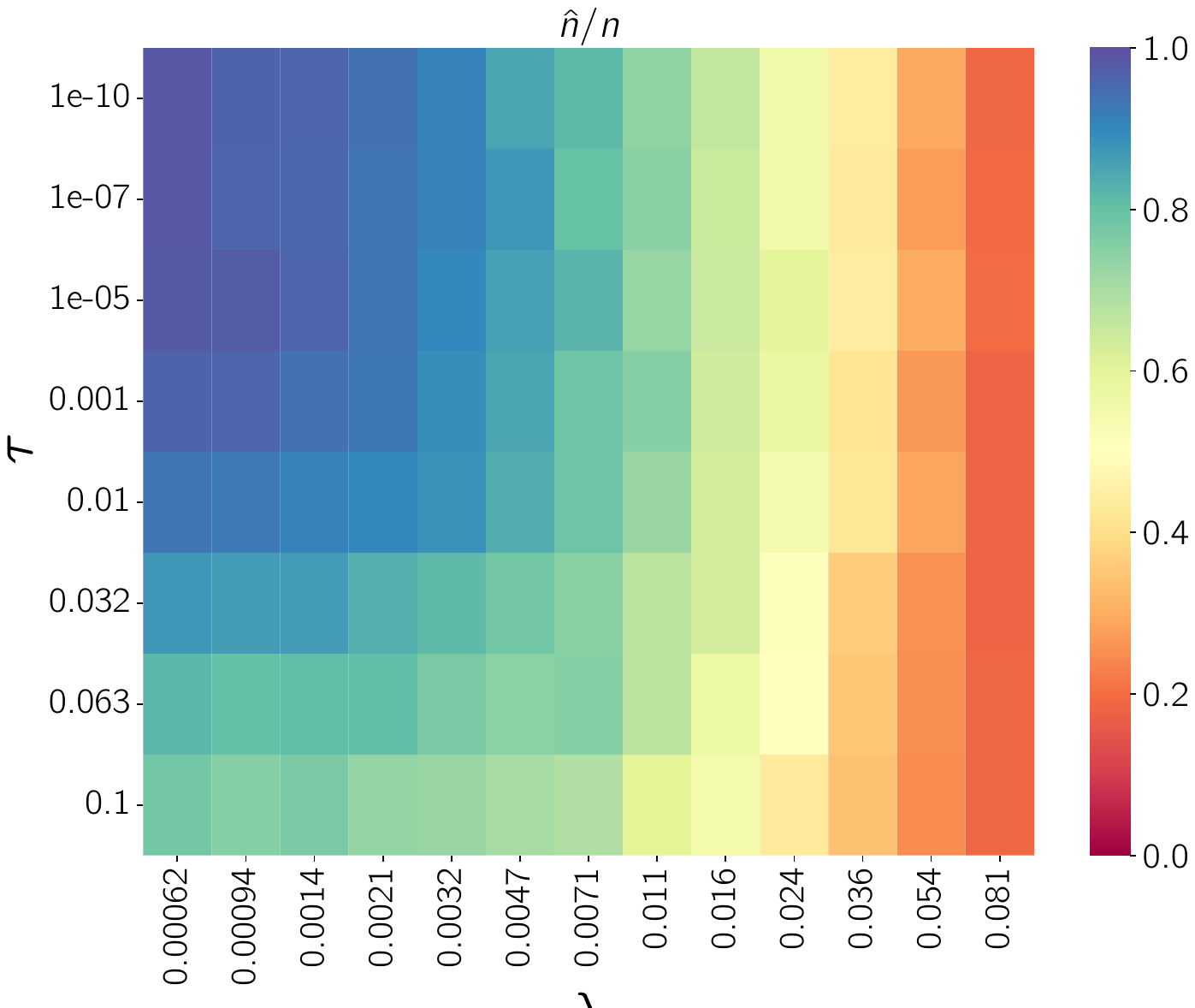}

    \includegraphics[width=0.32\textwidth]{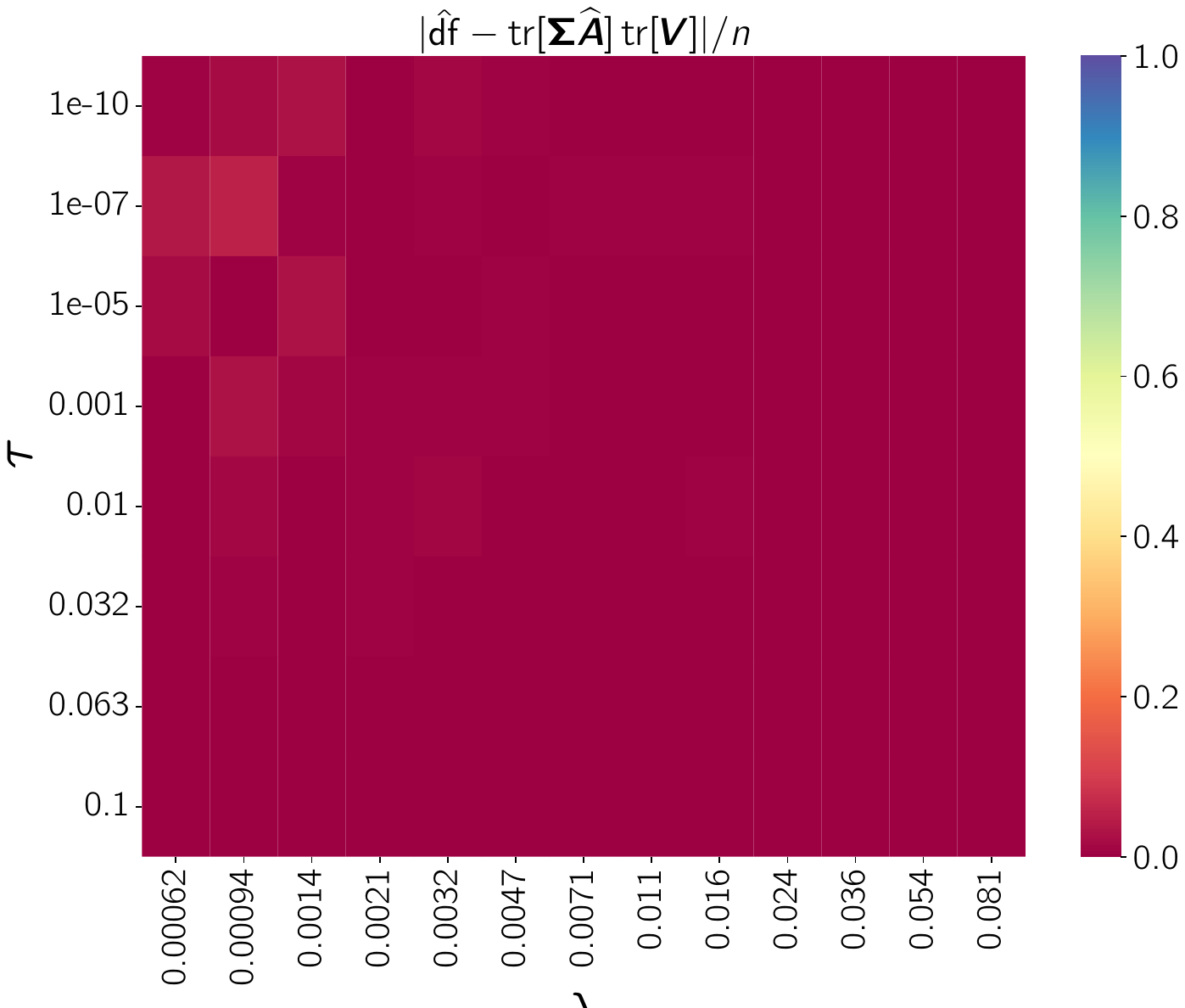}
    \includegraphics[width=0.32\textwidth]{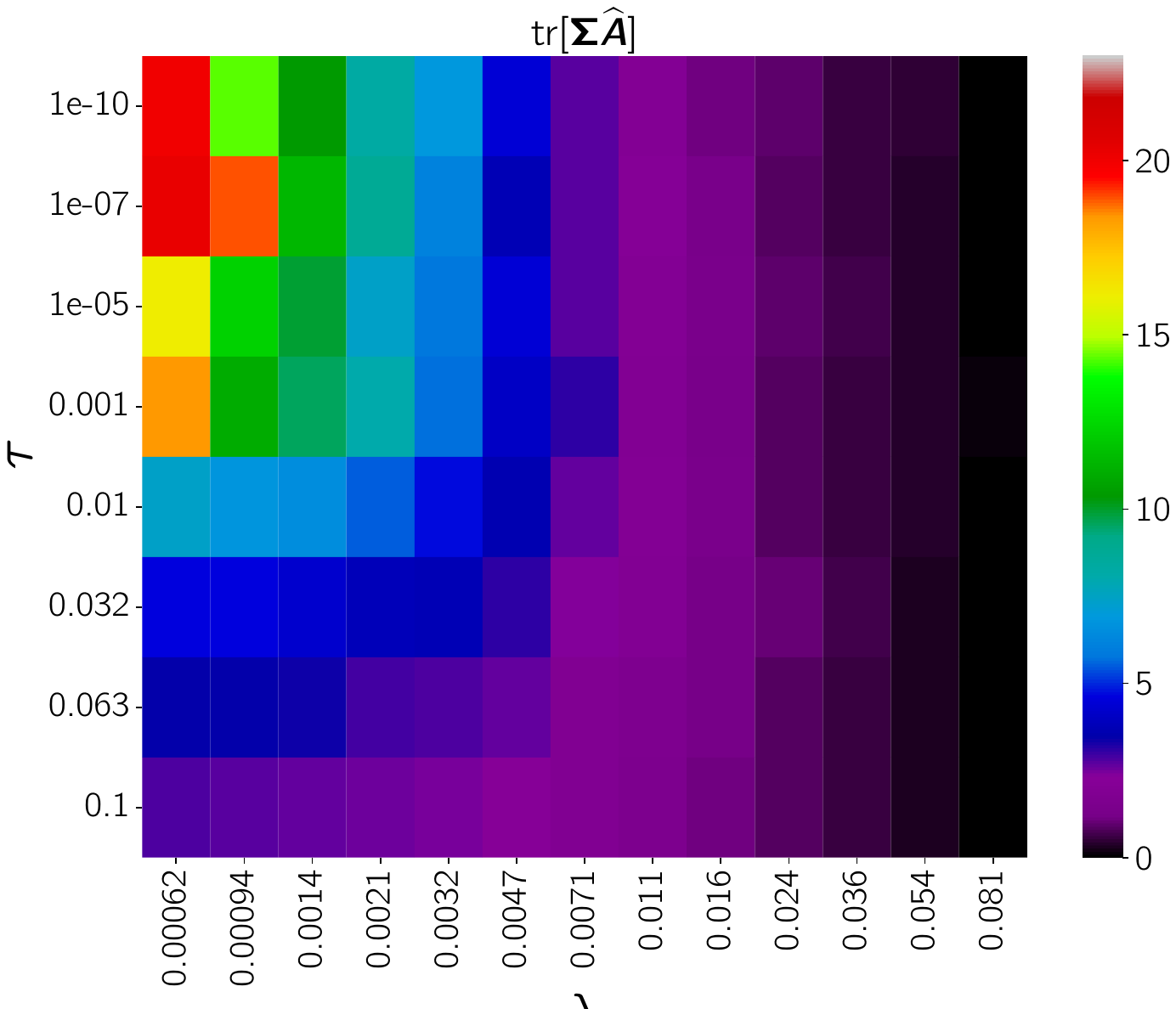}
    \includegraphics[width=0.32\textwidth]{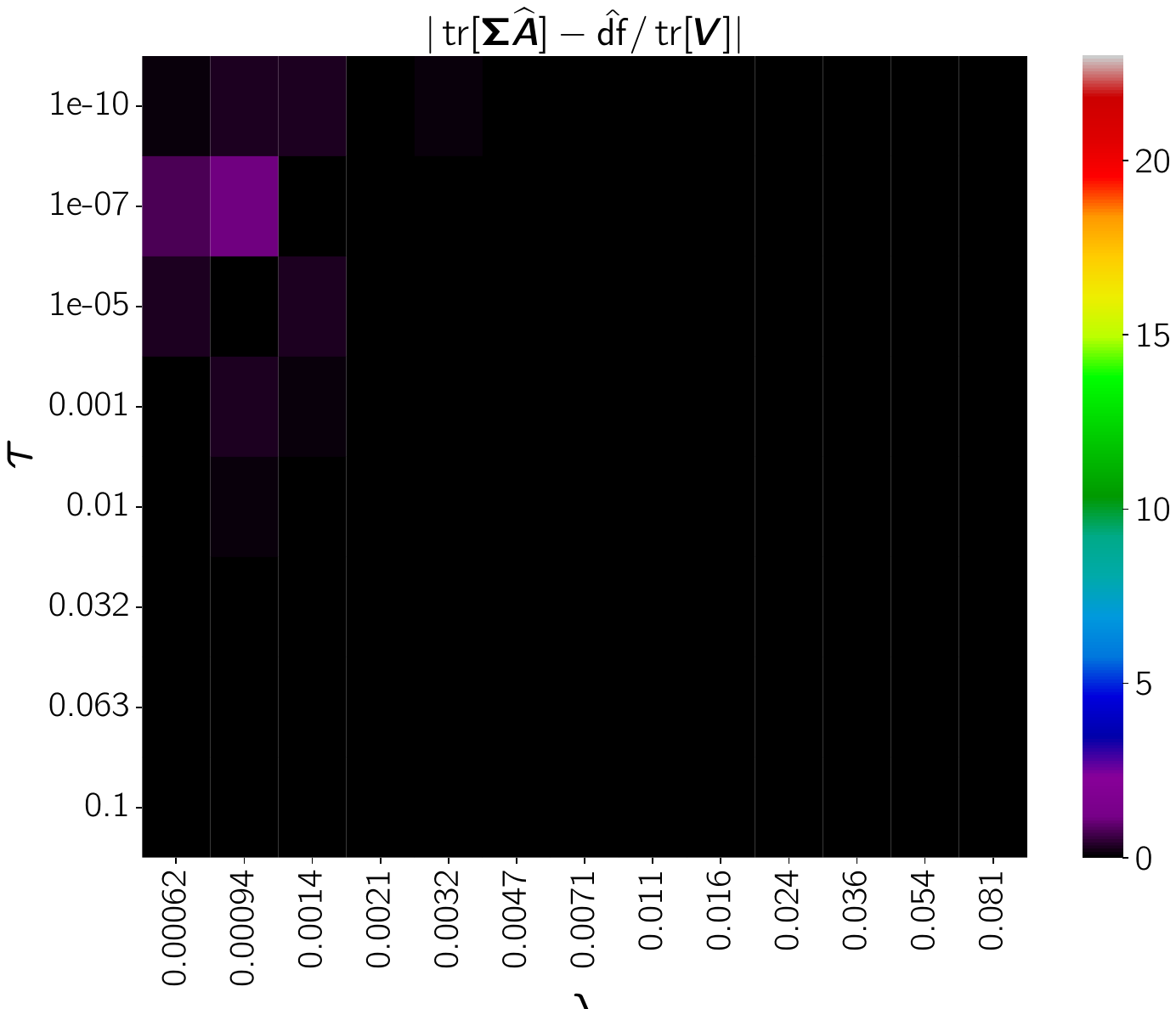}

    \includegraphics[width=0.32\textwidth]{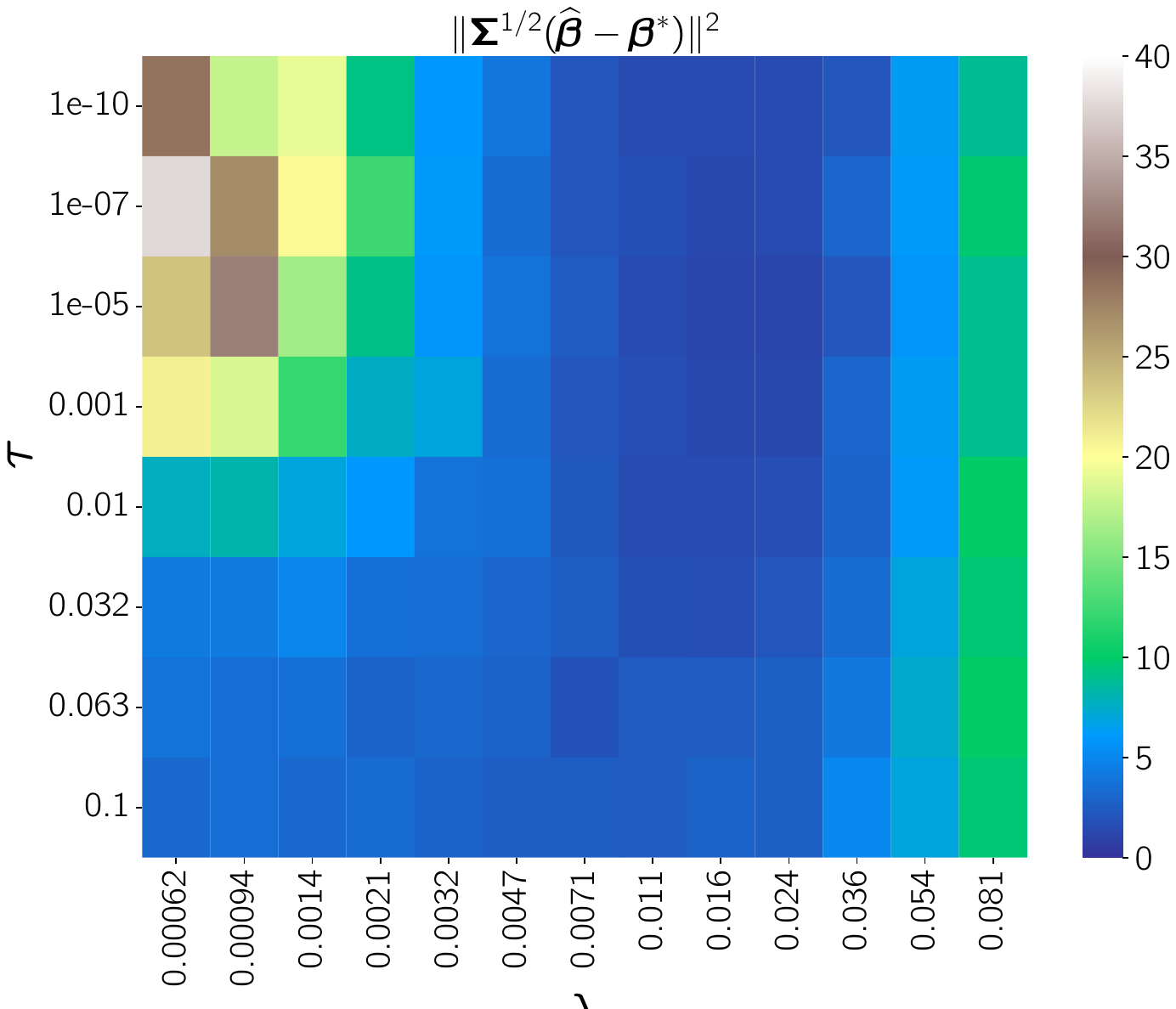}
    \includegraphics[width=0.32\textwidth]{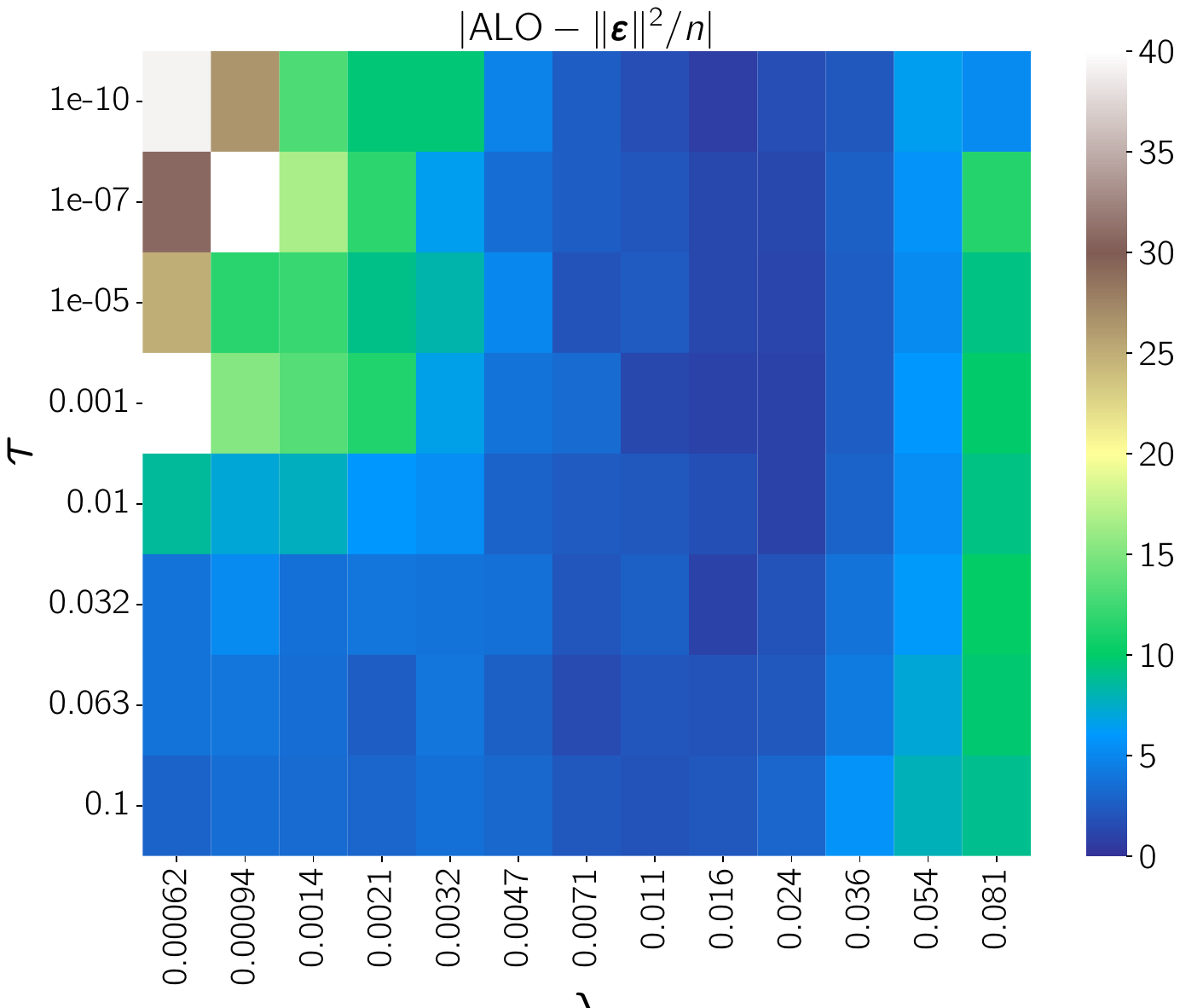}
    \includegraphics[width=0.32\textwidth]{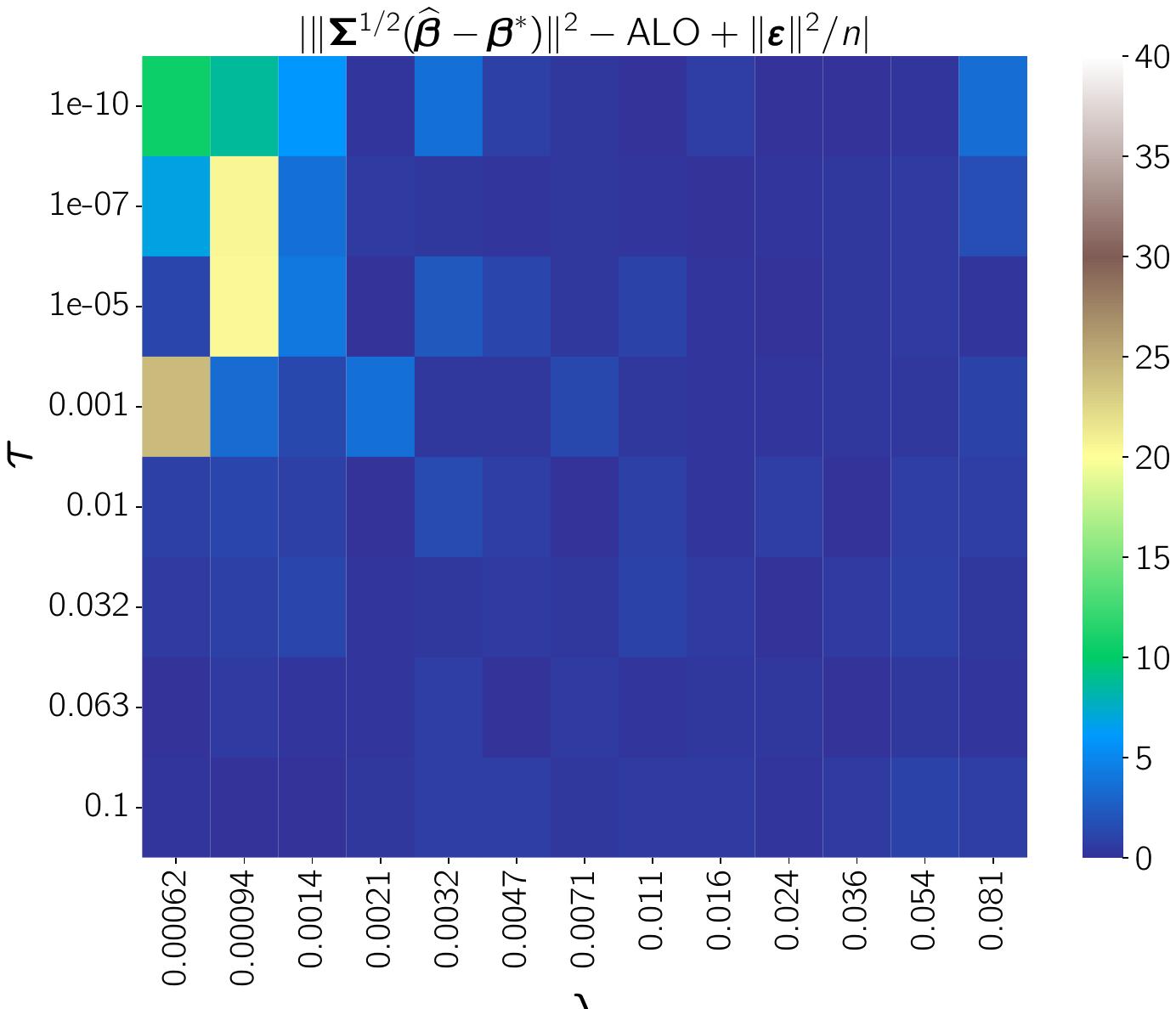}

    \includegraphics[width=0.32\textwidth]{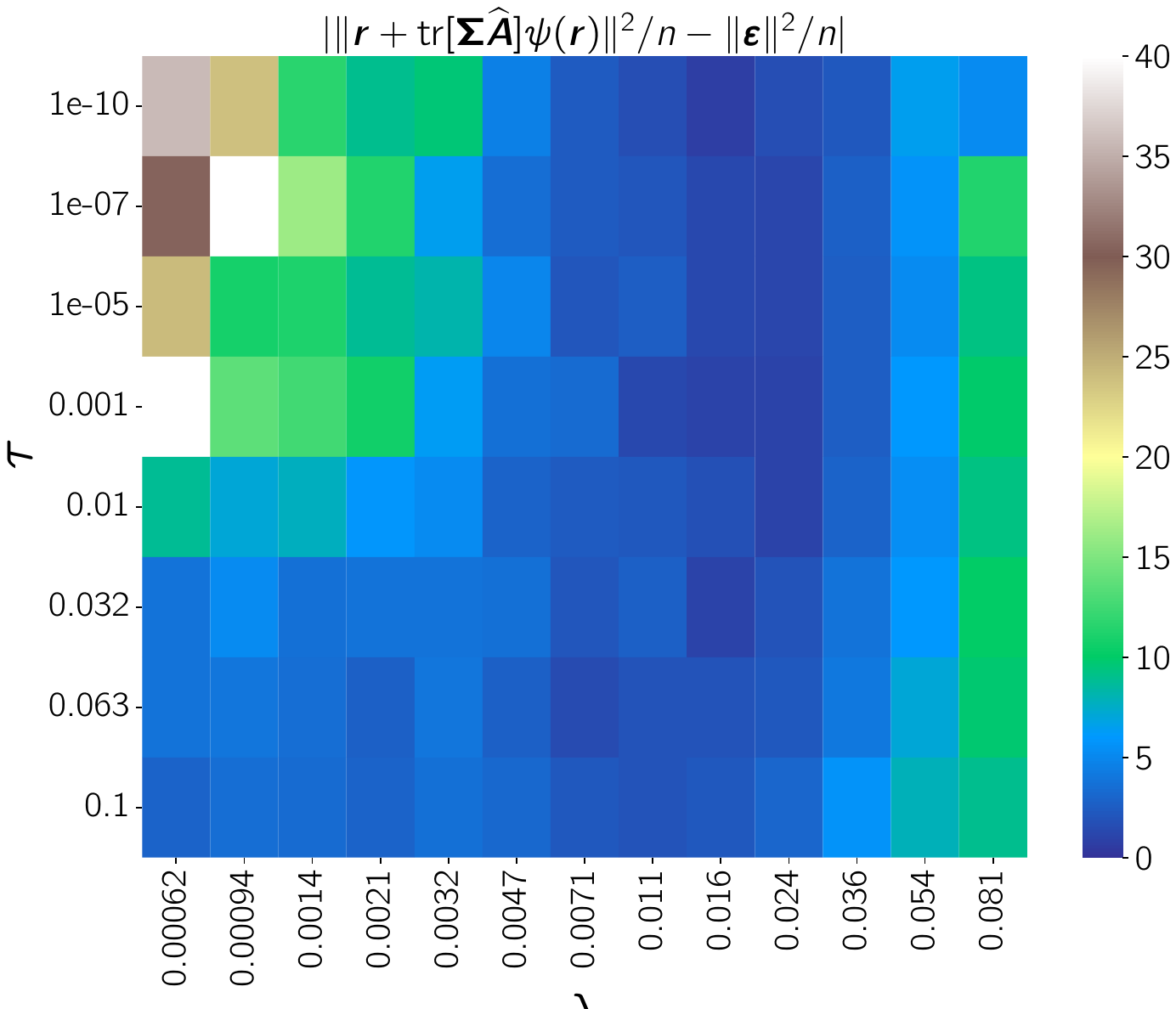}
    \includegraphics[width=0.32\textwidth]{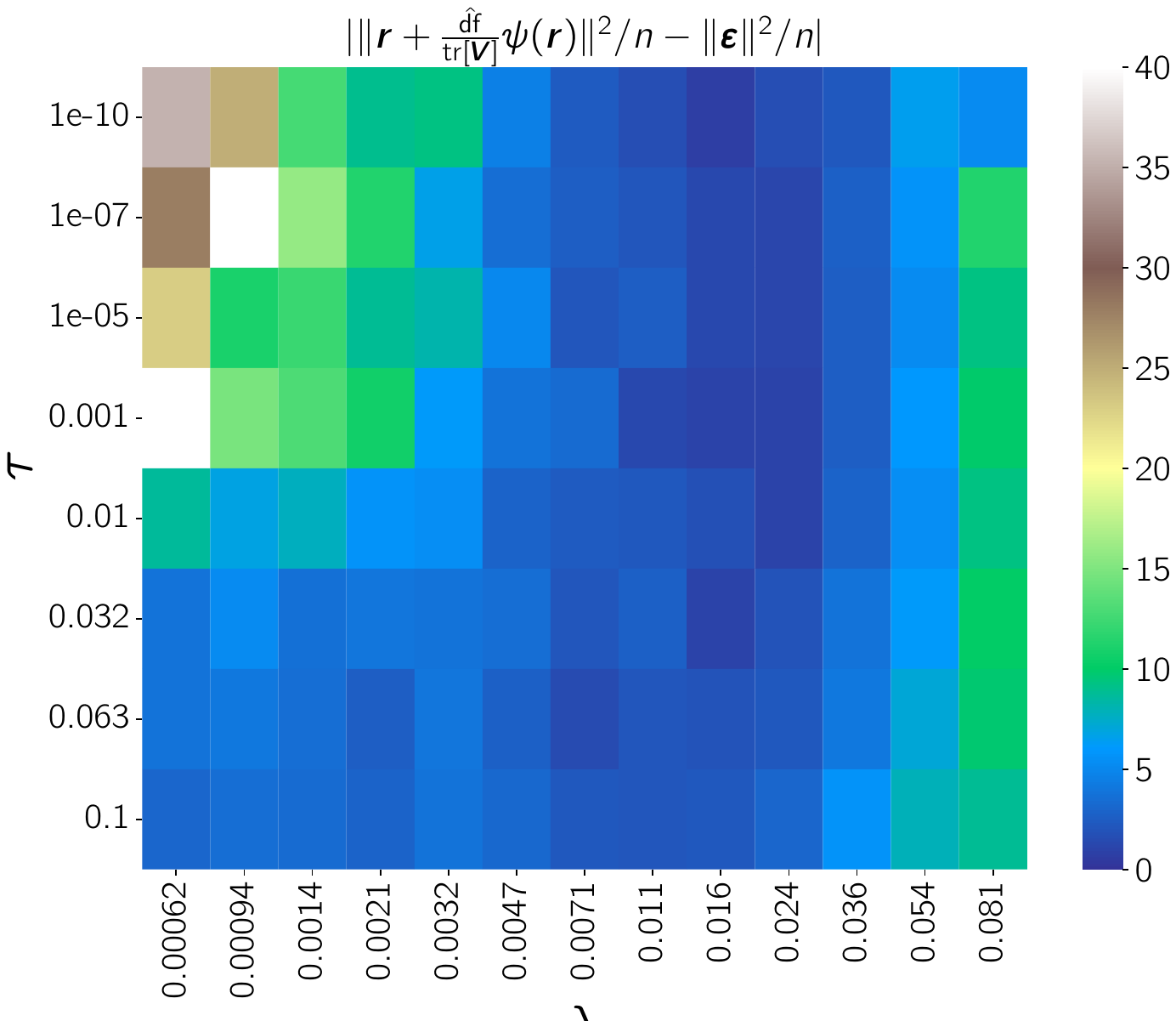}
    \includegraphics[width=0.32\textwidth]{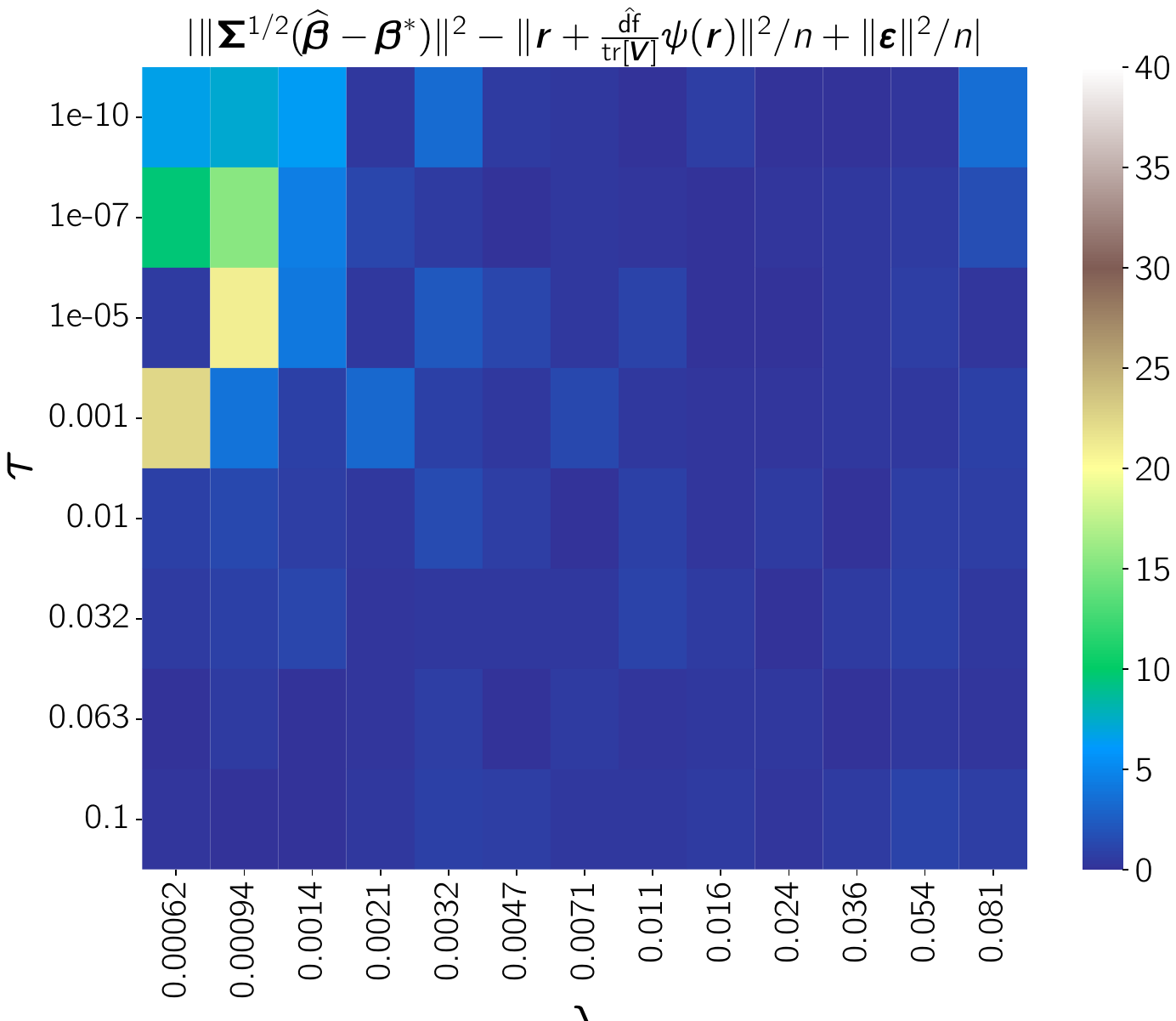}

    \caption{
        \label{fig7}
    Heatmaps for the Huber loss and Elastic-Net penalty on a grid of tuning parameters 
    with $\Lambda = 0.024 n^{1/2}$ and $(\lambda, \tau)$ where 
    $\lambda \in [0.00062, 0.081]$
    and 
    $\tau \in [10^{-10}, 0.1]$.
    Each cell over 1 repetition.
    See the simulation setup in \Cref{sec:simulations} in the paper for more details.}
\end{figure}


\newpage

\section{Additional Figures (non-Gaussian, Rademacher design)}
\label{sec:rademacher-figures}

\begin{figure}[ht]
    \centering
    \includegraphics[width=68mm]{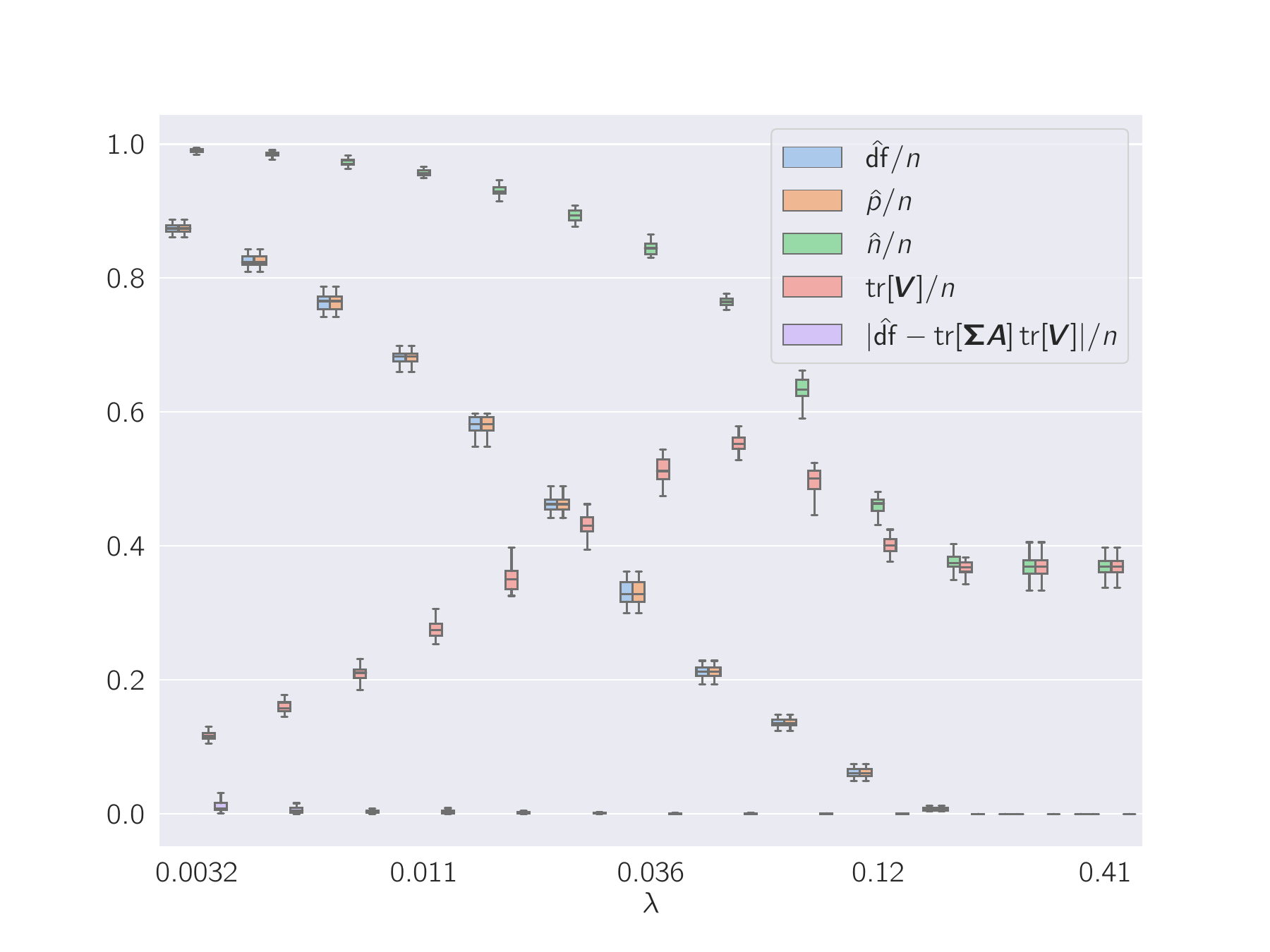}   
    \includegraphics[width=68mm]{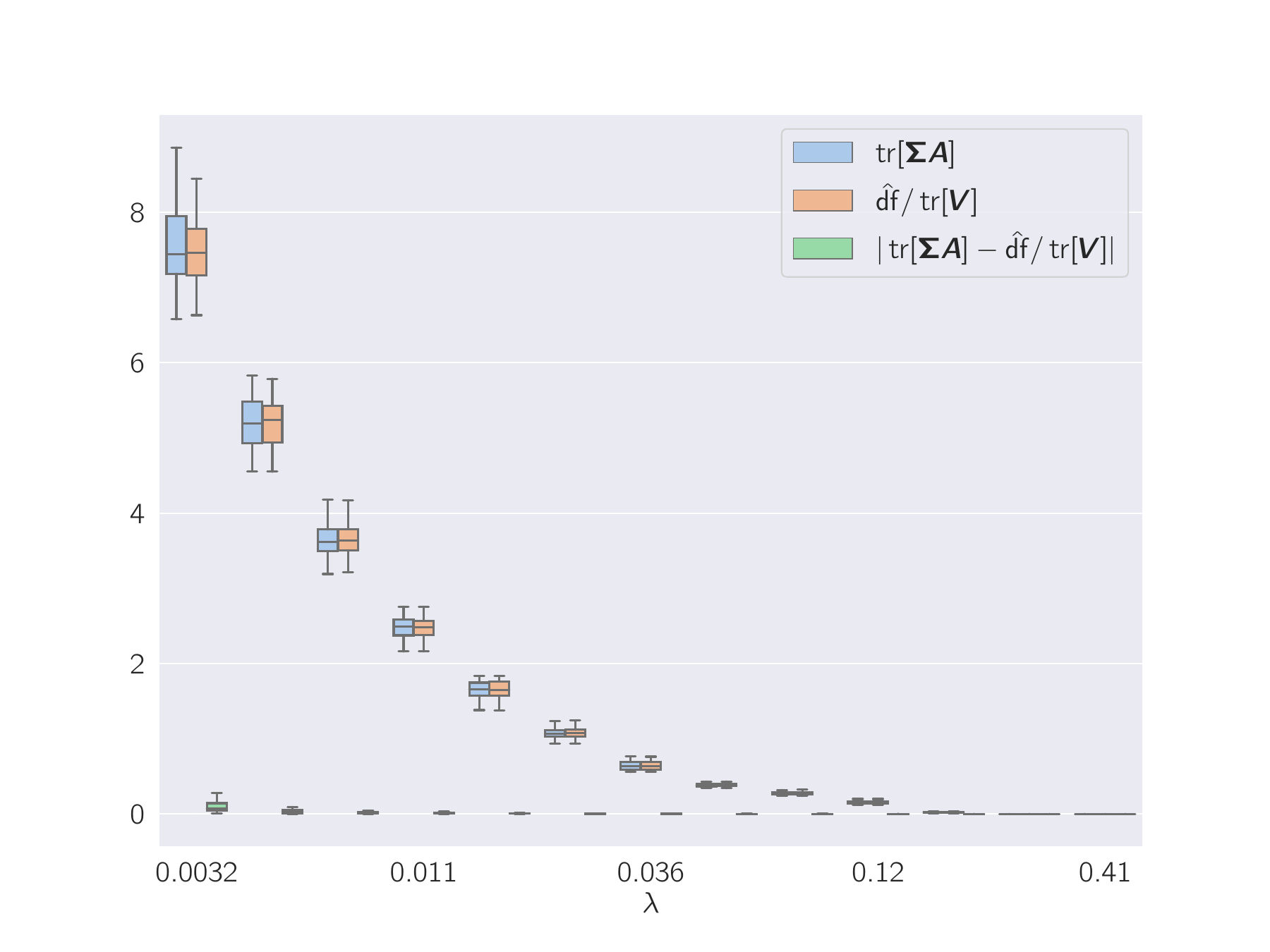}   
    \caption{
    Boxplots for $\df, \hat p, \hat n, \trace[\bV], \trace[\bSigma \hbA]$ and $|\trace [ \bSigma \hbA ] - \df / \trace [ \bV]|$ 
    in Huber Elastic-Net regression with $\tau = 10 ^{-10}$ and $\lambda \in [0.0032, 0.41].$
    The data are generated with $\bX$ having iid entries taking value $\pm 1$ each with probability 0.5 (so that $\bSigma=\bI_p$).
    Each box contains 30 data points.
    }
\end{figure}

\begin{figure}[ht]
\centering
\includegraphics[width=34mm]{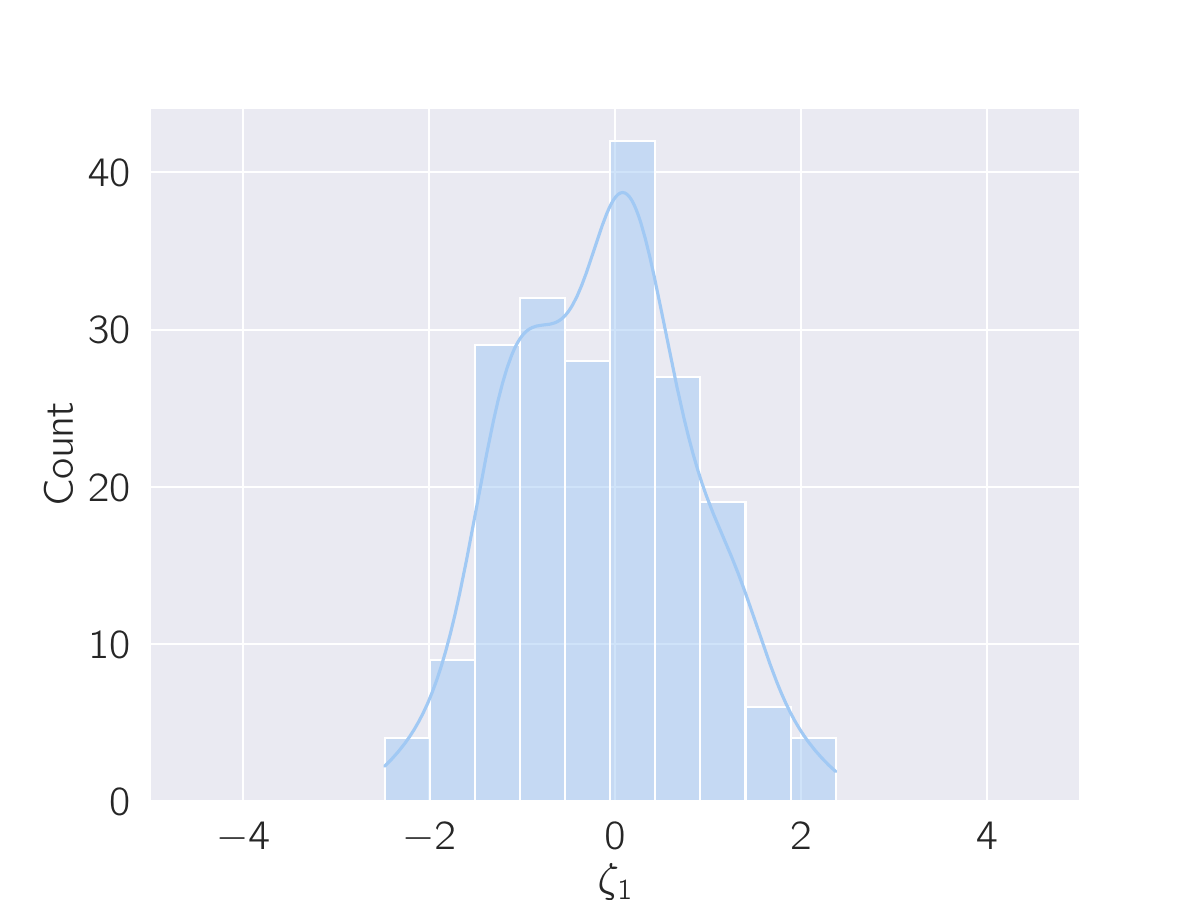}   
\includegraphics[width=34mm]{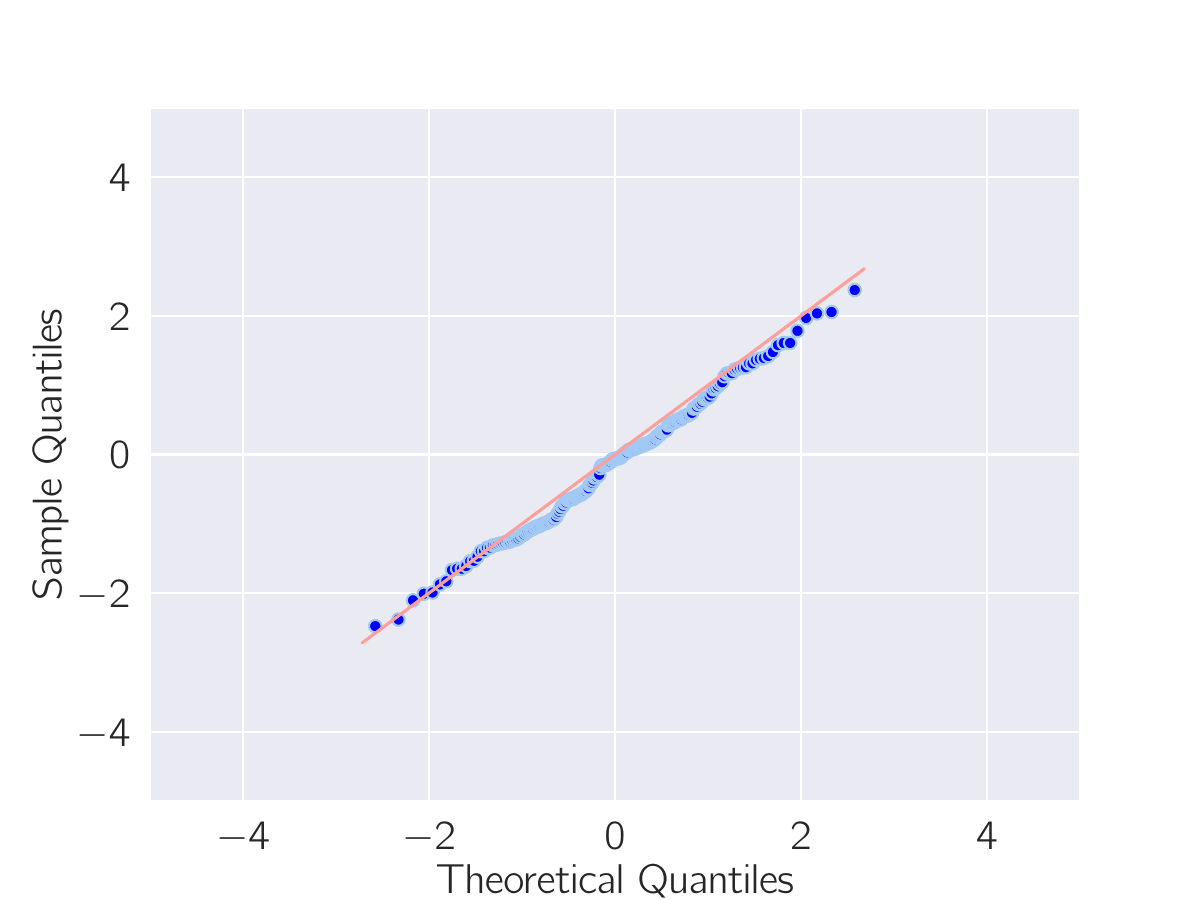}   
\includegraphics[width=34mm]{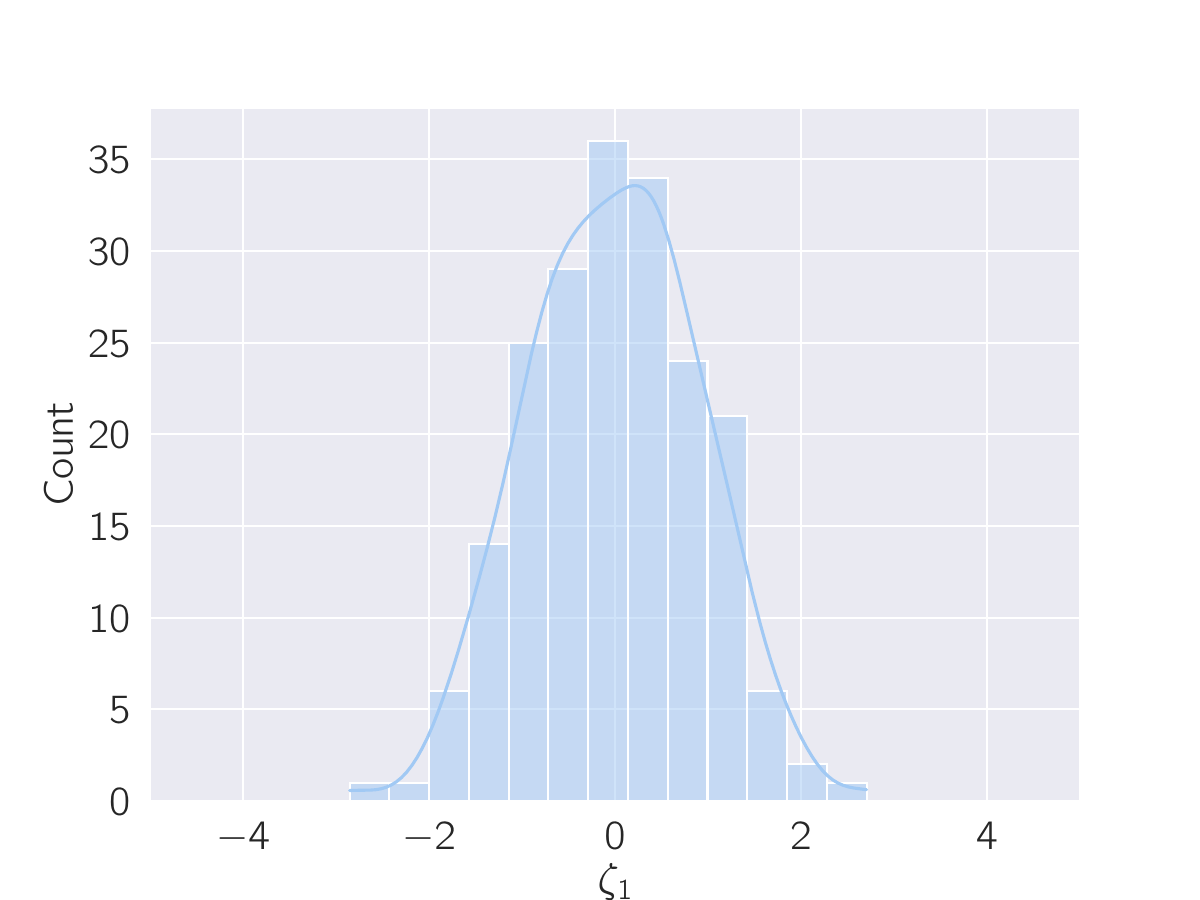}   
\includegraphics[width=34mm]{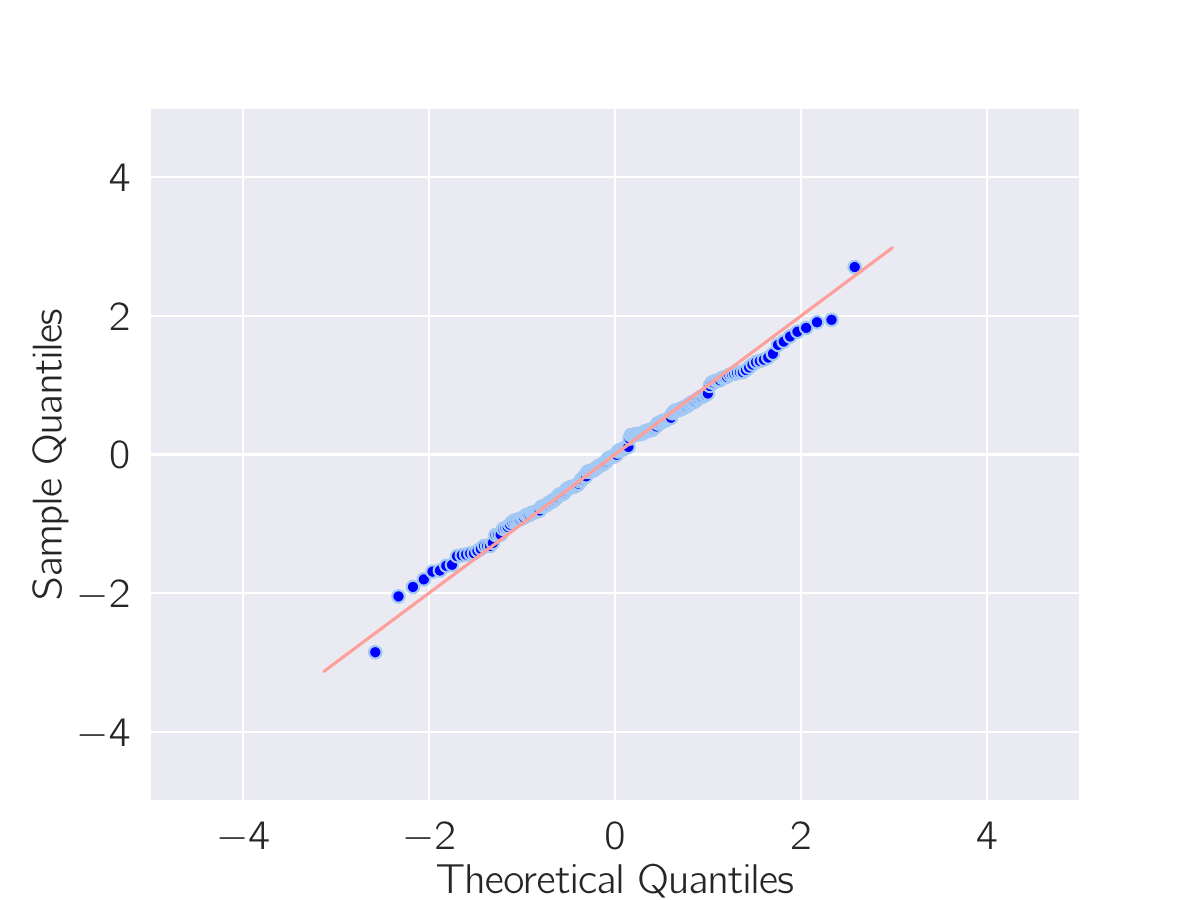}   
\includegraphics[width=34mm]{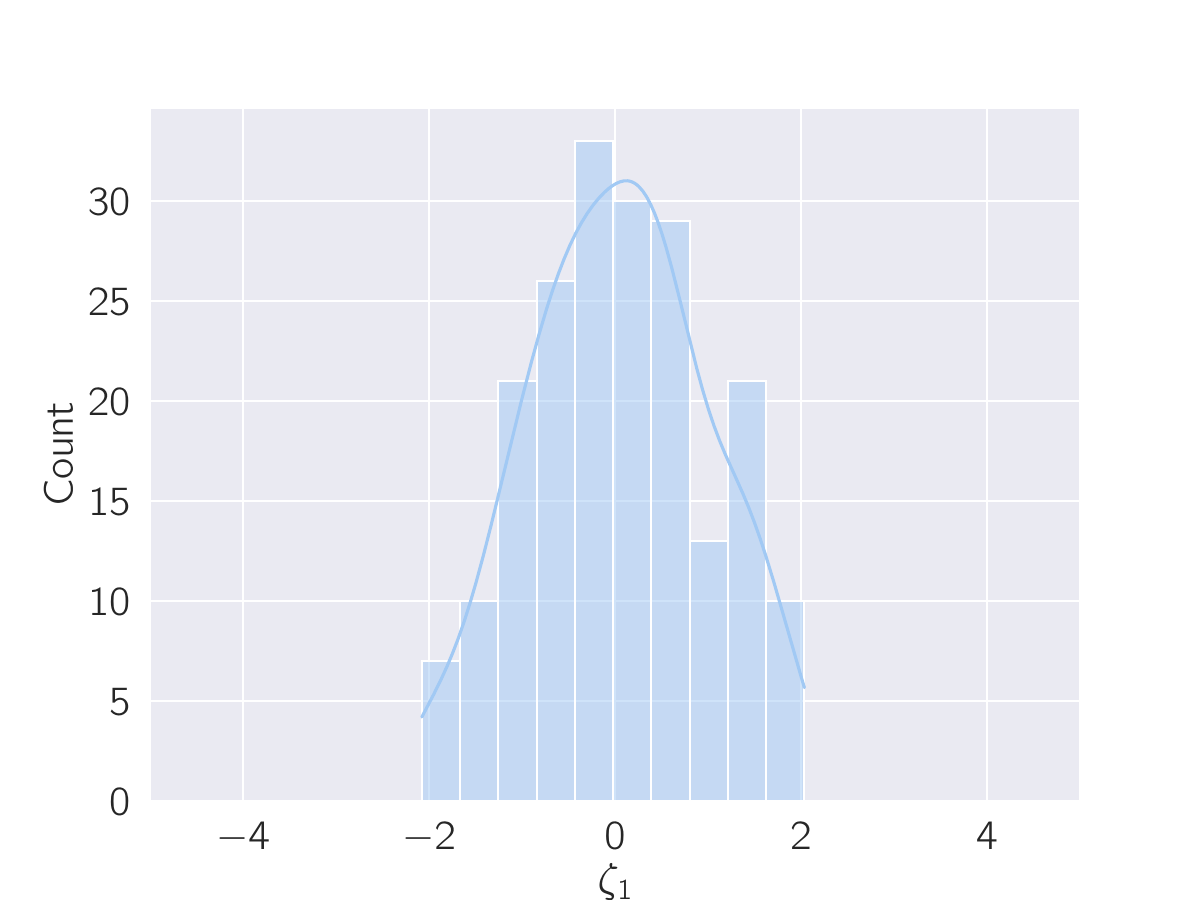}   
\includegraphics[width=34mm]{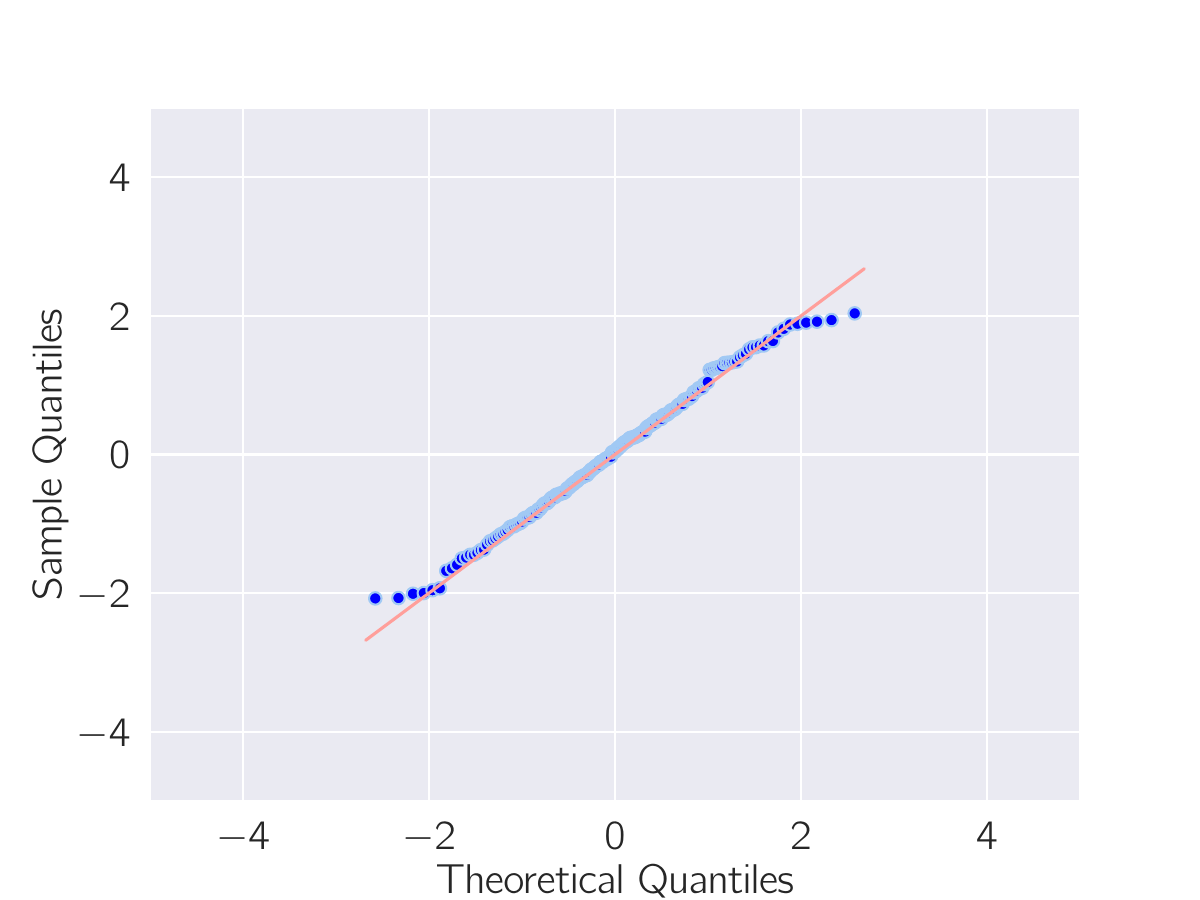}   
\includegraphics[width=34mm]{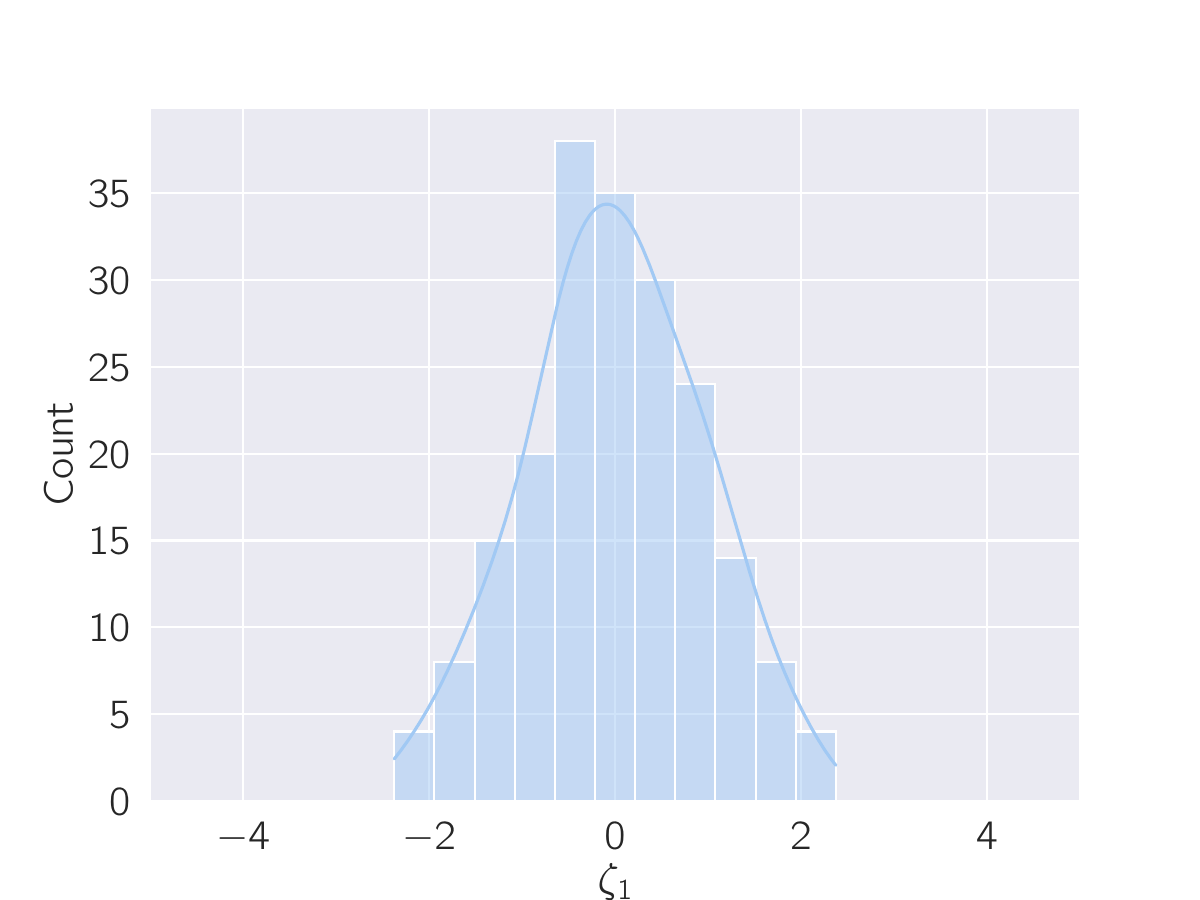}   
\includegraphics[width=34mm]{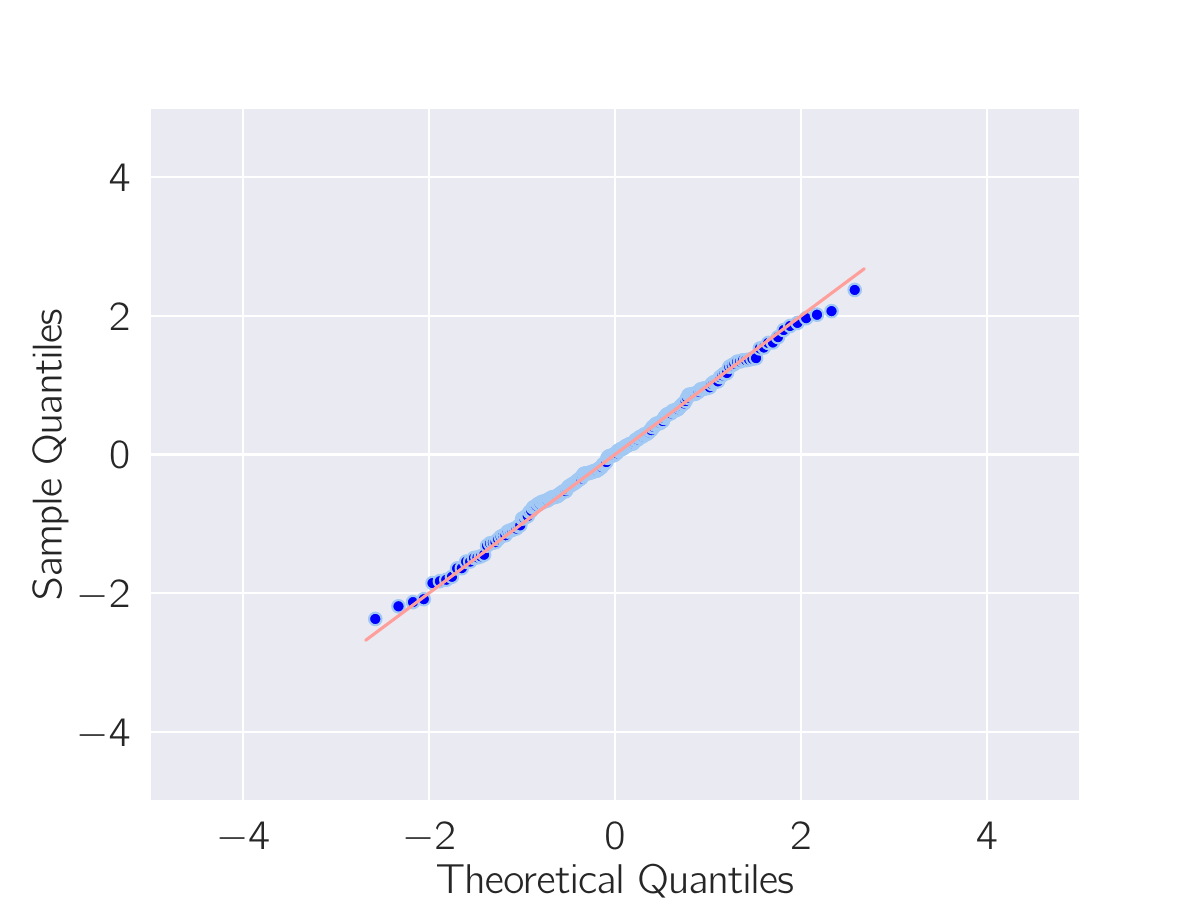}   
\caption{Histogram and QQ-plot for $\zeta_{1}$
     in \eqref{zeta_i}
under Huber Elastic-Net regression for different choices of
tuning parameters $(\lambda, \tau)$. 
Left Top: $(0.036, 10^{-10})$, 
Right Top: $(0.054,0.01)$, 
Left Bottom: $(0.036, 0.01)$, 
Right Bottom: $(0.024, 0.1)$.
Each figure contains 100 data points generated with Rademacher design matrix
(each entry has value $\pm 1$ with probability  0.5) and iid $\eps_i$ from the $t$-distribution with 2 degrees of freedom.
}
\end{figure}

\end{document}